\documentclass[12pt]{amsart}

\usepackage[table, svgnames, dvipsnames]{xcolor}
\usepackage[left=1in,top=1.2in,bottom=1in]{geometry}
\usepackage{tikz}
\usepackage{caption}
\usepackage{booktabs}
\usepackage{amssymb}

\usepackage{url}

\usepackage[pdfencoding=auto, psdextra]{hyperref}
\usepackage{cleveref}
\Crefname{figure}{Figure}{Figure}
\usepackage{subcaption}

\usepackage{graphicx} 
\usepackage{float}
\usepackage[labelfont=bf]{caption}
\newcommand{\Hthree}{\mathbb{H}^3}

\newcommand{\Sthree}{\mathbb{S}^3}

\newcommand{\knotcomp}{\Sthree \setminus K}

\newcommand{\ZZ}{\mathbb{Z}}

\newcommand{\tang}[1]{\theta_{#1}}

\newcommand\prgr{\Pi}
\newcommand\dpr{\widetilde{P}}
\newcommand\dprgr{\widetilde\Pi}
\newcommand\doh{\widetilde{O}}

\newcommand\co{\colon\thinspace}
\newcommand\be{\mathbf{e}}

\theoremstyle{plain}
\newtheorem{counter}{}[section]
\newtheorem{subcounter}{}[subsection]
\newtheorem{theorem}[counter]{Theorem}

\newtheorem{prop}[counter]{Proposition}
\newtheorem{lemma}[counter]{Lemma}
\newtheorem{corollary}[counter]{Corollary} 
\newtheorem{question}[counter]{Question}
\newtheorem{conjecture}[counter]{Conjecture}
 
\newtheorem{remark}[counter]{Remark}
\newtheorem{para}[subcounter]{}

\theoremstyle{definition}
\newtheorem{definition}[counter]{Definition} 
\newtheorem{example}[counter]{Example} 

\newtheorem{notation}[counter]{Notation}

\begin{document}

\begin{abstract} We describe four hyperbolic knot complements in $\mathbb{S}^3$, each of which covers a \emph{prism orbifold}: the quotient of $\mathbb{H}^3$ by the action of a discrete group generated by reflections in the faces of a polyhedron that has the combinatorial type of a triangular prism. The prism orbifolds are rigid-cusped and contain compact, totally geodesic hyperbolic triangle sub-orbifolds; as a result, the knot complements covering them have hidden symmetries and contain closed, embedded, totally geodesic surfaces.
\end{abstract}


\title{Knot complements decomposing into prisms}

\author[J. De{B}lois]{Jason De{B}lois\textsuperscript{$\dagger$}}

\author[A. Gharagozlou]{Arshia Gharagozlou\textsuperscript{$*$}}

\author[N. Hoffman]{Neil R Hoffman\textsuperscript{*}}

\address{\textsuperscript{$\dagger$}Department of Mathematics\\
University of Pittsburgh\\
Pittsburgh, PA}
\email{jdeblois@pitt.edu}

\address{\textsuperscript{*}  Department of Mathematics and Statistics \\
University of Minnesota Duluth\\
Duluth, MN} 
\email{neilhoff@d.umn.edu} 
\email{ghara027@d.umn.edu}

\date{\today}

\maketitle

\section{Introduction}

Three hyperbolic knot complements in $\mathbb{S}^3$ are known to cover a \emph{tetrahedral {reflection} orbifold}: the quotient of $\mathbb{H}^3$ by a discrete group generated by reflections in the four faces of a tetrahedron. For each of these knot complements, the cover to a tetrahedral orbifold is obtained from a \emph{platonic} decomposition, into isometric copies of a single non-compact, finite-volume polyhedron in the hyperbolic space $\mathbb{H}^3$ whose isometry group acts transitively on ``flags'' of the form (face, edge, ideal vertex). They are the figure-eight knot complement, which decomposes into two regular ideal tetrahedra as shown in Thurston's Notes \cite{Th_notes}, and the two ``dodecahedral knots'' {attributed to} Aitchison--Rubinstein\footnote{The dodecahedral knot with a fibered complement appears as Figure 1b in \cite{RileyComputers}. Riley refers to it there as `Thurston's knot', based on a 1981 lecture in which, according to Riley's account, Thurston labeled it as a ``totally asymmetric knot". However as shown by Aitchison--Rubinstein \cite[\S 12]{AitchRubgeodesic}, it is not in fact totally asymmetric. SnapPy \cite{SnapPy} computes the full symmetry group of its complement as $\ZZ/2\ZZ \times \ZZ/2\ZZ$.} \cite{AitchRubdodec}.

Those listed above are the only platonic knot complements in $\mathbb{S}^3$, by works of A.W.~Reid, who proved that the only arithmetic hyperbolic knot complement is that of the figure-eight \cite{reidarith}, and of the third author concerning the dodecahedral knots \cite{NeilExperiment}. 
After the tetrahedron, the next least-complicated three-dimensional polyhedron having all trivalent vertices is a triangular prism. 
We will say that a \emph{prism orbifold} is the quotient of $\mathbb{H}^3$ by the action of a discrete group generated by reflections in the five faces of such a prism. 
Our main result is:

\begin{theorem}\label{main_thm}
There exist hyperbolic knot complements in $\mathbb{S}^3$ that cover prism orbifolds.
\end{theorem}

In fact we exhibit four such knot complements, see \Cref{main_tech_thm}---or eight, if one distinguishes handedness, since they are chiral (\Cref{isoms})---and compute ancillary data including their volumes (\Cref{volumes}). \Cref{main_thm} gives counterexamples to two conjectures from the early-to-mid 1990's, on the non-existence of hyperbolic knot complements having hidden symmetries or closed, embedded, totally geodesic surfaces. A commensurability classification of the relevant prism orbifolds supplied in \Cref{commensurator} also shows that \Cref{main_tech_thm}'s manifolds are counterexamples to the later ``rigid cusp conjecture'' \cite[Conjecture 1.3]{BBCW2} on knot complements with hidden symmetries.

We discuss the conjectures mentioned above, and their relation to the present examples, immediately below in Sections \ref{subsec: hidden sym} and \ref{subsec: M-R}. \Cref{subsec: compute} describes the role of machine computation in our search for these examples.

The remainder of the paper is organized as follows. Section \ref{sec:background} gives general background on one-cusped prism orbifolds, following and expanding on work of Lakeland-Roth \cite{LakelandRoth}. In \Cref{sec: two_n_three} we zoom in on the prism orbifolds most relevant to the present work, computing their volumes to high precision and establishing minimality in their commensurability classes. Section \ref{sec:main_proof} describes the proof of \Cref{main_tech_thm}, appealing at times to machine computation; this is followed by a complete hand proof of one case in \Cref{sec:M21}, cf.~\Cref{main by hand}. \Cref{sec: mutants} establishes results on isometries, chirality, and mutation. We then discuss which prism orbifolds more generally might and cannot be covered by knot complements, in \Cref{sec:which prisms}, before concluding with further questions in \Cref{sec: ques}.

\subsection{Hidden symmetries of knot complements}\label{subsec: hidden sym} A \emph{hidden symmetry} of a space $X$ is a homeomorphism between finite-degree covers of $X$ that does not lift a self-homeomorphism of $X$. Among all hyperbolic manifolds, the arithmetic ones are characterized by having infinitely many hidden symmetries up to a natural equivalence. Non-arithmetic hyperbolic 3-manifolds having \emph{some} hidden symmetries are also not uncommon, though. For instance, any manifold that non-normally covers an orbifold has a hidden symmetry. 

However, \emph{hyperbolic knot complements in $\mathbb{S}^3$} with hidden symmetries seem exceedingly rare: until now, only the three platonic knot complements were known to have them, and many more are known not to. The reasons for this go back to a landmark paper of Neumann--Reid from the early 1990's \cite{NeumannReidArith}. Proposition 9.1 there implies that a hyperbolic knot complement in $\mathbb{S}^3$ has hidden symmetries if and only if it covers a rigid-cusped orbifold. Here we pause to recall that it follows from Margulis's Lemma that cusp cross-sections of hyperbolic $3$-orbifolds are Euclidean $2$-orbifolds (see eg.~\cite[\S 2]{DunbarMeyerhoff}). We call an orbifold \emph{rigid-cusped} if any such cusp cross-section is a \emph{rigid} Euclidean orbifold: one of the $(2,3,6)$, $(2,4,4)$, or $(3,3,3)$-triangle orbifolds, or their orientable double covers (``turnovers''). 

Proposition 9.1 of \cite{NeumannReidArith} implies that the dodecahedral knot complements have hidden symmetries, since they cover rigid-cusped tetrahedral orbifolds. Having observed this, and in light of Reid's result (recently proved at that time) that the only arithmetic knot complement is that of the figure-eight, Neumann--Reid asked if there exist knot complements with hidden symmetries other than those of the figure-eight and dodecahdral knots \cite[Q.~1, p.~307]{NeumannReidArith}. A few years later, with no new examples having emerged in the meantime, they conjectured in the Kirby problem list that in fact there are no others \cite[Problem 3.64(a)]{KirbyList}.

Since this conjecture was posed, multiple authors have contributed to a now-substantial literature showing that various classes of hyperbolic knot complements satisfy it: those of the two-bridge knots \cite{ReidWalsh} and of all ``highly twisted'' knots \cite{HMW}, for instance, and many others \cite{NeilCommClasses}, \cite{MacasiebMattman}, \cite{Millichap}, \cite{CDM}, \cite{CDHMMW}. These all rely on the criterion given in \cite[Prop.~9.1]{NeumannReidArith}. One can apply the same criterion, using rigorous computations with SnapPy within Sage \cite{sagemath}, to rule out eg.~all at-most-fifteen-crossing knots in the census besides the figure-eight. 

However since the prism orbifolds have rigid cusps (cf.~\Cref{one cusp prism}), \Cref{main_tech_thm} yields:

\newcommand\HidSymKnots{In addition to the figure-eight and dodecahedral knot complements, there are at least four more hyperbolic knot complements in $\mathbb{S}^3$ that have hidden symmetries.}
\theoremstyle{plain}
\newtheorem*{HidSymKnotsCor}{Corollary \ref{cor:HidSymKnots}}
\begin{HidSymKnotsCor}\HidSymKnots\end{HidSymKnotsCor}

In Conjecture 1.3 of their 2015 paper \cite{BBCW2} on commensurability and hidden symmetries of knot complements, Boileau-Boyer-Cebanu-Walsh conjectured that every knot complement with hidden symmetries covers an orbifold with a $(2,3,6)$-cusp. Our new examples are also counterexamples to this conjecture.

\newcommand\NotTTSKnots{There are at least four hyperbolic knot complements in $\mathbb{S}^3$ that have hidden symmetries and do not cover an orbifold with a $(2,3,6)$-cusp.}
\newtheorem*{NotTTSKnotsCor}{Corollary \ref{cor: Not236Knots}}
\begin{NotTTSKnotsCor}\NotTTSKnots\end{NotTTSKnotsCor}

This follows from \Cref{no cover}, which shows that the $(3,3,3)$-cusped prism orbifolds covered by our examples do not cover any other orbifold. That result also allows us to organize our examples of knot complements with hidden symmetries into \emph{commensurability classes}: collections of manifolds that pairwise share a finite-degree cover.

\begin{corollary}\label{cor_hidden_syms}
There are at least four distinct commensurability classes of hyperbolic 3-orbifolds that contain knot complements with hidden symmetries. 
\end{corollary}

For this result, our four new examples divide into two new commensurability classes, containing two each---see \Cref{main_tech_thm}---adding to those of the figure-eight and of the dodecahedral knot complements (which are commensurable with each other but not the figure-eight).

\subsection{The Menasco-Reid conjecture}\label{subsec: M-R} In the same collection as \cite{NeumannReidArith}, \emph{Topology '90}, Menasco--Reid published a paper in which they conjectured that no hyperbolic knot complement in $\mathbb{S}^3$ contains a closed, embedded totally geodesic surface \cite[Conjecture 1]{MenascoReid}, and established this for the complements of \emph{alternating} knots and of tunnel number one knots (which includes the two-bridge knots). This conjecture is also in the Kirby list \cite[Problem 1.76]{KirbyList}. It is now known to hold for many more families of knots, including almost alternating knots \cite{AdamsEtAlAlmostAlt}, toroidally alternating knots \cite{AdamsTorAlt}, knots of braid index $3$ \cite{LozanoPrzy} and $4$ \cite{Matsuda}, Montesinos knots \cite{Oertel}, and three-bridge and double-torus knots \cite{IchiharaOzawa}.

In the other direction, Leininger showed in \cite{Leininger_SmallCurv} that `one can get ``as close as possible'' to a counter-example', by constructing a sequence of hyperbolic knot complements in $\mathbb{S}^3$ containing embedded surfaces whose principal curvatures approach $0$. We go all the way: in \Cref{geodesic surface} below, we observe that every manifold cover of a prism orbifold contains a closed, embedded totally geodesic surface, hence obtaining the following from \Cref{main_tech_thm}.

\begin{corollary}\label{cor_geodesic}
There are at least four hyperbolic knot complements in $\mathbb{S}^3$ such that each contains a closed embedded totally geodesic surface. 
\end{corollary}

In fact, widening the lens to closed, not-necessarily-embedded, totally geodesic surfaces in hyperbolic knot complements  does not grow  the list of already-known examples much. The figure-eight \cite{MRparametrizingFuchsian} and dodecahedral knot \cite{AitchRubgeodesic} complements contain closed, immersed (not embedded) totally geodesic surfaces, but these are all that we  know of. Thus:

\begin{corollary}\label{cor: immersed geod}
    In addition to the figure-eight and dodecahedral knot complements, there are at least four more hyperbolic knot complements in $\mathbb{S}^3$ that have closed immersed totally geodesic surfaces.
\end{corollary}

To motivate our next corollary, we note that the three previously known knots in $\mathbb{S}^3$ whose complements have hidden symmetries are each alternating. But since the knot complements of \Cref{cor_geodesic} have closed, embedded totally geodesic surfaces, by \cite{MenascoReid} they are not alternating. More broadly, from the other works cited in this subsection's first paragraph we obtain:

\begin{corollary}\label{cor_non_alternating}
There exist  knots in $\mathbb{S}^3$ with complements admitting hidden symmetries that are neither alternating, almost alternating, toroidally alternating, nor Montesinos. 
\end{corollary} 

We finally mention the following consequence of the more-technical \Cref{cover mutants}.

\begin{corollary}
    There exist distinct knot complements in $\mathbb{S}^3$ that are related by mutation along a closed, embedded totally geodesic surface.
\end{corollary}

Here, \emph{mutation} refers to cutting along the surface and re-gluing the resulting pieces along their boundaries by a self-isometry. Like the well-known mutation along four-punctured spheres (cf.~\cite{Ruberman}), hyperbolic knot complements related in this way have the same volume.

\subsection{Computational Tools}\label{subsec: compute}

In \Cref{sec:main_proof} we present the main examples of this paper in such a way that their properties can in principle be extracted by human computation. We carry this out for one of our four examples in \Cref{sec:M21}, culminating in \Cref{main by hand} which implies \Cref{main_thm}. Hence this paper does contain a complete by-hand proof of \Cref{main_thm} and the existence of counterexamples to the conjectures mentioned below it. The proof of \Cref{main_tech_thm} in \Cref{sec:main_proof} appeals at points to machine computation, however, using both standard tools of computational topology (SnapPy \cite{SnapPy} and Regina \cite{regina}) and custom-coded Python scripts (available as ancillary files to the arxiv version of this article: arXiv:2507.01263). These computations are combinatorial in nature and fully rigorous.

These examples were also \emph{produced} with the aid of machine computation, and it is fair to say that we could not have produced them without it. We give a detailed description of this process in \Cref{sec:which prisms}. Here is a high-level overview of its steps.

\begin{para}[Pre-filtering]\label{pre-filter} For each member of an initial list of 54 individual prism orbifolds and 8 infinite families, we tested a sequence of necessary conditions for being covered by a knot complement. This process was expedited by custom-coded Python scripts and the use of Sage \cite{sagemath}. Twelve individual prism orbifolds, and two infinite sub-families, survived all tests.\end{para}

\begin{para}[Enumeration]\label{enumerate}
    The degree of any manifold cover of an orbifold is at least the least common multiple of its vertex group orders. The smallest of these values, among the prism orbifolds surviving pre-filtering, is $24$. For each of the three (3,3,3)-cusped prism orbifolds attaining this bound we used the \texttt{low-index} Python method coded by Culler et al \cite{lowindex}, which implements a computational group theory algorithm due to C.C.~Sims, to enumerate all subgroups of its orbifold $\pi_1$ up to index $24$, yielding $32\,425$, $29\,432$, and $306\,552$ subgroups.
\end{para}

\begin{para}[Filtering]\label{filter}
    Each subgroup from \ref{enumerate} is given in terms of a right-permutation representation, which we converted to a polyhedral decomposition of the corresponding cover (cf.~\Cref{cell decomp} below). Guided by this, we custom-coded Python scripts to test each cover for having one cusp, being a manifold, and having first homology $\mathbb{Z}$. These yielded $20$, $22$, and $12$ subgroups, respectively, corresponding to covers having all three properties.\end{para}

\begin{para}[Lens space recognition] For each subgroup remaining after \ref{filter}, we converted the corresponding cover's cell decomposition to a triangulation and used methods of Regina \cite{regina} and SnapPy \cite{SnapPy} to test for lens space fillings. These tests were successful for eight covers, belonging to four orientation-reversing-isometric pairs, which yield the four examples of \Cref{main_tech_thm} (eight, up to orientation-preserving isometry) after passing to further covers.
\end{para}

\subsection*{Acknowledgments} This work was initiated in the second author's University of Pittsburgh PhD thesis \cite{Arshia_thesis}, supervised by the first author, from which it draws substantially. The steps described  in \ref{pre-filter}, \ref{enumerate}, and \ref{filter} above were originally carried out there.

The work's strategy and execution benefited greatly from an AIM SQuaRE in which the first and third authors participated. We thank AIM (the American Institute of Mathematics) for providing a supportive and mathematically rich environment, and our fellow SQuaRE members Eric Chesebro, Michelle Chu, Priyadip Mondal, and Genevieve Walsh for numerous stimulating and helpful conversations. We thank Genevieve for her comments on a draft of this paper. We are extremely grateful to Alan Reid for his extensive help and guidance throughout our careers, and for always being willing to stir the pot. We also acknowledge the extensive efforts of Marc Culler, Nathan Dunfield, and Matthias Goerner for maintaining many of the underlying software packages that were essential to this project. 

This research was supported in part by the University of Pittsburgh Center for Research Computing and Data, RRID:SCR\_022735, through the resources provided. Specifically, this work used the H2P cluster, which is supported by NSF award number OAC-2117681. We are grateful to Kim F. Wong from the CRCD for technical assistance.

Finally, we very much appreciate the referee's thorough reading and detailed feedback, which included corrections to equation (2) and Lemma 5.5 and a number of useful suggestions for improvements on how the ancillary files of the accompanying arxiv post were presented and documented.

\section{Background: the classification and geometry of prism orbifolds}\label{sec:background}

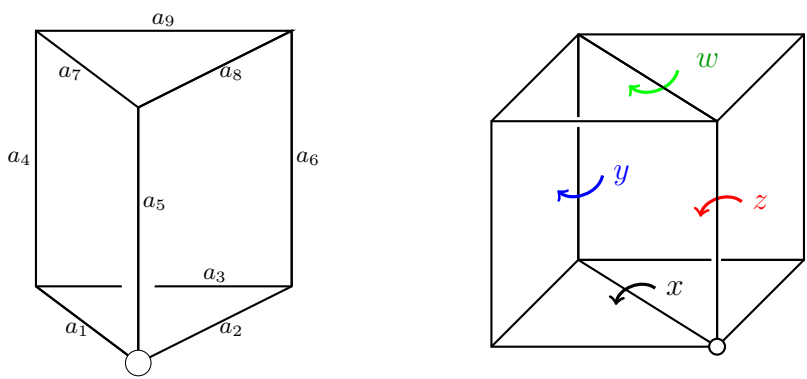
\begin{figure}
\begin{tikzpicture}

\begin{scope} [xshift=-1.5in, yshift=1.2in, scale = 1.7]
    \draw [thick] (0,-2) -- (-2,-2);
    \fill [color=white] (-1.2,-2.6) -- (-1.35,-1.9) -- (-1.05,-1.9);
    \draw [thick] (0,-2) -- (0,0) -- (-1.2,-0.6) -- (-1.2,-2.6) -- (-2,-2);
    \draw [thick] (0,0) -- (0,-2) -- (-1.2,-2.6) -- (-1.2,-0.6) -- (0,0) -- (-2,0) -- (-2,-2) -- (-1.2,-2.6);
    \draw [thick] (-1.2,-0.6) -- (-2,0);
    \fill [color=white] (-1.2,-2.6) circle [radius=0.1];
    \draw (-1.2,-2.6) circle [radius=0.1];

    \node [right] at (-0.05,-1) {\scriptsize $a_6$}; 
    \node [below right] at (-0.65,-2.2) {\scriptsize $a_2$}; 
    \node [below left] at (-1.5,-2.2) {\scriptsize $a_1$}; 
    \node [above] at (-0.6,-2.05) {\scriptsize $a_3$}; 
    \node [right] at (-1.25,-1.35) {\scriptsize $a_5$}; 
    \node [left] at (-1.95,-1) {\scriptsize $a_4$}; 
    \node [below right] at (-0.65,-0.2) {\scriptsize $a_8$}; 
    \node [below left] at (-1.55,-0.2) {\scriptsize $a_7$}; 
    \node [above] at (-1,-0.05) {\scriptsize $a_9$}; 
\end{scope}

\begin{scope}
    \coordinate (A) at (0,0,0);      
    \coordinate (B) at (3,0,0);      
    \coordinate (C) at (3,3,0);      
    \coordinate (D) at (0,3,0);      
    \coordinate (E) at (0,0,3);      
    \coordinate (F) at (3,0,3);      
    \coordinate (G) at (3,3,3);      
    \coordinate (H) at (0,3,3);      

    \draw[thick] (A) -- (B) -- (C) -- (D) -- cycle;  

    \draw[thick] (A) -- (0,1.77,0);  
    \draw[white, line width=3pt] (0,1.77,0) -- (0,1.9,0);
    \draw[thick] (0,1.9,0) -- (D); 

    \draw[thick] (A) -- (1.77,0,0); 
    \draw[white, line width=3pt] (1.77,0,0) -- (1.9,0,0); 
    \draw[thick] (1.9,0,0) -- (B);

    \draw[thick] (E) -- (F) -- (G) -- (H) -- cycle;  

    \draw[thick] (A) -- (E); 
    \draw[thick] (B) -- (F); 
    \draw[thick] (C) -- (G); 
    \draw[thick] (D) -- (H); 
    \draw[thick] (D) -- (G);
    \draw[thick] (A) -- (F);

    \draw[->, very thick, black, bend right=60] (1.6,0.2,1.5) to (0.5,-0.6,0);

    \draw[->, very thick, blue, bend left=60] (0.4,1.2,0.2) to (-0.2,1,0.2);

    \draw[->, very thick, red, bend right=60] (2.25,0.85,0.2) to (1.7,0.65,0.2);

    \draw[->, very thick, green, bend left=60] (1.4,2.6,0.2) to (0.75,2.4,0.2);

    \draw[thick] (D) -- (G);
    \draw[thick] (F) -- (G);
    
    \node[blue, right] at (0.4,1.2,0.2) { $y$};    
    \node[red, right] at (2.25,0.85,0.2) { $z$};   
    \node[green!60!black, above right] at (1.5,2.5,0.2) { $w$};   
    \node[black, right] at (1.6,0.2,1.5) {$x$};   


    
    \fill[white] (F) circle (3pt);
    \draw[thick] (F) circle (3pt);
\end{scope}

\end{tikzpicture}
\caption{\label{fig:piOne} A hyperbolic prism (left), and its double across a face, with group generators for $\dprgr(\mathbf{e})$ from the presentation (\ref{able}) shown.}
\end{figure}
In this paper, \emph{prism orbifold} will always refer to the quotient of $\mathbb{H}^3$ by the group generated by reflections in the sides of a non-compact, finite-volume polyhedron with the combinatorial type of a triangular prism. By Andreev's theorem \cite{Andreev}, such a polyhedron is determined up to isometry by its combinatorial type and the collection of its dihedral angles. Here, a nine-tuple $\mathbf{e} = (a_1,a_2,\hdots,a_9)$ of integers greater than 1 specifies a hyperbolic prism as pictured on the left in Figure \ref{fig:piOne}: for each $i$, the prism has dihedral angle $\tang{i}=\pi/a_i$ at its edge corresponding to the one labeled $a_i$ in the Figure. We will refer to the hyperbolic prism determined by $\be$ as $P(\be)$, the group generated by reflections in its faces as $\prgr(\be)$, and $O(\be)$ as the quotient of $\Hthree$ by $\prgr(\be)$.

\begin{remark}\label{one cusp prism}
    For each nine-tuple $\mathbf{e}$ considered in this paper, we require the following:\begin{itemize}
        \item The triple $(a_1,a_2,a_5)$ must equal one of $(2,3,6)$, $(3,3,3)$, or $(2,4,4)$ up to permutation of entries, so that the ``vertex'' shared by the edges labeled $a_1$, $a_2$, and $a_5$ (circled in the Figure) is ideal in the prism's geometric realization. 
        \item For each other vertex, shared by edges labeled $a_i$, $a_j$ and $a_k$, we require that $\frac{1}{a_i} + \frac{1}{a_j} + \frac{1}{a_k} > 1$, so that the corresponding vertex of the geometric realization belongs to $\mathbb{H}^3$.
    \end{itemize}
    These ensure that each prism orbifold considered here has exactly one cusp, which is rigid.
\end{remark}
 As the ultimate goal of this paper is to discuss examples of prism orbifolds covered by knot complements, the one-cusped case is the only one relevant for this investigation. 

The unique orientable double cover of the prism orbifold specified by a tuple $\be$ as above is the quotient $\doh(\be)$ of $\mathbb{H}^3$ by a group $\dprgr(\be)$ generated by rotations around the sides of one of the prism $P(\be)$'s quadrilateral faces. Precisely, let $\bar{P}(\be)$ be the image of $P(\be)$ under reflection across the plane containing the face with edges labeled $a_1$, $a_4$, $a_5$, and $a_7$, let $\dpr(\be) = P(\be)\cup\bar{P}(\be)$, and let $x$, $y$, $z$, and $w$ be the respective rotations fixing these four edges and by angles of $2\pi/a_i$ for the corresponding $i$, each taking a face of $P(\be)$ to one of $\bar{P}(\be)$. They thus act as face-pairings of $\dpr(\be)$, as pictured on the right side of Figure \ref{fig:piOne}. By a standard application of the Poincar\'e polyhedron theorem, the group $\dprgr(\be)$ that they generate is discrete, with $\dpr(\be)$ as a fundamental domain and the following presentation:\begin{equation}\label{able}
\dprgr(\be) \\
\cong\langle x, y,z, w\, |\, 
x^{a_1}, y^{a_4}, 
z^{a_5}, w^{a_7}, 
(y^{-1}x)^{a_3}, 
(z^{-1}x)^{a_2},
(z^{-1}y)^{a_6}, 
(y^{-1}w)^{a_9}, 
(z^{-1}w)^{a_8} \rangle
\end{equation}

(The first complete proof of the three-dimensional Poincar\'e polyhedron theorem is generally attributed to Maskit \cite{Maskit}. For a more recent reference see eg.~\cite[Theorem 13.5]{Ratcliffe}.) Each relation above featuring a product of two generators is associated to an edge of the doubled prism that does not belong to the face shared by the prism and its mirror image. Such an edge belongs to an edge cycle with one other. For example, the rightmost vertical edge of the doubled prism in Figure \ref{fig:piOne} is carried by each of the generators $y$ and $z$ to the leftmost, and hence fixed by $z^{-1}y$; it determines the relation $(z^{-1}y)^{a_6}=1$.

The prism orbifolds satisfying the requirements of \Cref{one cusp prism} were classified up to isometry by Lakeland-Roth in \cite{LakelandRoth}, as we recall in \Cref{subsec: classn}. We describe the $(3,3,3)$-cusped case of an explicit embedding of $P(\be)$ in $\mathbb{H}^3$, also from \cite{LakelandRoth}, in \Cref{subsec: embedding333}, then make some observations about a triangle suborbifold of $O(\be) = \mathbb{H}^3/\prgr(\be)$ in \Cref{tot geod}.

\begin{table}
\centering

\resizebox{\textwidth}{!}{%
\begin{tabular}{lrrrrlrrrrr}
\toprule
\char"0023 & $a_1$ &  $a_2$ &  $a_3$ &  $a_4$ &  $a_5$ &  $a_6$ &  $a_7$ &  $a_8$ &  $a_9$ \\
\midrule
$O^{236}_1$  & 2 & 3 & 3 & 4 & 6 & 2 & 2 & 2 & 2 \\
$O^{236}_2$  & 2 & 3 & 3 & 4 & 6 & 2 & 2 & 2 & 3 \\
$O^{236}_3$  & 2 & 3 & 3 & 5 & 6 & 2 & 2 & 2 & 2 \\
$O^{236}_4$  & 2 & 3 & 3 & 5 & 6 & 2 & 2 & 2 & 3 \\
$O^{236}_{5,n}$  & 2 & 6 & 2 & n & 3 & 2 & 2 & 2 & 2 \\
$O^{236}_{6,n}$  & 2 & 6 & 2 & n & 3 & 2 & 2 & 3 & 2 \\
$O^{236}_{7,n}$  & 2 & 6 & 2 & n & 3 & 2 & 2 & 4 & 2 \\
$O^{236}_{8,n}$  & 2 & 6 & 2 & n & 3 & 2 & 2 & 5 & 2 \\
$O^{236}_9$  & 2 & 3 & 2 & 4 & 6 & 2 & 2 & 2 & 2 \\
$O^{236}_{10}$ & 2 & 3 & 2 & 4 & 6 & 2 & 2 & 2 & 3 \\
$O^{236}_{11}$ & 2 & 3 & 2 & 5 & 6 & 2 & 2 & 2 & 2 \\
$O^{236}_{12}$ & 2 & 3 & 2 & 5 & 6 & 2 & 2 & 2 & 3 \\
$O^{236}_{13,n}$ & 2 & 3 & 2 & n & 6 & 2 & 2 & 2 & 2 \\
$O^{236}_{14}$ & 2 & 3 & 2 & 3 & 6 & 3 & 2 & 2 & 2 \\
$O^{236}_{15}$ & 2 & 3 & 2 & 3 & 6 & 3 & 2 & 2 & 3 \\
$O^{236}_{16}$ & 2 & 3 & 2 & 3 & 6 & 3 & 2 & 2 & 4 \\
$O^{236}_{17}$ & 2 & 3 & 2 & 3 & 6 & 3 & 2 & 2 & 5 \\
$O^{236}_{18}$ & 2 & 3 & 2 & 4 & 6 & 3 & 2 & 2 & 2 \\
$O^{236}_{19}$ & 2 & 3 & 2 & 4 & 6 & 3 & 2 & 2 & 3 \\
$O^{236}_{20}$ & 2 & 3 & 2 & 5 & 6 & 3 & 2 & 2 & 2 \\
\bottomrule
\end{tabular}
\begin{tabular}{lrrrrlrrrrr}
\toprule
\char"0023 & $a_1$ &  $a_2$ &  $a_3$ &  $a_4$ &  $a_5$ &  $a_6$ &  $a_7$ &  $a_8$ &  $a_9$ \\
\midrule
$O^{236}_{21}$ & 2 & 3 & 2 & 5 & 6 & 3 & 2 & 2 & 3 \\
$O^{236}_{22,n}$ & 2 & 3 & 2 & n & 6 & 3 & 2 & 2 & 2 \\
$O^{236}_{23}$ & 2 & 3 & 2 & 2 & 6 & 4 & 2 & 2 & 2 \\
$O^{236}_{24}$ & 2 & 3 & 2 & 2 & 6 & 4 & 2 & 2 & 3 \\
$O^{236}_{25}$ & 2 & 3 & 2 & 3 & 6 & 4 & 2 & 2 & 2 \\
$O^{236}_{26}$ & 2 & 3 & 2 & 3 & 6 & 4 & 2 & 2 & 3 \\
$O^{236}_{27}$ & 2 & 3 & 2 & 4 & 6 & 4 & 2 & 2 & 2 \\
$O^{236}_{28}$ & 2 & 3 & 2 & 4 & 6 & 4 & 2 & 2 & 3 \\
$O^{236}_{29}$ & 2 & 3 & 2 & 5 & 6 & 4 & 2 & 2 & 2 \\
$O^{236}_{30}$ & 2 & 3 & 2 & 5 & 6 & 4 & 2 & 2 & 3 \\
$O^{236}_{31,n}$ & 2 & 3 & 2 & n & 6 & 4 & 2 & 2 & 2 \\
$O^{236}_{32}$ & 2 & 3 & 2 & 2 & 6 & 5 & 2 & 2 & 2 \\
$O^{236}_{33}$ & 2 & 3 & 2 & 2 & 6 & 5 & 2 & 2 & 3 \\
$O^{236}_{34}$ & 2 & 3 & 2 & 3 & 6 & 5 & 2 & 2 & 2 \\
$O^{236}_{35}$ & 2 & 3 & 2 & 3 & 6 & 5 & 2 & 2 & 3 \\
$O^{236}_{36}$ & 2 & 3 & 2 & 4 & 6 & 5 & 2 & 2 & 2 \\
$O^{236}_{37}$ & 2 & 3 & 2 & 4 & 6 & 5 & 2 & 2 & 3 \\
$O^{236}_{38}$ & 2 & 3 & 2 & 5 & 6 & 5 & 2 & 2 & 2 \\
$O^{236}_{39}$ & 2 & 3 & 2 & 5 & 6 & 5 & 2 & 2 & 3 \\
$O^{236}_{40,n}$ & 2 & 3 & 2 & n & 6 & 5 & 2 & 2 & 2 \\
\bottomrule
\end{tabular}
}
\caption{Prism Orbifolds with (2,3,6) cusp. For $O^{236}_{i,n}$ with $i \in \{5,6,7,8\}$, we enforce $n\geq 7$ in order to ensure that the corresponding orbifold is hyperbolic. For $O^{236}_{i,n}$ with $i \in \{13,22,31,40 \}$, $n\geq 6$  is required for hyperbolicity (see \cite{LakelandRoth} for more details). }
\label{table236}
\end{table}

\subsection{The classification of prism orbifolds}\label{subsec: classn}

In \cite{LakelandRoth}, Lakeland-Roth checked the necessary and sufficient conditions of Andreev's theorem for hyperbolic prisms satisfying the conditions of \Cref{one cusp prism} and gave a complete classification up to isometry for those with each cusp geometry, in terms of nine-tuples $(a_1,\hdots,a_9)$ as described above. Tables \ref{table236} and \ref{table333}  reproduce the results of Sections 5.1 and 5.3 of \cite{LakelandRoth}, respectively, in each case with the addition of a column recording a name for each tuple. These names have the form ``$O^{236}_k$'' for the members of Table \ref{table236}, which all have $(a_1,a_2,a_5) = (2,3,6)$ (up to permutation of entries), or ``$O^{236}_{k,n}$'' for the table rows representing infinite families that depend on a natural number parameter $n$. The row names for Table \ref{table333} all have the form ``$O^{333}_k$'', since here $(a_1,a_2,a_5) = (3,3,3)$ in all cases. There are no infinite families in the $(3,3,3)$ case. 

By the prior work \cite{NeilOrbiCusps} of the third author, the prism orbifolds with $(2,4,4)$-cusp geometry are not relevant to the present work: none is covered by a knot complement in $S^3$.

In the rest of the paper, we will refer to prism orbifolds by the names given in Tables \ref{table236} and \ref{table333}. We also establish the following further notational conventions.

\begin{notation}\label{table names} For a row of Table \ref{table236} labeled ``$O^{236}_k$'', we will refer by ``$P^{236}_k$'' to the hyperbolic prism $P(\be)$ determined by the corresponding nine-tuple $\be$, and by ``$\prgr^{236}_k$'' and ``$\dprgr^{236}_k$'' to the corresponding reflection and rotation groups $\prgr(\be)$ and $\dprgr(\be)$, respectively, when convenient. Thus in particular, $O^{236}_k = \mathbb{H}^3/\prgr^{236}_k$. 

We will define $\doh^{236}_k = \mathbb{H}^3/\dprgr^{236}_k$ and call it the \emph{orientable prism orbifold}. We treat the infinite families $O^{236}_{k,n}$ in Table \ref{table236} and rows $O^{333}_k$ of Table \ref{table333} analogously.
\end{notation}

\begin{table}
\centering

\resizebox{\textwidth}{!}{%
\begin{tabular}{lrrrrlrrrrr}
\toprule
\char"0023 & $a_1$ &  $a_2$ &  $a_3$ &  $a_4$ &  $a_5$ &  $a_6$ &  $a_7$ &  $a_8$ &  $a_9$ \\
\midrule
$O^{333}_1$  & 3 & 3 & 2 & 3 & 3 & 4 & 2 & 2 & 2 \\
$O^{333}_2$  & 3 & 3 & 2 & 3 & 3 & 4 & 2 & 2 & 3 \\
$O^{333}_3$  & 3 & 3 & 2 & 3 & 3 & 4 & 2 & 3 & 2 \\
$O^{333}_4$  & 3 & 3 & 2 & 3 & 3 & 4 & 3 & 2 & 2 \\
$O^{333}_5$  & 3 & 3 & 2 & 3 & 3 & 4 & 4 & 2 & 2 \\
$O^{333}_6$  & 3 & 3 & 2 & 3 & 3 & 4 & 5 & 2 & 2 \\
$O^{333}_7$  & 3 & 3 & 2 & 3 & 3 & 5 & 2 & 2 & 2 \\
$O^{333}_8$  & 3 & 3 & 2 & 3 & 3 & 5 & 2 & 2 & 3 \\
$O^{333}_9$  & 3 & 3 & 2 & 3 & 3 & 5 & 2 & 3 & 2 \\
$O^{333}_{10}$ & 3 & 3 & 2 & 3 & 3 & 5 & 3 & 2 & 2 \\
$O^{333}_{11}$ & 3 & 3 & 2 & 3 & 3 & 5 & 4 & 2 & 2 \\
\bottomrule
\end{tabular}
\begin{tabular}{lrrrrlrrrrr}
\toprule
\char"0023 & $a_1$ &  $a_2$ &  $a_3$ &  $a_4$ &  $a_5$ &  $a_6$ &  $a_7$ &  $a_8$ &  $a_9$ \\
\midrule
$O^{333}_{12}$ & 3 & 3 & 2 & 3 & 3 & 5 & 5 & 2 & 2 \\
$O^{333}_{13}$ & 3 & 3 & 2 & 4 & 3 & 4 & 2 & 2 & 2 \\
$O^{333}_{14}$ & 3 & 3 & 2 & 4 & 3 & 4 & 2 & 2 & 3 \\
$O^{333}_{15}$ & 3 & 3 & 2 & 4 & 3 & 4 & 2 & 3 & 2 \\
$O^{333}_{16}$ & 3 & 3 & 2 & 4 & 3 & 5 & 2 & 2 & 2 \\
$O^{333}_{17}$ & 3 & 3 & 2 & 4 & 3 & 5 & 2 & 2 & 3 \\
$O^{333}_{18}$ & 3 & 3 & 2 & 4 & 3 & 5 & 2 & 3 & 2 \\
$O^{333}_{19}$ & 3 & 3 & 2 & 4 & 3 & 5 & 3 & 2 & 2 \\
$O^{333}_{20}$ & 3 & 3 & 2 & 5 & 3 & 5 & 2 & 2 & 2 \\
$O^{333}_{21}$ & 3 & 3 & 2 & 5 & 3 & 5 & 2 & 2 & 3 \\
$O^{333}_{22}$ & 3 & 3 & 2 & 5 & 3 & 5 & 2 & 3 & 2 \\
\bottomrule
\end{tabular}
}
\caption{Prism Orbifolds with (3,3,3) cusp}
\label{table333}
\end{table}

\subsection{Explicit embeddings and maximal cusp neighborhoods}\label{subsec: embedding333}

Another set of results of \cite{LakelandRoth} explicitly embed the prisms $P(\be)$ in $\mathbb{H}^3$, for nine-tuples $\be$ of possible edge labels. In this subsection we review the salient features of their construction; 
use it to describe a maximal horoball cusp neighborhood for each corresponding prism orbifold; then reproduce and elaborate on Lakeland-Roth's arithmetic determining exact locations of prism faces. 

Sections 3, 4.1, and 4.2 of \cite{LakelandRoth} describe embeddings of the prisms $P(\be)$ in the upper half-space model for $\mathbb{H}^3$---as a set, $\mathbb{C}\times(0,\infty)$, with ``ideal boundary'' the extended complex plane $\widehat{\mathbb{C}} = \mathbb{C}\cup\{\infty\}$---having the following common features:\begin{itemize}
    \item The ideal vertex of $P(\be)$, the endpoint of the edges labeled $a_1$, $a_2$, and $a_5$, is at $\infty$.
    \item The compact quadrilateral face of $P(\be)$ lies in the Euclidean hemisphere of unit radius centered at the origin.
    \item The non-compact triangular face of $P(\be)$ lies in a plane of the form $\ell\times(0,\infty)$, where:
        \begin{itemize}\item if $a_3 = 2$, $\ell\subset \mathbb{C}$ is the imaginary axis.
        \item if $a_3=3$, $\ell$ is the line $\mathrm{re}(z) = -\frac{1}{2}$.\end{itemize}
    In each case, projections to $\mathbb{C}$ of points of $P(\be)$ lie to the right of $\ell$.
    \item The (non-ideal) vertices of $P(\be)$ projecting into $\ell$---those of its non-compact triangular face---have the following complex coordinates, for $\theta_i = \pi/a_i$ as defined above:\begin{itemize}
        \item if $a_3 = 2$: $i\,y_1$ and $i\,y_2$, where $y_1 = \cos(\tang{4})/\sin(\tang{1})$ and $y_2 = -\cos(\tang{6})/\sin(\tang{2})$.
        \item if $a_3=3$: $z_1 =-\frac{1}{2}+(\cos(\tang{4})/\sin(\tang{1})+\cot(\tang{1})/2)i$ and\\ \mbox{}~\hfill $z_2 =-\frac{1}{2}-(\cos(\tang{6})/\sin(\tang{2})+\cot(\tang{2})/2)i$ 
    \end{itemize}
    \item The compact triangular face of $P(\be)$ lies in a Euclidean hemisphere of radius $r$ centered at $s+it\in\mathbb{C}$, for $s$, $t$, and $r$ solving equations (4), (5), and (6) of \cite{LakelandRoth} (here $(s,t)$ replace Lakeland and Roth's $(x,y)$ because we used $x,y$ as part of our generating set for $\dprgr(\be)$).
\end{itemize}

\begin{remark} The $a_3= 3$ case above corrects a typo in \cite{LakelandRoth}, which omits a negative sign when recording $z_2$ in Section 4.2 there (compare with the equation of the ``blue line'' given at the beginning of Section 4.1 there). In all the examples, in their paper $a_6=2$, hence $\cos(\tang{6})=0$ and the typo does not affect their results. However, their methods can be generalized to multi-cusped prism orbifolds, in those cases the typo corrected here would become apparent.\end{remark}

We now describe the equations defining $r$, $s$, and $t$, which are (4), (5), and (6) of \cite[\S 4.1]{LakelandRoth}. With a little algebra, equations (4) and (5) there are equivalent to:

\begin{align}\label{4n5}
\begin{pmatrix} s \\ t \end{pmatrix} = \frac{1}{\mbox{det}(A)}\begin{pmatrix} r\left(\cos(\tang{8})\sin(\tang{1})+\cos(\tang{7})\sin(\tang{2})\right) - \cos(\tang{6})\sin(\tang{1})-\cos(\tang{4})\sin(\tang{2}) \\
   r\left(-\cos(\tang{1})\cos(\tang{8})+\cos(\tang{2})\cos(\tang{7})\right)+\cos(\tang{1})\cos(\tang{6})-\cos(\tang{2})\cos(\tang{4})\end{pmatrix} \end{align}
where $\mbox{det}(A)= -\cos(\tang{2})\sin(\tang{1})-\cos(\tang{1})\sin(\tang{2})$.
   
Equation 6 from  \cite{LakelandRoth} is:\begin{align}\label{six}
    s^2 + t^2 = 1 + r^2 + 2r\cos\left(\tang{9}\right) \end{align}

For ease of notation, we will say
$s = s_A r + s_B$
and $t = t_A r+ t_B$. 
Then, if $a=s_A^2+t_A^2-1$, $b=2s_As_B+2t_A t_B -2 \cos(\tang{9})$, and $c=s_B^2+t_B^2-1$, then $r$ is a solution to $ar^2+br+c=0$, which we can solve for using the quadratic equation then determine the values of $s,t$ accordingly.

We will give more detail below about the equations defining $s$, $t$ and $r$, but first draw a basic conclusion about maximal embedded cusp neighborhoods in the $O^{236}_k$ and $O^{333}_k$.

\begin{lemma}\label{maximal general}
    For $\be$ in Table \ref{table236} or \ref{table333}, and $P(\be)$ embedded in $\mathbb{H}^3$ as described above, a maximal embedded horoball cusp neighborhood in the orbifold $\mathbb{H}^3/\prgr(\be)$ is obtained as the projection of the horoball at height $\max\{1,r\}$ centered at $\infty$, where $r$ solves equations (4), (5), and (6) of \cite{LakelandRoth}.
\end{lemma}

\begin{proof}
    Let $B_\infty$ be the horoball centered at $\infty$ having height $\max\{1,r\}$. Note that $B_{\infty}$ is tangent to a Euclidean hemisphere containing a compact face of $P(\be)$, and hence to its image under the reflection in this hemisphere, which belongs to $\prgr(\be)$. Therefore no horoball properly containing $B_{\infty}$ has image embedded in $\mathbb{H}^3/\prgr(\be)$. 
    Let $\Lambda(\be)$ be the peripheral subgroup of $\prgr(\be)$ fixing $\infty$, generated by reflections in the three non-compact faces of $P(\be)$. We claim that the projection of $B_{\infty}$ to $\mathbb{H}^3/\prgr(\be)$ factors through an embedding of $B_{\infty}/\Lambda(\be)$. 
    
    This follows from the fact that the intersection of $B_{\infty}$ with $P(\be)$ is a full fundamental domain for the action of $\Lambda(\be)$ on $B_{\infty}$, which is itself a consequence of $B_{\infty}$ having been chosen at height at least the radius of each Euclidean hemisphere containing a compact face of $P(\be)$. As a result, $B_{\infty}$ is entirely contained in the union of $\Lambda(\be)$-translates of $P(\be)$, so it does not overlap any of its distinct $\prgr(\be)$-translates.
\end{proof}

\begin{remark}
Though not relevant for this paper, the above lemma also applies to maximal embedded horoball cusp neighborhoods of prism orbifolds with $(2,4,4)$ cusps. 
\end{remark}

Following Lakeland and Roth \cite[Section 4]{LakelandRoth}, we can get matrix representations for $\dprgr(\be)$ as determined by the following cases:

{\bf Case 1: $a_3=2$.}
In this case, we have that 

$$ M_1 = \begin{pmatrix} 0 & -1\\ 1 & 0\end{pmatrix},$$

$$M_2 = \begin{pmatrix} e^{-i\tang{1}} & -y_1 i (e^{-i\tang{1}}-e^{i\tang{1}})\\ 0 & e^{i\tang{1}} \end{pmatrix}$$


$$M_3 = \begin{pmatrix} e^{i\tang{2}} & -y_2 i (e^{-i\tang{2}}-e^{i\tang{2}})\\ 0 & e^{-i\tang{2}} \end{pmatrix}$$


$$M_4 = \begin{pmatrix} \frac{1}{r}(-s+ti) & \frac{1}{r} (s^2+t^2)-r \\ \frac{1}{r} & \frac{1}{r}(-s-ti)  \end{pmatrix},$$

where $\theta_i = \frac{\pi}{a_i}$ and $y_1, y_2, r,s,$ and $t$ defined above.

{\bf Case 2: $a_3=3$.}
In this case, we have that 

$$ M_1 = \begin{pmatrix} -1 & -1\\ 1 & 0\end{pmatrix},$$

$$M_2 = \begin{pmatrix} e^{-i\tang{1}} & -z_1 i (e^{-i\tang{1}}-e^{i\tang{1}})\\ 0 & e^{i\tang{1}} \end{pmatrix}$$


$$M_3 = \begin{pmatrix} e^{i\tang{2}} & z_2 i (e^{-i\tang{2}}-e^{i\tang{2}})\\ 0 & e^{-i\tang{2}} \end{pmatrix}$$


$$M_4 = \begin{pmatrix} \frac{1}{r}(-s-1+ti) & \frac{1}{r} (-s-1+ti)(-s-ti)-r \\ \frac{1}{r} & \frac{1}{r}(-s-ti)  \end{pmatrix},$$

where $z_1,z_2,s,t$ and $r$ are defined above.

If we denote by $\rho: \dprgr(\be) \rightarrow PSL(2,\mathbb{C})$, then in either case we have that  $\rho(x)=M_1$,
$\rho(y) = M_2^{-1}M_1$, $\rho(z) = M_3^{-1}M_2$, $\rho(w)=M_4^{-1}M_2$.

\subsection{Totally geodesic surfaces}\label{tot geod} Here we make an observation that was not recorded in \cite{LakelandRoth}, on the existence of a certain triangle embedded in each prism, and hence on the existence of totally geodesic surfaces in manifold covers of the prism orbifolds.

\begin{lemma}\label{corandreev}
The hyperbolic prism $P(\be)$ determined by a nine-tuple $\be = (a_1, a_2, \dots, a_9)$ from one of Tables \ref{table236} or \ref{table333} contains a unique compact, totally geodesic triangle $T$ with its vertices in the edges labeled $a_4$, $a_5$, $a_6$, which meets each quadrilateral face at a right angle and hence has vertex angles $\tang{4}$, $\tang{5}$, and $\tang{6}$. Either $T$ is the compact triangular face of $P(\be)$, or it intersects only quadrilateral faces.
\end{lemma}

\begin{proof}
If $a_7 = a_8 = a_9 = 2$, then $T$ is the triangular face of $P\doteq P(\be)$ whose sides are edges labeled $a_7$, $a_8$, and $a_9$ (the compact one). We thus assume below that at least one of those three edge labels is larger than $2$. In this case we will find $T$ by separately proving existence of two sub-polyhedra $P^+$ and $P^-$ of $P$ that would result from cutting it along $T$, as pictured in Figure \ref{compact triangle}. $P$ is then recovered as the union of $P^+$ with $P^-$ along triangular faces, which produce $T$ when identified.

Each of $P^+$ and $P^-$ itself has the combinatorial type of a prism. Their edge labels are:\begin{align*} 
 (a_1^+, a_2^+, a_3^+, a_4^+, a_5^+, a_6^+, a_7^+, a_8^+, a_9^+) & = (2, 2, 2, a_4, a_5, a_6, a_7, a_8, a_9),\mbox{ and}\\ 
 (a_1^-, a_2^-, a_3^-, a_4^-, a_5^-, a_6^-, a_7^-, a_8^-, a_9^-) & = (a_1, a_2, a_3, a_4, a_5, a_6, 2, 2, 2), 
\end{align*}
 respectively. We just need to check that the collection of dihedral angles $\tang{i}^{\pm} = \pi/a_i^{\pm}$ determined by each tuple of labels satisfies the conditions of Andreev's theorem. As listed in \cite[Theorem 2]{Andreev}, these are:\begin{itemize}
     \item m0: $0<\tang{i}^{\pm} \le \pi/2$ for all $i$ holds by construction.
     \item m$1^*$: The sum of dihedral angles at edges meeting at a vertex is at least $\pi$. At each vertex of $P^{\pm}$ that is not a vertex of $P$, the dihedral angle equals $\pi/2$ at at least two such edges.
     \item m2: for the unique triangular prismatic element---the cycle of three quadrilateral faces---$\tang{4}^{\pm} + \tang{5}^{\pm} + \tang{6}^{\pm} < \pi$. This is inherited from $P$ in both cases.
     \item m3 holds vacuously: there is no quadrangular prismatic element.
     \item m4: the sum of all dihedral angles at edges of the two triangular faces is less than $3\pi$. For $P^+$ this follows from the hypothesis that at least one of $a_7$, $a_8$ and $a_9$ is greater than $2$; for $P^-$, the same holds for $a_1$ and $a_2$.
     \item m5 again holds vacuously, as there are no three faces of a prism that satisfy the conditions of its hypothesis.
 \end{itemize}
Andreev's theorem thus implies the existence of $P^+$ and $P^-$, and hence the desired triangle. It also asserts that $P^+$ and $P^-$ are prescribed uniquely up to isometry by their collection of edge labels, from which it follows that $T$ is the unique triangle meeting the quadrilateral faces at right angles. Uniqueness also implies that $T$ is unique in the case $a_7 = a_8 = a_9 = 2$, since otherwise $P$ would be isometric to a sub-polyhedron of itself.
\end{proof}

\begin{figure}[ht]
\centering
\begin{tikzpicture}[scale=1]

\begin{scope}[xshift=-6cm]
    \draw [thick] (0,-2) -- (-2,-2);
    \fill [color=white] (-1.2,-2.6) -- (-1.35,-1.9) -- (-1.05,-1.9);
    \draw [thick] (0,-2) -- (0,0) -- (-1.2,-0.6) -- (-1.2,-2.6) -- (-2,-2);
    \draw [thick] (0,0) -- (0,-2) -- (-1.2,-2.6) -- (-1.2,-0.6) -- (0,0) -- (-2,0) -- (-2,-2) -- (-1.2,-2.6);
    \draw [thick] (-1.2,-0.6) -- (-2,0);
    \fill [color=white] (-1.2,-2.6) circle [radius=0.1];
    \draw (-1.2,-2.6) circle [radius=0.1];
    
    \draw [dashed] (0,-1) -- (-2,-1) -- (-1.2,-1.6) -- (0, -1);
    \fill [color=red, opacity=0.2] (0,-1) -- (-2,-1) -- (-1.2,-1.6);
    \fill [color=white] (-1.2,-1.6) -- (-1.35,-0.9) -- (-1.05,-0.9);
    \draw [thick] (-1.2,-0.6) -- (-1.2,-2);

    \node [below] at (-1.6, -1.3) {\scriptsize \textcolor{red}{2}};
    \node [below] at (-0.6, -1.3) {\scriptsize \textcolor{red}{2}};
    \node [above] at (-0.8, -1) {\scriptsize \textcolor{red}{2}};

    \node at (-1,-3.3) {\Large $P$};
        
\end{scope}

\node at (-3.8,-1.1) {\Huge $=$};

\begin{scope}
    \draw [thick] (0,-2) -- (-2,-2);
    \fill [color=white] (-1.2,-2.6) -- (-1.35,-1.9) -- (-1.05,-1.9);
    
    \draw [thick] (0,-2) -- (0,-1);
    \draw [thick] (-2,-2) -- (-2,-1);
    \draw [thick] (-1.2,-1.6) -- (-1.2,-2.6);

    \draw [thick] (-2,-2) -- (-1.2,-2.6);
    \draw [thick] (0,-2) -- (-1.2,-2.6);
    
    
    \fill [color=white] (-1.2,-2.6) circle [radius=0.1];
    
    \draw (-1.2,-2.6) circle [radius=0.1];

    \draw [dashed] (0,-1) -- (-2,-1) -- (-1.2,-1.6) -- (0, -1);
    \fill [color=red, opacity=0.2] (0,-1) -- (-2,-1) -- (-1.2,-1.6);

    \node [below] at (-1.6, -1.3) {\scriptsize \textcolor{red}{2}};
    \node [below] at (-0.6, -1.3) {\scriptsize \textcolor{red}{2}};
    \node [above] at (-0.8, -1) {\scriptsize \textcolor{red}{2}};

    \node at (-1,-3.3) {\Large $P^-$};
\end{scope}

\node at (2,-1.1) {\Huge $\cup$};

\begin{scope}[xshift=6cm]

    \draw [thick] (0,-1) -- (0,0);
    \draw [thick] (-2,-1) -- (-2,0);
    \draw [thick] (-1.2,-0.6) -- (-1.2,-1.6);

    \draw [thick] (0,0) -- (-2,0) -- (-1.2,-0.6) -- (0, 0);

    \draw [dashed] (0,-1) -- (-2,-1) -- (-1.2,-1.6) -- (0, -1);
    \fill [color=red, opacity=0.2] (0,-1) -- (-2,-1) -- (-1.2,-1.6);
    \fill [color=white] (-1.2,-1.6) -- (-1.35,-0.9) -- (-1.05,-0.9);
    \draw [thick] (-1.2,-0.6) -- (-1.2,-1.6);

    \node [below] at (-1.6, -1.3) {\scriptsize \textcolor{red}{2}};
    \node [below] at (-0.6, -1.3) {\scriptsize \textcolor{red}{2}};
    \node [above] at (-0.8, -1) {\scriptsize \textcolor{red}{2}};

    \node at (-1,-3.3) {\Large $P^+$};
\end{scope}

\end{tikzpicture}
\caption{The compact triangle from \Cref{corandreev}.}
\label{compact triangle}
\end{figure}
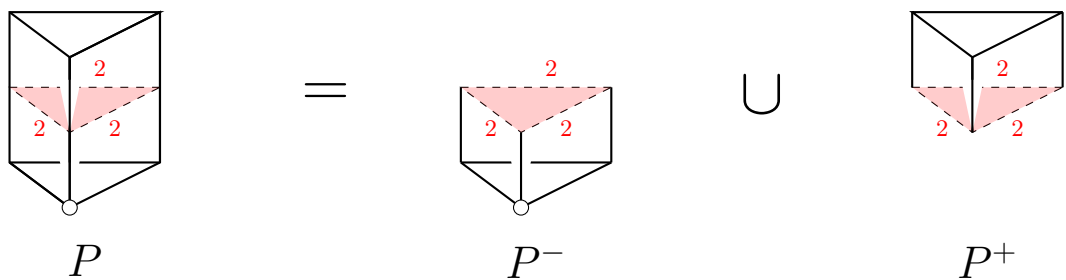

\begin{corollary}\label{triangle group} For any nine-tuple $\be = (a_1, a_2, \dots, a_9)$ from Table \ref{table236} or \ref{table333}, the subgroup of $\prgr(\be)$ generated by reflections in the quadrilateral faces of $P(\be)$ is a Fuchsian group generated by reflections in the sides of a triangle with vertex angles $\tang{4}$, $\tang{5}$, and $\tang{6}$.\end{corollary}

\begin{proof} For such a tuple $\be$, let $H$ be the plane containing the triangle $T$ of \Cref{corandreev}. By that result, $H$ meets any quadrilateral face $Q$ of $P(\be)$ at right angles and therefore is stabilized by the reflection in this face. The reflection in the plane containing $Q$ restricts to $H$ as reflection in the geodesic containing the side $Q\cap H$ of $T$. Again because $H$ meets each quadrilateral face at right angles, the angle between any two sides of $T$ equals the dihedral angle along the edge of intersection of the corresponding quadrilateral faces.
\end{proof}

\begin{prop}\label{geodesic surface} Suppose $\be$ is a nine-tuple from one of Tables \ref{table236} or \ref{table333}, let  $P(\be)$ be the associated hyperbolic prism, and let $T\subset P(\be)$ be the totally geodesic triangle from \Cref{corandreev}. The preimage of $T$ in any finite-degree manifold cover $M$ of $P(\be)$ is a (possibly disconnected) closed, embedded totally  geodesic surface $S$. If at least one of $a_7$, $a_8$, or $a_9$ does not equal $2$ then $S$ separates $M$.
\end{prop}

\begin{proof} Fix a tuple $\be$ belonging to one of Tables \ref{table236} or \ref{table333}, and let $T\subset P(\be)$ be the triangle supplied by Lemma \ref{corandreev}. The action of the reflection group $\prgr(\be)$ tiles $\mathbb{H}^3$ with a union of translates of $P(\be)$. Any manifold cover of the prism orbifold is cellulated by the projections of these translates. For such a manifold cover $M$, cellulated by copies $P_1,\hdots,P_n$ of $P(\be)$, let $T_1,\hdots,T_n$ be the corresponding copies of $T$. We claim that $S = T_1 \cup \cdots \cup T_n$ is a closed, embedded, totally geodesic surface in $M$.

One can show that $S$ is a totally geodesic surface by giving each $p\in S$ a chart that takes the intersection of $S$ with its domain into a totally geodesic copy of $\mathbb{H}^2$ in $\mathbb{H}^3$. If $p$ lies in the interior of $T_i$ for some $i$, we can take this chart to be the inverse of the restriction of the universal covering to the interior of some $\widetilde{P}_i\subset\mathbb{H}^3$ that projects to $P_i$. 

For $p$ in the interior of an edge of $T_i$, itself contained in a quadrilateral face $Q$ of $P_i$, let $P_j$ be the copy of $P(\be)$ sharing this face with $P_i$. Then for $\widetilde{P}_i$ as before, there is a copy $\widetilde{P}_j$ of $P(\be)$ projecting to $P_j$ that intersects $\widetilde{P}_i$ in a quadrilateral face $\widetilde{Q}$ projecting to $Q$. Because $\prgr(\be)$ is generated by reflections in the faces of $P(\be)$, $\widetilde{P}_j$ is the reflection of $\widetilde{P}_i$ across $\widetilde{Q}$. The preimage $\widetilde{T}_i$ of $T_i$ in $\widetilde{P}_i$ intersects $\widetilde{Q}$ at right angles; hence it and its image under the reflection across $\widetilde{Q}$ lie in a single totally geodesic plane. By the uniqueness assertion of Lemma \ref{corandreev}, this reflected image of $\widetilde{T}_i$ is the preimage $\widetilde{T}_j$ of $T_j$ under the projection $\widetilde{P}_j\to P_j$. There is thus a chart map around $p$, defined on the unions of the interiors of $P_i$ and $P_j$ with that of $Q$, sending the interior of $T_i\cup T_j$ into $\widetilde{T}_i\cup\widetilde{T}_j$.

One builds a chart for $p$ in a vertex of $T_i$ in a similar way, working around the sequence of faces containing the vertex. This shows that $S$ is a totally geodesic surface. Since each $T_i$ is compact and $M$ is a finite-degree cover, the surface $S$ is closed. For embeddedness, we recall that each $T_i$ meets only quadrilateral faces, and hence intersects a unique $T_j$ along each edge, in the case that $a_7$, $a_8$ and $a_9$ do not all equal $2$. In the case that $a_7=a_8=a_9=2$, we must additionally observe that since $T_i$ is the compact triangular face of $P_i$, it is equal to some $T_k$ for a copy $P_k$ of $P(\be)$ adjacent to $P_i$, and hence it intersects $T_j$ along an edge if and only if $P_i$ or $P_k$ shares a quadrilateral face with $P_j$.

To see that $S$ is separating if at least one of $a_7$, $a_8$, or $a_9$ does not equal $2$, we recall from the proof of \Cref{corandreev} that in this case, $T$ separates $P(\be)$ into sub-prisms $P^+$ and $P^-$. The resulting decomposition of each $P_i$ into $P_i^{\pm}$ is preserved by the face identifications, so $S = \bigcup T_i$ separates $M$ into submanifolds $M^+ = \bigcup P_i^+$ and $M^-=\bigcup P_i^-$.
\end{proof}

\section{Geometric particulars of \texorpdfstring{$O^{333}_2$}{O2} and \texorpdfstring{$O^{333}_3$}{O3}}\label{sec: two_n_three}

Here we will draw on the previous section's perspective to derive particular results about the orbifolds $O^{333}_2$ and $O^{333}_3$ of primary interest in this paper. Here is the data from Table \ref{table333} for these two orbifolds:

\begin{center}
\begin{tabular}{lrrrrrlrrrrrr}

\toprule
\char"0023 & $a_1$ &  $a_2$ &  $a_3$ &  $a_4$ &  $a_5$ &  $a_6$ &  $a_7$ &  $a_8$ &  $a_9$ \\
\midrule
$O^{333}_2$ & 3 & 3 & 2 & 3 & 3 & 4 & 2 & 2 & 3 \\
$O^{333}_3$ & 3 & 3 & 2 & 3 & 3 & 4 & 2 & 3 & 2 \\
\bottomrule
\end{tabular}
\end{center}

Section \ref{quick n easy} records two quick consequences of the work in \Cref{tot geod} for the present orbifolds: that they have equal volumes and are non-arithmetic, in Remarks \ref{scis cong} and \ref{non-arithmetic}, respectively. In \Cref{subsec: volumes} below we will give exact formulas for the volumes of $O^{333}_2$ and $O^{333}_3$ as sums of single integrals. Numerically computing these integrals yields the approximate value $0.938069938216186$ in both cases. The main result of \Cref{min orb}, \Cref{no cover}, asserts that neither $O^{333}_2$ nor $O^{333}_3$ has a non-trivial cover to another orbifold.

\subsection{Initial observations}\label{quick n easy} Here we record two quick consequences of the work in \Cref{tot geod} when it is specialized to the present setting.

\begin{remark}\label{scis cong}
    $O^{333}_2$ is \mbox{\rm OR-scissors congruent} to $O^{333}_3$, meaning that the two orbifolds decompose into collections of hyperbolic polyhedra which are identical up to possibly orientation-reversing isometry. In particular, they have the same volume.
\end{remark}

The decomposition here comes from the polyhedra $P^+$ and $P^-$ obtained by cutting the prism along the triangle $T$ from Lemma \ref{corandreev} as in its proof. Each of the two $\be$ in the table above yields the same non-compact prism $P^-$ after cutting, with associated nine-tuple
\[ \be^- = (3, 3, 2, 3, 3, 4, 2, 2, 2) \]
The compact prisms $P^+$ associated to the two $\be$ are mirror images of each other. We picture them in Figure \ref{prism caps} with the one for $O^{333}_2$ on the left. In each case the face shared with $P^-$---the triangle $T$ from \Cref{corandreev}---is shaded. In this case the vertex angles of $T$ are $\pi/3$, $\pi/3$, and $\pi/4$. As a result we obtain:

\begin{figure}[ht]
\begin{tikzpicture}

\begin{scope}
    \draw [thick] (-1.2,-1.4) -- (-1.2,0.6);
    \fill [opacity=0.1] (0,-2) -- (-1.2,-1.4) -- (-2,-2);
    \fill [color=white] (-1.2,0) circle [radius=0.15];
    \draw [thick] (0,0) -- (0,-2) -- (-1.2,-1.4) -- (-2,-2) -- (-2,0) -- (-1.2,0.6) -- (0,0) -- (-2,0);
    \draw [thick] (0,-2) -- (-2,-2);

    \node [right] at (0,-1) {\scriptsize $4$};
    \node [above right] at (-0.7,-1.8) {\scriptsize $2$};
    \node [above left] at (-1.45,-1.75) {\scriptsize $2$};
    \node [below] at (-1.1,-1.95) {\scriptsize $2$};
    \node [right] at (-1.25,-0.6) {\scriptsize $3$};
    \node [left] at (-2,-1) {\scriptsize $3$};
    \node [above right] at (-0.7,0.2) {\scriptsize $2$};
    \node [above left] at (-1.5,0.2) {\scriptsize $2$};
    \node [below] at (-0.65,0.05) {\scriptsize $3$};
\end{scope}

\begin{scope}[xshift=-2.5in]
    \draw [thick] (0,-2) -- (-2,-2);
    \fill [opacity=0.1] (0,-2) -- (-2,-2) -- (-1.2,-2.6);
    \fill [color=white] (-1.2,-2.6) -- (-1.35,-1.9) -- (-1.05,-1.9);
    \draw [thick] (0,0) -- (0,-2) -- (-1.2,-2.6) -- (-1.2,-0.6) -- (0,0) -- (-2,0) -- (-2,-2) -- (-1.2,-2.6);
    \draw [thick] (-1.2,-0.6) -- (-2,0);

    \node [right] at (0,-1) {\scriptsize $4$};
    \node [below right] at (-0.7,-2.2) {\scriptsize $2$};
    \node [below left] at (-1.5,-2.2) {\scriptsize $2$};
    \node [above] at (-0.6,-2.05) {\scriptsize $2$};
    \node [right] at (-1.25,-1.35) {\scriptsize $3$};
    \node [left] at (-2,-1) {\scriptsize $3$};
    \node [below right] at (-0.7,-0.2) {\scriptsize $2$};
    \node [below left] at (-1.5,-0.2) {\scriptsize $2$};
    \node [above] at (-1,-0.05) {\scriptsize $3$};
\end{scope}

\end{tikzpicture}
\caption{The compact prisms $P^+$ of \Cref{corandreev}'s proof, for $O^{333}_2$ and $O^{333}_3$.}
\label{prism caps}
\end{figure}
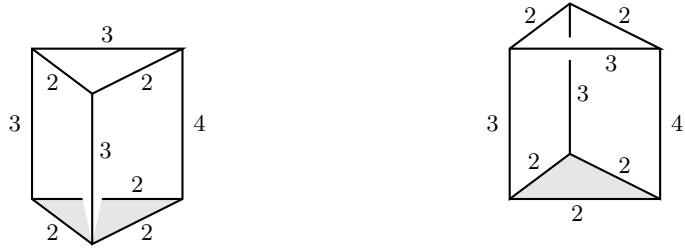

\begin{prop}
\label{non-arithmetic} For $i=2$ or $3$, $O^{333}_i$ is non-arithmetic.\end{prop}

\begin{proof}
The Fuchsian $(3,3,4)$-triangle group, which is a subgroup of the orbifold fundamental group of $O^{333}_2$ and $O^{333}_3$ by \Cref{triangle group}, has $\mathbb{Q}(\sqrt{2})$ as its field of definition (cf.~eg.~\cite[\S 13.3]{book}). It follows that the invariant trace fields of $O^{333}_2$ and $O^{333}_3$ each contain $\mathbb{Q}(\sqrt{2})$. But every non-compact arithmetic $3$-orbifold is commensurable with a Bianchi group $\operatorname{PGL}_2(\mathcal{O}_d)$ for some square-free positive integer $d$, where $\mathcal{O}_d$ is the ring of integers of the imaginary quadratic field $\mathbb{Q}(\sqrt{-d})$ with $\mathbb{Q}(\sqrt{-d})$ its invariant trace field (see \cite[Theorem 8.2.3]{book} for example). The only real subfield of any such field is $\mathbb{Q}$. \end{proof}

\subsection{Volume computation}\label{subsec: volumes} Here we will set up a computation for the volumes of $O^{333}_i$, for $i=1$, $2$, and $3$, as a sum of four single integrals, recording its results in \Cref{rem:quot_vol}. The computation for $i=1$ is relevant since $O^{333}_1$ has the same volume as the sub-prism $P^-$ common to $P^{333}_2$ and $P^{333}_3$ as in the proof of \Cref{corandreev}.

Following \Cref{table names}, for each $i\in\{1,2,3\}$ let $P^{333}_i$ be the hyperbolic prism yielding $O^{333}_i$, arranged in $\mathbb{H}^3$ as in \Cref{subsec: embedding333}. In all cases, the quadrilateral face of $P^{333}_i$ that does not contain its ideal point lies in the unit Euclidean hemisphere centered at the origin; it meets the triangular face containing the ideal point in the hemisphere's equator over the $y$-axis; and $y_1 = 1/\sqrt{3}$ and $y_2 = -\sqrt{2/3}$ are the $y$-coordinates of the endpoints of the edge of intersection between these faces. This is part of the embedding described in \cite[\S 4.1]{LakelandRoth}.

The upshot of the previous paragraph for the ``view from infinity'' is that $y_1$ and $y_2$ as above are the vertices on the $y$-axis of the equilateral triangle $\Delta$ formed as a cross-section of $P^{333}_i$ by a(ny) horosphere centered at $\infty$. The three lines containing edges of $\Delta$ are $x=0$, $y_l(x) = \frac{1}{\sqrt{3}}x-\sqrt{\frac{2}{3}}$, and $y_u(x) = \frac{-1}{\sqrt{3}}x+\frac{1}{\sqrt{3}}$, again matching \cite[\S 4.1]{LakelandRoth}.

Let $T_i^{(c)}$ refer to the compact triangular face of $P^{333}_i$, for $i\in\{1,2,3\}$. Equations (\ref{4n5}) and (\ref{six}), governing the center $(s,t)$ and radius $r$ of the Euclidean hemisphere containing $T^{(c)}_i$, specialize here to the below:
\setcounter{equation}{3}
\begin{align}
    \begin{pmatrix} s \\ t \end{pmatrix} & = \begin{pmatrix} \frac{1+\sqrt{2}}{2} \\ \frac{1-\sqrt{2}}{2\sqrt{3}} \end{pmatrix} &
    \begin{pmatrix} s \\ t \end{pmatrix} & = \begin{pmatrix} \frac{1+\sqrt{2}}{2} \\ \frac{1-\sqrt{2}}{2\sqrt{3}} \end{pmatrix} &
    \begin{pmatrix} s \\ t \end{pmatrix} & = \begin{pmatrix} -\frac{r}{2} + \frac{1+\sqrt{2}}{2} \\ \frac{1}{\sqrt{3}}\left(\frac{r}{2} + \frac{1-\sqrt{2}}{2}\right) \end{pmatrix} \nonumber\\
 s^2+t^2 & = r^2 + 1 & s^2+t^2 & = r^2 + r + 1    & s^2+t^2 & = r^2 + 1 \label{spec r}
\end{align}
The columns above correspond to the case $i=1$, $2$, and $3$, left-to-right. Call unique solutions having positive $r$-value $(X_i,Y_i,R_i)$ in the respective cases. Note that $(X_1,Y_1) = (X_2,Y_2)$ is a vertex of $\Delta$. This reflects the fact that each of $P^{333}_1$ and $P^{333}_2$ has dihedral angle $\pi/2$ at each of the two edges of $T^{(c)}_i$ that meet at its vertex of intersection with an edge containing the ideal point. Solving the final equation in these cases yields
\[ R_1 = \sqrt{\frac{\sqrt{2}}{3}}\approx 0.686589048 \quad\mbox{and}\quad R_2  = \frac{1}{2}\left[-1+\sqrt{1+\frac{4\sqrt{2}}{3}}\right] \approx 0.349355356. \]
Formulas for $X_3$, $Y_3$ and $R_3$ are messier, but we have approximately:
\[ X_3 \approx 1.0486436547687,\ \ Y_3 \approx -0.028084427175087, \ \ R_3 \approx 0.31692625283566. \]

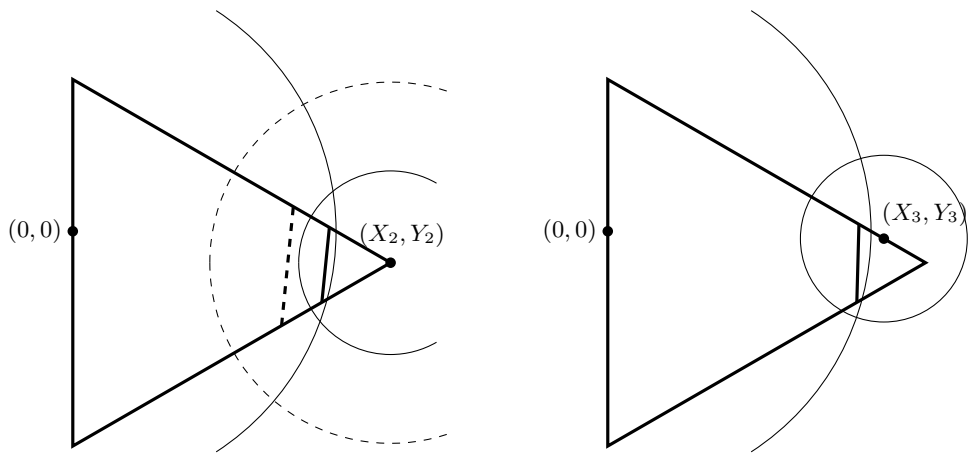
\begin{figure}[ht]
\begin{tikzpicture}

\begin{scope}[scale=3.5]

\draw [thin] (1,0) arc (0:57:1);
\draw [thin] (1,0) arc (0:-57:1);
\draw [very thick] (0,0.577) -- (0,-0.8165) -- (1.2071,-0.1196) -- cycle;
\draw [very thick] (0.946,-0.27) -- (0.954,0.026);

\fill (0,0) circle [radius=0.02];
\node [left] at (0,0) {\scriptsize{$(0,0)$}};
\fill (1.0486,-0.0281) circle [radius=0.02];
\node [above right] at (1,-0.02) {\scriptsize{$(X_3,Y_3)$}};
\draw [thin] (1.0486,-0.0281) circle [radius=0.3169];

\end{scope}

\begin{scope}[xshift=-2.8in, scale=3.5]

\draw [thin] (1,0) arc (0:57:1);
\draw [thin] (1,0) arc (0:-57:1);
\draw [very thick] (0,0.577) -- (0,-0.8165) -- (1.2071,-0.1196) -- cycle;
\draw [very thick] (0.946,-0.27) -- (0.9746,0.0147);
\draw [very thick, dashed] (0.7929,-0.3587) -- (0.8377,0.0937);

\fill (0,0) circle [radius=0.02];
\node [left] at (0,0) {\scriptsize{$(0,0)$}};
\fill (1.2071,-0.1196) circle [radius=0.02];
\node [above] at (1.25,-0.1) {\scriptsize{$(X_2,Y_2)$}};
\draw [thin] (0.8581,-0.1196) arc (180:60:0.349);
\draw [thin] (0.8581,-0.1196) arc (180:300:0.349);
\draw [thin, dashed] (0.5205,-0.1196) arc (180:70:0.6866);
\draw [thin, dashed] (0.5205,-0.1196) arc (180:290:0.6866);

\end{scope}

\end{tikzpicture}

\caption{$P^{333}_1$ and $P^{333}_2$ (left), and $P^{333}_3$ (right), viewed from above, with the circles bounding the hemispheres that contain their non-vertical faces.}
\label{rigeur}
\end{figure}

For each of $i=1$, $2$ and $3$, with $P^{333}_i$ arranged as in Figure \ref{rigeur}, its edge of intersection between the two faces that do not contain the ideal point is contained in the vertical plane over the line $\ell_i$ through the points of intersection between the unit and non-unit circles. This line is perpendicular to the one joining the two circles' centers, so it has slope $-X_i/Y_i$. Its equation is therefore
\[ y - t_iY_i = -\frac{X_i}{Y_i}(x-t_iX_i)\quad \rightsquigarrow\quad y = -\frac{X_i}{Y_i}x + \frac{t_i(X_i^2+Y_i^2)}{Y_i}, \]
where $t_i(X_i,Y_i)$ is the point at which $\ell_i$ meets the line through the two circle centers. This satisfies $t_i(X_i^2+Y_i^2) = \frac{1}{2}\left[1-R_i^2 + X_i^2 + Y_i^2\right]$, by a calculation with the Pythagorean theorem, so using equations (\ref{spec r}) above we obtain:
\[ \ell_1(x) = -\frac{X_1}{Y_1}x + \frac{1}{Y_1},\quad \ell_2(x) = -\frac{X_2}{Y_2}x + \frac{R_2+2}{2Y_2}\quad\mbox{and}\quad \ell_3(x) = -\frac{X_3}{Y_3}x + \frac{1}{Y_3}. \]

For $i=1$, $2$ or $3$, let $X_{l,i}$ be $x$-value of the point of intersection between $\ell_i$ and $y_l$, and let $X_{u,i}$ be the $x$-value of the point of intersection between $\ell_i$ and $y_u$. These can be exactly computed; we give their approximations below:\begin{align*}
    & X_{l,1} \approx 0.79289321881345 & & X_{l,2} \approx 0.94637894091280 & & X_{l,3} = X_{l,2}\ \mbox{(exactly)} \\
    & X_{u,1} \approx 0.83770871866842 & & X_{u,2} \approx 0.97458817776402 & & X_{u,3} \approx 0.95431912485788
    \end{align*}
The volume computation itself is a sum of four triple integrals, each over a region consisting of points in $P^{333}_i$ ($i=1$, $2$ or $3$) above a sub-quadrilateral of $\Delta$ bounded by four lines:\begin{enumerate}
    \item $x = 0$ and $x = X_{l,i}$, and the lines $y_l(x)$ and $y_u(x)$;
    \item $x = X_{l,i}$ and $x = X_{u,i}$, and the lines $\ell_i(x)$ and $y_u(x)$;
    \item $x = X_{l,i}$ and $x = X_{u,i}$, and the lines $y_l(x)$ and $\ell_i(x)$; and
    \item $x = X_{u,i}$ and $x = \frac{1+\sqrt{2}}{2}$, and the lines $y_l(x)$ and $y_u(x)$.
\end{enumerate}
The lower bound on $z$ for points in regions (1) and (2) is the unit hemisphere; ie.~$z = \sqrt{1-x^2-y^2}$. Here is the integral over the first region:\begin{align*}
    \int_0^{X_{l,i}}\int_{y_l(x)}^{y_u(x)} & \int_{\sqrt{1-x^2-y^2}}^{\infty} \frac{1}{z^3} \mathit{dz}\,\mathit{dy}\,\mathit{dx} = \int_0^{X_{l,i}}\int_{y_l(x)}^{y_u(x)} \frac{1}{2}\,\frac{1}{1-x^2-y^2} \mathit{dy}\,\mathit{dx} \\
        & = \int_0^{X_{l,i}} \frac{1}{4\sqrt{1-x^2}}\ln \left(\frac{(\sqrt{1-x^2}+y_u(x))(\sqrt{1-x^2}-y_l(x))}{(\sqrt{1-x^2}-y_u(x))(\sqrt{1-x^2}+y_l(x))}\right)\,\mathit{dx}
\end{align*}
The integral with respect to $z$ is straightforward, and for the integral with respect to $y$ we use partial fractions, to arrive at the above.
For the integral in region (2), the $x$-bounds are replaced by $X_{l,i}$ and $X_{u,i}$, and $y_l(x)$ is replaced by $\ell_i(x)$. 

The integral for region (3) has the same $x$ bounds as in region (2), but now $\ell_i(x)$ is an upper bound for $y$, $y_l(x)$ is a lower bound, and the lower bound on $z$ is given by the \emph{non-unit} hemisphere; ie. $z \ge \sqrt{R_i^2 - (x-X_i)^2-(y-Y_i)^2}$. After integrating with respect to $z$ and then $y$ in the same way as the other cases, the resulting integral has the form below:
\begin{align*} 
    \int_{X_{l,i}}^{X_{u,i}} & \frac{1}{4\sqrt{R_i^2-(x-X_i)^2}}\cdot \\
        &\ln \left(\frac{(\sqrt{R_i^2-(x-X_i)^2}+\ell_i(x)-Y_i)(\sqrt{R_i^2-(x-X_i)^2}-y_l(x)+Y_i)}{(\sqrt{R_i^2-(x-X_i)^2}-\ell_i(x)+Y_i)(\sqrt{R_i^2-(x-X_i)^2}+y_l(x)-Y_i)}\right)\mathit{dx}
\end{align*}
The integral for region (4) is similar to the above, but with $x$-bounds of $X_{u,i}$ (below) and $(1+\sqrt{2})/2$ (above), and $y_u(x)$ replacing $\ell_i(x)$ in the integrand. 

\begin{remark}\label{rem:quot_vol} Taking the sum of the four numerically computed integrals in each case yields: 
\[ \mathrm{vol}(P^{333}_1) \approx 0.672771983317043,\ \ \mathrm{vol}(P^{333}_2),\ \mathrm{vol}(P^{333}_3) \approx 0.938069938216186.\]
(Totals for $i=2$ and $3$ are identical to this precision, as expected by \Cref{scis cong}, although individual integrals' values differ.) The orientable orbifolds $\doh^{333}_i$ thus have volumes approximately equal to $1.34554396663409$, for $i=1$, and $1.87613987643237$ for $i=2, 3$.
\end{remark}

\subsection{Maximal horoballs and commensurability}\label{min orb} We begin by giving an intrinsic characterization of maximal horoball cusp neighborhoods in $O^{333}_i$, for $i\in\{1,2,3\}$, by specializing the discussion in \Cref{subsec: embedding333} to the present context. Below as in the previous subsection, we take $P^{333}_i$ to be the hyperbolic prism that is the fundamental domain for the orbifold fundamental group of $O^{333}_i$, given the Lakeland-Roth embedding in $\mathbb{H}^3$.

\begin{lemma}\label{max cusp} For $i\in\{1,2,3\}$, the maximal horoball cusp neighborhood $\mathcal{N}\subset O^{333}_i$ has
\[ \mathrm{vol}(\mathcal{N}) = \frac{3+2\sqrt{2}}{8\sqrt{3}} \approx 0.4206304962 \]
and a unique point of self-tangency, contained in the projection to $O^{333}_i$ of the edge of intersection between the compact quadrilateral and non-compact triangular faces of $P^{333}_i$.\end{lemma}

\begin{proof}
    By \Cref{maximal general}, the maximal horoball cusp neighborhood in $O^{333}_i$ is the projection of the horoball $B_{\infty}$ centered at $\infty$ of height $\max\{1,r\}$, where $r$ is the radius of the Euclidean hemisphere containing the compact triangular face $T^{(c)}_i$ of $P^{333}_i$. We computed the radii of this non-unit hemisphere for the relevant $i$ in the previous subsection, respectively obtaining $R_1\approx 0.687$, $R_2 \approx 0.317$ and $R_3\approx 0.349$. Therefore $B_{\infty}$ is at height $1$ and does not intersect the Euclidean hemisphere of radius $R_i$; or, in particular, the face $T^{(c)}_i$ of $P^{333}_i$ that it contains. Its unique point of tangency with the unit hemisphere centered at the origin is the point directly above the origin, which lies in the edge shared by the compact quadrilateral face and non-compact triangular face of $P^{333}_i$ (compare Figure \ref{rigeur}).

    Reflection in the compact quadrilateral face $Q$ of $P^{333}_i$ sends $B_{\infty}$ to the horosphere centered at the origin and tangent to $Q$ at $B_{\infty}\cap Q$, which is fixed by this reflection. Since the reflection is a generator for the orbifold fundamental group of $O^{333}_i$, these two horoballs belong to the same orbit and hence their point of tangency projects to a point of self-tangency in $O^{333}_i$. This is the unique point of self-tangency since $B_{\infty}$ has no other points of intersection with compact faces of the fundamental domain $P^{333}_i$ for the orbifold group.

    Because $B_{\infty}$ is at height $1$, the volume of its intersection with $P^{333}_i$ is half the area of the Euclidean equilateral triangle $\Delta$ which is the projection of $B_{\infty}\cap P^{333}_i$ to the plane. The vertices of $\Delta$ are $(0,y_1)$ and $(0,y_2)$, so its area is $(y_1-y_2)^2\frac{\sqrt{3}}{4}$, where $y_1 = 1/\sqrt{3}$ and $y_2 = -\sqrt{2}/\sqrt{3}$. This gives the volume of $\mathcal{N}$.
\end{proof}

We will use the horoball packing determined by the maximal cusp to prove:

\begin{prop}\label{no cover}
    For $i=2$ or $3$, $O^{333}_i$ as defined in \Cref{table names} does not non-trivially cover another orbifold.
\end{prop}

\begin{figure}
\begin{tikzpicture}

\begin{scope}[scale=1.5]

\draw [thin] (0,0) circle [radius=1];
\draw [very thick] (0,0.577) -- (0,-0.8165) -- (1.2071,-0.1196) -- (0,0.577) -- (1.2071,1.2742) -- (1.2071,-0.1196);
\draw [very thick] (1.2071,1.2742) -- (0,1.9712) -- (0,0.577) -- (-1.2071,1.2742) -- (0,1.9712);
\draw [very thick] (-1.2071,1.2742) -- (-1.2071,-0.1196) -- (0,0.577);
\draw [very thick] (-1.2071,-0.1196) -- (0,-0.8165);

\draw (0,0) circle [radius=0.04];
\node [left] at (0,0) {\scriptsize{$o$}};
\draw (0.5,0.866) circle [radius=0.04];
\node [above] at (0.5,0.866) {\scriptsize $\rho_u(o)$};
\draw (-0.5,0.866) circle [radius=0.04];
\node [above] at (-0.45,0.92) {\scriptsize $\rho_y(\rho_u(o))$};

\fill (0,0.577) circle [radius=0.01];
\node [right] at (0,0.577) {\scriptsize{$(0,y_1)$}};
\fill (0,-0.8165) circle [radius=0.01];
\node [right] at (0,-0.8165) {\scriptsize{$(0,y_2)$}};

\node at (0.5,-0.1196) {$\Delta$};

\end{scope}

\end{tikzpicture}

\caption{Some translates of $\Delta$ and points of tangency.}
\label{bit o tiling}
\end{figure}
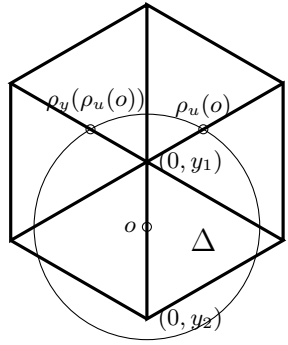

\begin{proof}[Proof of \Cref{no cover}]
    For an orbifold cover $O^{333}_i\to O$, a maximal cusp neighborhood in $O$ pulls back to a maximal cusp neighborhood in $O^{333}_i$. But since $O^{333}_i$ is one-cusped, the cusp neighborhood $\mathcal{N}$ described in \Cref{max cusp} is the unique such neighborhood in $O^{333}_i$. Let $O = \mathbb{H}^3/\Pi$, for $\Pi$ containing $\prgr^{333}_i$. It follows that the packing $\mathcal{B}$ of $\mathbb{H}^3$ by $\prgr^{333}_i$-translates of the horoball $B_{\infty}$ of \Cref{max cusp} is also $\Pi$-invariant. In particular, the collection of ``full-size'' horoballs of $\mathcal{B}$---those tangent to $B_{\infty}$---is invariant under the action of the peripheral subgroup $\Lambda$ of $\Pi$ stabilizing $\infty$.

    Again since $O^{333}_i$ is one-cusped, the degree of the cover $O^{333}_i\to O$ equals the degree of its restriction to the cusp neighborhood $\mathcal{N}$, covering a cusp neighborhood in $O$. This degree is the index $[\Lambda:\Lambda^{333}_i]$ of the peripheral subgroup $\Lambda^{333}_i$ of $\prgr^{333}_i$ in $\Lambda$. We will show that this index is $1$, and hence that the cover is trivial, by establishing that $\Lambda^{333}_i$ is the full symmetry group of the collection of points of tangency of full-size horoballs with $B_{\infty}$.

    To start, we note that because $\mathcal{N}$ has a unique point of self-tangency (by \Cref{max cusp}), $\Lambda^{333}_i$ acts transitively on this collection. Equivalently, as in the proof of \Cref{max cusp} the fundamental domain $\Delta$ for the action of $\Lambda^{333}_i$ on $B_{\infty}$ contains a unique point $o$ of tangency with a full-size horoball---directly above the origin. The reflection $\rho_y$ in the side of $\Delta$ that projects into the $y$-axis fixes $o$. We claim that the image of $o$ under reflection $\rho_u$ in the other side of $\Delta$ containing $(0,y_1)$ (the one that projects into the line with equation $y_u(x)$ from \Cref{subsec: volumes}) is one of exactly two closest points of tangency with full-sized horoballs, among all of them, having minimal distance from $o$; the other one being $\rho_y(\rho_u(o))$.

    Before addressing the claim, let us note that it implies that $\Lambda^{333}_i$ is the full symmetry group: any symmetry of the collection of full-sized horoballs can be right-multiplied by an element of $\Lambda^{333}_i$ so that the product fixes $o$, by transitivity; then, after further right-multiplying by $\rho_y$ if necessary, so that it also fixes each of the two closest points of tangency to $o$. But the only isometry of $\partial B_{\infty}$ fixing all three members of this collection pointwise is the identity, so we have expressed the original symmetry as the inverse of an element of $\Lambda^{333}_i$.

    The proof of the claim is essentially contained in Figure \ref{bit o tiling}. Each of $\rho_u(o)$ and $\rho_y(\rho_u(o))$ is at distance $1$ from $o$, by a Euclidean trigonometry calculation. The circle of radius $1$ around $o$ is entirely contained in the union of translates of $\Delta$ by the subgroups of $\Lambda^{333}_i$ fixing $(0,y_1)$ and $(0,y_2)$ (with only the former union pictured), so the other possible points of tangency within that circle are points of the orbit of $o$ under the subgroup fixing $(0,y_2)$ (again, not pictured). But these have distance $\sqrt{2}>1$ from $o$, and the claim is proved.
\end{proof}

Our next result follows from \Cref{no cover} and the non-arithmeticity of $\doh^{333}_2$ and $\doh^{333}_3$. Below, two orbifolds are \emph{commensurable} if they have a common cover, of finite degree over each; dually, two discrete groups of isometries of $\mathbb{H}^3$ are \emph{commensurable} if they have a common finite-index subgroup. The \emph{commensurator} of $\Gamma$ is the group of isometries $f$ such that $\Gamma$ and $f\Gamma f^{-1}$ are commensurable (it can be directly shown to be a group). The \emph{commensurability class} of an orbifold is the set of all orbifolds commensurable to it; and ``minimality'' is measured  with respect to orbifold covers.

\begin{corollary}\label{commensurator}
    The commensurator of $\Pi^{333}_i$ is $\Pi^{333}_i$, and $O^{333}_i$ is the unique minimal orbifold in its commensurability class. In particular, $O^{333}_2$ and $O^{333}_3$ are not commensurable.
\end{corollary}

\begin{proof}
Since $O^{333}_i$ is non-arithmetic (cf. \Cref{non-arithmetic}), by Margulis's arithmeticity theorem (see \cite[Theorem~10.3.5]{book}) the commensurator of $\prgr^{333}_i$ is itself a discrete group $\Gamma_i$, and the commensurability class of $O^{333}_i$ has $\mathbb{H}^3/\Gamma_i$ as its unique minimal orbifold. $O^{333}_i$ therefore covers this orbifold, so by \Cref{no cover}, it equals the minimal orbifold and $\prgr^{333}_i = \Gamma_i$. Since $O^{333}_2$ and $O^{333}_3$ are non-isometric, it follows that they are also not commensurable.
\end{proof}

\section{Proof of the main theorem}\label{sec:main_proof}

This section is primarily devoted to establishing the following result.

\begin{theorem}\label{main_tech_thm}
Let $\doh^{333}_2$ and $\doh^{333}_3$ be the orientable prism orbifolds determined by the relevant rows of \Cref{table333}. For each $i\in\{2,3\}$ and $j\in\{1,2\}$, the right-permutation representation $\sigma_{i,j}$ of \Cref{table_pr} determines a degree-24 cover $M_{i,j}\to \doh^{333}_i$ by a one-cusped manifold $M_{i,j}$, with $H_1(M_{i,j})\cong\mathbb{Z}$, such that $M_{i,1}$ is a knot complement in the lens space $L(13,3)$, and $M_{i,2}$ is a knot complement in $L(22,5)$ for each $i$. In consequence:
\begin{itemize}
\item There is a knot complement $\widetilde{M}_{2,1}$ in $\Sthree$ covering $\doh^{333}_2$ with degree $312=24*13$.
\item There is a knot complement $\widetilde{M}_{2,2}$ in $\Sthree$ covering $\doh^{333}_2$ with degree $528=24*22$.
\item There is a knot complement $\widetilde{M}_{3,1}$ in $\Sthree$ covering $\doh^{333}_3$ with degree $312=24*13$.
\item There is a knot complement $\widetilde{M}_{3,2}$ in $\Sthree$ covering $\doh^{333}_3$ with degree $528=24*22$.
\end{itemize}
The manifolds $\widetilde{M}_{2,j}$ are incommensurable, and in particular non-isometric, to the $\widetilde{M}_{3,j}$. 
\end{theorem}

The remainder of this section establishes the assertions of \Cref{main_tech_thm}. In \Cref{subsec:the M_ij} we establish existence of the $\sigma_{i,j}$ and $M_{i,j}$ and, in \Cref{cell decomp}, give polyhedral decompositions of the $M_{i,j}$ by copies of the corresponding doubled prisms $\dpr^{333}_i$; then establish that the $M_{i,j}$ are one-cusped manifolds. In \Cref{spinal subsec} we describe a two-dimensional spine for the $M_{i,j}$ and how it is used to show that $H_1(M_{i,j})\cong\ZZ$ for each $i$ and $j$. In \Cref{subsec: lens space} we convert the polyhedral decompositions of the $M_{i,j}$ to triangulations, then to ideal triangulations, using the tools Regina \cite{regina} and SnapPy \cite{SnapPy} of computational three-manifold topology to establish that they are knot complements in lens spaces; then use this to prove \Cref{main_tech_thm}.

First, however, we record some consequences of \Cref{main_tech_thm} and our prior results. 

\begin{corollary}\label{cor:HidSymKnots}\HidSymKnots\end{corollary}

\begin{proof}
We observe first that by Proposition 9.1 of \cite{NeumannReidArith}, each $\widetilde{M}_{i,j}$ above has hidden symmetries since the prism orbifolds $\doh^{333}_2$ and $\doh^{333}_3$ each have a rigid $(3,3,3)$-turnover cusp. That these are distinct from the dodecahedral and figure-eight knot complements can be seen by considering hyperbolic volume: the figure-eight knot complement has volume just over $2$, and the dodecahedral knot complements each have volume approximately $41.16$.
\end{proof}

\begin{corollary}\label{cor: Not236Knots}\NotTTSKnots\end{corollary}

\begin{proof}
    By \Cref{commensurator} the orbifolds $O^{333}_2$ (covered by $\widetilde{M}_{2,1}$ and $\widetilde{M}_{2,2}$) and $O^{333}_3$ (covered by $\widetilde{M}_{3,1}$ and $\widetilde{M}_{3,2}$) are each minimal in their commensurability classes. This implies in particular that there is no two- or six-torsion in either commensurability class, and hence that no $\widetilde{M}_{i,j}$ covers a $(2,3,6)$-cusped orbifold.
\end{proof}

\begin{corollary}\label{volumes}
For all $i\in\{2,3\}$ and $j\in\{1,2\}$, the $M_{i,j}$ from \Cref{main_tech_thm} have equal volume. This is approximately 
$45.0273570343769$, accurate to at least $7$ digits right of the decimal point. The knot complements $\widetilde{M}_{i,1}$ in $\mathbb{S}^3$ have equal volumes, approximately $585.3556414469$; likewise the $\widetilde{M}_{i,2}$ have equal volumes, approximately $990.6018547563$.

For each $i$ and $j$, let $S_{i,j}\subset M_{i,j}$ be the separating totally geodesic surface supplied by \Cref{geodesic surface}. The compact and non-compact submanifolds $M_{i,j}^+$ and $M_{i,j}^-$ of $M_{i,j}$ bounded by $S_{i,j}$, respectively, have approximate volumes $12.73430184$ and $32.2930552$.
\end{corollary}

\begin{proof} By \Cref{scis cong}, $O^{333}_2$ and $O^{333}_3$ have equal volumes. Since the $M_{i,j}$ are each degree-$24$ covers of one or the other, it follows that their volumes are all equal. Their volume recorded above is obtained by multiplying those of the $\doh^{333}_i$ from \Cref{rem:quot_vol} by $24$. The volumes of the $\widetilde{M}_{i,j}$ are obtained by further multiplying by the appropriate values ($13$ and $22$). 

For each $i,j$, the surface $S_{i,j}$ supplied by \Cref{geodesic surface} is separating, since neither nine-tuple $\be$ associated to either $\doh^{333}_i$ has $a_7=a_8=a_9 = 2$. 
    For $i=2$ or $3$, as in the proof of \Cref{corandreev}, the triangle $T\subset P^{333}_i$ that tiles $S_{i,j}$ divides $P^{333}_i$ into sub-prisms $P^+$ and $P^-$. As discussed in the proof of \Cref{geodesic surface}, the compact and non-compact submanifolds $M_{i,j}^+$ and $M_{i,j}^-$ of $M_{i,j}$, respectively, that are bounded by $S_{i,j}$ divide into copies of the respective doubles $\dpr^+$ and $\dpr^-$ of $P^+$ and $P^-$. The volume of $M_{i,j}^{\pm}$ is therefore $48$ times that of $P^{\pm}$.

    The approximate volumes above follow from \Cref{rem:quot_vol}. It is implicit in the computations of \Cref{subsec: volumes} leading up to that Remark that for $i=2$ or $3$, $P^- = P^{333}_1$. We now argue this explicitly. Each edge of $P^-$ that is not an edge of $T$ is contained in one of $P^{333}_i$, and shares that edge's dihedral angle. Each edge of $P^-$ that is contained in $T$ has dihedral angle $\pi/2$, so the nine-tuple $\be^-$ recording the labels of edges of $T$ is as follows:
\[ \be' = (3,3,2,3,3,4,2,2,2). \]
Here the first six labels are for the edges contained in those of $P^{333}_i$---which have identical labels for $i=2$ and $3$---and the last three are of the edges of $T$. Comparing with \Cref{table333}, we see that this nine-tuple matches the one defining $P^{333}_1$. Multiplying the volume of $P^{333}_1 = P^-$ recorded in \Cref{rem:quot_vol} by $48$ thus yields the approximate value of $32.2930552$ for the volume of $M_{i,j}^-$ above. The volume of $M_{i,j}^+$ is the volume of $M_{i,j}$ minus that of $M_{i,j}^-$.
\end{proof}

\begin{remark}\label{knot vol}
Using SnapPy \cite{SnapPy} with the ideal triangulations of the $M_{i,j}$ described below in the proof of \Cref{lem:lifts}, we can independently compute their approximate volumes to within a rigorously determined interval of accuracy. Each such computation agrees with that of \Cref{volumes} to at least 7 digits right of the decimal point, within this interval.
\end{remark}

\subsection{Degree-24 covers and their properties.}\label{subsec:the M_ij} In \Cref{main_tech_thm}, the manifolds $M_{i,j}$ are described in terms of a ``permutation representation'' of the generators of certain orientable orbifold groups $\dprgr(\be)$, presented as in (\ref{able}), on right cosets of finite-index subgroups. We begin by defining the term.

\begin{definition}\label{permrep}
    The \emph{(right-)permutation representation} of a group $G$ associated to an index-$n$ subgroup $H<G$ and a set $\mathcal{S} = \{Hg_0,\hdots,H g_{n-1}\}$ of distinct right-coset representatives for $H$ in $G$ is the map $\sigma\co G\to S_n$, where $S_n$ is the symmetric group on $n$ letters, that records the actions of elements of $G$ on $\mathcal{S}$ by right-multiplication. That is, for any $g\in G$ and $i\in\{0,\hdots,n-1\}$, $\sigma(g)(i) = j$, where $(H g_i)g = H g_j$.
\end{definition}

Note that because a \emph{right-}action is used to define the permutation representation $\sigma$ above, it reverses the order of multiplication; ie.~$\sigma(gh) = \sigma(h) \circ \sigma(g)$ for any $g$ and $h$ in $G$. We now construct the $M_{i,j}$ using right-permutation representations $\sigma_{i,j}$ defined in \Cref{table_pr}.

\begin{table}
\centering
\centering
\begin{tabular}{ccrl}
\toprule
\textbf{i} & name & & permutation representation \\
\midrule
$\mathbf{2}$ & $\sigma_{2,1}$ & $x\mapsto$ & $[1, 2, 0, 12, 10, 19, 18, 3, 20, 8, 16, 6, 7, 17, 13, 5, 4, 14, 11, 15, 9, 22, 23, 21]$ \\ 
    & & $y\mapsto$ & $[3, 6, 10, 4, 0, 20, 7, 1, 19, 14, 11, 2, 18, 15, 23, 22, 12, 8, 16, 17, 21, 5, 13, 9]$ \\ 
    & & $z\mapsto$ & $[2, 8, 5, 13, 17, 0, 22, 12, 9, 1, 4, 18, 23, 14, 3, 19, 15, 10, 20, 16, 11, 6, 21, 7]$ \\ 
    & & $w\mapsto$ & $[1, 0, 9, 15, 18, 8, 12, 22, 5, 2, 20, 17, 6, 16, 19, 3, 13, 11, 4, 14, 10, 23, 7, 21]$ \\
\midrule
$\mathbf{2}$ & $\sigma_{2,2}$ & $x\mapsto$ & $[1, 2, 0, 13, 11, 17, 12, 3, 14, 8, 6, 15, 10, 7, 9, 4, 23, 18, 5, 20, 21, 19, 16, 22]$ \\ 
    & & $y\mapsto$ & $[3, 6, 11, 4, 0, 19, 7, 1, 20, 16, 2, 10, 15, 12, 5, 13, 18, 23, 9, 14, 22, 17, 8, 21]$\\ 
    & & $z\mapsto$ & $[2, 8, 5, 14, 16, 0, 20, 19, 9, 1, 22, 4, 6, 3, 13, 18, 11, 15, 17, 21, 12, 7, 23, 10]$ \\ 
    & & $w\mapsto$ & $[6, 10, 12, 11, 13, 20, 0, 15, 23, 22, 1, 3, 2, 4, 16, 7, 14, 19, 21, 17, 5, 18, 9, 8]$ \\
\midrule
$\mathbf{3}$ & $\sigma_{3,1}$ & $x\mapsto$ & $[1, 2, 0, 13, 10, 18, 21, 3, 14, 8, 15, 6, 23, 7, 9, 4, 20, 16, 19, 5, 17, 11, 12, 22]$ \\ 
    & & $y\mapsto$ & $[3, 6, 10, 4, 0, 20, 7, 1, 5, 17, 11, 2, 19, 21, 12, 13, 18, 22, 23, 14, 8, 15, 9, 16]$ \\ 
    & & $z\mapsto$ & $[2, 8, 5, 7, 16, 0, 19, 14, 9, 1, 23, 21, 22, 12, 3, 10, 17, 4, 6, 18, 11, 20, 13, 15]$ \\ 
    & & $w\mapsto$ & $[6, 3, 12, 1, 7, 21, 0, 4, 15, 16, 14, 19, 2, 20, 10, 8, 9, 23, 22, 11, 13, 5, 18, 17]$ \\
\midrule
$\mathbf{3}$ & $\sigma_{3,2}$ & $x\mapsto$ & $[1, 2, 0, 13, 10, 17, 19, 3, 14, 8, 15, 6, 20, 7, 9, 4, 23, 18, 5, 11, 21, 12, 16, 22]$ \\ 
    & & $y\mapsto$ & $[3, 6, 10, 4, 0, 12, 7, 1, 20, 16, 11, 2, 14, 19, 5, 13, 18, 23, 9, 15, 22, 17, 8, 21]$ \\ 
    & & $z\mapsto$ & $[2, 8, 5, 14, 16, 0, 20, 12, 9, 1, 4, 22, 21, 3, 13, 18, 10, 15, 17, 6, 19, 7, 23, 11]$ \\ 
    & & $w\mapsto$ & $[6, 3, 12, 1, 7, 10, 0, 4, 18, 22, 5, 14, 2, 17, 11, 23, 20, 13, 8, 21, 16, 19, 9, 15]$
\end{tabular}
\caption{Permutation representations for $\dprgr^{333}_i$, $i = 2$ or $3$.}
\label{table_pr}
\end{table}

\begin{lemma}\label{itsapermrep}
    For $i=2$ or $3$, let $\doh^{333}_i = \mathbb{H}^3/\dprgr^{333}_i$ be the orientable prism orbifold double-covering $O^{333}_i$ from Table \ref{table333}. Specializing the presentation for $\dprgr^{333}_i$ from (\ref{able}) to:\begin{align*}
    \dprgr^{333}_2\ & \cong\ \langle x, y,z, w\, |\, x^{3}, y^{3}, z^{3}, w^{2}, 
        (y^{-1}x)^{2}, (z^{-1}x)^{3}, (z^{-1}y)^{4}, (y^{-1}w)^{3}, (z^{-1}w)^{2} \rangle,\mbox{ and}\\
    \dprgr^{333}_3\ & \cong\ \langle x, y,z, w\, |\, x^{3}, y^{3}, z^{3}, w^{2}, 
        (y^{-1}x)^{2}, (z^{-1}x)^{3}, (z^{-1}y)^{4}, (y^{-1}w)^{2}, (z^{-1}w)^{3} \rangle,
    \end{align*}
    each map $\sigma_{i,j}\co \{x,y,z,w\}\to S_{24}$ recorded in \Cref{table_pr} extends to a right-permutation representation of $\dprgr^{333}_i$ with image acting transitively on $\{0,1,\hdots,23\}$. The stabilizer $G_{i,j}$ of $0$ therefore has index $24$ in $\dprgr^{333}_i$, and $\sigma_{i,j}$ records the generators' action by right-multiplication on right cosets of $G_{i,j}$.
\end{lemma}

\begin{definition}\label{first manifold covers}
    For each $i\in\{2,3\}$ and $j\in\{1,2\}$, taking $G_{i,j}<\dprgr^{333}_i$ as in \Cref{itsapermrep}, let $M_{i,j} = \mathbb{H}^3/G_{i,j}$ be the corresponding cover of $\doh^{333}_i$.
\end{definition}

\begin{proof} We define the right-representation extending $\sigma_{i,j}$ as follows. For a word $g = \chi_1^{\alpha_1}\cdots \chi_k^{\alpha_k}$ in the generators $\{x,y,z,w\}$, with each ${\alpha_i}\in\mathbb{Z}$, and $n\in\{0,1,\hdots,23\}$, let
\[ \sigma_{i,j}(g)(n) = \left(\sigma_{i,j}(\chi_k)^{\alpha_k} \circ \hdots \circ \sigma_{i,j}(\chi_1)^{\alpha_1}\right)(n).\]
(Note that the composition here is in the opposite order from the usual product in $S_{24}$.) Well-definedness of this map on $\dprgr^{333}_i$ is equivalent to it sending each relator to the identity, which can be checked directly. See eg.~\Cref{M21 permreps} below for the case of $\sigma_{2,1}$. There, writing each relation from the presentation for $\dprgr^{333}_2$ given above in the form $g_j^{a_{i_j}}$ for $j\in\{1,\hdots,9\}$, cycle decompositions for the images of the $g_j$ are given. Each consists of $a_{i_j}$-cycles, verifying the $(i,j) = (2,1)$-case.

Another straightforward check establishes that each $\sigma_{i,j}$ determines a transitive action on $\{0,1,\hdots,23\}$. Therefore for any $i,j$, the stabilizer $G_{i,j}$ of $0$ in $\dprgr^{333}_i$ has index $24$. In particular, elements $g, h\in\dprgr^{333}_i$ determine the same right coset of $G_{i,j}$ if and only if $gh^{-1}\in G_{i,j}$; ie.~if and only if $0.g = 0.h$. For a complete set $\{g_0,\hdots,g_{23}\}$ of right coset representatives, where $0.g_n = n$ for each $n$, we then have $0.(g_nx) = n.x$, so $g_{n.x}$ represents the right coset $\left(G_{i,j}g_n\right)x$, and likewise for the other generators.\end{proof}

We now describe polyhedral decompositions of arbitrary covers of orbifolds $\doh(\be)$ determined by permutation representations of their orbifold fundamental groups $\dprgr(\be)$.

\begin{lemma}\label{cell decomp} Suppose $\dprgr(\be)$ is an orientable prism group, presented as in (\ref{able}), for a nine-tuple $\be$ from Table \ref{table236} or \ref{table333}. Label the faces of the doubled prism $\dpr(\be)$ on which the generators $x,y,z,w$ of $\dprgr(\be)$ act as face-pairings according to Figure \ref{numberings} (right), i.e.~as $0_+$, $1_+$, $2_+$, $3_+$ and $0_-$, $1_-$, $2_-$, $3_-$ so that $x$ takes $0_-$ to $0_+$, $y$ takes $1_-$ to $1_+$, $z$ takes $2_-$ to $2_+$, and $w$ takes $3_-$ to $3_+$.

Further suppose that \( H \) is an index-\( n \) subgroup of \( \dprgr(\be) \), and for a complete set $\mathcal{S} = \{g_0,\hdots,g_{n-1}\}$ of right coset representatives for $H$, let \( \sigma_x, \sigma_y, \sigma_z, \sigma_w \) be the values of the associated right-permutation representation $\sigma$ from \Cref{permrep}, on the generators. Then \( \mathbb{H}^3 / H \) is isometric to the orbifold obtained from distinct copies \( \dpr_0, \dots, \dpr_{n-1} \) of $\dpr(\be)$ by the equivalence relation generated by the face-pairing below. For every $i \in \{0, 1,\hdots, n-1\}$:

\begin{itemize}   
    \item Identify face \( 0_+ \) of \( \dpr_i \) to face \( 0_- \) of \( \dpr_j \), where \( j = \sigma_x(i) \).
    \item Identify face \( 1_+ \) of \( \dpr_i \) to face \( 1_- \) of \( \dpr_j \), where \( j = \sigma_y(i) \).
    \item Identify face \( 2_+ \) of \( \dpr_i \) to face \( 2_- \) of \( \dpr_j \), where \( j = \sigma_z(i) \).
    \item Identify face \( 3_+ \) of \( \dpr_i \) to face \( 3_- \) of \( \dpr_j \), where \( j = \sigma_w(i) \).
\end{itemize}

These identifications use the appropriate generator, pre- and post-composed by the appropriate marking. For example, the first identification map is \( g_j x g_i^{-1} \), where \( g_i: \dpr(\be) \rightarrow \dpr_i \) and \( g_j: \dpr(\be) \rightarrow \dpr_j \) are the markings. In particular, for each $f\in\{0,1,2,3\}$ they take the edge $f_+\cap f_-$ of $\dpr_i$ to the corresponding edge of $\dpr_j$.
\end{lemma}

\begin{figure}[ht]
\begin{tikzpicture}

\begin{scope}[xshift=-2.5in]
    \coordinate (A) at (0,0,0);      
    \coordinate (B) at (3,0,0);      
    \coordinate (C) at (3,3,0);      
    \coordinate (D) at (0,3,0);      
    \coordinate (E) at (0,0,3);      
    \coordinate (F) at (3,0,3);      
    \coordinate (G) at (3,3,3);      
    \coordinate (H) at (0,3,3);      

    \draw[thick, orange] (C) -- (D);

    \draw[thick] (A) -- (0,1.1,0); 
    \draw[white, line width=3pt] (0,1.1,0) -- (0,1.5,0);
    \draw[thick] (0,1.5,0) -- (0,1.77,0);    
    \draw[white, line width=3pt] (0,1.77,0) -- (0,1.9,0);
    \draw[thick] (0,1.9,0) -- (D);

    \draw[thick, green] (A) -- (1.77,0,0);
    \draw[white, line width=3pt] (1.77,0,0) -- (1.9,0,0);
    \draw[thick, green] (1.9,0,0) -- (B);

    \draw[thick, violet] (H) -- (G);

    \draw[thick, brown] (E) -- (0,1.35,3);
    \draw[white, line width=3pt] (0,1.35,3) -- (0,1.66,3);
    \draw[thick, brown] (0,1.65,3) -- (H);

    \draw[thick, brown] (B) -- (3,1.3,0);
    \draw[white, line width=3pt] (3,1.3,0) -- (3,1.6,0);
    \draw[thick, brown] (3,1.6,0) -- (C);

    \draw[thick] (F) -- (3,1.55,3);
    \draw[white, line width=3pt] (3,1.55,3) -- (3,1.85,3);
    \draw[thick] (3,1.85,3) -- (G);
    
    \draw[thick, green] (A) -- (E);

    \draw[thick, blue] (B) -- (F); 

    \draw[thick, blue] (E) -- (F);

    \draw[thick, violet] (C) -- (G);

    \draw[thick, orange] (D) -- (H); 
    
    \draw[thick] (D) -- (1.35,3,1.35);
    \draw[white, line width=3pt] (1.35,3,1.35) -- (1.65,3,1.65);
    \draw[thick] (1.65,3,1.65) -- (G);

    \draw[thick] (A) -- (1.35,0,1.35);
    \draw[white, line width=3pt] (1.35,0,1.35) -- (1.65,0,1.65);
    \draw[thick] (1.65,0,1.65) -- (F);

    \node[blue] at (1.5,0,3) { $1_+$};
    \node[green!60!black] at (0,0.1,1.5) { $2_+$};
    \node[green!60!black] at (1.4,0,0) { $2_-$};
    \node[black] at (3,1.7,3) { $4$};
    \node[blue] at (3.15,0,1.5) { $1_-$};
    \node[black] at (1.5,0,1.5) { $0$};
    \node[black] at (0,1.3,0) { $3$};
    \node[brown] at (3.1,1.4,0) { $5_-$};
    \node[brown] at (0,1.5,3) { $5_+$};
    \node[black] at (1.5,3,1.5) { $6$};
    \node[violet] at (3,3,1.5) { $7_-$};
    \node[violet] at (1.5,3.05,3.2) { $7_+$};
    \node[orange] at (1.5,2.95,0) { $8_-$};
    \node[orange] at (0.01,3,1.5) { $8_+$};

    \fill[white] (F) circle (3pt);
    \draw[thick] (F) circle (3pt);
\end{scope}

\begin{scope}

    \coordinate (A) at (0,0,0);      
    \coordinate (B) at (3,0,0);      
    \coordinate (C) at (3,3,0);      
    \coordinate (D) at (0,3,0);      
    \coordinate (E) at (0,0,3);      
    \coordinate (F) at (3,0,3);      
    \coordinate (G) at (3,3,3);      
    \coordinate (H) at (0,3,3);      

    \draw[thick] (A) -- (B) -- (C) -- (D) -- cycle;  

    \draw[thick] (A) -- (0,1.77,0);  
    \draw[white, line width=3pt] (0,1.77,0) -- (0,1.9,0);
    \draw[thick] (0,1.9,0) -- (D); 

    \draw[thick] (A) -- (1.77,0,0); 
    \draw[white, line width=3pt] (1.77,0,0) -- (1.9,0,0); 
    \draw[thick] (1.9,0,0) -- (B);

    \draw[thick] (E) -- (F) -- (G) -- (H) -- cycle;  

    \draw[thick] (A) -- (E); 
    \draw[thick] (B) -- (F); 
    \draw[thick] (C) -- (G); 
    \draw[thick] (D) -- (H); 
    \draw[thick] (D) -- (G);
    \draw[thick] (A) -- (F);

    
    \fill [color=black, opacity=0.1] (A) -- (E) -- (F);
    \fill [color=black, opacity=0.1] (A) -- (B) -- (F);


    \fill [color=green, opacity=0.1] (D) -- (G) -- (H);
    \fill [color=green, opacity=0.1] (D) -- (G) -- (C);

    
    \node[blue] at (-0.2,1,0.2) { $1_+$};    
    \node[red] at (1.7,0.65,0.2) { $2_+$};   
    \node[green!60!black] at (0.75,2.3,0.2) { $3_+$};    
    \node[blue] at (0.4,1.2,0.2) { $1_-$};    
    \node[red] at (2.25,0.85,0.2) { $2_-$};   
    \node[green!60!black] at (1.5,2.5,0.2) { $3_-$};   
    \node[black] at (0.5,-0.6,0) { $0_+$};    
    \node[black] at (1.6,0.2,1.5) { $0_-$};    
    
    \fill[white] (F) circle (3pt);
    \draw[thick] (F) circle (3pt);
\end{scope}

\end{tikzpicture}
\caption{Edge numbering (left); Face numbering (right - note that the back faces are numbered by $1_+$ and $1_-$  while front are numbered by $2_+$ and $2_-$).}
\label{numberings}
\end{figure}
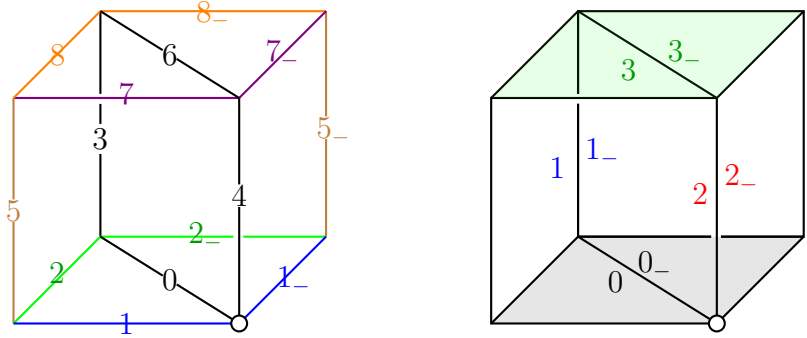

\begin{proof}
Let \( \dpr_i = g_i(\dpr(\be)) \) for every $i$, where \( \{g_0, \dots, g_{n-1}\} \) is the complete set of right coset representatives for \( H \) in \( \dprgr(\be) \). Since \( \dpr(\be) \) is a fundamental domain for \( \dprgr(\be) \), we have \( \mathbb{H}^3 = \bigcup_{g \in \dprgr(\be)} g(\dpr(\be)) \). Thus, for any \( g \in \dprgr(\be) \), expressing $g$ as \( h g_i \) for some \( h \in H \) and coset representative \( g_i \), we have \( g(\dpr(\be)) = h(\dpr_i) \). Hence, $\bigcup_{i=0}^{n-1} \dpr_i$ maps onto \( \mathbb{H}^3 / H \) under the covering projection from $\mathbb{H}^3$.

For some \( h, h' \in H \) and \( i, j \in \{0, \dots, n - 1\} \), if the interior of \( h(\dpr_i) \) meets that of \( h'(\dpr_j) \), then again since \( \dpr(\be) \) is a fundamental domain for \( \dprgr(\be) \), we must have that \( h(\dpr_i) = h'(\dpr_j) \) and \( h g_i = h' g_j \). Hence \( i = j \), since \( g_i \) and \( g_j \) are coset representatives, and it follows that the map
\[
\bigcup_{i=0}^{n-1} \dpr_i \to \mathbb{H}^3 / H
\]
identifies only points of faces of the \( \dpr_i \).

The face gluings between \( \dpr_i \)'s are determined by the permutation representations of the generators of \( \dprgr(\be) \). For example, for a fixed \( i \) $\in \{0, 1, 2, ..., n-1\}$, taking \( j = \sigma_z(i) \), we  have \( g_i z = h g_j \) for some \( h \in H \). Noting that $z(\dpr(\be))$ intersects $\dpr(\be)$ along its face $2_-$, which is the face $2_+$ of $\dpr(\be)$, hence that $g_iz(\dpr(\be))$ intersects $\dpr_i = g_i(\dpr(\be))$ along the corresponding face(s), we obtain the following information about
\[ \dpr_j = g_j(\dpr(\be)) = h^{-1}g_iz(\dpr(\be)). \]
Since the projection to $\doh(\be)$ is $H$-invariant, the projection of $\dpr_j$ meets that of $\dpr_i = g_i(\dpr(\be))$ along the face corresponding to $2_-$ in the former and $2_+$ in the latter. The cases of the other generators are identical, and the result is proved.
\end{proof}

\begin{lemma}\label{torsion order} Suppose $\dprgr(\be)$ is an orientable prism group, presented as in (\ref{able}), for a nine-tuple $\be$ from Table \ref{table236} or \ref{table333}, and that $\sigma\co \dprgr(\be)\to S_n$ is the right-permutation representation associated to the right cosets of an index-\( n \) subgroup \( H \) of \( \dprgr(\be) \). Let $\dpr_0,\hdots,\dpr_{n-1}$ be the three-cells of the decomposition of $\mathbb{H}^3/H$ given by \Cref{cell decomp}. Fix an edge $e$ of the prism $P(\be)$, and write the relation of (\ref{able}) associated to the edge cycle of $e$ as $g^{a_j}$ for a word $g$ in the generators $\{x,y,z,w\}$ and some $j\in\{1,\hdots,9\}$. Then for any $k\in\{0,\hdots,n-1\}$, the edge $e_k$ of $\dpr_k$ corresponding to $e$ is identified to the edge $e_{k'}$ of $\dpr_{k'}$ corresponding to $e$ if and only $k$ and $k'$ belong to the same cycle of $\sigma(g)$.

In particular, $\mathbb{H}^3/H$ is a manifold if and only if for each relation $g^{a_j}$, $\sigma(g)$ is the product of $\frac{n}{a_j}$ disjoint $a_j$ cycles.
\end{lemma}

\begin{proof}
    Fix an edge $e$. The first conclusion has two slightly different cases, corresponding to whether or not $e$ is fixed by one of the generators $x$, $y$, $z$, or $w$. Suppose first that it is, so that $g$ equals this generator. For each $k$, let $f_k$ and $\bar{f}_k$ be the faces of $\dpr_k$ containing the edge $e_k$ corresponding to $e$, and respectively belonging to the copies of $P(\be)$ and $\bar{P}(\be)$ in $\dpr_k$. By \Cref{cell decomp}, $\bar{f}_k$ is identified to $f_{k.g}$, taking $e_k$ to $e_{k.g}$. Moreover, through the other face $\bar{f}_{k.g}$ of $\dpr_{k.g}$ containing $e_{k.g}$ we see $\dpr_{k.g^2}$, and so on. Thus all edges identified to the copy of $e$ in $\dpr_k$ are other copies of $e$, and such an edge in $\dpr_{k'}$ is identified to this copy if and only if $k' = k.g^p$ for some power $p$; ie.~if and only if $k'$ belongs to the cycle of $k$ under the $g$-action.

    Now suppose that $e$ is not fixed by a generator, and let $\bar{e}$ be its mirror image under the reflection of $\dpr(\be)$ exchanging $P(\be)$ and $\bar{P}(\be)$. In this case there are four relevant faces of $\dpr(\be)$: the two of $P(\be)$ containing $e$ and their mirror images containing $\bar{e}$. For the sake of tangibility let us suppose that $e$ is the edge labeled $a_6$ of the left-hand prism of \Cref{fig:piOne}; thus $e$ and $\bar{e}$ are the left- and right-most vertical edges of the doubled prism $\dpr(\be)$ pictured on the right side of Figure \ref{fig:piOne}. These two edges comprise the edge cycle determining the relation $(z^{-1}y)^{a_6}$ in the presentation (\ref{able}), so $g = z^{-1}y$ in this case. 
    
    Fixing some $k$, let $e_k$ and $\bar{e}_k$ be the edges of $\dpr_k$ respectively corresponding to $e$ and $\bar{e}$. The face of $\dpr_k$ labeled $2_-$ in \Cref{numberings} is identified to the face of $\dpr_{k.z^{-1}}$ labeled $2_+$, taking $\bar{e}_{k.z^{-1}}$ to $e_k$. Proceeding around this edge class through face $1_+$ of $\dpr_{k.z^{-1}}$, we next encounter $\dpr_{k.(z^{-1}y)}$ with its edge $e_{k.z^{-1}y}$ identified to $\bar{e}_{k.z^{-1}}$, and hence also to $e_k$. This index $k.(z^{-1}y)$ is the next one in the cycle of $\sigma_{2,1}(y)\circ\sigma_{2,1}(z)^{-1}$ containing $k$. Continuing around the edge, we find that in this case the sequence of indices of polyhedra containing that edge contains the cycle elements as every other index, alternating with their images under the action of $z^{-1}$. Furthermore, for $k'$ belonging to the cycle of $k$, the edge of $P_{k'}$ identified to $e_k$ is $e_{k'}$; for $P_{k'.z^{-1}}$ the edge is $\bar{e}_{k'.z^{-1}}$. The first conclusion is thus established for this particular edge class. An entirely analogous argument shows the same for all other edge classes containing two edges.

    For the second conclusion, we note that $\mathbb{H}^3/H$ is a manifold if and only if every point in the interior of an edge $e$ of some $\dpr_k$ has a neighborhood isometric to one in $\mathbb{H}^3$, since this is clear for points in the interior of $\dpr_k$ or one of its faces, and the singular locus of a hyperbolic $3$-orbifold has no isolated vertices. This in turn reduces to checking, for each such edge $e$, that the dihedral angle sum is $2\pi$ around all edges identified to $e$. There are two cases here, matching the two cases above. If $g\in\{x,y,z,w\}$ then $\dpr_k$ has dihedral angle $2\pi/a_j$ at $e$, and all other edges in the cycle of $e$ are isometric copies of the same edge of $\dpr(\be)$; hence with identical dihedral angles. If $g\notin\{x,y,z,w\}$ then $\dpr_k$ has angle $\pi/a_j$ at $e$ and also at its mirror image $\bar{e}$, and in this case copies of $e$ alternate with copies of $\bar{e}$ in the edge cycle. Thus in either case, the total dihedral angle is $2\pi$ if and only if there are $a_j$ copies of $e$ in the edge cycle. By the first conclusion, this occurs if and only if the cycle of $\sigma_{i,j}(g)$ containing $k$ has length $a_{j}$.
\end{proof}

\begin{corollary}\label{torfree}
    For $i=2$ or $3$ and $j=1$ or $2$, taking $G_{i,j}<\dprgr^{333}_i$ as in \Cref{itsapermrep}, the corresponding cover $M_{i,j} = \mathbb{H}^3/G_{i,j}$ of $\doh^{333}_i$ is a manifold.
\end{corollary}

\begin{proof}  In light of \Cref{torsion order}, we must simply check that for each relation in the presentation for $\dprgr^{333}_i$ obtained by specializing (\ref{able}), of the form $g^{a_j}$ for a word $g$ of length $1$ or $2$ in the generators $x,y,z,w$,  
$\sigma_{i,j}(g)$ is the product of $\frac{24}{a_j}$ disjoint $a_j$ cycles. 
This is a computation. (For $M_{2,1}$, cf.~\Cref{M21 permreps}.)
\end{proof}

\begin{lemma}\label{cusp count}  Suppose $\dprgr(\be)$ is an orientable prism group, presented as in (\ref{able}), for a nine-tuple $\be$ from Table \ref{table236} or \ref{table333}, and that $\sigma\co \dprgr(\be)\to S_n$ is the right-permutation representation associated to the right cosets of an index-\( n \) subgroup \( H \) of \( \dprgr(\be) \). The cusps of $\mathbb{H}^3/H$ correspond to orbits of $\langle \sigma(x),\sigma(z)\rangle$ acting on $\{0,\hdots,n-1\}$.
\end{lemma}

\begin{proof} Let $\dpr_0,\hdots,\dpr_{n-1}$ be the three-cells of the decomposition of $\mathbb{H}^3/H$ by copies of $\dpr(\be)$ given by \Cref{cell decomp}. The cusps of $\mathbb{H}^3/H$ are in bijective correspondence with equivalence classes of the ideal vertices of the $\dpr_k$. Since the sole ideal vertex of $\dpr(\be)$ is contained in the faces $1_+$, $1_-$, $2_+$, and $2_-$, the first two being paired by $x$ and the last two by $z$, from \Cref{cell decomp} we have that these equivalence classes correspond to orbits under the action of $\langle \sigma(x),\sigma(z)\rangle$ on $\{0,\hdots,n-1\}$ as claimed.
\end{proof}

\begin{corollary}\label{onecusp}
    For $i=2$ or $3$ and $j=1$ or $2$, taking $G_{i,j}<\dprgr^{333}_i$ as in \Cref{itsapermrep}, the corresponding cover $M_{i,j} = \mathbb{H}^3/G_{i,j}$ of $\doh^{333}_i$ has one cusp.
\end{corollary}

\begin{proof}
     In light of \Cref{cusp count},  for each $(i,j)$ as in the hypotheses, we must check that the action of $\langle \sigma_{i,j}(x),\sigma_{i,j}(z)\rangle$ on $\{0,1,\hdots,23\}$ is transitive. This is straightforward.
\end{proof}

\subsection{Spines, and a first crack at first homology}\label{spinal subsec} Here we will construct a spine for the manifolds $M_{i,j}$ from \Cref{first manifold covers} and use it to prove this subsection's main theorem:

\begin{theorem}\label{homologyZ}
    For $i=2$ or $3$ and $j=1$ or $2$, taking $G_{i,j}<\dprgr^{333}_i$ as in \Cref{itsapermrep}, the corresponding cover $M_{i,j} = \mathbb{H}^3/G_{i,j}$ of $\doh^{333}_i$ has $H_1 (M_{i,j})\cong\mathbb{Z}$.
\end{theorem}

The spine in question is more generally a deformation retract of the underlying space of an arbitrary orbifold cover of an arbitrary $\doh(\be)$.

\begin{lemma}\label{spine} Suppose \( \dprgr(\be) \) is an orientable prism group, presented as in (\ref{able}), for a nine-tuple $\be$ from Table \ref{table236} or \ref{table333}, and that \( H \) is an index-\( n \) subgroup of \( \dprgr(\be) \). Encoding the respective actions of \( x, y, z, w \) by right multiplication on a set of right coset representatives for $H$ as \( \sigma_x, \sigma_y, \sigma_z, \sigma_w \) as in \Cref{cell decomp}, let $\dpr_0,\hdots,\dpr_{n-1}$ be the $3$-cells of that result's cell decomposition of $\mathbb{H}^3/H$. There is a 2-dimensional subcomplex which is a deformation retract of the underlying topological space of $\mathbb{H}^3/H$, with two-cells in bijective correspondence with the faces of the $\dpr_i$ numbered $1_+$ and $3_+$ in Figure \ref{numberings}.
\end{lemma}

\begin{proof} A standard argument shows that there is a deformation retract $F_0\co\dpr(\be)\times I \to \dpr(\be)$  to the union of its compact faces. With $P(\be)$ arranged as discussed in \Cref{subsec: embedding333}, this can be taken to occur along vertical straight lines. If one further takes care to arrange that $F_0(x(p),t) = x(F_0(p,t))$ for all $p$ in the face $0_-$, and $F_0(z(p),t)=z(F_0(p,t))$ for all $p\in 2_-$, then defining $F(g(p),t) = g(F_0(p,t))$ for arbitrary $p\in \dpr(\be)$ and $g\in\dprgr(\be)$ produces a well-defined, $\dprgr(\be)$-equivariant deformation retract from $\mathbb{H}^3$---which is tiled by $\dprgr(\be)$-translates of $\dpr(\be)$---to the union of compact faces of these translates. This induces the deformation retract of $\mathbb{H}^3/H$ claimed in the Lemma's statement.

Note that the compact faces of $\dpr(\be)$ are those numbered $1_-$, $1_+$, $3_-$, and $3_+$ in Figure \ref{numberings}. From the definition of the cell decomposition in \Cref{cell decomp}, each face numbered $1_-$ has the same projection as another numbered $1_+$, and likewise for $3_-$ and $3_+$, so the image of the deformation retract is the union of the projections of faces numbered $1_+$ and $3_+$ only.
\end{proof}

\begin{proof}[Proof of \Cref{homologyZ}]
For each $i$ and $j$ in question, since $M_{i,j}$ is a manifold the spine constructed in \Cref{spine} carries $\pi_1 M_{i,j} \cong G_{i,j}$. We may therefore produce a presentation for $G_{i,j}$ using the standard strategy for a $2$-complex: fix a maximal tree in the one-skeleton, associate a generator to each edge that does not belong to it, and associate a relation to each two-cell by reading off the edges of its boundary. Having done so, $H_1(M_{i,j})$ is obtained as the abelianization of this presentation.

What results is a rather forbiddingly complex complex; for instance, by Lemma \ref{spine} it has 48 two-cells. We have written Python code to compute the presentations in question; the scripts are included in the ancillary files in the directory \path{anc/Finding_low_index_subgroups/code}, where a sequence of modules can be run to reproduce the computations (see the \texttt{README} file therein). These computations are described further below in \Cref{subsec: filtering}. It returns $\mathbb{Z}$ for each $H_1(M_{i,j})$ here.
\end{proof}

\begin{remark} It is also possible to compute the homology of manifold covers of the $\doh(\be)$ by triangulating them and using tools of computational topology as in the next subsection. However, the spine constructed here (actually a coarsening of it) is fundamental to the hand proof laid out in \Cref{sec:M21}. And in any case, since this subsection's computational methods are independent from those of the next, each is useful for corroborating the other.\end{remark}

\subsection{Ideal triangulations and surgery}\label{subsec: lens space} This section establishes our main theorem, on the existence of knot complements covering $\doh^{333}_2$ and $\doh^{333}_3$.

\begin{lemma}\label{triangulation}
    Suppose $\dprgr(\be)$ is an orientable prism group, presented as in (\ref{able}), for a nine-tuple $\be$ from Table \ref{table236} or \ref{table333}, and that $\sigma\co \dprgr(\be)\to S_n$ is the right-permutation representation associated to the right cosets of an index-\( n \) subgroup \( H \) of \( \dprgr(\be) \), such that $\mathbb{H}^3/H$ has $c$ cusps. The underlying topological space of $\mathbb{H}^3/H$ has a triangulation by $6n$ tetrahedra, with some finite vertex classes and $c$ ideal vertex classes, obtained by subdividing the polyhedral decomposition of \Cref{cell decomp}.
\end{lemma}

\begin{remark}\label{triangulation code}
    For $\sigma$, $\dprgr(\be)$, and $H$ as in \Cref{triangulation}, we have coded a method that outputs the triangulation described in that result, in Regina format \cite{regina}, given the input of the right-permutation representation $\sigma$. These scripts are included as ancillary files with this arxiv posting (see \path{anc/Analysis_of_covers/code/PrismCoverBuilder.py}).
\end{remark}

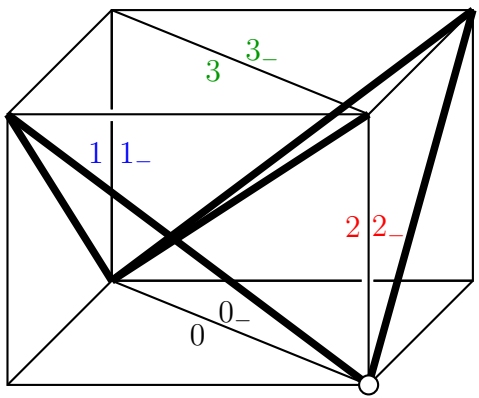
\begin{figure}[h!]
    \centering
    \begin{minipage}[t]{0.48\textwidth}
        \centering
        \vspace{0pt}
        \begin{tikzpicture}[scale=1.2]
        
            \coordinate (A) at (0,0,0);      
            \coordinate (B) at (4,0,0);      
            \coordinate (C) at (4,3,0);      
            \coordinate (D) at (0,3,0);      
            \coordinate (E) at (0,0,3);      
            \coordinate (F) at (4,0,3);      
            \coordinate (G) at (4,3,3);      
            \coordinate (H) at (0,3,3);      
        
            \draw[thick] (A) -- (B) -- (C) -- (D) -- cycle;  

            \draw[thick] (A) -- (0,1.77,0);  
            \draw[white, line width=3pt] (0,1.77,0) -- (0,1.9,0);
            \draw[thick] (0,1.9,0) -- (D); 
        
            \draw[thick] (A) -- (2.77,0,0); 
            \draw[white, line width=3pt] (2.77,0,0) -- (2.9,0,0); 
            \draw[thick] (2.9,0,0) -- (B);

            \draw[thick] (E) -- (F) -- (G) -- (H) -- cycle;  

            \draw[thick] (A) -- (E); 
            \draw[thick] (B) -- (F); 
            \draw[thick] (C) -- (G); 
            \draw[thick] (D) -- (H); 
            \draw[thick] (D) -- (G);
            \draw[thick] (A) -- (F);
        
            \draw[line width=1mm] (H) -- (F);
            \draw[line width=1mm] (C) -- (F);
            \draw[line width=1mm] (A) -- (H);
            \draw[line width=1mm] (A) -- (C);
            \draw[line width=1mm] (A) -- (G);

            


        
            
            \node[blue] at (-0.1,1.5,0.2) { $1_+$};
            \node[blue] at (0.35,1.46,0.2) { $1_-$};
            
            \node[red] at (2.75,0.68,0.2) { $2_+$};
            \node[red] at (3.15,0.65,0.2) { $2_-$}; 
            
            \node[green!60!black] at (1.2,2.4,0.2) { $3_+$};    
            \node[green!60!black] at (1.75,2.6,0.2) { $3_-$};   
            
            \node[black] at (0.95,-0.6,0) { $0_+$};    
            \node[black] at (1.95,0.2,1.5) { $0_-$};    
            
            \fill[white] (F) circle (3pt);
            \draw[thick] (F) circle (3pt);
            
        \end{tikzpicture}
    \end{minipage}
    \hfill
    \begin{minipage}[t]{0.48\textwidth}
        \centering
        \vspace{30pt} 
        \small 
        \renewcommand{\arraystretch}{1.1} 
        \setlength{\tabcolsep}{2pt} 
        \resizebox{\linewidth}{!}{ 
        \begin{tabular}{|l|c|c|c|c|}
            \hline
            Tet no. & 012 & 013 & 023 & 123 \\
            \hline
            6k & - & 6k+3(013) & - & 6k+1(120) \\
            \hline
            6k+1 & 6k(312) & 6k+4(013) & 6k+2(023) & - \\
            \hline
            6k+2 & - & - & 6k+1(023) & - \\
            \hline
            6k+3 & - & 6k(013) & - & 6k+4(120) \\
            \hline
            6k+4 & 6k+3(312) & 6k+1(013) & 6k+5(023) & - \\
            \hline
            6k+5 & - & - & 6k+4(023) & - \\
            \hline
        \end{tabular}
        }
    \end{minipage}
    \caption{\label{fig:tetrahedral_subdivision_of_prism} A decomposition of the prism into tetrahedra such that faces can be paired by gluings as in \Cref{fig:piOne}.}
\end{figure}

\begin{proof} We first sub-divide $P(\be)$ into three tetrahedra, then reflect this across its quadrilateral face that intersects the interior of $\dpr(\be)$ to produce a reflection-invariant decomposition of $\dpr(\be) = P(\be)\cup\overline{P}(\be)$ into six tetrahedra. The resulting triangulation is pictured in \Cref{fig:tetrahedral_subdivision_of_prism}. Its constituent tetrahedra can be numbered as follows: from $0$ to $2$ in $P(\be)$, with $0$ containing the triangular face labeled ``$3_-$'' in \Cref{numberings} and $2$ containing the one labeled ``$0_-$'' there; and with $a+3$ the mirror image of $a$ in $\overline{P}(\be)$ for each $a\in\{0,1,2\}$.

The face identifications of tetrahedra $0$ through $5$ that are internal to $\dpr(\be)$ are depicted on the right in \Cref{fig:tetrahedral_subdivision_of_prism}. To triangulate the underlying space of $\mathbb{H}^3/H$, taking $\dpr_0,\hdots,\dpr_{n-1}$ to be the cells of the decomposition given by \Cref{cell decomp}, we triangulate each $\dpr_k$ by tetrahedra numbered $6k,6k+1,\hdots,6k+5$ using the decomposition of $\dpr(\be)$ from the paragraph above and the marking isometry from \Cref{cell decomp}. The identifications of tetrahedral faces contained in the boundaries of the $\dpr_k$ are then inherited from the face identifications described in \Cref{cell decomp}.
\end{proof}

\begin{lemma}\label{lem:lifts}
For $i\in\{2,3\}$ and $j\in\{1,2\}$, let $M_{i,j}$ be the cover of $\doh^{333}_i$ described in \Cref{first manifold covers}. The $M_{i,j}$ are pairwise non-isometric, and those covering $\doh^{333}_2$ are not commensurable with those covering $\doh^{333}_3$. Furthermore:
\begin{itemize}
    \item $M_{2,1}$ has a unique non-hyperbolic Dehn filling, homeomorphic to the lens space $L(13,3)$.
    \item $M_{2,2}$ has a unique non-hyperbolic Dehn filling, homeomorphic to the lens space $L(22,5)$.
    \item $M_{3,1}$ has a unique non-hyperbolic Dehn filling, homeomorphic to the lens space $L(13,3)$. 
    \item $M_{3,2}$ has a unique non-hyperbolic Dehn filling, homeomorphic to the lens space $L(22,5)$.
\end{itemize}
\end{lemma}

\begin{proof} Recall from \Cref{torfree} that each $M_{i,j}$ is a manifold, and from \Cref{onecusp} that it has one cusp. It follows from \Cref{commensurator} that the $M_{2,j}$ are not commensurable, hence also not isometric, to the $M_{3,j}$. That $M_{2,1}$ is not isometric to $M_{2,2}$, and likewise for $M_{3,1}$ and $M_{3,2}$, can be seen from their having unique lens space fillings yielding distinct lens spaces. In fact, area considerations imply that each $M_{i,j}$ has at most one filling slope of length $6$ or less on a maximal horoball cusp cross-section, as we show in the next paragraph. It therefore follows from the Six Theorem \cite{AgolSix}, \cite{Lackenby_sixes} that $M_{i,j}$ has at most one non-hyperbolic filling slope. 

A maximal cusp cross-section for $M_{i,j}$ pulls back from the cross-section for $O^{333}_i$ described in \Cref{max cusp}, which has area greater than $0.84$ (twice the cusp \emph{volume} listed in the Lemma's statement). Since the cover $M_{i.j}\to O^{333}_i$ has degree $48$, factoring as $M_{i,j}\stackrel{24:1}{\longrightarrow} \doh^{333}_i \stackrel{2:1}{\to} O^{333}_i$, the cusp cross-section of $M_{i.j}$ has area greater than $48*0.84 = 40.32$. A set of linearly independent slopes in the cusp lattice that each had length at most $6$ would span a sub-lattice having coarea at most $36$, a contradiction.

We establish the existence of lens space fillings with the aid of computers. Using the code available in the ancillary files of this arxiv post, we triangulate the $M_{i,j}$ as described in \Cref{triangulation}. Since the $M_{i,j}$ cover the $\doh^{333}_i$ with degree $24$, the resulting triangulations have $96$ tetrahedra; since the $M_{i,j}$ are single-cusped, they have a single ideal vertex class as well as some finite vertices. Applying the method {\tt intelligentSimplify} in Regina \cite{regina} to each of these four triangulations results in an \emph{ideal} triangulation with a single (ideal) vertex class (see \cite{JacoRubinstein}, in particular Prop.~5.15 there, for background on how to change a triangulation to reduce the number of vertex classes to $1$), with a number of tetrahedra in the low 50's.

The ideal triangulations are then passed to SnapPy \cite{SnapPy} via their triangulation isomorphism signatures, which recognizes them as having a finite cyclic filling (see \path{anc/Analysis_of_covers/data} for files that containing these signatures). Experimentation with fillings in SnapPy produces one for each $M_{i,j}$ that is recognized to have a single-generator, single-relator fundamental group. The SnapPy \texttt{filled\_triangulation} method rigorously produces and simplifies a closed triangulation for the filled manifold. The resulting, considerably simpler triangulations---with single-digit numbers of tetrahedra---are then passed back to Regina, which compares them to an existing census of triangulations of closed manifolds that it ships with (see \cite{burton}), recognizing each as the lens space identified in the statement.
\end{proof}

It is an essentially algebraic-topological fact that for a knot complement $M$ in a lens space $L(p,q)$ such that $H_1(M)\cong\ZZ$, the preimage $\widetilde{M}$ of $M$ in the universal cover $\mathbb{S}^3$ of $L(p,q)$ is also a knot complement; that is, that the corresponding knot in $L(p,q)$ has connected preimage in $\mathbb{S}^3$. This was laid out in work of Gonzalez-Acu\~na--Whitten \cite{GonzalezAcunaWhitten1992}. The statement below combines \cite[Proposition 4.7]{GonzalezAcunaWhitten1992} and its preceding remark, updated to acknowledge the positive resolution of the Geometrization Conjecture.

\begin{prop}[Gonzalez-Acu\~{n}a--Whitten \cite{GonzalezAcunaWhitten1992}]\label{prop:GonalezAcunaWhitten}
Let $K$ be a hyperbolic knot in a lens space $L(p,q)$ and denote by $Q\cong L(p,q) \setminus K$. If $H_1(Q,\ZZ)\cong \ZZ$, then the universal covering of $L(p,q)$ restricts to a $p$-fold covering of $\pi:\Sthree \setminus K' \rightarrow Q$ for some hyperbolic knot $K'$. 
\end{prop}

This section's main theorem quickly follows from this and the results established above.

\begin{proof}[Proof of \Cref{main_tech_thm}] For $i\in\{2,3\}$ and $j\in\{1,2\}$, the cover $M_{i,j}\to\doh^{333}_i$ defined in \Cref{first manifold covers} is a knot complement in a lens space $L(p,q)$, by \Cref{lem:lifts}. By \Cref{homologyZ} it has first homology isomorphic to $\ZZ$. Therefore the preimage $\widetilde{M}_{i,j}$ of $M_{i,j}$ in the universal cover $\mathbb{S}^3$ of the relevant $L(p,q)$ is a knot complement, by \Cref{prop:GonalezAcunaWhitten}. Since the $M_{2,j}$ are not commensurable to the $M_{3,j}$, again by \Cref{lem:lifts}, the same is true for the $\widetilde{M}_{2,j}$ and $\widetilde{M}_{3,j}$. 
By \Cref{lem:lifts} the $M_{2,j}$ are not commensurable to the $M_{3,j}$, so this also holds for the $\widetilde{M}_{2,j}$ and $\widetilde{M}_{3,j}$.
\end{proof}

\section{Hand-computations for $M_{2,1}$}\label{sec:M21}

This section gives a paper-and-pencil proof of the main theorem for $M_{2,1}$. First we recall relevant data from earlier parts of the paper. Here is the nine-tuple describing $O^{333}_2$:

\begin{table}[ht]
\begin{tabular}{lrrrrrrrrrr}
\toprule
\char"0023 & $a_1$ &  $a_2$ &  $a_3$ &  $a_4$ &  $a_5$ &  $a_6$ &  $a_7$ &  $a_8$ &  $a_9$ \\
\midrule
$O^{333}_2$  & 3 & 3 & 2 & 3 & 3 & 4 & 2 & 2 & 3 \\
\end{tabular}
\end{table}

\noindent The presentation (\ref{able}) thus specifies to 
\[ \dprgr^{333}_2
\cong\langle x, y,z, w\, |\, 
x^{3}, y^{3}, 
z^{3}, w^{2}, 
(y^{-1}x)^{2}, 
(z^{-1}x)^{3},
(z^{-1}y)^{4}, 
(y^{-1}w)^{3}, 
(z^{-1}w)^{2} \rangle. \]
Next, the right-representation $\sigma_{2,1}$ specifying the cover $M_{2,1}\to \doh^{333}_2$, from \Cref{table_pr}.

\begin{table}[ht]
\begin{tabular}{ccrl}
\toprule
\textbf{i} & name & & permutation representation \\
\midrule
$\mathbf{2}$ & $\sigma_{2,1}$ & $x\mapsto$ & $[1, 2, 0, 12, 10, 19, 18, 3, 20, 8, 16, 6, 7, 17, 13, 5, 4, 14, 11, 15, 9, 22, 23, 21]$ \\ 
    & & $y\mapsto$ & $[3, 6, 10, 4, 0, 20, 7, 1, 19, 14, 11, 2, 18, 15, 23, 22, 12, 8, 16, 17, 21, 5, 13, 9]$ \\ 
    & & $z\mapsto$ & $[2, 8, 5, 13, 17, 0, 22, 12, 9, 1, 4, 18, 23, 14, 3, 19, 15, 10, 20, 16, 11, 6, 21, 7]$ \\ 
    & & $w\mapsto$ & $[1, 0, 9, 15, 18, 8, 12, 22, 5, 2, 20, 17, 6, 16, 19, 3, 13, 11, 4, 14, 10, 23, 7, 21]$ \\
\midrule
\end{tabular}
\end{table}

Recall from the proof of \Cref{itsapermrep} that the formulas above define $\sigma_{2,1}\co\dprgr^{333}_2\to S_{24}$ as follows: for a word $g = \chi_1^{\alpha_1}\cdots\chi_k^{\alpha_k}$ in the generators $x,y,z,w$, and $n\in\{0,1,\hdots,23\}$, take:
\[ \sigma_{2,1}(g)(n) = \left(\sigma_{2,1}(\chi_k)^{\alpha_k} \circ \hdots \circ \sigma_{2,1}(\chi_1)^{\alpha_1}\right)(n). \]
The cycle decompositions given below are used in that proof to show that $\sigma_{2,1}$ is well-defined.

\begin{example}\label{M21 permreps} Here for each relator of $\dprgr^{333}_2$ in the presentation above, of the form $g_j^{a_{i_j}}$ for some $j\in\{1,\hdots,9\}$, we exhibit the cycle decomposition of $\sigma_{2,1}(g_j)$. 

\begin{para}\label{ex} $\sigma_{2,1}(x) = (0,1,2)(3,12,7)(4,10,16)(5,19,15)(6,18,11)(8,20,9)(13,17,14)(21,22,23)$ \end{para}
\begin{para}\label{why} $\sigma_{2,1}(y) = (0,3,4)(1,6,7)(2,10,11)(5,20,21)(8,19,17)(9,14,23)(12,18,16)(13,15,22)$ \end{para}
\begin{para}\label{zee} $\sigma_{2,1}(z) = (0,2,5)(1,8,9)(3,13,14)(4,17,10)(6,22,21)(7,12,23)(11,18,20)(15,19,16)$ \end{para}
\begin{para}\label{dub} \small $\sigma_{2,1}(w) = (0,1)(2,9)(3,15)(4,18)(5,8)(6,12)(7,22)(10,20)(11,17)(13,16)(14,19)(21,23)$ \end{para}
\begin{para}\label{whyex} $\sigma_{2,1}(x)\sigma_{2,1}(y)^{-1} = (0,10)(1,3)(2,6)(4,12)(5,22)(7,18)(8,14)(9,21)(11,16)(13,23)$\\
    \strut\hfill $(15,17)(19,20)$\end{para}
\begin{para}\label{zeeyex} $\sigma_{2,1}(x)\sigma_{2,1}(z)^{-1} = (0,19,5)(1,8,2)(3,13,12)(4,16,15)(6,22,18)(7,21,23)(9,20,11)$\\
    \strut\hfill$(10,14,17)$\end{para}
\begin{para}\label{whyzee} $\sigma_{2,1}(y)\sigma_{2,1}(z)^{-1} = (0,20,16,17)(1,14,15,12)(2,3,23,18)(4,11,21,13)(5,10,8,6)$\\
    \strut\hfill $(7,9,19,22)$\end{para}
\begin{para}\label{whydub} $\sigma_{2,1}(w)\sigma_{2,1}(y)^{-1} = (0,18,6)(1,22,3)(2,17,14)(4,15,16)(5,23,19)(7,12,13)(8,11,20)$\\
    \strut\hfill$(9,21,10)$\end{para}
\begin{para}\label{zeedub} $\sigma_{2,1}(w)\sigma_{2,1}(z)^{-1} = (0,8)(1,2)(3,19)(4,20)(5,9)(6,23)(7,21)(10,11)(12,22)(13,15)$\\
    \strut\hfill$(14,16)(17,18)$ \end{para}
\end{example}

As recorded in \Cref{torfree}, 
the fact that each element is collection of $\frac{24}{a_j}$ disjoint $a_j$ cycles (with $a_j=2,3,4$) 
implies that $M_{2,1}$ is a manifold; 
per \Cref{onecusp}, the fact that $\langle \sigma_{2,1}(x),\sigma_{2,1}(z)\rangle$ acts transitively implies that it has one cusp. We will also use the cycle decompositions given above to coarsen the spine for $M_{2,1}$ constructed in \Cref{spine}.

\begin{lemma}\label{M_21 spine} The cover $M_{2,1}\to \doh^{333}_2$ prescribed by the right-permutation representation $\sigma_{2,1}\co\dprgr^{333}_2\to S_{24}$ given in \Cref{table_pr} has a spine $\Sigma$ with $16$ two-cells, $22$ edges, and 6 vertices. Numbering polyhedra of the decomposition of $M_{2,1}$ given in \Cref{cell decomp} as $\dpr_0,\hdots,\dpr_{23}$, and labeling their faces as in \Cref{numberings}, twelve $2$-cells of $\Sigma$ are each the union of two faces labeled $1_-$, and four are each the union of six faces labeled $3_-$.
\begin{para}\label{whyex again} The union of faces of $\dpr_k$ and $\dpr_{k'}$ labeled $1_-$ is a face of $\Sigma$ for these pairs $(k,k')$:
\[ (0,10)(1,3)(2,6)(4,12)(5,22)(7,18)(8,14)(9,21)(11,16)(13,23)(15,17)(19,20) \]\end{para}
\begin{para}\label{three minus} The $6$-tuples of indices of polyhedra whose faces labeled $3_-$ comprise a face of $\Sigma$ are: 
\[  (0,1,2,9,5,8), (3,15,13,16,14,19), (4,18,17,11,10,20), (6,12,22,7,21,23) \]
These $6$-tuples are arranged so that the face corresponding to each entry shares an edge with the faces corresponding to the entries immediately before and after it (taken cyclically).\end{para}
\end{lemma}

\begin{proof} Let $X$ be the spine for $M_{2,1}$ described in \Cref{spine}. We produce the complex $\Sigma$ from $X$ by removing edges of $X$ having valence two, and for each such edge, replacing the two $2$-cells of $X$ containing it by their union. Below we give more detail on this process. 

Because $a_3 = 2$ in the tuple $(a_1,\hdots,a_9)$ determining $O^{333}_2$, each copy of the edge of $\dpr^{333}_2$ with that label in \Cref{fig:piOne} has valence two in $X$: it follows from \Cref{torsion order} that such an edge of $P_k$ is contained in the faces labeled $1_-$ of $P_k$ and $P_{k'}$ only, where $k'$ is the other member of the cycle of $k$ for $\sigma_{2,1}(y^{-1}x)$. (The relator of (\ref{able}) corresponding to the edge cycle of this edge is $(y^{-1}x)^{a_3}$.) The set of these cycles is thus in bijective correspondence with the set of twelve faces of $\Sigma$ that are each the union of two of $X$ labeled $1_-$. Reproducing this set from \ref{whyex} yields \ref{whyex again}.

The other two edges with valence two in $X_0$ are those labeled $a_7$ and $a_8$ in \Cref{fig:piOne}, since here $a_7 = a_8 = 2$, and only the faces labeled $3_{\pm}$ containing these edges belong to $X_0$. Each $2$-cell of $\Sigma$ containing a face labeled $3_-$ will contain six, since this face contains the edges labeled $a_7$ and $a_8$, with an angle of $\pi/3$ at their vertex of intersection. These $2$-cells of $X$ correspond to unions of cycles for $\sigma_{2,1}(w)$, since $w$ fixes the edge labeled $a_7$, and $\sigma_{2,1}(z^{-1}w)$ corresponding to $a_8$. Using the decompositions recorded in \ref{dub} and \ref{zeedub}, respectively, we obtain the sets of indices recorded in \ref{three minus}.

The vertices of $\Sigma$ are identical to those of $X$, as are the remaining edges---those that were not removed in the construction of the faces of $\Sigma$.
\end{proof}

We use the spine $\Sigma$ from \Cref{M_21 spine} to produce the presentation for $\pi_1 M_{2,1}$ given below.

\begin{prop}\label{M_21}
    For the cover $M_{2,1}\to \doh^{333}_2$ prescribed by the right-permutation representation $\sigma_{2,1}\co\dprgr^{333}_2\to S_{24}$ given in \Cref{table_pr},\begin{align*}
    \pi_1 M_{2,1} & \cong \langle\,A_1,A_2,A_4,\hdots,A_7,B_1,\hdots,B_5,C_1,C_3,\hdots,C_7\,|\\
        & \ \qquad C_1^{-1}C_7^{-1}C_4C_6^{-1},\ C_1C_3^{-1}C_5C_3^{-1}C_4^{-1},\ C_3C_6^{-1}C_7C_6^{-1},\ C_5^{-1}C_1C_5^{-1}C_7C_4^{-1},\\
        & \ \ \qquad A_2C_7^{-1}B_4^{-1},\ B_1C_1A_1^{-1}C_1^{-1}B_2^{-1},\ B_2A_2^{-1}A_1B_4^{-1},\ B_3C_3A_6C_5^{-1}B_1^{-1},\\
        & \ \ \qquad B_4C_4A_7C_1^{-1}B_5^{-1},\ B_5C_5A_1^{-1}A_6B_2^{-1},\ B_4C_6A_4^{-1}A_5B_1^{-1},\ B_5C_7A_5^{-1}C_7^{-1}B_3^{-1},\\
        & \ \qquad\left. B_3C_6A_2^{-1}A_6C_3^{-1},\ B_3C_5A_7^{-1}A_5C_4^{-1}B_2^{-1},\ B_1C_3A_7^{-1}A_4,\ B_5C_4A_4^{-1}C_6^{-1}\ \right\rangle
\end{align*}
\end{prop}

Before proving this result, we record a consequence.

\begin{corollary}\label{M_21 H1} For $M_{2,1}$ as in \Cref{first manifold covers}, $H_1(M_{2,1})\cong\mathbb{Z}$ is generated by $\gamma \doteq 5C_6-C_7$. (Here ``$C_i$'' refers to the homology class of the corresponding generator from the presentation for $\pi_1 M_{2,1}$ given in \Cref{M_21}.)\end{corollary}

\begin{proof}This can be showed by abelianizing the presentation given in \Cref{M_21}. For a fully by-hand computation, one can convert each relation into an equation over $\mathbb{Z}$ in variables corresponding to the generators---so for instance, the second relation would yield $C_1 - 2C_3 - C_4 + C_5 = 0$---and row-reduce the matrix corresponding to the resulting linear system over $\mathbb{Z}$. After doing so, we obtain the matrix of \Cref{H1 REF}. 

\begin{table}[ht]
\begin{tabular}{l | rrrrrrrrrrrrrrrrr}

 & $A_1$ & $A_2$ & $A_4$ & $A_5$ & $A_6$ & $A_7$ & $B_1$ & $B_2$ & $B_3$ & $B_4$ & $B_5$ & $C_1$ & $C_3$ & $C_4$ & $C_5$ & $C_6$ & $C_7$ \\ \hline
$R_6'$& -1 & & & & & & 1 & -1 & & & & & & & & & \\
$R_5'$& & 1 & & & & & & & & -1 & & & & & & & -1 \\
$R_{15}'$& & & 1 & & & -1 & 1 & & & & & & 1 & & & & \\
$R_{12}'$& & & & -1 & & & & & -1 & & 1 & & & & & & \\
$R_{10}'$& & & & & 1 & & -1 & & & & 1 & & & & 1 & & \\
$R_9'$& & & & & & 1 & & & & 1 & -1 & -1 & & 1 & & & \\
$R_7'$& & & & & & & 1 & & & -2 & & & & & & & -1 \\
$R_{14}'$& & & & & & & & -1 & & 1 & & -1 & & & 1 & & \\
$R_{8}'$& & & & & & & & & 1 & & -1 & & 1 & & -2 & & \\
$R_{13}'$& & & & & & & & & & 1 & & & -2 & & 1 & 1 & \\
$R_{11}'$& & & & & & & & & & & -1 & -1 & 6 & 1 & -4 & -1 & \\
$R_1'$& & & & & & & & & & & & -1 & & 1 & & -1 & -1 \\
$R_3'$& & & & & & & & & & & & & 1 & & & -2 & 1 \\
$R_{16}'$& & & & & & & & & & & & & & 1 & -3 & 11 & -5 \\
$R_2'$& & & & & & & & & & & & & & & 1 & -5 & 1 \\
$R_4'$& & & & & & & & & & & & & & & & -11 & 2 \\
\end{tabular}
\caption{A presentation matrix for $H_1(M_{2,1})$.}
\label{H1 REF}
\end{table}

In \Cref{H1 REF}, a row label $R_i'$ above means the row is descended from the $i$th relator, numbered starting at top left and proceeding left-to-right and top-to-bottom from there---after row-reduction. Regarding this as a matrix, its kernel is generated by:\begin{align*}
    & (A_1,A_2,A_4,A_5,A_6,A_7,B_1,B_2,B_3,B_4,B_5,C_1,C_3,C_4,C_5,C_6,C_7) \\
    & \quad = (14,-4,1,-5,9,-25,-19,-33,-22,-15,-27,17,-7,30,-1,2,11).
\end{align*}
One can also use this as a correctness check for the row reduction, by verifying that it also belongs to kernel of the original presentation matrix for $H_1$ obtained directly from the presentation for $\pi_1$.

The result shows that the class $5C_6-C_7$ generates $H_1(M_{2,1})$, with $-11C_6+2C_7$ vanishing and all other classes depending on those of $C_6$ and $C_7$.\end{proof}

\begin{proof}[Proof of \Cref{M_21}] We continue with the spine $\Sigma$ from \Cref{M_21 spine}, now analyzing its edges. These fall into three classes: those lifted from edges of $P^{333}_2$ labeled $a_4$, $a_6$, and $a_9$. By \Cref{torsion order}, each such lifted edge corresponds to a cycle associated to a relator of $\dprgr^{333}_2$; respectively, \ref{why}, \ref{whyzee}, and \ref{whydub}. There are $8$, $6$, and $8$ such edges, respectively, since $a_4 = a_9 = 3$ and $a_6=4$, for a total of $22$.

$\Sigma$ has six vertices total, coming from the vertices of the face $1_-$ of $P^{333}_2$. Each vertex of this face that belongs to the edge labeled $a_6$ has connected preimage in $M_{2,1}$ since its stabilizer in $\dprgr^{333}_2$ has order $24$, being a $(2,3,4)$-rotation group. Each of the other two vertices has two preimages in $M_{2,1}$, as its stabilizer is a $(2,3,3)$-rotation group of order twelve. These preimages may be identified with the orbits of the $\sigma_{2,1}$-images of the associated vertex groups: $\langle x,y\rangle$ in one case, and $\langle y,w\rangle$ in the other. Consulting the cycle decompositions of \Cref{M_21}, we obtain the following in the respective cases.

\begin{para}\label{ex+why} $\{0,1,2,3,4,6,7,10,11,12,16,18\}, \{5,8,9,13,14,17,15,19,20,21,22,23\}$\end{para}
\begin{para}\label{why+dub} $\{0,1,3,4,6,7,12,13,15,16,18,22\}, \{2,5,8,9,10,11,14,17,19,20,21,23\}$\end{para}

As a first step toward presenting $\pi_1 M_{2,1}$ using the complex $\Sigma$, we choose orientations for its edges. We orient the edges of $P^{333}_2$ labeled $a_4$ and $a_6$ pointing up in \Cref{fig:piOne}, and the one labeled $a_9$ pointing right-to-left, and carry these orientations to the corresponding edge classes of $M_{2,1}$ using the isometric embeddings $P^{333}_2\subset \dpr^{333}_2\to \dpr_k$ for $k\in \{0,\hdots,23\}$. Name these edge classes $a_0,\hdots,a_7$, $b_0,\hdots,b_5$, and $c_0,\hdots,c_7$, respectively, corresponding to order of appearance in the cycle decompositions \ref{why}, \ref{whyzee}, and \ref{whydub}, and let $\bar{a}_i$ etc.~denote an edge with orientation reversed.

With the orientations above, the boundary of each $2$-cell of $\Sigma$ can be read off as a concatenation of edges. For instance, the $2$-cell corresponding to the cycle $(0,10)$ of $\sigma_{2,1}(y^{-1}x)$ has boundary $b_0c_0\bar{a}_0a_2\bar{c}_7\bar{b}_4$; the first three edges belonging to the face $1_-$ of $\dpr_0$, and the last three to that of $\dpr_{10}$. Each $2$-cell composed of copies of $1_-$ will have a similar form, where by \Cref{torsion order}, we discern the indices of edges bounding a face of a particular $\dpr_k$ by searching for the cycle containing $k$ in the appropriate cycle decompositions. 

The $2$-cells that are the union of faces labeled $3_-$ have boundary consisting of edges entirely labeled $c$. From the tuples of \ref{three minus}, using their cyclic arrangment (see the note directly below it), we obtain in order:

\begin{para}\label{del three minus} $c_0\bar{c}_1c_2\bar{c}_7c_4\bar{c}_6$, $c_1\bar{c}_3c_5\bar{c}_3c_2\bar{c}_4$, $c_3\bar{c}_0c_2\bar{c}_6c_7\bar{c}_6$, $c_0\bar{c}_5c_1\bar{c}_5c_7\bar{c}_4$
\end{para}

The edge orientations alternate above because this is also true of the orientations that the copies of $3_-$ inherit from their ambient $\dpr_k$.

The next step to presenting $\pi_1 M_{2,1}$ is to choose a maximal subtree $T$ of the one-skeleton of $\Sigma$. As edges of $T$ we take $b_0$, joining the vertices with stabilizer of order $24$; $c_0$ and $c_2$, each of which join an endpoint of $b_0$ to a distinct preimage of the vertex stabilized by $y^{-1}w$; and $a_0$ and $a_3$, which respectively share a single endpoint with $c_0$ and $c_2$ and have distinct far endpoints. For each $i\ne 3$ between $1$ and $7$, let $A_i$ denote the element of $\pi_1 M_{2,1}$ that is the concatenation of the arc in $T$ from a designated basepoint vertex to the initial point of $a_i$, $a_i$ itself, and the arc in $T$ from the terminal vertex of $a_i$ back to the basepoint. Similarly define $B_j$, for $j\in\{1,2,3,4,5\}$, and $C_k$ for $k\ne 2$ between $1$ and $7$. 

The relations on the first line of the presentation given in the statement now come from the faces of type $3_-$; the others, from faces of type $1_-$. The first of the latter type comes from the face corresponding to the cycle $(0,10)$ for $\sigma_{2,1}(y^{-1}x)$, which we recall from above has boundary $b_0c_0\bar{a}_0a_2\bar{c}_7\bar{b}_4$. The edges of this cycle that do not belong to the maximal tree thus give rise to the product $A_2C_7^{-1}B_4^{-1}$ in $\pi_1$ of the one-skeleton. The others are obtained similarly.
\end{proof}

We now describe a peripheral subgroup of $\pi_1 M_{2,1}$, ie.~one representing the fundamental group of a torus cusp cross-section of $M_{2,1}$. Below a \emph{geometric basis} for the fundamental group of a Euclidean torus $T$ consists of an element $\mu_0$ representing the shortest geodesic on $T$ and one, $\lambda_0$ representing the shortest geodesic that is not a power of the first.

\begin{lemma}\label{filling slope} In the presentation for $\pi_1 M_{2,1}$ given in \Cref{M_21}, there is a geometric basis for a peripheral subgroup consisting of $\mu_0 = C_1A_2$ and 
\[ \lambda_0 = B_2^{-1}B_1A_5^{-1}C_7^{-1}C_6B_2^{-1}B_4C_4A_5C_7^{-1}C_5C_1^{-1}B_1^{-1}B_2C_4A_5^{-1}A_4C_6^{-1}C_7C_6^{-1}B_3^{-1}A_4^{-1}C_6^{-1}C_1^{-1}. \] 
These represent elements having lengths $1+\sqrt{2}$ and $4\sqrt{3}(1+\sqrt{2})$, respectively, in a maximal cusp cross-section. The homological longitude is $\lambda = \mu_0^{55}\lambda_0^{-13}$, and the shortest homological meridian (meaning a curve whose inclusion-induced image generates $H_1(M_{2,1})$) is $\mu_0^{-17}\lambda_0^{4}$.
\end{lemma}

\begin{remark}
    For any orientable one-cusped hyperbolic manifold $M$ with $H_1(M)\cong\mathbb{Z}$, the inclusion map of a cusp cross-section into $M$ induces a surjection on first homology. This follows from ``half lives--half dies'' with arbitrary field coefficients, or see \cite[Lemma 3.0]{GonzalezAcunaWhitten1992}.
\end{remark}

\begin{figure}[ht]
\begin{tikzpicture}

\begin{scope}

\draw [thick] (1.732,0) -- (0,1) -- (-1.732,0) -- (0,-1) -- cycle;
\draw [thin] (0,-1) -- (0,1);
\draw [thin] (-0.866,0.5) -- (0,0) -- (0.866,0.5);
\fill (0,-1) circle [radius=0.08];

\node at (0.8,-0.1) {$1_-$};
\node at (-0.7,-0.06) {$1_+$};
\node at (0.4,0.5) {\small $3_-$};
\node at (-0.25,0.5) {\small $3_+$};

\end{scope}

\begin{scope}[xshift=2in]

\draw [thick] (1.732,0) -- (0,1) -- (-1.732,0) -- (0,-1) -- cycle;
\draw [thin] (0,-1) -- (0,1);
\draw [thin] (-0.866,0.5) -- (0,0) -- (0.866,0.5);
\fill (0,-1) circle [radius=0.08];

\draw [very thick,->] (1.732,0) -- (0.866,0.5); 
\draw [very thick,->] (0.866,0.5) -- (0,0);
\draw [very thick,->] (0,-1) -- (0,0);

\node at (0.8,-0.1) {$i$};
\node at (-0.7,-0.06) {$(i.y)$};
\node [above right] at (1.199,0.2) {\footnotesize $b_{j_i}$};
\node [above] at (0.4,0.25) {\footnotesize $c_{k_i}$};
\node [right] at (-0.05,-0.5) {\footnotesize ${a}_{l_i}$};

\end{scope}

\end{tikzpicture}
\caption{Left: the diamond's decomposition into faces induced by the deformation retract, labeled with numbers from \Cref{numberings}. Right, labeling paths of the diamond in $\dpr_i$ by their images.}
\label{fig: diamonds}
\end{figure}
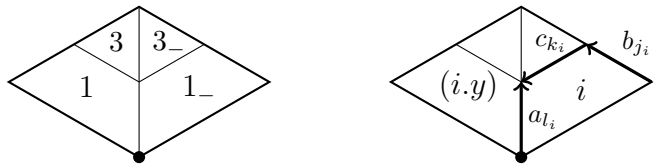

\begin{proof} A maximal horospherical cusp cross-section for $M_{2,1}$ is a union of $24$ diamonds (ie.~rhombi), each a cross-section of a copy of the doubled prism $\dpr^{333}_2$ by a horosphere centered at its ideal vertex. The half of such a diamond coming from the prism $\dpr^{333}_2$ is discussed in \Cref{max cusp}. It is an equilateral Euclidean triangle with side length $(1+\sqrt{2})/\sqrt{3}$, pictured in bold in \Cref{rigeur}; the diamond's other half is this triangle's reflection across its side projecting into a line with negative slope there.

The deformation retract described in \Cref{spine}, which takes $M_{2,1}$ to a spine that is a union of certain faces of copies $\dpr_i$ of the doubled prism, carries the diamond in each such $\dpr_i$ homeomorphically to the intersection of $\dpr_i$ with that spine---the union of the faces of $\dpr_i$ labeled $1_{\pm}$ and $3_{\pm}$ in \Cref{numberings}. The preimages of these four faces tile the diamond as pictured on the left in \Cref{fig: diamonds}. 

The ``model diamond'' in $\dpr^{333}_2$ has one vertex that is fixed by the face-pairing generator $x$ (rotation around the triangle vertex at $\left(0,\frac{1}{\sqrt{3}}\right)$ in \Cref{rigeur}), and its opposite vertex is fixed by $z$ (rotation about $(X_2,Y_2)$ from \Cref{subsec: volumes}). If we label the corresponding vertices of the diamond in each $\dpr_i$ accordingly, for $i\in\{0,\hdots,23\}$, then it follows from \Cref{cell decomp} that the two sides containing $x$ of the diamond in $\dpr_i$ are identified to sides of the diamonds in $\dpr_{i.x^{\pm 1}}$, and those containing $z$ are identified to sides in $\dpr_{i.z^{\pm 1}}$. Here the actions of $x$ and $z$ on the indices $i$ are determined by the right-permutation representation $\sigma_{2,1}$, respectively with cycle decompositions \ref{ex} and \ref{zee}.

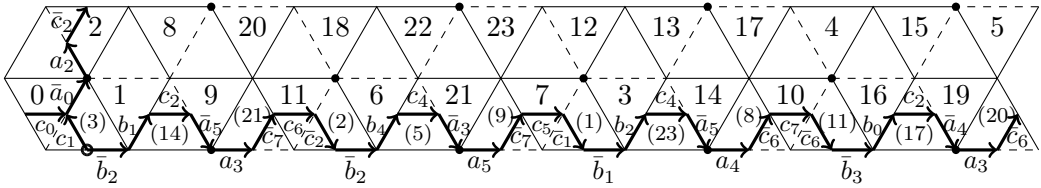
\begin{figure}[ht]
\begin{tikzpicture}

\begin{scope}[scale=1.1]
\draw (-1,0) -- (0,0) -- (0.5,-0.866) -- (1,0) -- (0.5,0.866) -- (0,0);
\draw (1,0) -- (2,0) -- (1.5,-0.866) -- (-0.5,-0.866) -- (-1,0) -- (-0.5,0.866) -- (1.5,0.866) -- (2,0);

\draw (0,-0.866) circle [radius=0.07];
\draw (0,-0.866) circle [radius=0.05];
\draw (0,-0.866) circle [radius=0.06];

\fill (0,0) circle [radius=0.05];
\fill (1.5,0.866) circle [radius=0.05];
\fill (1.5,-0.866) circle [radius=0.05];

\draw [dashed] (0,0) -- (1,0);
\draw [dashed] (-0.5,0.866) -- (0,0) -- (-0.5,-0.866);
\draw [dashed] (2.5,0.866) -- (1.5,0.866) -- (1,0) -- (1.5,-0.866) -- (2.5,-0.866);

\node [below] at (-0.6,0.05) {\small $0$};
\node [below] at (0.1,0.866) {\small $2$};
\node [below] at (0.4,0.05) {\small $1$};
\node [below] at (1.5,0.05) {\small $9$};
\node [below] at (1,0.866) {\small $8$};
\node [below] at (2,0.866) {\small $20$};
\node at (0.07,-0.533) {\tiny $(3)$};
\node [above] at (1,-0.946) {\tiny $(14)$};
\node [below] at (2,-0.2) {\tiny $(21)$};

\draw [very thick, ->] (-0.25,-0.433) -- (0,0);
\node [below] at (-0.3,0.05) {\footnotesize $\bar{a}_0$};
\draw [very thick, ->] (0,0) -- (-0.25,0.433);
\node [above] at (-0.3,-0.05) {\footnotesize $a_2$};
\draw [very thick, ->] (-0.25,0.433) -- (0,0.866);
\node [below] at (-0.3,0.916) {\footnotesize $\bar{c}_2$};

\draw [very thick, ->] (-0.75,-0.433) -- (-0.25,-0.433);
\node [below] at (-0.5,-0.383) {\scriptsize $c_0$};
\draw [very thick, ->] (0,-0.866) -- (-0.25,-0.433);
\node [above] at (-0.25,-0.966) {\scriptsize $c_1$};
\draw [very thick, ->] (0,-0.866) -- (0.5,-0.866);
\node [below] at (0.25,-0.816) {\footnotesize $\bar{b}_2$};
\draw [very thick, ->] (0.5,-0.866) -- (0.75,-0.433);
\node [above] at (0.5,-0.816) {\scriptsize $b_1$};
\draw [very thick, ->] (0.75,-0.433) -- (1.25,-0.433);
\node [above] at (1,-0.483) {\footnotesize $c_2$};
\draw [very thick, ->] (1.25,-0.433) -- (1.5,-0.866);
\node [above] at (1.5,-0.816) {\footnotesize $\bar{a}_5$};
\draw [very thick, ->] (1.5,-0.866) -- (2,-0.866);
\node [below] at (1.75,-0.816) {\footnotesize $a_3$};
\draw [very thick, ->] (2,-0.866) -- (2.25,-0.433);
\node [above] at (2.25,-0.966) {\footnotesize $\bar{c}_7$};

\end{scope}

\begin{scope}[scale=1.1, xshift=3cm]
\draw (-1,0) -- (0,0) -- (0.5,-0.866) -- (1,0) -- (0.5,0.866) -- (0,0);
\draw (1,0) -- (2,0) -- (1.5,-0.866) -- (-0.5,-0.866) -- (-1,0) -- (-0.5,0.866) -- (1.5,0.866) -- (2,0);

\fill (0,0) circle [radius=0.05];
\fill (1.5,0.866) circle [radius=0.05];
\fill (1.5,-0.866) circle [radius=0.05];

\draw [dashed] (0,0) -- (1,0);
\draw [dashed] (-0.5,0.866) -- (0,0) -- (-0.5,-0.866);
\draw [dashed] (2.5,0.866) -- (1.5,0.866) -- (1,0) -- (1.5,-0.866) -- (2.5,-0.866);

\node [below] at (-0.5,0.05) {\small $11$};
\node [below] at (0,0.866) {\small $18$};
\node [below] at (0.5,0.05) {\small $6$};
\node [below] at (1.5,0.05) {\small $21$};
\node [below] at (1,0.866) {\small $22$};
\node [below] at (2,0.866) {\small $23$};
\node at (0.07,-0.533) {\tiny $(2)$};
\node [above] at (1,-0.946) {\tiny $(5)$};
\node [below] at (2,-0.2) {\tiny $(9)$};

\draw [very thick, ->] (-0.75,-0.433) -- (-0.25,-0.433);
\node [below] at (-0.5,-0.383) {\scriptsize $c_6$};
\draw [very thick, ->] (-0.25,-0.433) -- (0,-0.866);
\node [above] at (-0.25,-0.966) {\scriptsize $\bar{c}_2$};
\draw [very thick, ->] (0,-0.866) -- (0.5,-0.866);
\node [below] at (0.25,-0.816) {\footnotesize $\bar{b}_2$};
\draw [very thick, ->] (0.5,-0.866) -- (0.75,-0.433);
\node [above] at (0.5,-0.816) {\scriptsize $b_4$};
\draw [very thick, ->] (0.75,-0.433) -- (1.25,-0.433);
\node [above] at (1,-0.483) {\footnotesize $c_4$};
\draw [very thick, ->] (1.25,-0.433) -- (1.5,-0.866);
\node [above] at (1.5,-0.816) {\footnotesize $\bar{a}_3$};
\draw [very thick, ->] (1.5,-0.866) -- (2,-0.866);
\node [below] at (1.75,-0.816) {\footnotesize $a_5$};
\draw [very thick, ->] (2,-0.866) -- (2.25,-0.433);
\node [above] at (2.25,-0.966) {\footnotesize $\bar{c}_7$};

\end{scope}

\begin{scope}[scale=1.1, xshift=6cm]
\draw (-1,0) -- (0,0) -- (0.5,-0.866) -- (1,0) -- (0.5,0.866) -- (0,0);
\draw (1,0) -- (2,0) -- (1.5,-0.866) -- (-0.5,-0.866) -- (-1,0) -- (-0.5,0.866) -- (1.5,0.866) -- (2,0);

\fill (0,0) circle [radius=0.05];
\fill (1.5,0.866) circle [radius=0.05];
\fill (1.5,-0.866) circle [radius=0.05];

\draw [dashed] (0,0) -- (1,0);
\draw [dashed] (-0.5,0.866) -- (0,0) -- (-0.5,-0.866);
\draw [dashed] (2.5,0.866) -- (1.5,0.866) -- (1,0) -- (1.5,-0.866) -- (2.5,-0.866);

\node [below] at (-0.5,0.05) {\small $7$};
\node [below] at (0,0.866) {\small $12$};
\node [below] at (0.5,0.05) {\small $3$};
\node [below] at (1.5,0.05) {\small $14$};
\node [below] at (1,0.866) {\small $13$};
\node [below] at (2,0.866) {\small $17$};
\node at (0.07,-0.533) {\tiny $(1)$};
\node [above] at (1,-0.946) {\tiny $(23)$};
\node [below] at (2,-0.2) {\tiny $(8)$};

\draw [very thick, ->] (-0.75,-0.433) -- (-0.25,-0.433);
\node [below] at (-0.5,-0.383) {\scriptsize $c_5$};
\draw [very thick, ->] (-0.25,-0.433) -- (0,-0.866);
\node [above] at (-0.25,-0.966) {\scriptsize $\bar{c}_1$};
\draw [very thick, ->] (0,-0.866) -- (0.5,-0.866);
\node [below] at (0.25,-0.816) {\footnotesize $\bar{b}_1$};
\draw [very thick, ->] (0.5,-0.866) -- (0.75,-0.433);
\node [above] at (0.5,-0.816) {\scriptsize $b_2$};
\draw [very thick, ->] (0.75,-0.433) -- (1.25,-0.433);
\node [above] at (1,-0.483) {\footnotesize $c_4$};
\draw [very thick, ->] (1.25,-0.433) -- (1.5,-0.866);
\node [above] at (1.5,-0.816) {\footnotesize $\bar{a}_5$};
\draw [very thick, ->] (1.5,-0.866) -- (2,-0.866);
\node [below] at (1.75,-0.816) {\footnotesize $a_4$};
\draw [very thick, ->] (2,-0.866) -- (2.25,-0.433);
\node [above] at (2.25,-0.966) {\footnotesize $\bar{c}_6$};

\end{scope}

\begin{scope}[scale=1.1, xshift=9cm]
\draw (-1,0) -- (0,0) -- (0.5,-0.866) -- (1,0) -- (0.5,0.866) -- (0,0);
\draw (1,0) -- (2,0) -- (1.5,-0.866) -- (-0.5,-0.866) -- (-1,0) -- (-0.5,0.866) -- (1.5,0.866) -- (2,0);

\draw (2.5,0.866) -- (2,0) -- (2.5,-0.866);

\fill (0,0) circle [radius=0.05];
\fill (1.5,0.866) circle [radius=0.05];
\fill (1.5,-0.866) circle [radius=0.05];

\draw [dashed] (0,0) -- (1,0);
\draw [dashed] (-0.5,0.866) -- (0,0) -- (-0.5,-0.866);
\draw [dashed] (2.5,0.866) -- (1.5,0.866) -- (1,0) -- (1.5,-0.866) -- (2.5,-0.866);

\node [below] at (-0.5,0.05) {\small $10$};
\node [below] at (0,0.866) {\small $4$};
\node [below] at (0.5,0.05) {\small $16$};
\node [below] at (1.5,0.05) {\small $19$};
\node [below] at (1,0.866) {\small $15$};
\node [below] at (2,0.866) {\small $5$};
\node at (0.07,-0.533) {\tiny $(11)$};
\node [above] at (1,-0.946) {\tiny $(17)$};
\node [below] at (2,-0.2) {\tiny $(20)$};

\draw [very thick, ->] (-0.75,-0.433) -- (-0.25,-0.433);
\node [below] at (-0.5,-0.383) {\scriptsize $c_7$};
\draw [very thick, ->] (-0.25,-0.433) -- (0,-0.866);
\node [above] at (-0.25,-0.966) {\scriptsize $\bar{c}_6$};
\draw [very thick, ->] (0,-0.866) -- (0.5,-0.866);
\node [below] at (0.25,-0.816) {\footnotesize $\bar{b}_3$};
\draw [very thick, ->] (0.5,-0.866) -- (0.75,-0.433);
\node [above] at (0.5,-0.816) {\scriptsize $b_0$};
\draw [very thick, ->] (0.75,-0.433) -- (1.25,-0.433);
\node [above] at (1,-0.483) {\footnotesize $c_2$};
\draw [very thick, ->] (1.25,-0.433) -- (1.5,-0.866);
\node [above] at (1.5,-0.816) {\footnotesize $\bar{a}_4$};
\draw [very thick, ->] (1.5,-0.866) -- (2,-0.866);
\node [below] at (1.75,-0.816) {\footnotesize $a_3$};
\draw [very thick, ->] (2,-0.866) -- (2.25,-0.433);
\node [above] at (2.25,-0.966) {\footnotesize $\bar{c}_6$};

\end{scope}

\end{tikzpicture}
\caption{A fundamental domain for a maximal cusp cross-section of $M_{2,1}$, tiled by diamonds and with curves representing $\mu_0$ and $\lambda_0$ in bold.}
\label{fig: M21 cusp}
\end{figure}

A fundamental domain for the maximal cusp cross section's cell decomposition resulting from these edge pairings is pictured in \Cref{fig: M21 cusp}. Each diamond in the figure is divided by a dashed line into triangles coming from the copy of $P^{333}_2$ and its mirror image $\bar{P}^{333}_2$, and labeled with the index $i$ of the doubled prism containing it. This label is placed on the $P^{333}_2$ side. By \Cref{cell decomp}, the edges on that side face edges on the $\bar{P}^{333}_2$ side of the diamonds in $P_{i.x^{-1}}$ and $P_{i.z^{-1}}$. Vertices corresponding to the one fixed by $x$ are dotted in the Figure.

The cross-section itself is obtained by identifying the top and bottom edges of the pictured fundamental domain by a vertical translation, and left to right by a horizontal translation. Because all triangles pictured in the Figure are actually equilateral in the cross-section's Euclidean metric, it follows that the cusp is rectangular and that the closed curves coming from vertical and horizontal arcs comprise a geometric basis with lengths given in the result's statement. To represent these curves in the presentation for $\pi_1 M_{2,1}$ given in \Cref{M_21}, we place a basepoint at the circled location in diamond $0$ (lower left of \Cref{fig: M21 cusp}) and homotope them to concatenated paths whose deformation-retracted images lie in the spine's one-skeleton, as pictured in bold in the Figure.

The right-hand side of \Cref{fig: diamonds} shows how to determine the labels of paths in such a concatenation. For a diamond lying in $\dpr_i$, decomposed into quadrilaterals and squares according to its image under the deformation retract as described above, the bold arrows denote oriented edges mapping to edge classes $a_{l_i}$, $c_{k_i}$, and $b_{j_i}$ belonging to the one-skeleton of the spine described in the proof of \Cref{M_21}. The indices $l_i$, $k_i$, and $j_i$ are those of the cycles containing $i$ in \ref{why}, \ref{whydub}, and \ref{whyzee}, respectively, as described there. These edges' mirror images in the diamond are also carried into the spine's one-skeleton by the deformation retract. The corresponding indices for these edges are determined by the corresponding cycles containing $i.y$, which is the index of the copy of $\dpr^{333}_2$ on the other side of face $1_+$ of $\dpr_i$. This is in parentheses on the right-hand side of \Cref{fig: diamonds}, and in relevant diamonds in \Cref{fig: M21 cusp}.

Fundamental group elements representing $\mu_0$ and $\lambda_0$ are obtained by reading off the labels in \Cref{fig: M21 cusp} in sequence, recalling from the proof of \Cref{M_21} that $a_0$, $a_3$, $b_0$, $c_0$, and $c_2$ belong to the maximal tree in the one-skeleton and hence do not themselves correspond to fundamental group generators.

From the description of $\lambda_0$ as an element of $\pi_1$, we obtain its first homology class:
\[ [\lambda_0] = -A_5 - B_2 - B_3 + B_4 - 2C_1 + 2C_4 + C_5 - 2C_6 - C_7 \]
Using \Cref{H1 REF} we find that in first homology this equals $-55\gamma$, where $\gamma = 5C_6 - C_7$ is the generator for $H_1(M_{2,1})$ from \Cref{M_21 H1}. The homology class $A_2+C_1$ can be found to be $-13\gamma$. It follows that $-17[\mu_0]+4[\lambda_0] = \gamma$, and that $55[\mu_0]-13[\lambda_0] = 0$. Therefore $\lambda \doteq \mu_0^{55}\lambda_0^{-13}$ is a homological longitude and $\mu\doteq \mu_0^{-17}\lambda_0^{4}$ a homological meridian, as claimed in the result's statement. Other homological meridians have the form $\mu\lambda^k = \mu_0^{-55k-17}\lambda_0^{4-13k}$ for $k\in\mathbb{Z}$. In general, it follows from the Pythagorean theorem that $\mu_0^a\lambda_0^b$ has length
\[ \sqrt{a^2|\mu_0|^2 + b^2|\lambda_0|^2} = (1+\sqrt{2})\sqrt{a^2 + 48\,b^2}. \]
From this and the coefficients it is easy to see that $\mu$ is the shortest homological meridian.
\end{proof}

\subsection{The two sides of $S_{2,1}$} Here we consider the closed, separating totally geodesic surface in $M_{2,1}$ described in \Cref{geodesic surface}, which we call $S_{2,1}$. Cutting $M_{2,1}$ along $S_{2,1}$ yields a compact piece $M_{2,1}^+$, which we analyze in \Cref{M_21plus}, and a non-compact piece $M_{2,1}^-$. We show that filling $M_{2,1}^-$ along the geometric meridian $\mu_0$ identified in \Cref{filling slope} yields a handlebody, and, using this description and \Cref{M_21plus}, establish the main theorem. This all uses a perspective on the spine $\Sigma$ from \Cref{M_21 spine} elaborated in the example below.

\begin{example}\label{ex: M_21 spine} \Cref{fig: spine from cusp} depicts the spine $\Sigma$ for $M_{2,1}$ from \Cref{M_21 spine}, viewed from the cusp. As in \Cref{fig: M21 cusp}, we should understand top to be identified to bottom and left to right by horizontal and vertical translations, respectively. The spine's edges are bold in the Figure, labeled as $a_i$, $b_j$, $c_k$ following the naming convention from the proof of \Cref{M_21}. 

\begin{figure}[ht]
\begin{tikzpicture}

\begin{scope}[scale=1.25]
\draw [gray] (-1,0) -- (0,0) -- (0.5,-0.866) -- (1,0) -- (0.5,0.866) -- (0,0);
\draw [gray] (1,0) -- (2,0) -- (1.5,-0.866) -- (-0.5,-0.866) -- (-1,0) -- (-0.5,0.866) -- (1.5,0.866) -- (2,0);

\fill (0,0) circle [radius=0.05];
\fill (1.5,0.866) circle [radius=0.05];
\fill (1.5,-0.866) circle [radius=0.05];

\draw [dashed] (0,0) -- (1,0);
\draw [dashed] (-0.5,0.866) -- (0,0) -- (-0.5,-0.866);
\draw [dashed] (2.5,0.866) -- (1.5,0.866) -- (1,0) -- (1.5,-0.866) -- (2.5,-0.866);

\node [below] at (-0.5,0.05) {\small $0$};
\node [above] at (-0.5,-0.05) {\tiny $(10)$};
\node [below] at (0.1,0.866) {\small $2$};
\node at (0.45,0.35) {\tiny $(6)$};
\node [below] at (0.4,0) {\small $1$};
\node [below] at (1.6,0) {\small $9$};
\node [below] at (1,0.866) {\small $8$};
\node at (1.55,0.4) {\tiny $(19)$};
\node [below] at (1.9,0.816) {\small $20$};
\node at (0.07,-0.533) {\tiny $(3)$};
\node [above] at (1,-0.886) {\tiny $(14)$};
\node at (1.95,-0.533) {\tiny $(21)$};

\draw [thick] (0,0) -- (-0.25,0.433) -- (0,0.866) -- (0.5,0.866);
\node [above] at (0,0.08) {\scriptsize $a_2$};
\node [below] at (-0.25,0.916) {\scriptsize $c_2$};
\draw [thick] (-0.25,0.433) -- (-0.75,0.433) -- (-1,0);
\node [above] at (-0.5,0.383) {\scriptsize $c_7$};
\node [above] at (-1,0.05) {\scriptsize $b_4$};
\draw [thick] (0,0) -- (-0.25,-0.433) -- (0,-0.866);
\node [below] at (0,-0.08) {\scriptsize $a_0$};
\node [above] at (-0.25,-0.916) {\scriptsize $c_1$};
\draw [thick] (-0.25,-0.433) -- (-0.75,-0.433) -- (-1,0);
\node [below] at (-0.5,-0.383) {\scriptsize $c_0$};
\node [below] at (-1,0) {\scriptsize $b_0$};

\draw [thick] (0,0) -- (0.5,0) -- (0.75,-0.433);
\node [above] at (0.3,-0.05) {\scriptsize $a_1$};
\node [below] at (0.75,0) {\scriptsize $c_1$};
\draw [thick] (0.5,0) -- (0.75,0.433) -- (0.5,0.866);
\node [above] at (0.75,0) {\scriptsize $c_0$};
\node [below] at (0.5,0.816) {\scriptsize $b_4$};
\draw [thick] (0.75,0.433) -- (1.25,0.433) -- (1.5,0.866) -- (2,0.866);
\node [below] at (1,0.483) {\scriptsize $c_6$};
\node [below] at (1.5,0.816) {\scriptsize $a_4$};
\draw [thick] (1.25,0.433) -- (1.5,0) -- (2,0);
\node [above] at (1.25,-0.05) {\scriptsize $c_4$};
\node [above] at (1.75,-0.05) {\scriptsize $b_5$};
\draw [thick] (1.5,0) -- (1.25,-0.433);
\node [below] at (1.25,0.05) {\scriptsize $c_7$};

\draw [thick] (0,-0.866) -- (0.5,-0.866);
\node [below] at (0.25,-0.816) {\footnotesize $b_2$};
\draw [thick] (0.5,-0.866) -- (0.75,-0.433);
\node [above] at (0.5,-0.816) {\scriptsize $b_1$};
\draw [thick] (0.75,-0.433) -- (1.25,-0.433);
\node [above] at (1,-0.483) {\scriptsize $c_2$};
\draw [thick] (1.25,-0.433) -- (1.5,-0.866);
\node [above] at (1.5,-0.816) {\scriptsize $a_5$};
\draw [thick] (1.5,-0.866) -- (2,-0.866);
\node [below] at (1.75,-0.816) {\footnotesize $a_3$};
\draw [thick] (2,-0.866) -- (2.25,-0.433);
\node [above] at (2.25,-0.916) {\scriptsize $c_7$};
\draw [thick] (2,0.866) -- (2.25,0.433);
\node [below] at (2.25,0.916) {\scriptsize $c_6$};

\end{scope}

\begin{scope}[scale=1.25, xshift=3cm]
\draw [gray] (-1,0) -- (0,0) -- (0.5,-0.866) -- (1,0) -- (0.5,0.866) -- (0,0);
\draw [gray] (1,0) -- (2,0) -- (1.5,-0.866) -- (-0.5,-0.866) -- (-1,0) -- (-0.5,0.866) -- (1.5,0.866) -- (2,0);

\fill (0,0) circle [radius=0.05];
\fill (1.5,0.866) circle [radius=0.05];
\fill (1.5,-0.866) circle [radius=0.05];

\draw [dashed] (0,0) -- (1,0);
\draw [dashed] (-0.5,0.866) -- (0,0) -- (-0.5,-0.866);
\draw [dashed] (2.5,0.866) -- (1.5,0.866) -- (1,0) -- (1.5,-0.866) -- (2.5,-0.866);

\node [below] at (-0.5,0.05) {\small $11$};
\node [above] at (-0.5,-0.05) {\tiny $(16)$};
\node [below] at (0.1,0.866) {\small $18$};
\node at (0.45,0.35) {\tiny $(7)$};
\node [below] at (0.4,0) {\small $6$};
\node [below] at (1.6,0) {\small $21$};
\node [below] at (1,0.866) {\small $22$};
\node at (1.55,0.4) {\tiny $(13)$};
\node [below] at (1.9,0.816) {\small $23$};
\node at (0.07,-0.533) {\tiny $(2)$};
\node [above] at (1,-0.886) {\tiny $(5)$};
\node at (1.95,-0.533) {\tiny $(9)$};

\draw [thick] (0,0) -- (-0.25,0.433) -- (0,0.866) -- (0.5,0.866);
\node [above] at (0,0.08) {\scriptsize $a_6$}; %
\node [below] at (-0.25,0.916) {\scriptsize $c_0$}; %
\draw [thick] (-0.25,0.433) -- (-0.75,0.433) -- (-1,0);
\node [above] at (-0.5,0.383) {\scriptsize $c_3$}; %
\node [above] at (-1,0.05) {\scriptsize $b_0$}; %
\draw [thick] (0,0) -- (-0.25,-0.433) -- (0,-0.866);
\node [below] at (0,-0.08) {\scriptsize $a_2$}; %
\node [above] at (-0.25,-0.916) {\scriptsize $c_2$}; %
\draw [thick] (-0.25,-0.433) -- (-0.75,-0.433) -- (-1,0);
\node [below] at (-0.5,-0.383) {\scriptsize $c_6$}; %
\node [below] at (-1,0) {\scriptsize $b_3$}; %

\draw [thick] (0,0) -- (0.5,0) -- (0.75,-0.433);
\node [above] at (0.3,-0.05) {\scriptsize $a_1$}; %
\node [below] at (0.75,0) {\scriptsize $c_0$}; %
\draw [thick] (0.5,0) -- (0.75,0.433) -- (0.5,0.866);
\node [above] at (0.75,0) {\scriptsize $c_5$}; %
\node [below] at (0.5,0.816) {\scriptsize $b_5$}; %
\draw [thick] (0.75,0.433) -- (1.25,0.433) -- (1.5,0.866) -- (2,0.866);
\node [below] at (1,0.483) {\scriptsize $c_1$}; %
\node [below] at (1.5,0.816) {\scriptsize $a_7$}; %
\draw [thick] (1.25,0.433) -- (1.5,0) -- (2,0);
\node [above] at (1.25,-0.05) {\scriptsize $c_5$}; %
\node [above] at (1.75,-0.05) {\scriptsize $b_3$}; %
\draw [thick] (1.5,0) -- (1.25,-0.433);
\node [below] at (1.25,0.05) {\scriptsize $c_7$}; %

\draw [thick] (0,-0.866) -- (0.5,-0.866);
\node [below] at (0.25,-0.816) {\footnotesize $b_2$}; %
\draw [thick] (0.5,-0.866) -- (0.75,-0.433);
\node [above] at (0.5,-0.816) {\scriptsize $b_4$}; %
\draw [thick] (0.75,-0.433) -- (1.25,-0.433);
\node [above] at (1,-0.483) {\scriptsize $c_4$}; %
\draw [thick] (1.25,-0.433) -- (1.5,-0.866);
\node [above] at (1.5,-0.816) {\scriptsize $a_3$}; %
\draw [thick] (1.5,-0.866) -- (2,-0.866);
\node [below] at (1.75,-0.816) {\footnotesize $a_5$}; %
\draw [thick] (2,-0.866) -- (2.25,-0.433);
\node [above] at (2.25,-0.916) {\scriptsize $c_7$}; %
\draw [thick] (2,0.866) -- (2.25,0.433);
\node [below] at (2.25,0.916) {\scriptsize $c_4$}; %

\end{scope}

\begin{scope}[scale=1.25, xshift=6cm]
\draw [gray] (-1,0) -- (0,0) -- (0.5,-0.866) -- (1,0) -- (0.5,0.866) -- (0,0);
\draw [gray] (1,0) -- (2,0) -- (1.5,-0.866) -- (-0.5,-0.866) -- (-1,0) -- (-0.5,0.866) -- (1.5,0.866) -- (2,0);

\fill (0,0) circle [radius=0.05];
\fill (1.5,0.866) circle [radius=0.05];
\fill (1.5,-0.866) circle [radius=0.05];

\draw [dashed] (0,0) -- (1,0);
\draw [dashed] (-0.5,0.866) -- (0,0) -- (-0.5,-0.866);
\draw [dashed] (2.5,0.866) -- (1.5,0.866) -- (1,0) -- (1.5,-0.866) -- (2.5,-0.866);

\node [below] at (-0.5,0.05) {\small $7$}; %
\node [above] at (-0.5,-0.05) {\tiny $(18)$}; %
\node [below] at (0.1,0.866) {\small $12$}; %
\node at (0.45,0.35) {\tiny $(4)$}; %
\node [below] at (0.4,0) {\small $3$}; %
\node [below] at (1.6,0) {\small $14$}; %
\node [below] at (1,0.866) {\small $13$}; %
\node at (1.55,0.4) {\tiny $(15)$}; %
\node [below] at (1.9,0.816) {\small $17$}; %
\node at (0.07,-0.533) {\tiny $(1)$}; %
\node [above] at (1,-0.886) {\tiny $(23)$}; %
\node at (1.95,-0.533) {\tiny $(8)$}; %

\draw [thick] (0,0) -- (-0.25,0.433) -- (0,0.866) -- (0.5,0.866);
\node [above] at (0,0.08) {\scriptsize $a_6$}; %
\node [below] at (-0.25,0.916) {\scriptsize $c_5$}; %
\draw [thick] (-0.25,0.433) -- (-0.75,0.433) -- (-1,0);
\node [above] at (-0.5,0.383) {\scriptsize $c_0$}; %
\node [above] at (-1,0.05) {\scriptsize $b_2$}; %
\draw [thick] (0,0) -- (-0.25,-0.433) -- (0,-0.866);
\node [below] at (0,-0.08) {\scriptsize $a_1$}; %
\node [above] at (-0.25,-0.916) {\scriptsize $c_1$}; %
\draw [thick] (-0.25,-0.433) -- (-0.75,-0.433) -- (-1,0);
\node [below] at (-0.5,-0.383) {\scriptsize $c_5$}; %
\node [below] at (-1,0) {\scriptsize $b_5$}; %

\draw [thick] (0,0) -- (0.5,0) -- (0.75,-0.433);
\node [above] at (0.3,-0.05) {\scriptsize $a_0$}; %
\node [below] at (0.75,0) {\scriptsize $c_1$}; %
\draw [thick] (0.5,0) -- (0.75,0.433) -- (0.5,0.866);
\node [above] at (0.75,0) {\scriptsize $c_3$}; %
\node [below] at (0.5,0.816) {\scriptsize $b_3$}; %
\draw [thick] (0.75,0.433) -- (1.25,0.433) -- (1.5,0.866) -- (2,0.866);
\node [below] at (1,0.483) {\scriptsize $c_5$}; %
\node [below] at (1.5,0.816) {\scriptsize $a_7$}; %
\draw [thick] (1.25,0.433) -- (1.5,0) -- (2,0);
\node [above] at (1.25,-0.05) {\scriptsize $c_3$}; %
\node [above] at (1.75,-0.05) {\scriptsize $b_1$}; %
\draw [thick] (1.5,0) -- (1.25,-0.433);
\node [below] at (1.25,0.05) {\scriptsize $c_2$}; %

\draw [thick] (0,-0.866) -- (0.5,-0.866);
\node [below] at (0.25,-0.816) {\footnotesize $b_1$}; %
\draw [thick] (0.5,-0.866) -- (0.75,-0.433);
\node [above] at (0.5,-0.816) {\scriptsize $b_2$}; %
\draw [thick] (0.75,-0.433) -- (1.25,-0.433);
\node [above] at (1,-0.483) {\scriptsize $c_4$}; %
\draw [thick] (1.25,-0.433) -- (1.5,-0.866);
\node [above] at (1.5,-0.816) {\scriptsize $a_5$}; %
\draw [thick] (1.5,-0.866) -- (2,-0.866);
\node [below] at (1.75,-0.816) {\footnotesize $a_4$}; %
\draw [thick] (2,-0.866) -- (2.25,-0.433);
\node [above] at (2.25,-0.916) {\scriptsize $c_6$}; %
\draw [thick] (2,0.866) -- (2.25,0.433);
\node [below] at (2.25,0.916) {\scriptsize $c_2$}; %

\end{scope}

\begin{scope}[scale=1.25, xshift=9cm]
\draw [gray] (-1,0) -- (0,0) -- (0.5,-0.866) -- (1,0) -- (0.5,0.866) -- (0,0);
\draw [gray] (1,0) -- (2,0) -- (1.5,-0.866) -- (-0.5,-0.866) -- (-1,0) -- (-0.5,0.866) -- (1.5,0.866) -- (2,0);
\draw [gray] (2.5,-0.866) -- (2,0) -- (2.5,0.866);

\fill (0,0) circle [radius=0.05];
\fill (1.5,0.866) circle [radius=0.05];
\fill (1.5,-0.866) circle [radius=0.05];

\draw [dashed] (0,0) -- (1,0);
\draw [dashed] (-0.5,0.866) -- (0,0) -- (-0.5,-0.866);
\draw [dashed] (2.5,0.866) -- (1.5,0.866) -- (1,0) -- (1.5,-0.866) -- (2.5,-0.866);

\node [below] at (-0.5,0.05) {\small $10$}; %
\node [above] at (-0.5,-0.05) {\tiny $(0)$}; %
\node [below] at (0.1,0.866) {\small $4$}; %
\node at (0.45,0.35) {\tiny $(12)$}; %
\node [below] at (0.4,0) {\small $16$}; %
\node [below] at (1.6,0) {\small $19$}; %
\node [below] at (1,0.866) {\small $15$}; %
\node at (1.55,0.4) {\tiny $(22)$}; %
\node [below] at (1.9,0.816) {\small $5$}; %
\node at (0.07,-0.533) {\tiny $(11)$}; %
\node [above] at (1,-0.886) {\tiny $(17)$}; %
\node at (1.95,-0.533) {\tiny $(20)$}; %

\draw [thick] (0,0) -- (-0.25,0.433) -- (0,0.866) -- (0.5,0.866);
\node [above] at (0,0.08) {\scriptsize $a_0$}; %
\node [below] at (-0.25,0.916) {\scriptsize $c_3$}; %
\draw [thick] (-0.25,0.433) -- (-0.75,0.433) -- (-1,0);
\node [above] at (-0.5,0.383) {\scriptsize $c_0$}; %
\node [above] at (-1,0.05) {\scriptsize $b_0$}; %
\draw [thick] (0,0) -- (-0.25,-0.433) -- (0,-0.866);
\node [below] at (0,-0.08) {\scriptsize $a_2$}; %
\node [above] at (-0.25,-0.916) {\scriptsize $c_6$}; %
\draw [thick] (-0.25,-0.433) -- (-0.75,-0.433) -- (-1,0);
\node [below] at (-0.5,-0.383) {\scriptsize $c_7$}; %
\node [below] at (-1,0) {\scriptsize $b_4$}; %

\draw [thick] (0,0) -- (0.5,0) -- (0.75,-0.433);
\node [above] at (0.3,-0.05) {\scriptsize $a_6$}; %
\node [below] at (0.75,0) {\scriptsize $c_3$}; %
\draw [thick] (0.5,0) -- (0.75,0.433) -- (0.5,0.866);
\node [above] at (0.75,0) {\scriptsize $c_5$}; %
\node [below] at (0.5,0.816) {\scriptsize $b_1$}; %
\draw [thick] (0.75,0.433) -- (1.25,0.433) -- (1.5,0.866) -- (2,0.866);
\node [below] at (1,0.483) {\scriptsize $c_3$}; %
\node [below] at (1.5,0.816) {\scriptsize $a_7$}; %
\draw [thick] (1.25,0.433) -- (1.5,0) -- (2,0);
\node [above] at (1.25,-0.05) {\scriptsize $c_1$}; %
\node [above] at (1.75,-0.05) {\scriptsize $b_5$}; %
\draw [thick] (1.5,0) -- (1.25,-0.433);
\node [below] at (1.25,0.05) {\scriptsize $c_4$}; %

\draw [thick] (0,-0.866) -- (0.5,-0.866);
\node [below] at (0.25,-0.816) {\footnotesize $b_3$}; %
\draw [thick] (0.5,-0.866) -- (0.75,-0.433);
\node [above] at (0.5,-0.816) {\scriptsize $b_0$}; %
\draw [thick] (0.75,-0.433) -- (1.25,-0.433);
\node [above] at (1,-0.483) {\scriptsize $c_2$}; %
\draw [thick] (1.25,-0.433) -- (1.5,-0.866);
\node [above] at (1.5,-0.816) {\scriptsize $a_4$}; %
\draw [thick] (1.5,-0.866) -- (2,-0.866);
\node [below] at (1.75,-0.816) {\footnotesize $a_3$}; %
\draw [thick] (2,-0.866) -- (2.25,-0.433) -- (2,0);
\node [above] at (2.25,-0.916) {\scriptsize $c_6$}; %
\draw [thick] (2,0.866) -- (2.25,0.433) -- (2,0);
\node [below] at (2.25,0.916) {\scriptsize $c_4$}; %

\end{scope}

\end{tikzpicture}
\caption{The spine of $M_{2,1}$, viewed from its cusp.}
\label{fig: spine from cusp}
\end{figure}
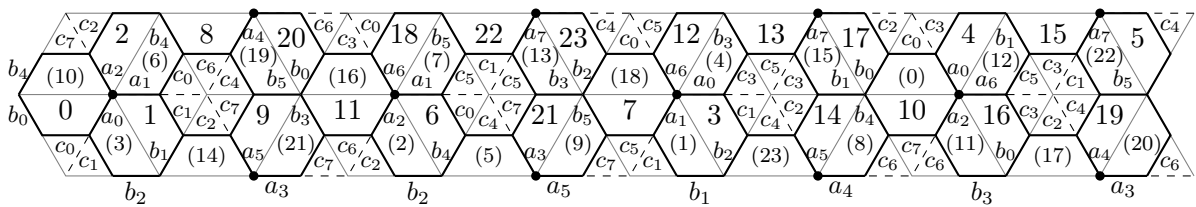

Each face of the spine $\Sigma$ appears as two hexagons in the Figure, viewed from both sides. Let $\Sigma_0$ denote this  ``visible pre-spine''. Formally, $\Sigma_0$ is the path-closure of $M_{2,1}-\Sigma$, and $M_{2,1}$ is recovered from $\Sigma_0\times [0,\infty)$ by pairing hexagons of $\Sigma_0\times\{0\}$ in label-preserving fashion. We note that \Cref{fig: spine from cusp} is not geometrically accurate: in an accurate cusp view, the $c_i$-labeled hexagons would appear smaller than the others, and no hexagon would be equi-angular. The left side of \Cref{rigeur} can be used to recover a geometrically accurate view.

In \Cref{fig: spine from cusp} each face of $\Sigma_0$ bounded by all edges of $c$ type is a union of six faces of the $\dpr_i$ labeled ``$3_-$'' in \Cref{numberings}. Each of the others is a union of two faces labeled ``$1_-$'' in \Cref{numberings}. Each of these faces has a label ``$i$'' on the side contained in the $\dpr_i$ with ideal vertex pointing into the cusp, and a label ``$j$'' on the other side, where $j = i.y^{-1}x$ as in \ref{whyex again}. Each of these thus appears once with labels $i,(j)$ and once with labels $j,(i)$. 
\end{example}

\begin{example}\label{ex: tot geod from cusp} \Cref{fig: tot geod from cusp} depicts the closed, embedded totally geodesic surface $S_{2,1}$ in $M_{2,1}$ constructed in the proof of \Cref{geodesic surface}, viewed from the cusp of $M_{2,1}$.

\begin{figure}[ht]
\begin{tikzpicture}

\begin{scope}[scale=1.25]
\draw [gray] (-1,0) -- (0,0) -- (0.5,-0.866) -- (1,0) -- (0.5,0.866) -- (0,0);
\draw [gray] (1,0) -- (2,0) -- (1.5,-0.866) -- (-0.5,-0.866) -- (-1,0) -- (-0.5,0.866) -- (1.5,0.866) -- (2,0);

\fill (0,0) circle [radius=0.05];
\fill (1.5,0.866) circle [radius=0.05];
\fill (1.5,-0.866) circle [radius=0.05];

\draw [dashed] (0,0) -- (1,0);
\draw [dashed] (-0.5,0.866) -- (0,0) -- (-0.5,-0.866);
\draw [dashed] (2.5,0.866) -- (1.5,0.866) -- (1,0) -- (1.5,-0.866) -- (2.5,-0.866);

\node [below] at (-0.5,0.05) {\small $0$};
\node [above] at (-0.5,-0.05) {\tiny $(10)$};
\node [below] at (0.1,0.866) {\small $2$};
\node at (0.45,0.35) {\tiny $(6)$};
\node [below] at (0.4,0) {\small $1$};
\node [below] at (1.6,0) {\small $9$};
\node [below] at (1,0.866) {\small $8$};
\node at (1.55,0.4) {\tiny $(19)$};
\node [below] at (1.9,0.816) {\small $20$};
\node at (0.07,-0.533) {\tiny $(3)$};
\node [above] at (1,-0.886) {\tiny $(14)$};
\node at (1.95,-0.533) {\tiny $(21)$};

\draw [thick] (0,0) -- (-0.25,0.433) -- (0,0.866) -- (0.5,0.866);
\node [above] at (0,0.08) {\scriptsize $a_2$};
\node [below] at (-0.25,0.916) {\scriptsize $c_2$};
\draw [thick] (-0.25,0.433) -- (-0.75,0.433) -- (-1,0);
\node [above] at (-0.5,0.383) {\scriptsize $c_7$};
\node [above] at (-1,0.05) {\scriptsize $b_4$};
\draw [thick] (0,0) -- (-0.25,-0.433) -- (0,-0.866);
\node [below] at (0,-0.08) {\scriptsize $a_0$};
\node [above] at (-0.25,-0.916) {\scriptsize $c_1$};
\draw [thick] (-0.25,-0.433) -- (-0.75,-0.433) -- (-1,0);
\node [below] at (-0.5,-0.383) {\scriptsize $c_0$};
\node [below] at (-1,0) {\scriptsize $b_0$};

\draw [thick] (0,0) -- (0.5,0) -- (0.75,-0.433);
\node [above] at (0.3,-0.05) {\scriptsize $a_1$};
\node [below] at (0.75,0) {\scriptsize $c_1$};
\draw [thick] (0.5,0) -- (0.75,0.433) -- (0.5,0.866);
\node [above] at (0.75,0) {\scriptsize $c_0$};
\node [below] at (0.5,0.816) {\scriptsize $b_4$};
\draw [thick] (0.75,0.433) -- (1.25,0.433) -- (1.5,0.866) -- (2,0.866);
\node [below] at (1,0.483) {\scriptsize $c_6$};
\node [below] at (1.5,0.816) {\scriptsize $a_4$};
\draw [thick] (1.25,0.433) -- (1.5,0) -- (2,0);
\node [above] at (1.25,-0.05) {\scriptsize $c_4$};
\node [above] at (1.75,-0.05) {\scriptsize $b_5$};
\draw [thick] (1.5,0) -- (1.25,-0.433);
\node [below] at (1.25,0.05) {\scriptsize $c_7$};

\draw [thick] (0,-0.866) -- (0.5,-0.866);
\node [below] at (0.25,-0.816) {\footnotesize $b_2$};
\draw [thick] (0.5,-0.866) -- (0.75,-0.433);
\node [above] at (0.5,-0.816) {\scriptsize $b_1$};
\draw [thick] (0.75,-0.433) -- (1.25,-0.433);
\node [above] at (1,-0.483) {\scriptsize $c_2$};
\draw [thick] (1.25,-0.433) -- (1.5,-0.866);
\node [above] at (1.5,-0.816) {\scriptsize $a_5$};
\draw [thick] (1.5,-0.866) -- (2,-0.866);
\node [below] at (1.75,-0.816) {\footnotesize $a_3$};
\draw [thick] (2,-0.866) -- (2.25,-0.433);
\node [above] at (2.25,-0.916) {\scriptsize $c_7$};
\draw [thick] (2,0.866) -- (2.25,0.433);
\node [below] at (2.25,0.916) {\scriptsize $c_6$};

\fill [opacity=0.1] (1,0) circle [radius=0.7];
\draw [thick] (1,0) circle [radius=0.7];
\fill [opacity=0.1] (-0.5,-0.866) -- (0.2,-0.866) -- (-0.85,-0.259);
\fill [opacity=0.1] (0.2,-0.866) arc (0:120:0.7);
\draw [thick] (0.2,-0.866) arc (0:120:0.7);
\fill [opacity=0.1] (-0.5,0.866) -- (0.2,0.866) -- (-0.85,0.259);
\fill [opacity=0.1] (0.2,0.866) arc (0:-120:0.7);
\draw [thick] (0.2,0.866) arc (0:-120:0.7);
\fill [opacity=0.1] (1.8,-0.866) -- (2.5,-0.866) -- (2.15,-0.259);
\fill [opacity=0.1] (1.8,-0.866) arc (180:120:0.7);
\draw [thick] (1.8,-0.866) arc (180:120:0.7);
\fill [opacity=0.1] (1.8,0.866) -- (2.5,0.866) -- (2.15,0.259);
\fill [opacity=0.1] (1.8,0.866) arc (180:240:0.7);
\draw [thick] (1.8,0.866) arc (180:240:0.7);

\end{scope}

\begin{scope}[scale=1.25, xshift=3cm]
\draw [gray] (-1,0) -- (0,0) -- (0.5,-0.866) -- (1,0) -- (0.5,0.866) -- (0,0);
\draw [gray] (1,0) -- (2,0) -- (1.5,-0.866) -- (-0.5,-0.866) -- (-1,0) -- (-0.5,0.866) -- (1.5,0.866) -- (2,0);

\fill (0,0) circle [radius=0.05];
\fill (1.5,0.866) circle [radius=0.05];
\fill (1.5,-0.866) circle [radius=0.05];

\draw [dashed] (0,0) -- (1,0);
\draw [dashed] (-0.5,0.866) -- (0,0) -- (-0.5,-0.866);
\draw [dashed] (2.5,0.866) -- (1.5,0.866) -- (1,0) -- (1.5,-0.866) -- (2.5,-0.866);

\node [below] at (-0.5,0.05) {\small $11$};
\node [above] at (-0.5,-0.05) {\tiny $(16)$};
\node [below] at (0.1,0.866) {\small $18$};
\node at (0.45,0.35) {\tiny $(7)$};
\node [below] at (0.4,0) {\small $6$};
\node [below] at (1.6,0) {\small $21$};
\node [below] at (1,0.866) {\small $22$};
\node at (1.55,0.4) {\tiny $(13)$};
\node [below] at (1.9,0.816) {\small $23$};
\node at (0.07,-0.533) {\tiny $(2)$};
\node [above] at (1,-0.886) {\tiny $(5)$};
\node at (1.95,-0.533) {\tiny $(9)$};

\draw [thick] (0,0) -- (-0.25,0.433) -- (0,0.866) -- (0.5,0.866);
\node [above] at (0,0.08) {\scriptsize $a_6$}; %
\node [below] at (-0.25,0.916) {\scriptsize $c_0$}; %
\draw [thick] (-0.25,0.433) -- (-0.75,0.433) -- (-1,0);
\node [above] at (-0.5,0.383) {\scriptsize $c_3$}; %
\node [above] at (-1,0.05) {\scriptsize $b_0$}; %
\draw [thick] (0,0) -- (-0.25,-0.433) -- (0,-0.866);
\node [below] at (0,-0.08) {\scriptsize $a_2$}; %
\node [above] at (-0.25,-0.916) {\scriptsize $c_2$}; %
\draw [thick] (-0.25,-0.433) -- (-0.75,-0.433) -- (-1,0);
\node [below] at (-0.5,-0.383) {\scriptsize $c_6$}; %
\node [below] at (-1,0) {\scriptsize $b_3$}; %

\draw [thick] (0,0) -- (0.5,0) -- (0.75,-0.433);
\node [above] at (0.3,-0.05) {\scriptsize $a_1$}; %
\node [below] at (0.75,0) {\scriptsize $c_0$}; %
\draw [thick] (0.5,0) -- (0.75,0.433) -- (0.5,0.866);
\node [above] at (0.75,0) {\scriptsize $c_5$}; %
\node [below] at (0.5,0.816) {\scriptsize $b_5$}; %
\draw [thick] (0.75,0.433) -- (1.25,0.433) -- (1.5,0.866) -- (2,0.866);
\node [below] at (1,0.483) {\scriptsize $c_1$}; %
\node [below] at (1.5,0.816) {\scriptsize $a_7$}; %
\draw [thick] (1.25,0.433) -- (1.5,0) -- (2,0);
\node [above] at (1.25,-0.05) {\scriptsize $c_5$}; %
\node [above] at (1.75,-0.05) {\scriptsize $b_3$}; %
\draw [thick] (1.5,0) -- (1.25,-0.433);
\node [below] at (1.25,0.05) {\scriptsize $c_7$}; %

\draw [thick] (0,-0.866) -- (0.5,-0.866);
\node [below] at (0.25,-0.816) {\footnotesize $b_2$}; %
\draw [thick] (0.5,-0.866) -- (0.75,-0.433);
\node [above] at (0.5,-0.816) {\scriptsize $b_4$}; %
\draw [thick] (0.75,-0.433) -- (1.25,-0.433);
\node [above] at (1,-0.483) {\scriptsize $c_4$}; %
\draw [thick] (1.25,-0.433) -- (1.5,-0.866);
\node [above] at (1.5,-0.816) {\scriptsize $a_3$}; %
\draw [thick] (1.5,-0.866) -- (2,-0.866);
\node [below] at (1.75,-0.816) {\footnotesize $a_5$}; %
\draw [thick] (2,-0.866) -- (2.25,-0.433);
\node [above] at (2.25,-0.916) {\scriptsize $c_7$}; %
\draw [thick] (2,0.866) -- (2.25,0.433);
\node [below] at (2.25,0.916) {\scriptsize $c_4$}; %

\fill [opacity=0.1] (1,0) circle [radius=0.7];
\draw [thick] (1,0) circle [radius=0.7];
\fill [opacity=0.1] (-0.5,-0.866) -- (0.2,-0.866) -- (-0.85,-0.259);
\fill [opacity=0.1] (0.2,-0.866) arc (0:120:0.7);
\draw [thick] (0.2,-0.866) arc (0:120:0.7);
\fill [opacity=0.1] (-0.5,0.866) -- (0.2,0.866) -- (-0.85,0.259);
\fill [opacity=0.1] (0.2,0.866) arc (0:-120:0.7);
\draw [thick] (0.2,0.866) arc (0:-120:0.7);
\fill [opacity=0.1] (1.8,-0.866) -- (2.5,-0.866) -- (2.15,-0.259);
\fill [opacity=0.1] (1.8,-0.866) arc (180:120:0.7);
\draw [thick] (1.8,-0.866) arc (180:120:0.7);
\fill [opacity=0.1] (1.8,0.866) -- (2.5,0.866) -- (2.15,0.259);
\fill [opacity=0.1] (1.8,0.866) arc (180:240:0.7);
\draw [thick] (1.8,0.866) arc (180:240:0.7);

\end{scope}

\begin{scope}[scale=1.25, xshift=6cm]
\draw [gray] (-1,0) -- (0,0) -- (0.5,-0.866) -- (1,0) -- (0.5,0.866) -- (0,0);
\draw [gray] (1,0) -- (2,0) -- (1.5,-0.866) -- (-0.5,-0.866) -- (-1,0) -- (-0.5,0.866) -- (1.5,0.866) -- (2,0);

\fill (0,0) circle [radius=0.05];
\fill (1.5,0.866) circle [radius=0.05];
\fill (1.5,-0.866) circle [radius=0.05];

\draw [dashed] (0,0) -- (1,0);
\draw [dashed] (-0.5,0.866) -- (0,0) -- (-0.5,-0.866);
\draw [dashed] (2.5,0.866) -- (1.5,0.866) -- (1,0) -- (1.5,-0.866) -- (2.5,-0.866);

\node [below] at (-0.5,0.05) {\small $7$}; %
\node [above] at (-0.5,-0.05) {\tiny $(18)$}; %
\node [below] at (0.1,0.866) {\small $12$}; %
\node at (0.45,0.35) {\tiny $(4)$}; %
\node [below] at (0.4,0) {\small $3$}; %
\node [below] at (1.6,0) {\small $14$}; %
\node [below] at (1,0.866) {\small $13$}; %
\node at (1.55,0.4) {\tiny $(15)$}; %
\node [below] at (1.9,0.816) {\small $17$}; %
\node at (0.07,-0.533) {\tiny $(1)$}; %
\node [above] at (1,-0.886) {\tiny $(23)$}; %
\node at (1.95,-0.533) {\tiny $(8)$}; %

\draw [thick] (0,0) -- (-0.25,0.433) -- (0,0.866) -- (0.5,0.866);
\node [above] at (0,0.08) {\scriptsize $a_6$}; %
\node [below] at (-0.25,0.916) {\scriptsize $c_5$}; %
\draw [thick] (-0.25,0.433) -- (-0.75,0.433) -- (-1,0);
\node [above] at (-0.5,0.383) {\scriptsize $c_0$}; %
\node [above] at (-1,0.05) {\scriptsize $b_2$}; %
\draw [thick] (0,0) -- (-0.25,-0.433) -- (0,-0.866);
\node [below] at (0,-0.08) {\scriptsize $a_1$}; %
\node [above] at (-0.25,-0.916) {\scriptsize $c_1$}; %
\draw [thick] (-0.25,-0.433) -- (-0.75,-0.433) -- (-1,0);
\node [below] at (-0.5,-0.383) {\scriptsize $c_5$}; %
\node [below] at (-1,0) {\scriptsize $b_5$}; %

\draw [thick] (0,0) -- (0.5,0) -- (0.75,-0.433);
\node [above] at (0.3,-0.05) {\scriptsize $a_0$}; %
\node [below] at (0.75,0) {\scriptsize $c_1$}; %
\draw [thick] (0.5,0) -- (0.75,0.433) -- (0.5,0.866);
\node [above] at (0.75,0) {\scriptsize $c_3$}; %
\node [below] at (0.5,0.816) {\scriptsize $b_3$}; %
\draw [thick] (0.75,0.433) -- (1.25,0.433) -- (1.5,0.866) -- (2,0.866);
\node [below] at (1,0.483) {\scriptsize $c_5$}; %
\node [below] at (1.5,0.816) {\scriptsize $a_7$}; %
\draw [thick] (1.25,0.433) -- (1.5,0) -- (2,0);
\node [above] at (1.25,-0.05) {\scriptsize $c_3$}; %
\node [above] at (1.75,-0.05) {\scriptsize $b_1$}; %
\draw [thick] (1.5,0) -- (1.25,-0.433);
\node [below] at (1.25,0.05) {\scriptsize $c_2$}; %

\draw [thick] (0,-0.866) -- (0.5,-0.866);
\node [below] at (0.25,-0.816) {\footnotesize $b_1$}; %
\draw [thick] (0.5,-0.866) -- (0.75,-0.433);
\node [above] at (0.5,-0.816) {\scriptsize $b_2$}; %
\draw [thick] (0.75,-0.433) -- (1.25,-0.433);
\node [above] at (1,-0.483) {\scriptsize $c_4$}; %
\draw [thick] (1.25,-0.433) -- (1.5,-0.866);
\node [above] at (1.5,-0.816) {\scriptsize $a_5$}; %
\draw [thick] (1.5,-0.866) -- (2,-0.866);
\node [below] at (1.75,-0.816) {\footnotesize $a_4$}; %
\draw [thick] (2,-0.866) -- (2.25,-0.433);
\node [above] at (2.25,-0.916) {\scriptsize $c_6$}; %
\draw [thick] (2,0.866) -- (2.25,0.433);
\node [below] at (2.25,0.916) {\scriptsize $c_2$}; %

\fill [opacity=0.1] (1,0) circle [radius=0.7];
\draw [thick] (1,0) circle [radius=0.7];
\fill [opacity=0.1] (-0.5,-0.866) -- (0.2,-0.866) -- (-0.85,-0.259);
\fill [opacity=0.1] (0.2,-0.866) arc (0:120:0.7);
\draw [thick] (0.2,-0.866) arc (0:120:0.7);
\fill [opacity=0.1] (-0.5,0.866) -- (0.2,0.866) -- (-0.85,0.259);
\fill [opacity=0.1] (0.2,0.866) arc (0:-120:0.7);
\draw [thick] (0.2,0.866) arc (0:-120:0.7);
\fill [opacity=0.1] (1.8,-0.866) -- (2.5,-0.866) -- (2.15,-0.259);
\fill [opacity=0.1] (1.8,-0.866) arc (180:120:0.7);
\draw [thick] (1.8,-0.866) arc (180:120:0.7);
\fill [opacity=0.1] (1.8,0.866) -- (2.5,0.866) -- (2.15,0.259);
\fill [opacity=0.1] (1.8,0.866) arc (180:240:0.7);
\draw [thick] (1.8,0.866) arc (180:240:0.7);

\end{scope}

\begin{scope}[scale=1.25, xshift=9cm]
\draw [gray] (-1,0) -- (0,0) -- (0.5,-0.866) -- (1,0) -- (0.5,0.866) -- (0,0);
\draw [gray] (1,0) -- (2,0) -- (1.5,-0.866) -- (-0.5,-0.866) -- (-1,0) -- (-0.5,0.866) -- (1.5,0.866) -- (2,0);
\draw [gray] (2.5,-0.866) -- (2,0) -- (2.5,0.866);

\fill (0,0) circle [radius=0.05];
\fill (1.5,0.866) circle [radius=0.05];
\fill (1.5,-0.866) circle [radius=0.05];

\draw [dashed] (0,0) -- (1,0);
\draw [dashed] (-0.5,0.866) -- (0,0) -- (-0.5,-0.866);
\draw [dashed] (2.5,0.866) -- (1.5,0.866) -- (1,0) -- (1.5,-0.866) -- (2.5,-0.866);

\node [below] at (-0.5,0.05) {\small $10$}; %
\node [above] at (-0.5,-0.05) {\tiny $(0)$}; %
\node [below] at (0.1,0.866) {\small $4$}; %
\node at (0.45,0.35) {\tiny $(12)$}; %
\node [below] at (0.4,0) {\small $16$}; %
\node [below] at (1.6,0) {\small $19$}; %
\node [below] at (1,0.866) {\small $15$}; %
\node at (1.55,0.4) {\tiny $(22)$}; %
\node [below] at (1.9,0.816) {\small $5$}; %
\node at (0.07,-0.533) {\tiny $(11)$}; %
\node [above] at (1,-0.886) {\tiny $(17)$}; %
\node at (1.95,-0.533) {\tiny $(20)$}; %

\draw [thick] (0,0) -- (-0.25,0.433) -- (0,0.866) -- (0.5,0.866);
\node [above] at (0,0.08) {\scriptsize $a_0$}; %
\node [below] at (-0.25,0.916) {\scriptsize $c_3$}; %
\draw [thick] (-0.25,0.433) -- (-0.75,0.433) -- (-1,0);
\node [above] at (-0.5,0.383) {\scriptsize $c_0$}; %
\node [above] at (-1,0.05) {\scriptsize $b_0$}; %
\draw [thick] (0,0) -- (-0.25,-0.433) -- (0,-0.866);
\node [below] at (0,-0.08) {\scriptsize $a_2$}; %
\node [above] at (-0.25,-0.916) {\scriptsize $c_6$}; %
\draw [thick] (-0.25,-0.433) -- (-0.75,-0.433) -- (-1,0);
\node [below] at (-0.5,-0.383) {\scriptsize $c_7$}; %
\node [below] at (-1,0) {\scriptsize $b_4$}; %

\draw [thick] (0,0) -- (0.5,0) -- (0.75,-0.433);
\node [above] at (0.3,-0.05) {\scriptsize $a_6$}; %
\node [below] at (0.75,0) {\scriptsize $c_3$}; %
\draw [thick] (0.5,0) -- (0.75,0.433) -- (0.5,0.866);
\node [above] at (0.75,0) {\scriptsize $c_5$}; %
\node [below] at (0.5,0.816) {\scriptsize $b_1$}; %
\draw [thick] (0.75,0.433) -- (1.25,0.433) -- (1.5,0.866) -- (2,0.866);
\node [below] at (1,0.483) {\scriptsize $c_3$}; %
\node [below] at (1.5,0.816) {\scriptsize $a_7$}; %
\draw [thick] (1.25,0.433) -- (1.5,0) -- (2,0);
\node [above] at (1.25,-0.05) {\scriptsize $c_1$}; %
\node [above] at (1.75,-0.05) {\scriptsize $b_5$}; %
\draw [thick] (1.5,0) -- (1.25,-0.433);
\node [below] at (1.25,0.05) {\scriptsize $c_4$}; %

\draw [thick] (0,-0.866) -- (0.5,-0.866);
\node [below] at (0.25,-0.816) {\footnotesize $b_3$}; %
\draw [thick] (0.5,-0.866) -- (0.75,-0.433);
\node [above] at (0.5,-0.816) {\scriptsize $b_0$}; %
\draw [thick] (0.75,-0.433) -- (1.25,-0.433);
\node [above] at (1,-0.483) {\scriptsize $c_2$}; %
\draw [thick] (1.25,-0.433) -- (1.5,-0.866);
\node [above] at (1.5,-0.816) {\scriptsize $a_4$}; %
\draw [thick] (1.5,-0.866) -- (2,-0.866);
\node [below] at (1.75,-0.816) {\footnotesize $a_3$}; %
\draw [thick] (2,-0.866) -- (2.25,-0.433) -- (2,0);
\node [above] at (2.25,-0.916) {\scriptsize $c_6$}; %
\draw [thick] (2,0.866) -- (2.25,0.433) -- (2,0);
\node [below] at (2.25,0.916) {\scriptsize $c_4$}; %

\fill [opacity=0.1] (1,0) circle [radius=0.7];
\draw [thick] (1,0) circle [radius=0.7];
\fill [opacity=0.1] (-0.5,-0.866) -- (0.2,-0.866) -- (-0.85,-0.259);
\fill [opacity=0.1] (0.2,-0.866) arc (0:120:0.7);
\draw [thick] (0.2,-0.866) arc (0:120:0.7);
\fill [opacity=0.1] (-0.5,0.866) -- (0.2,0.866) -- (-0.85,0.259);
\fill [opacity=0.1] (0.2,0.866) arc (0:-120:0.7);
\draw [thick] (0.2,0.866) arc (0:-120:0.7);
\fill [opacity=0.1] (1.8,-0.866) -- (2.5,-0.866) -- (2.15,-0.259);
\fill [opacity=0.1] (1.8,-0.866) arc (180:120:0.7);
\draw [thick] (1.8,-0.866) arc (180:120:0.7);
\fill [opacity=0.1] (1.8,0.866) -- (2.5,0.866) -- (2.15,0.259);
\fill [opacity=0.1] (1.8,0.866) arc (180:240:0.7);
\draw [thick] (1.8,0.866) arc (180:240:0.7);

\end{scope}

\end{tikzpicture}
\caption{The closed totally geodesic surface $S_{2,1}$ in $M_{2,1}$.}
\label{fig: tot geod from cusp}
\end{figure}
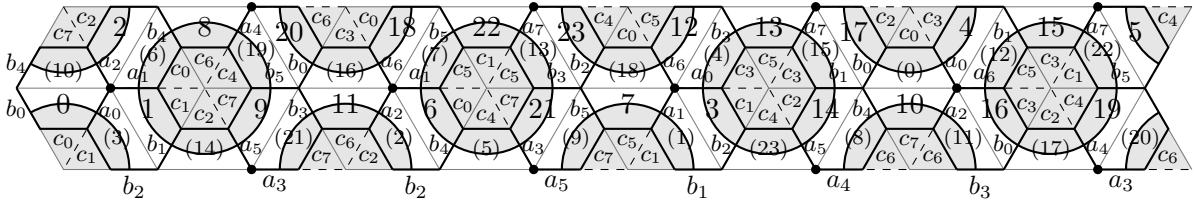

Each shaded circle in the figure represents a totally geodesic planar hexagon in $\Sigma_0\times[0,\infty)$, meeting $\Sigma_0\times\{0\}$ at right angles, that is the union of twelve copies of the triangle $T$ from \Cref{corandreev}. More precisely, each hexagon is six copies of $T\cup\overline{T}$, where $\overline{T}$ is the mirror image of $T$ in $\dpr^{333}_2$. Each hexagon has an angle of $2\pi/3$ at its vertices lying in edges labeled $a_i$ and $\pi/2$ at its vertices lying in edges labeled $b_j$. These polygons' edges are identified to each other under the pairings of hexagons of $\Sigma_0\times\{0\}$ that recover $M_{2,1}$, and an Euler characteristic or area computation shows that the polygons' union is a genus-two surface.
\end{example}

\begin{prop}\label{M_21plus} Let $\Sigma^+$ be the smallest subcomplex of the spine $\Sigma$ for $M_{2,1}$ from \Cref{M_21 spine} containing all $2$-cells that are unions of faces labeled $3_-$. For the surface $S_{2,1}\subset M_{2,1}$ of \Cref{ex: tot geod from cusp}, the compact component $M_{2,1}^+$ that results from cutting $M_{2,1}$ along $S_{2,1}$ deformation retracts to $\Sigma^+$ along geodesic arcs perpendicular to $S_{2,1}$. We thus have\begin{align*}
    \pi_1 M^+_{2,1} & \cong \langle\,C_1,C_3,C_4,C_5,C_6,C_7\,|\\
        & \ \qquad C_1^{-1}C_7^{-1}C_4C_6^{-1},\ C_1C_3^{-1}C_5C_3^{-1}C_4^{-1},\ C_3C_6^{-1}C_7C_6^{-1},\ C_5^{-1}C_1C_5^{-1}C_7C_4^{-1}\,\rangle
 \end{align*}
The fundamental group of $S_{2,1}$, presented as follows:
\[ \pi_1 S_{2,1} \cong \langle E_3, E_5, E_7, E_8\,|\, E_5E_3E_5^{-1}E_8E_3^{-1}E_7^{-1}E_8^{-1}E_7 \rangle \]
includes in $\pi_1 M^+_{2,1}$ via
\[ E_3 \mapsto C_3C_1,\ E_5\mapsto C_3C_5^{-1}C_1C_4C_6^{-1},\ E_7\mapsto C_4^{-1}C_5C_1^{-1}C_5C_3^{-1},\ E_8\mapsto C_3C_5^{-1}C_1C_6. \]
\end{prop}

\begin{remark}\label{basis change} Substituting $D = E_5^{-1}E_8$ for $E_5$ as a generator of $\pi_1 S_{2,1}$, we can replace its presentation given above by $\langle E_3,D,E_7,E_8\,|\,D^{-1}E_3DE_3^{-1}E_7^{-1}E_8^{-1}E_7E_8\rangle$; and $D\mapsto C_6C_4^{-1}C_6$.\end{remark}

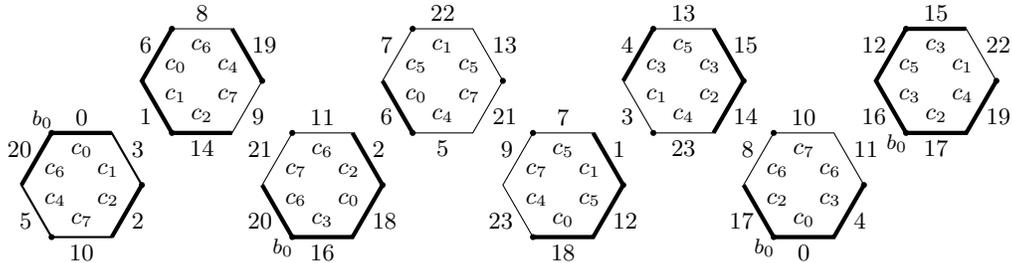
\begin{figure}[ht]
\begin{tikzpicture}

\begin{scope}[scale=0.8]
    \draw [thick] (0,0) -- (1,0) -- (1.5,0.866) -- (1,1.732) -- (0,1.732) -- (-0.5,0.866) -- cycle;
    \fill (0,0) circle [radius=0.05];
    \fill (0,1.732) circle [radius=0.05];
    \fill (1.5,0.866) circle [radius=0.05];

    \node [below right] at (1.15,0.583) {\scriptsize $2$};
    \node [above right] at (1.15,1.149) {\scriptsize $3$};
    \node [above] at (0.5,1.682) {\scriptsize $0$};
    \node [above left] at (-0.15,1.149) {\scriptsize $20$};
    \node [below left] at (-0.15,0.583) {\scriptsize $5$};
    \node [below] at (0.5,0.05) {\scriptsize $10$};

    \node [above left] at (0.2,1.632) {\tiny $b_0$};

    \node [above left] at (1.3,0.333) {\scriptsize $c_2$};
    \node [below left] at (1.3,1.399) {\scriptsize $c_1$};
    \node [below] at (0.5,1.732) {\scriptsize $c_0$};
    \node [below right] at (-0.3,1.399) {\scriptsize $c_6$};
    \node [above right] at (-0.3,0.333) {\scriptsize $c_4$};
    \node [above] at (0.5,0) {\scriptsize $c_7$};

    \draw [ultra thick] (1,0) -- (1.5,0.866);
    \draw [ultra thick] (1,1.732) -- (0,1.732) -- (-0.5,0.866);
\end{scope}
    
\begin{scope}[scale=0.8, xshift=2cm, yshift=1.732cm]
    \draw (0,0) -- (1,0) -- (1.5,0.866) -- (1,1.732) -- (0,1.732) -- (-0.5,0.866) -- cycle;
    \fill (0,0) circle [radius=0.05];
    \fill (0,1.732) circle [radius=0.05];
    \fill (1.5,0.866) circle [radius=0.05];

    \node [below right] at (1.15,0.583) {\scriptsize $9$};
    \node [above right] at (1.15,1.149) {\scriptsize $19$};
    \node [above] at (0.5,1.682) {\scriptsize $8$};
    \node [above left] at (-0.15,1.149) {\scriptsize $6$};
    \node [below left] at (-0.15,0.583) {\scriptsize $1$};
    \node [below] at (0.5,0.05) {\scriptsize $14$};

    \node [above left] at (1.3,0.333) {\scriptsize $c_7$};
    \node [below left] at (1.3,1.399) {\scriptsize $c_4$};
    \node [below] at (0.5,1.732) {\scriptsize $c_6$};
    \node [below right] at (-0.3,1.399) {\scriptsize $c_0$};
    \node [above right] at (-0.3,0.333) {\scriptsize $c_1$};
    \node [above] at (0.5,0) {\scriptsize $c_2$};

    \draw [ultra thick] (1,0) -- (0,0) -- (-0.5,0.866) -- (0,1.732);
    \draw [ultra thick] (1,1.732) -- (1.5,0.866);
\end{scope}

\begin{scope}[scale=0.8, xshift=4cm]
    \draw (0,0) -- (1,0) -- (1.5,0.866) -- (1,1.732) -- (0,1.732) -- (-0.5,0.866) -- cycle;
    \fill (0,0) circle [radius=0.05];
    \fill (0,1.732) circle [radius=0.05];
    \fill (1.5,0.866) circle [radius=0.05];

    \node [below right] at (1.15,0.583) {\scriptsize $18$};
    \node [above right] at (1.15,1.149) {\scriptsize $2$};
    \node [above] at (0.5,1.682) {\scriptsize $11$};
    \node [above left] at (-0.15,1.149) {\scriptsize $21$};
    \node [below left] at (-0.15,0.583) {\scriptsize $20$};
    \node [below] at (0.5,0.05) {\scriptsize $16$};

    \node [below left] at (0.2,0.15) {\tiny $b_0$};

    \node [above left] at (1.3,0.333) {\scriptsize $c_0$};
    \node [below left] at (1.3,1.399) {\scriptsize $c_2$};
    \node [below] at (0.5,1.732) {\scriptsize $c_6$};
    \node [below right] at (-0.3,1.399) {\scriptsize $c_7$};
    \node [above right] at (-0.3,0.333) {\scriptsize $c_6$};
    \node [above] at (0.5,0) {\scriptsize $c_3$};

    \draw [ultra thick] (-0.5,0.866) -- (0,0) -- (1,0) -- (1.5,0.866) -- (1,1.732);
\end{scope}

\begin{scope}[scale=0.8, xshift=6cm, yshift=1.732cm]
    \draw (0,0) -- (1,0) -- (1.5,0.866) -- (1,1.732) -- (0,1.732) -- (-0.5,0.866) -- cycle;
    \fill (0,0) circle [radius=0.05];
    \fill (0,1.732) circle [radius=0.05];
    \fill (1.5,0.866) circle [radius=0.05];

    \node [below right] at (1.15,0.583) {\scriptsize $21$};
    \node [above right] at (1.15,1.149) {\scriptsize $13$};
    \node [above] at (0.5,1.682) {\scriptsize $22$};
    \node [above left] at (-0.15,1.149) {\scriptsize $7$};
    \node [below left] at (-0.15,0.583) {\scriptsize $6$};
    \node [below] at (0.5,0.05) {\scriptsize $5$};

    \node [above left] at (1.3,0.333) {\scriptsize $c_7$};
    \node [below left] at (1.3,1.399) {\scriptsize $c_5$};
    \node [below] at (0.5,1.732) {\scriptsize $c_1$};
    \node [below right] at (-0.3,1.399) {\scriptsize $c_5$};
    \node [above right] at (-0.3,0.333) {\scriptsize $c_0$};
    \node [above] at (0.5,0) {\scriptsize $c_4$};

    \draw [ultra thick] (0,0) -- (-0.5,0.866);
\end{scope}

\begin{scope}[scale=0.8, xshift=8cm]
    \draw (0,0) -- (1,0) -- (1.5,0.866) -- (1,1.732) -- (0,1.732) -- (-0.5,0.866) -- cycle;
    \fill (0,0) circle [radius=0.05];
    \fill (0,1.732) circle [radius=0.05];
    \fill (1.5,0.866) circle [radius=0.05];

    \node [below right] at (1.15,0.583) {\scriptsize $12$};
    \node [above right] at (1.15,1.149) {\scriptsize $1$};
    \node [above] at (0.5,1.682) {\scriptsize $7$};
    \node [above left] at (-0.15,1.149) {\scriptsize $9$};
    \node [below left] at (-0.15,0.583) {\scriptsize $23$};
    \node [below] at (0.5,0.05) {\scriptsize $18$};

    \node [above left] at (1.3,0.333) {\scriptsize $c_5$};
    \node [below left] at (1.3,1.399) {\scriptsize $c_1$};
    \node [below] at (0.5,1.732) {\scriptsize $c_5$};
    \node [below right] at (-0.3,1.399) {\scriptsize $c_7$};
    \node [above right] at (-0.3,0.333) {\scriptsize $c_4$};
    \node [above] at (0.5,0) {\scriptsize $c_0$};

    \draw [ultra thick] (0,0) -- (1,0) -- (1.5,0.866) -- (1,1.732);
\end{scope}

\begin{scope}[scale=0.8, xshift=10cm, yshift=1.732cm]
    \draw (0,0) -- (1,0) -- (1.5,0.866) -- (1,1.732) -- (0,1.732) -- (-0.5,0.866) -- cycle;
    \fill (0,0) circle [radius=0.05];
    \fill (0,1.732) circle [radius=0.05];
    \fill (1.5,0.866) circle [radius=0.05];

    \node [below right] at (1.15,0.583) {\scriptsize $14$};
    \node [above right] at (1.15,1.149) {\scriptsize $15$};
    \node [above] at (0.5,1.682) {\scriptsize $13$};
    \node [above left] at (-0.15,1.149) {\scriptsize $4$};
    \node [below left] at (-0.15,0.583) {\scriptsize $3$};
    \node [below] at (0.5,0.05) {\scriptsize $23$};

    \node [above left] at (1.3,0.333) {\scriptsize $c_2$};
    \node [below left] at (1.3,1.399) {\scriptsize $c_3$};
    \node [below] at (0.5,1.732) {\scriptsize $c_5$};
    \node [below right] at (-0.3,1.399) {\scriptsize $c_3$};
    \node [above right] at (-0.3,0.333) {\scriptsize $c_1$};
    \node [above] at (0.5,0) {\scriptsize $c_4$};

    \draw [ultra thick] (1,0) -- (1.5,0.866) -- (1,1.732);
    \draw [ultra thick] (-0.5,0.866) -- (0,1.732);
\end{scope}

\begin{scope}[scale=0.8, xshift=12cm]
    \draw (0,0) -- (1,0) -- (1.5,0.866) -- (1,1.732) -- (0,1.732) -- (-0.5,0.866) -- cycle;
    \fill (0,0) circle [radius=0.05];
    \fill (0,1.732) circle [radius=0.05];
    \fill (1.5,0.866) circle [radius=0.05];

    \node [below right] at (1.15,0.583) {\scriptsize $4$};
    \node [above right] at (1.15,1.149) {\scriptsize $11$};
    \node [above] at (0.5,1.682) {\scriptsize $10$};
    \node [above left] at (-0.15,1.149) {\scriptsize $8$};
    \node [below left] at (-0.15,0.583) {\scriptsize $17$};
    \node [below] at (0.5,0.05) {\scriptsize $0$};

    \node [below left] at (0.2,0.15) {\tiny $b_0$};

    \node [above left] at (1.3,0.333) {\scriptsize $c_3$};
    \node [below left] at (1.3,1.399) {\scriptsize $c_6$};
    \node [below] at (0.5,1.732) {\scriptsize $c_7$};
    \node [below right] at (-0.3,1.399) {\scriptsize $c_6$};
    \node [above right] at (-0.3,0.333) {\scriptsize $c_2$};
    \node [above] at (0.5,0) {\scriptsize $c_0$};

    \draw [ultra thick] (-0.5,0.866) -- (0,0) -- (1,0) -- (1.5,0.866);
\end{scope}

\begin{scope}[scale=0.8, xshift=14.2cm, yshift=1.732cm]
    \draw (0,0) -- (1,0) -- (1.5,0.866) -- (1,1.732) -- (0,1.732) -- (-0.5,0.866) -- cycle;
    \fill (0,0) circle [radius=0.05];
    \fill (0,1.732) circle [radius=0.05];
    \fill (1.5,0.866) circle [radius=0.05];

    \node [below right] at (1.15,0.583) {\scriptsize $19$};
    \node [above right] at (1.15,1.149) {\scriptsize $22$};
    \node [above] at (0.5,1.682) {\scriptsize $15$};
    \node [above left] at (-0.15,1.149) {\scriptsize $12$};
    \node [below left] at (-0.15,0.583) {\scriptsize $16$};
    \node [below] at (0.5,0.05) {\scriptsize $17$};

    \node [below left] at (0.2,0.15) {\tiny $b_0$};

    \node [above left] at (1.3,0.333) {\scriptsize $c_4$};
    \node [below left] at (1.3,1.399) {\scriptsize $c_1$};
    \node [below] at (0.5,1.732) {\scriptsize $c_3$};
    \node [below right] at (-0.3,1.399) {\scriptsize $c_5$};
    \node [above right] at (-0.3,0.333) {\scriptsize $c_3$};
    \node [above] at (0.5,0) {\scriptsize $c_2$};

    \draw [ultra thick] (0,1.732) -- (-0.5,0.866) -- (0,0) -- (1,0) -- (1.5,0.866);
    \draw [ultra thick] (0,1.732) -- (1,1.732);
\end{scope}

\end{tikzpicture}
\caption{The hexagons decomposing $S_{2,1}$.}
\label{fig: S_21}
\end{figure}

\begin{proof}[Proof of \Cref{M_21plus}] Informally, we can view the deformation retract $M_{2,1}^+\to \Sigma^+$ as shrinking each shaded disk in \Cref{fig: tot geod from cusp} down to the hexagon pictured inside of it. More formally, the map is built cell-by-cell using the polyhedral decomposition of $M_{2,1}$ from \Cref{cell decomp}. Any cell $\dpr_k$ of that decomposition intersects $M_{2,1}^+$ in the union of a copy of the truncated prism $P^+$ from the proof of \Cref{corandreev} and its mirror image across a quadrilateral face; $S_{2,1}$ intersects $\dpr_k$ in the union of the triangle $T$ of \Cref{corandreev} and its mirror image across the corresponding edge. Because $T$ intersects each quadrilateral face of $P^+$ at right angles, for any $x\in T$, the perpendicular geodesic ray to $T$ based at $x$ that points into $P^+$ exits $P^+$ in the opposite triangular face: the face of $P^+\subset P\subset \dpr$ labeled $3_-$ in \Cref{numberings}. A standard construction thus produces a deformation retract from $P^+$ to the face labeled $3_-$ along geodesic arcs perpendicular to $T$, and applying this to all copies of $P^+$ and its mirror image in the $\dpr_k$ determines the desired map on $M_{2,1}^+$.

The shaded disks of \Cref{fig: tot geod from cusp} are re-pictured as hexagons in \Cref{fig: S_21}, with their edges labeled on the outside by numbers of the quadrilateral faces of the $\dpr_k$ that they lie in. Precisely: for each $k\in\{0,\hdots,23\}$ the edge labeled ``$k$'' is the intersection of $S_{2,1}$ with the face of $\dpr_k$ labeled $1_-$. Each such edge has one endpoint in an edge of the spine $\Sigma$ for $M_{2,1}$ that is labeled $a_i$, and the other in one labeled $b_j$, for some $i$ and $j$ using the edge labeling convention from the proof of \Cref{M_21}. We orient each hexagon edge pointing away from its $b_j$-endpoint (these are circled in bold in the Figure).

The bold edges of Figure \ref{fig: S_21} belong to a tree in the one-skeleton of the decomposition of $S_{2,1}$ determined by the hexagons, with combinatorics pictured schematically below.

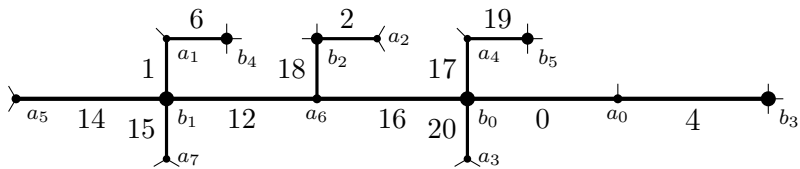
\begin{figure}[ht]
\begin{tikzpicture}
    \fill (0,0) circle [radius=0.06];
    \fill (2,0) circle [radius=0.1];
    \fill (4,0) circle [radius=0.06];
    \fill (6,0) circle [radius=0.1];
    \fill (8,0) circle [radius=0.06];
    \fill (10,0) circle [radius=0.1];

    \draw [ultra thick] (0,0) -- (10,0);

    \fill (2,-0.8) circle [radius=0.05];
    \fill (6,-0.8) circle [radius=0.05];
    \fill (2,0.8) circle [radius=0.05];
    \fill (4,0.8) circle [radius=0.08];
    \fill (6,0.8) circle [radius=0.05];
    \fill (2.8,0.8) circle [radius=0.08];
    \fill (4.8,0.8) circle [radius=0.05];
    \fill (6.8,0.8) circle [radius=0.08];

    \draw [very thick] (2,-0.8) -- (2,0.8) -- (2.8,0.8);
    \draw [very thick] (4,0) -- (4,0.8) -- (4.8,0.8);
    \draw [very thick] (6,-0.8) -- (6,0.8) -- (6.8,0.8);

    \draw (-0.1,-0.17) -- (0,0) -- (-0.1,0.17);
    \draw (1.83,-0.9) -- (2,-0.8) -- (2.17,-0.9);
    \draw (1.86,0.94) -- (2,0.8);
    \draw (2.8,0.6) -- (2.8,1);
    \draw (2.8,0.8) -- (3,0.8);
    \draw (3.8,0.8) -- (4,0.8) -- (4,1);
    \draw (4.9,0.97) -- (4.8,0.8) -- (4.9,0.63);
    \draw (5.83,-0.9) -- (6,-0.8) -- (6.17,-0.9);
    \draw (5.86,0.94) -- (6,0.8);
    \draw (6.8,0.6) -- (6.8,1);
    \draw (6.8,0.8) -- (7,0.8);
    \draw (8,0) -- (8,0.2);
    \draw (10,-0.2) -- (10,0.2);
    \draw (10,0) -- (10.2,0);

    \node [below right] at (6,0) {\tiny $b_0$};
    \node [below] at (7,0) {\small $0$};
    \node [below] at (8,0) {\tiny $a_0$};
    \node [left] at (6,0.4) {\small $17$};
    \node [below right] at (6,0.8) {\tiny $a_4$};
    \node [above] at (6.4,0.8) {\small $19$};
    \node [below right] at (6.8,0.8) {\tiny $b_5$};
    \node [below] at (5,0) {\small $16$};
    \node [below] at (4,0) {\tiny $a_6$};
    \node [left] at (6,-0.4) {\small $20$};
    \node [right] at (6,-0.8) {\tiny $a_3$};
    \node [below] at (3,0) {\small $12$};
    \node [below right] at (2,0) {\tiny $b_1$};
    \node [left] at (4,0.4) {\small $18$};
    \node [below right] at (4,0.8) {\tiny $b_2$};
    \node [above] at (4.4,0.8) {\small $2$};
    \node [right] at (4.8,0.8) {\tiny $a_2$};
    \node [below] at (9,0) {$4$};
    \node [below right] at (10,0) {\tiny $b_3$};
    \node [left] at (2,0.4) {\small $1$};
    \node [below right] at (2,0.8) {\tiny $a_1$};
    \node [above] at (2.4,0.8) {\small $6$};
    \node [below right] at (2.8,0.8) {\tiny $b_4$};
    \node [below] at (1,0) {\small $14$};
    \node [below right] at (0,0) {\tiny $a_5$};
    \node [left] at (2,-0.4) {\small $15$};
    \node [right] at (2,-0.8) {\tiny $a_7$};

\end{tikzpicture}
\end{figure}

Locating a basepoint for $\pi_1 S_{2,1}$ at its intersection with the edge $b_0$, a standard presentation has a generator for each edge with label ``$k$'' that lies outside the tree above: the loop $E_k$ which runs within the tree from $b_0$ to the edge's initial point, traverses the edge, and returns to $b_0$ within the tree. Each hexagon of \Cref{fig: S_21} gives rise to a relation:\begin{align*}
    & E_9 = E_8^{-1} && E_5E_{21}^{-1}E_{13}E_{22}^{-1}E_7 = 1 && E_{13} = E_3^{-1}E_{23} && E_{22} = 1 \\
    & E_{10} = E_5E_3^{-1} && E_{21}=E_{11} && E_{23}=E_7^{-1}E_9 && E_{11} = E_{8}^{-1}E_{10}
\end{align*}
Above, all relations but the second of the top row have been solved for a different generator. After removing $E_{22}$, which is trivial, iterated substitutions in the others yield the following descriptions of the other generators in terms of $E_3$, $E_5$, $E_7$, and $E_8$:
\[ E_9 = E_8^{-1},\ E_{10} = E_5E_3^{-1},\ E_{21}=E_{11}=E_{8}^{-1}E_5E_3^{-1},\ E_{13} = E_3^{-1}E_7^{-1}E_8^{-1},\ E_{23}=E_7^{-1}E_8^{-1}. \]
Plugging into the remaining relation yields the presentation from this result's statement.

The inclusion-induced map from $\pi_1 S_{2,1}$ to $\pi_1 M_{2,1}^+$ can now be read off from the labels inside the hexagons. Each edge's inner label records the label of its image edge of $\Sigma^+$ under the deformation retract, according to the numbering scheme from the proof of \Cref{M_21}. This yields the images described in the statement, bearing in mind that each generator is a concatenation of paths in the tree above with an edge transversal. For instance, the edge labeled $3$ has its initial point on $b_2$ and terminal point on $a_0$, so $E_3 = e_{16}\bar{e}_{18}e_3\bar{e}_0$ as a concatenation of edge paths. (Here $e_i$ is the path transversing the edge labeled $i$ in the orientation direction.)
\end{proof}

\begin{corollary}\label{M_21plus H1} For $M_{2,1}^+$ as in \Cref{M_21plus}, $H_1(M_{2,1}^+)\cong\mathbb{Z}^2$ is generated by  $C_4$ and $\gamma \doteq 5C_6-C_7 = C_5$. (Here ``$C_i$'' refers to the homology class of the corresponding generator from the presentation for $\pi_1 M_{2,1}^+$ given in \Cref{M_21plus}.) The inclusion-induced map $H_1 (S_{2,1})\to H_1 (M_{2,1}^+)$ is surjective, with kernel generated by $E_7-E_3+2E_8-D$ and $13E_3-3D-16E_8$.\end{corollary}

\begin{proof}
    Turning the abelianized relations of $\pi_1 M_{2,1}$ from \Cref{M_21plus} into equations and row-reducing as in \Cref{M_21 H1}, we obtain the following:
    \[ C_1 = C_4 - C_6 - C_7,\ C_3 = 2C_6 - C_7,\ C_5 = 5C_6 - C_7,\ -11C_6 + 2C_7 = 0\]
    It follows that $H_1 (M_{2,1}^+)\cong\mathbb{Z}^2$ is freely generated by $C_4$ and $\gamma = 5C_6-C_7$; the latter because $\mathrm{Det}\left(\begin{smallmatrix} -11 & 2 \\ 5 & -1 \end{smallmatrix}\right) = 1$. From this we obtain in addition that $C_6 = -2\gamma$ and $C_7 = -11\gamma$ in $H_1(M_{2,1}^+)$. We use this to compute the inclusion-induced images of the generators of $\pi_1 S_{2,1}$ from \Cref{basis change}:\begin{align*}
    & E_3 \mapsto C_1 + C_3 = C_4 + 20\gamma &&  E_7\mapsto -C_1-C_3-C_4+2C_5 =  -2C_4-18\gamma\\
    & D\mapsto 2C_6-C_4 = -C_4 - 4\gamma && E_8\mapsto C_1 + C_3 - C_5 + C_6 = C_4 + 17\gamma \end{align*}
    A row-reduction process produces the generators recorded above for the kernel of the inclusion-induced map. One can check directly that each maps to $0$, and that 
    \[ \{D,D-4E_3+5E_8,-3D+13E_3-16E_8,E_7-D-E_3+2E_8\} \]
    is a generating set for $H_1(S_{2,1})$ such that $D-4E_3 + 5E_8\mapsto \gamma$.
\end{proof}

\begin{prop}\label{M_21 minus}
    For the non-compact component $M_{2,1}^-$ that results from cutting $M_{2,1}$ along $S_{2,1}$, where the surface $S_{2,1}\subset M_{2,1}$ is as in \Cref{ex: tot geod from cusp}, the result $M_{2,1}^-(\mu_0)$ of filling $M_{2,1}^-$ along the slope $\mu_0$ from \Cref{filling slope} is a handlebody, with a full set of compressing disks whose boundaries are represented in $\pi_1 S_{2,1}$ by $D$ and $E_8$.
\end{prop}

\begin{remark}\label{good stuff} \Cref{M_21 minus} is stronger than what we can conclude about the $M_{i,j}^-$ directly from \Cref{lem:lifts}. Indeed, there exist surfaces in $\mathbb{S}^3$ that are incompressible to one side and do not bound a handlebody to the other: see \cite[Ex.~2.4]{Ozawa}, cf.~the Remark on \cite[p.~133]{Homma}.
\end{remark}

\begin{proof} To prove the result, we will exhibit a pair of disjoint punctured disks properly embedded in $M_{2,1}^-$ with the following properties:\begin{itemize}
    \item The union of their boundary circles in $S_{2,1}$ is non-separating.
    \item Their components of intersection with a cusp cross-section for $M_{2,1}$ are each parallel to $\mu_0$ from \Cref{filling slope}.
\end{itemize}
It follows from these properties that in the Dehn-filled manifold $M_{2,1}^-(\mu_0)$, the closures of the two punctured disks form a full set of compressing disks for $S_{2,1}$. We will then compute their representations in $\pi_1 S_{2,1}$.

\begin{figure}[ht]
\begin{tikzpicture}

\begin{scope}[scale=1.25]
\draw [gray] (-1,0) -- (0,0) -- (0.5,-0.866) -- (1,0) -- (0.5,0.866) -- (0,0);
\draw [gray] (1,0) -- (2,0) -- (1.5,-0.866) -- (-0.5,-0.866) -- (-1,0) -- (-0.5,0.866) -- (1.5,0.866) -- (2,0);

\fill (0,0) circle [radius=0.05];
\fill (1.5,0.866) circle [radius=0.05];
\fill (1.5,-0.866) circle [radius=0.05];

\draw [dashed] (0,0) -- (1,0);
\draw [dashed] (-0.5,0.866) -- (0,0) -- (-0.5,-0.866);
\draw [dashed] (2.5,0.866) -- (1.5,0.866) -- (1,0) -- (1.5,-0.866) -- (2.5,-0.866);

\node [below] at (-0.5,0.05) {\small $0$};
\node [above] at (-0.5,-0.05) {\tiny $(10)$};
\node [below] at (0.1,0.866) {\small $2$};
\node at (0.45,0.35) {\tiny $(6)$};
\node [below] at (0.4,0) {\small $1$};
\node [below] at (1.6,0) {\small $9$};
\node [below] at (1,0.866) {\small $8$};
\node at (1.55,0.4) {\tiny $(19)$};
\node [below] at (1.9,0.816) {\small $20$};
\node at (0.07,-0.533) {\tiny $(3)$};
\node [above] at (1,-0.886) {\tiny $(14)$};
\node at (1.95,-0.533) {\tiny $(21)$};

\draw [thick] (0,0) -- (-0.25,0.433) -- (0,0.866) -- (0.5,0.866);
\node [above] at (0,0.08) {\scriptsize $a_2$};
\node [below] at (-0.25,0.916) {\scriptsize $c_2$};
\draw [thick] (-0.25,0.433) -- (-0.75,0.433) -- (-1,0);
\node [above] at (-0.5,0.383) {\scriptsize $c_7$};
\node [above] at (-1,0.05) {\scriptsize $b_4$};
\draw [thick] (0,0) -- (-0.25,-0.433) -- (0,-0.866);
\node [below] at (0,-0.08) {\scriptsize $a_0$};
\node [above] at (-0.25,-0.916) {\scriptsize $c_1$};
\draw [thick] (-0.25,-0.433) -- (-0.75,-0.433) -- (-1,0);
\node [below] at (-0.5,-0.383) {\scriptsize $c_0$};
\node [below] at (-1,0) {\scriptsize $b_0$};

\draw [thick] (0,0) -- (0.5,0) -- (0.75,-0.433);
\node [above] at (0.3,-0.05) {\scriptsize $a_1$};
\node [below] at (0.75,0) {\scriptsize $c_1$};
\draw [thick] (0.5,0) -- (0.75,0.433) -- (0.5,0.866);
\node [above] at (0.75,0) {\scriptsize $c_0$};
\node [below] at (0.5,0.816) {\scriptsize $b_4$};
\draw [thick] (0.75,0.433) -- (1.25,0.433) -- (1.5,0.866) -- (2,0.866);
\node [below] at (1,0.483) {\scriptsize $c_6$};
\node [below] at (1.5,0.816) {\scriptsize $a_4$};
\draw [thick] (1.25,0.433) -- (1.5,0) -- (2,0);
\node [above] at (1.25,-0.05) {\scriptsize $c_4$};
\node [above] at (1.75,-0.05) {\scriptsize $b_5$};
\draw [thick] (1.5,0) -- (1.25,-0.433);
\node [below] at (1.25,0.05) {\scriptsize $c_7$};

\draw [thick] (0,-0.866) -- (0.5,-0.866);
\node [below] at (0.25,-0.816) {\footnotesize $b_2$};
\draw [thick] (0.5,-0.866) -- (0.75,-0.433);
\node [above] at (0.5,-0.816) {\scriptsize $b_1$};
\draw [thick] (0.75,-0.433) -- (1.25,-0.433);
\node [above] at (1,-0.483) {\scriptsize $c_2$};
\draw [thick] (1.25,-0.433) -- (1.5,-0.866);
\node [above] at (1.5,-0.816) {\scriptsize $a_5$};
\draw [thick] (1.5,-0.866) -- (2,-0.866);
\node [below] at (1.75,-0.816) {\footnotesize $a_3$};
\draw [thick] (2,-0.866) -- (2.25,-0.433);
\node [above] at (2.25,-0.916) {\scriptsize $c_7$};
\draw [thick] (2,0.866) -- (2.25,0.433);
\node [below] at (2.25,0.916) {\scriptsize $c_6$};

\fill [opacity=0.1] (1,0) circle [radius=0.7];
\draw [thick] (1,0) circle [radius=0.7];
\fill [opacity=0.1] (-0.5,-0.866) -- (0.2,-0.866) -- (-0.85,-0.259);
\fill [opacity=0.1] (0.2,-0.866) arc (0:120:0.7);
\draw [thick] (0.2,-0.866) arc (0:120:0.7);
\fill [opacity=0.1] (-0.5,0.866) -- (0.2,0.866) -- (-0.85,0.259);
\fill [opacity=0.1] (0.2,0.866) arc (0:-120:0.7);
\draw [thick] (0.2,0.866) arc (0:-120:0.7);
\fill [opacity=0.1] (1.8,-0.866) -- (2.5,-0.866) -- (2.15,-0.259);
\fill [opacity=0.1] (1.8,-0.866) arc (180:120:0.7);
\draw [thick] (1.8,-0.866) arc (180:120:0.7);
\fill [opacity=0.1] (1.8,0.866) -- (2.5,0.866) -- (2.15,0.259);
\fill [opacity=0.1] (1.8,0.866) arc (180:240:0.7);
\draw [thick] (1.8,0.866) arc (180:240:0.7);

\draw [very thick, dash pattern=on 1pt off 1pt, color=red] (-0.5,-0.866) -- (-0.5,0.866);
\draw [very thick, color=red] (-0.5,-0.866) -- (-0.5,-0.166);
\draw [very thick, color=red] (-0.5,0.866) -- (-0.5,0.166);

\draw [very thick, dash pattern=on 1pt off 1pt, color=blue] (1,-0.866) -- (1,0.866);
\draw [very thick, color=blue] (1,-0.7) -- (1,0.7);

\end{scope}

\begin{scope}[scale=1.25, xshift=3cm]
\draw [gray] (-1,0) -- (0,0) -- (0.5,-0.866) -- (1,0) -- (0.5,0.866) -- (0,0);
\draw [gray] (1,0) -- (2,0) -- (1.5,-0.866) -- (-0.5,-0.866) -- (-1,0) -- (-0.5,0.866) -- (1.5,0.866) -- (2,0);

\fill (0,0) circle [radius=0.05];
\fill (1.5,0.866) circle [radius=0.05];
\fill (1.5,-0.866) circle [radius=0.05];

\draw [dashed] (0,0) -- (1,0);
\draw [dashed] (-0.5,0.866) -- (0,0) -- (-0.5,-0.866);
\draw [dashed] (2.5,0.866) -- (1.5,0.866) -- (1,0) -- (1.5,-0.866) -- (2.5,-0.866);

\node [below] at (-0.5,0.05) {\small $11$};
\node [above] at (-0.5,-0.05) {\tiny $(16)$};
\node [below] at (0.1,0.866) {\small $18$};
\node at (0.45,0.35) {\tiny $(7)$};
\node [below] at (0.4,0) {\small $6$};
\node [below] at (1.6,0) {\small $21$};
\node [below] at (1,0.866) {\small $22$};
\node at (1.55,0.4) {\tiny $(13)$};
\node [below] at (1.9,0.816) {\small $23$};
\node at (0.07,-0.533) {\tiny $(2)$};
\node [above] at (1,-0.886) {\tiny $(5)$};
\node at (1.95,-0.533) {\tiny $(9)$};

\draw [thick] (0,0) -- (-0.25,0.433) -- (0,0.866) -- (0.5,0.866);
\node [above] at (0,0.08) {\scriptsize $a_6$}; %
\node [below] at (-0.25,0.916) {\scriptsize $c_0$}; %
\draw [thick] (-0.25,0.433) -- (-0.75,0.433) -- (-1,0);
\node [above] at (-0.5,0.383) {\scriptsize $c_3$}; %
\node [above] at (-1,0.05) {\scriptsize $b_0$}; %
\draw [thick] (0,0) -- (-0.25,-0.433) -- (0,-0.866);
\node [below] at (0,-0.08) {\scriptsize $a_2$}; %
\node [above] at (-0.25,-0.916) {\scriptsize $c_2$}; %
\draw [thick] (-0.25,-0.433) -- (-0.75,-0.433) -- (-1,0);
\node [below] at (-0.5,-0.383) {\scriptsize $c_6$}; %
\node [below] at (-1,0) {\scriptsize $b_3$}; %

\draw [thick] (0,0) -- (0.5,0) -- (0.75,-0.433);
\node [above] at (0.3,-0.05) {\scriptsize $a_1$}; %
\node [below] at (0.75,0) {\scriptsize $c_0$}; %
\draw [thick] (0.5,0) -- (0.75,0.433) -- (0.5,0.866);
\node [above] at (0.75,0) {\scriptsize $c_5$}; %
\node [below] at (0.5,0.816) {\scriptsize $b_5$}; %
\draw [thick] (0.75,0.433) -- (1.25,0.433) -- (1.5,0.866) -- (2,0.866);
\node [below] at (1,0.483) {\scriptsize $c_1$}; %
\node [below] at (1.5,0.816) {\scriptsize $a_7$}; %
\draw [thick] (1.25,0.433) -- (1.5,0) -- (2,0);
\node [above] at (1.25,-0.05) {\scriptsize $c_5$}; %
\node [above] at (1.75,-0.05) {\scriptsize $b_3$}; %
\draw [thick] (1.5,0) -- (1.25,-0.433);
\node [below] at (1.25,0.05) {\scriptsize $c_7$}; %

\draw [thick] (0,-0.866) -- (0.5,-0.866);
\node [below] at (0.25,-0.816) {\footnotesize $b_2$}; %
\draw [thick] (0.5,-0.866) -- (0.75,-0.433);
\node [above] at (0.5,-0.816) {\scriptsize $b_4$}; %
\draw [thick] (0.75,-0.433) -- (1.25,-0.433);
\node [above] at (1,-0.483) {\scriptsize $c_4$}; %
\draw [thick] (1.25,-0.433) -- (1.5,-0.866);
\node [above] at (1.5,-0.816) {\scriptsize $a_3$}; %
\draw [thick] (1.5,-0.866) -- (2,-0.866);
\node [below] at (1.75,-0.816) {\footnotesize $a_5$}; %
\draw [thick] (2,-0.866) -- (2.25,-0.433);
\node [above] at (2.25,-0.916) {\scriptsize $c_7$}; %
\draw [thick] (2,0.866) -- (2.25,0.433);
\node [below] at (2.25,0.916) {\scriptsize $c_4$}; %

\fill [opacity=0.1] (1,0) circle [radius=0.7];
\draw [thick] (1,0) circle [radius=0.7];
\fill [opacity=0.1] (-0.5,-0.866) -- (0.2,-0.866) -- (-0.85,-0.259);
\fill [opacity=0.1] (0.2,-0.866) arc (0:120:0.7);
\draw [thick] (0.2,-0.866) arc (0:120:0.7);
\fill [opacity=0.1] (-0.5,0.866) -- (0.2,0.866) -- (-0.85,0.259);
\fill [opacity=0.1] (0.2,0.866) arc (0:-120:0.7);
\draw [thick] (0.2,0.866) arc (0:-120:0.7);
\fill [opacity=0.1] (1.8,-0.866) -- (2.5,-0.866) -- (2.15,-0.259);
\fill [opacity=0.1] (1.8,-0.866) arc (180:120:0.7);
\draw [thick] (1.8,-0.866) arc (180:120:0.7);
\fill [opacity=0.1] (1.8,0.866) -- (2.5,0.866) -- (2.15,0.259);
\fill [opacity=0.1] (1.8,0.866) arc (180:240:0.7);
\draw [thick] (1.8,0.866) arc (180:240:0.7);

\end{scope}

\begin{scope}[scale=1.25, xshift=6cm]
\draw [gray] (-1,0) -- (0,0) -- (0.5,-0.866) -- (1,0) -- (0.5,0.866) -- (0,0);
\draw [gray] (1,0) -- (2,0) -- (1.5,-0.866) -- (-0.5,-0.866) -- (-1,0) -- (-0.5,0.866) -- (1.5,0.866) -- (2,0);

\fill (0,0) circle [radius=0.05];
\fill (1.5,0.866) circle [radius=0.05];
\fill (1.5,-0.866) circle [radius=0.05];

\draw [dashed] (0,0) -- (1,0);
\draw [dashed] (-0.5,0.866) -- (0,0) -- (-0.5,-0.866);
\draw [dashed] (2.5,0.866) -- (1.5,0.866) -- (1,0) -- (1.5,-0.866) -- (2.5,-0.866);

\node [below] at (-0.5,0.05) {\small $7$}; %
\node [above] at (-0.5,-0.05) {\tiny $(18)$}; %
\node [below] at (0.1,0.866) {\small $12$}; %
\node at (0.45,0.35) {\tiny $(4)$}; %
\node [below] at (0.4,0) {\small $3$}; %
\node [below] at (1.6,0) {\small $14$}; %
\node [below] at (1,0.866) {\small $13$}; %
\node at (1.55,0.4) {\tiny $(15)$}; %
\node [below] at (1.9,0.816) {\small $17$}; %
\node at (0.07,-0.533) {\tiny $(1)$}; %
\node [above] at (1,-0.886) {\tiny $(23)$}; %
\node at (1.95,-0.533) {\tiny $(8)$}; %

\draw [thick] (0,0) -- (-0.25,0.433) -- (0,0.866) -- (0.5,0.866);
\node [above] at (0,0.08) {\scriptsize $a_6$}; %
\node [below] at (-0.25,0.916) {\scriptsize $c_5$}; %
\draw [thick] (-0.25,0.433) -- (-0.75,0.433) -- (-1,0);
\node [above] at (-0.5,0.383) {\scriptsize $c_0$}; %
\node [above] at (-1,0.05) {\scriptsize $b_2$}; %
\draw [thick] (0,0) -- (-0.25,-0.433) -- (0,-0.866);
\node [below] at (0,-0.08) {\scriptsize $a_1$}; %
\node [above] at (-0.25,-0.916) {\scriptsize $c_1$}; %
\draw [thick] (-0.25,-0.433) -- (-0.75,-0.433) -- (-1,0);
\node [below] at (-0.5,-0.383) {\scriptsize $c_5$}; %
\node [below] at (-1,0) {\scriptsize $b_5$}; %

\draw [thick] (0,0) -- (0.5,0) -- (0.75,-0.433);
\node [above] at (0.3,-0.05) {\scriptsize $a_0$}; %
\node [below] at (0.75,0) {\scriptsize $c_1$}; %
\draw [thick] (0.5,0) -- (0.75,0.433) -- (0.5,0.866);
\node [above] at (0.75,0) {\scriptsize $c_3$}; %
\node [below] at (0.5,0.816) {\scriptsize $b_3$}; %
\draw [thick] (0.75,0.433) -- (1.25,0.433) -- (1.5,0.866) -- (2,0.866);
\node [below] at (1,0.483) {\scriptsize $c_5$}; %
\node [below] at (1.5,0.816) {\scriptsize $a_7$}; %
\draw [thick] (1.25,0.433) -- (1.5,0) -- (2,0);
\node [above] at (1.25,-0.05) {\scriptsize $c_3$}; %
\node [above] at (1.75,-0.05) {\scriptsize $b_1$}; %
\draw [thick] (1.5,0) -- (1.25,-0.433);
\node [below] at (1.25,0.05) {\scriptsize $c_2$}; %

\draw [thick] (0,-0.866) -- (0.5,-0.866);
\node [below] at (0.25,-0.816) {\footnotesize $b_1$}; %
\draw [thick] (0.5,-0.866) -- (0.75,-0.433);
\node [above] at (0.5,-0.816) {\scriptsize $b_2$}; %
\draw [thick] (0.75,-0.433) -- (1.25,-0.433);
\node [above] at (1,-0.483) {\scriptsize $c_4$}; %
\draw [thick] (1.25,-0.433) -- (1.5,-0.866);
\node [above] at (1.5,-0.816) {\scriptsize $a_5$}; %
\draw [thick] (1.5,-0.866) -- (2,-0.866);
\node [below] at (1.75,-0.816) {\footnotesize $a_4$}; %
\draw [thick] (2,-0.866) -- (2.25,-0.433);
\node [above] at (2.25,-0.916) {\scriptsize $c_6$}; %
\draw [thick] (2,0.866) -- (2.25,0.433);
\node [below] at (2.25,0.916) {\scriptsize $c_2$}; %

\fill [opacity=0.1] (1,0) circle [radius=0.7];
\draw [thick] (1,0) circle [radius=0.7];
\fill [opacity=0.1] (-0.5,-0.866) -- (0.2,-0.866) -- (-0.85,-0.259);
\fill [opacity=0.1] (0.2,-0.866) arc (0:120:0.7);
\draw [thick] (0.2,-0.866) arc (0:120:0.7);
\fill [opacity=0.1] (-0.5,0.866) -- (0.2,0.866) -- (-0.85,0.259);
\fill [opacity=0.1] (0.2,0.866) arc (0:-120:0.7);
\draw [thick] (0.2,0.866) arc (0:-120:0.7);
\fill [opacity=0.1] (1.8,-0.866) -- (2.5,-0.866) -- (2.15,-0.259);
\fill [opacity=0.1] (1.8,-0.866) arc (180:120:0.7);
\draw [thick] (1.8,-0.866) arc (180:120:0.7);
\fill [opacity=0.1] (1.8,0.866) -- (2.5,0.866) -- (2.15,0.259);
\fill [opacity=0.1] (1.8,0.866) arc (180:240:0.7);
\draw [thick] (1.8,0.866) arc (180:240:0.7);

\draw [very thick, color=blue] (2,-0.866) -- (1.894,-0.516);
\draw [very thick, color=blue, dash pattern=on 1pt off 1pt] (1.894,-0.516) -- (1.606,-0.35);
\draw [very thick, color=blue] (1.606,-0.35) -- (1.5,0) -- (1.606,0.35);
\draw [very thick, color=blue, dash pattern=on 1pt off 1pt] (1.894,0.516) -- (1.606,0.35);
\draw [very thick, color=blue] (2,0.866) -- (1.894,0.516);

\end{scope}

\begin{scope}[scale=1.25, xshift=9cm]
\draw [gray] (-1,0) -- (0,0) -- (0.5,-0.866) -- (1,0) -- (0.5,0.866) -- (0,0);
\draw [gray] (1,0) -- (2,0) -- (1.5,-0.866) -- (-0.5,-0.866) -- (-1,0) -- (-0.5,0.866) -- (1.5,0.866) -- (2,0);
\draw [gray] (2.5,-0.866) -- (2,0) -- (2.5,0.866);

\fill (0,0) circle [radius=0.05];
\fill (1.5,0.866) circle [radius=0.05];
\fill (1.5,-0.866) circle [radius=0.05];

\draw [dashed] (0,0) -- (1,0);
\draw [dashed] (-0.5,0.866) -- (0,0) -- (-0.5,-0.866);
\draw [dashed] (2.5,0.866) -- (1.5,0.866) -- (1,0) -- (1.5,-0.866) -- (2.5,-0.866);

\node [below] at (-0.5,0.05) {\small $10$}; %
\node [above] at (-0.5,-0.05) {\tiny $(0)$}; %
\node [below] at (0.1,0.866) {\small $4$}; %
\node at (0.45,0.35) {\tiny $(12)$}; %
\node [below] at (0.4,0) {\small $16$}; %
\node [below] at (1.6,0) {\small $19$}; %
\node [below] at (1,0.866) {\small $15$}; %
\node at (1.55,0.4) {\tiny $(22)$}; %
\node [below] at (1.9,0.816) {\small $5$}; %
\node at (0.07,-0.533) {\tiny $(11)$}; %
\node [above] at (1,-0.886) {\tiny $(17)$}; %
\node at (1.95,-0.533) {\tiny $(20)$}; %

\draw [thick] (0,0) -- (-0.25,0.433) -- (0,0.866) -- (0.5,0.866);
\node [above] at (0,0.08) {\scriptsize $a_0$}; %
\node [below] at (-0.25,0.916) {\scriptsize $c_3$}; %
\draw [thick] (-0.25,0.433) -- (-0.75,0.433) -- (-1,0);
\node [above] at (-0.5,0.383) {\scriptsize $c_0$}; %
\node [above] at (-1,0.05) {\scriptsize $b_0$}; %
\draw [thick] (0,0) -- (-0.25,-0.433) -- (0,-0.866);
\node [below] at (0,-0.08) {\scriptsize $a_2$}; %
\node [above] at (-0.25,-0.916) {\scriptsize $c_6$}; %
\draw [thick] (-0.25,-0.433) -- (-0.75,-0.433) -- (-1,0);
\node [below] at (-0.5,-0.383) {\scriptsize $c_7$}; %
\node [below] at (-1,0) {\scriptsize $b_4$}; %

\draw [thick] (0,0) -- (0.5,0) -- (0.75,-0.433);
\node [above] at (0.3,-0.05) {\scriptsize $a_6$}; %
\node [below] at (0.75,0) {\scriptsize $c_3$}; %
\draw [thick] (0.5,0) -- (0.75,0.433) -- (0.5,0.866);
\node [above] at (0.75,0) {\scriptsize $c_5$}; %
\node [below] at (0.5,0.816) {\scriptsize $b_1$}; %
\draw [thick] (0.75,0.433) -- (1.25,0.433) -- (1.5,0.866) -- (2,0.866);
\node [below] at (1,0.483) {\scriptsize $c_3$}; %
\node [below] at (1.5,0.816) {\scriptsize $a_7$}; %
\draw [thick] (1.25,0.433) -- (1.5,0) -- (2,0);
\node [above] at (1.25,-0.05) {\scriptsize $c_1$}; %
\node [above] at (1.75,-0.05) {\scriptsize $b_5$}; %
\draw [thick] (1.5,0) -- (1.25,-0.433);
\node [below] at (1.25,0.05) {\scriptsize $c_4$}; %

\draw [thick] (0,-0.866) -- (0.5,-0.866);
\node [below] at (0.25,-0.816) {\footnotesize $b_3$}; %
\draw [thick] (0.5,-0.866) -- (0.75,-0.433);
\node [above] at (0.5,-0.816) {\scriptsize $b_0$}; %
\draw [thick] (0.75,-0.433) -- (1.25,-0.433);
\node [above] at (1,-0.483) {\scriptsize $c_2$}; %
\draw [thick] (1.25,-0.433) -- (1.5,-0.866);
\node [above] at (1.5,-0.816) {\scriptsize $a_4$}; %
\draw [thick] (1.5,-0.866) -- (2,-0.866);
\node [below] at (1.75,-0.816) {\footnotesize $a_3$}; %
\draw [thick] (2,-0.866) -- (2.25,-0.433) -- (2,0);
\node [above] at (2.25,-0.916) {\scriptsize $c_6$}; %
\draw [thick] (2,0.866) -- (2.25,0.433) -- (2,0);
\node [below] at (2.25,0.916) {\scriptsize $c_4$}; %

\fill [opacity=0.1] (1,0) circle [radius=0.7];
\draw [thick] (1,0) circle [radius=0.7];
\fill [opacity=0.1] (-0.5,-0.866) -- (0.2,-0.866) -- (-0.85,-0.259);
\fill [opacity=0.1] (0.2,-0.866) arc (0:120:0.7);
\draw [thick] (0.2,-0.866) arc (0:120:0.7);
\fill [opacity=0.1] (-0.5,0.866) -- (0.2,0.866) -- (-0.85,0.259);
\fill [opacity=0.1] (0.2,0.866) arc (0:-120:0.7);
\draw [thick] (0.2,0.866) arc (0:-120:0.7);
\fill [opacity=0.1] (1.8,-0.866) -- (2.5,-0.866) -- (2.15,-0.259);
\fill [opacity=0.1] (1.8,-0.866) arc (180:120:0.7);
\draw [thick] (1.8,-0.866) arc (180:120:0.7);
\fill [opacity=0.1] (1.8,0.866) -- (2.5,0.866) -- (2.15,0.259);
\fill [opacity=0.1] (1.8,0.866) arc (180:240:0.7);
\draw [thick] (1.8,0.866) arc (180:240:0.7);

\draw [very thick, dash pattern=on 1pt off 1pt, color=red] (-0.5,-0.866) -- (-0.5,0.866);
\draw [very thick, color=red] (-0.5,-0.866) -- (-0.5,-0.166);
\draw [very thick, color=red] (-0.5,0.866) -- (-0.5,0.166);

\draw [very thick, dash pattern=on 1pt off 1pt, color=blue] (1,-0.866) -- (1,0.866);
\draw [very thick, color=blue] (1,-0.7) -- (1,0.7);

\end{scope}

\end{tikzpicture}
\caption{Compressing disk boundaries for $S_{2,1}$ in $M_{2,1}(\mu_0)$.}
\label{fig: comp disks from cusp}
\end{figure}
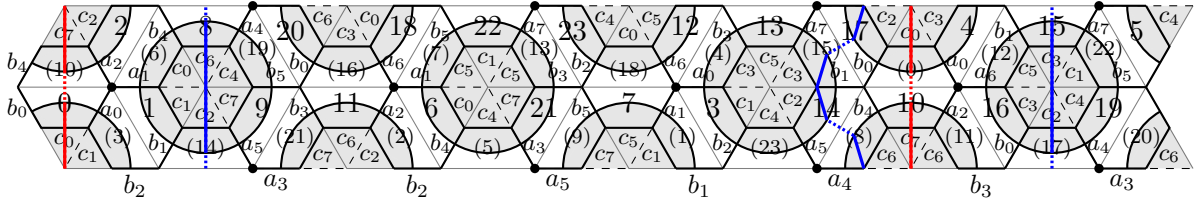

Let $\widetilde{M}_{2,1}^-$ refer to the preimage in $\Sigma_0\times[0,\infty)$ of $M_{2,1}^-$. Each red or blue arc of \Cref{fig: comp disks from cusp} indicates a closed curve on the boundary of $\widetilde{M}_{2,1}^-$ that bounds a punctured disk in $\widetilde{M}_{2,1}^-$ (after straightening the crooked blue arc at finite height). The solid sub-arcs represent these curves' intersection with $S_{2,1}$, and the dotted sub-arcs their intersection with $\Sigma_0\times\{0\}$. $M_{2,1}^-$ is recovered from $\widetilde{M}_{2,1}^-$ by applying the face-pairings of $\Sigma_0\times\{0\}$ to its intersection with $\partial \widetilde{M}_{2,1}^-$, which also pairs off the dotted sub-arcs of the red and blue curves. This joins the punctured disks together to form multiply-punctured disks, as pictured schematically in \Cref{fig: caploz}.

\begin{figure}[ht]
\begin{tikzpicture}

\begin{scope}[scale=0.9]
    \fill [color=red, opacity=0.1] (0,0) circle [radius=1.5];
    \draw [thick, color=red] (0,0) circle [radius=1.5];
    \draw [thick, dotted, color=red] (0,-1.5) -- (0,1.5);
    \draw (-0.1,0) -- (0.1,0);
    \node [right] at (0,-0.75) {$0$};
    \node [left] at (0,0.75) {$10$};
    \fill [color=white] (0.75,0) circle [radius=0.08];
    \draw [dotted] (0.75,0) circle [radius=0.1];
    \fill [color=white] (-0.75,0) circle [radius=0.08];
    \draw [dotted] (-0.75,0) circle [radius=0.1];
\end{scope}

\begin{scope}[scale=0.9, xshift=4.5cm]
    \fill [color=blue, opacity=0.1] (0,-1.5) -- (3,-1.5) -- (3,1.5) -- (0,1.5);
    \draw [thick, color=blue] (0,-1.5) -- (3,-1.5);
    \draw [thick, color=blue] (0,1.5) -- (3,1.5);
    \draw [thick, dotted, color=blue] (0,-1.5) -- (0,1.5);
    \draw [thick, dotted, color=blue] (3,-1.5) -- (3,1.5);
    \fill [color=blue, opacity=0.1] (3,-1.5) arc (-90:90:1.5);
    \draw [thick, color=blue] (3,-1.5) arc (-90:90:1.5);
    \fill [color=blue, opacity=0.1] (0,1.5) arc (90:270:1.5);
    \draw [thick, color=blue] (0,1.5) arc (90:270:1.5);
    \draw (-0.1,0) -- (0.1,0);
    \draw (2.9,0) -- (3.1,0);
    \node [left] at (0,0.75) {$15$};
    \node [right] at (0,-0.75) {$17$};
    \node [left] at (3,0.75) {$14$};
    \node [right] at (3,-0.75) {$8$};

    \draw [dotted] (1.5,0) circle [radius=0.75];
    \draw [dotted] (1.5,0.75) -- (1.5,1.5);
    \draw [dotted] (1.5,-0.75) -- (1.5,-1.5);
    \draw [dotted] (2.03,0.53) -- (3,1.5);
    \draw [dotted] (2.03,-0.53) -- (3,-1.5);
    \draw [dotted] (0.97,0.53) -- (0,1.5);
    \draw [dotted] (0.97,-0.53) -- (0,-1.5);

    \fill [color=white] (-0.75,0) circle [radius=0.08];
    \draw [dotted] (-0.75,0) circle [radius=0.1];
    \fill [color=white] (1.5,0) circle [radius=0.08];
    \draw [dotted] (1.5,0) circle [radius=0.1];
    \fill [color=white] (3.75,0) circle [radius=0.08];
    \draw [dotted] (3.75,0) circle [radius=0.1];

\end{scope}

\end{tikzpicture}
\caption{Two-punctured caplet; three-punctured lozenge.}
\label{fig: caploz}
\end{figure}
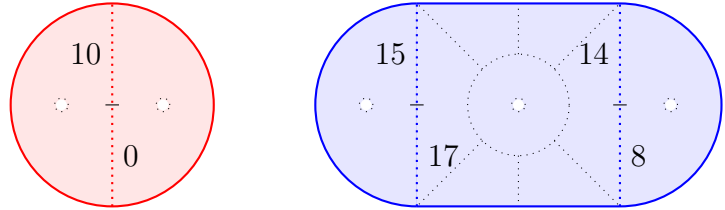

On the left in \Cref{fig: caploz}, the two punctured disks in $\widetilde{M}_{2,1}^-$ are joined to form a two-punctured disk in $M_{2,1}$ with its boundary on $S_{2,1}$. On the right, the three punctured disks in $\widetilde{M}_{2,1}^-$ join to form a three-punctured disk in $M_{2,1}$. Figure \ref{fig: S_21 with curves} depicts the disk boundaries on $S_{2,1}$ as solid curves. These are respectively isotopic to the same-colored dashed curves that lie in the one-skeleton of the decomposition into hexagons.

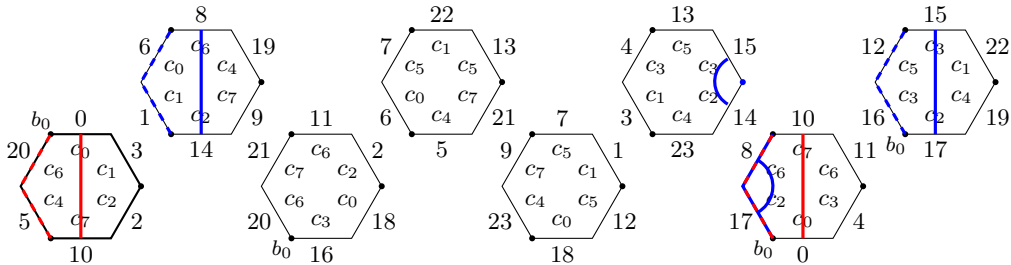
\begin{figure}[ht]
\begin{tikzpicture}

\begin{scope}[scale=0.8]
    \draw [thick] (0,0) -- (1,0) -- (1.5,0.866) -- (1,1.732) -- (0,1.732) -- (-0.5,0.866) -- cycle;
    \fill (0,0) circle [radius=0.05];
    \fill (0,1.732) circle [radius=0.05];
    \fill (1.5,0.866) circle [radius=0.05];

    \node [below right] at (1.15,0.583) {\scriptsize $2$};
    \node [above right] at (1.15,1.149) {\scriptsize $3$};
    \node [above] at (0.5,1.682) {\scriptsize $0$};
    \node [above left] at (-0.15,1.149) {\scriptsize $20$};
    \node [below left] at (-0.15,0.583) {\scriptsize $5$};
    \node [below] at (0.5,0.05) {\scriptsize $10$};

    \node [above left] at (0.2,1.632) {\tiny $b_0$};

    \node [above left] at (1.3,0.333) {\scriptsize $c_2$};
    \node [below left] at (1.3,1.399) {\scriptsize $c_1$};
    \node [below] at (0.5,1.732) {\scriptsize $c_0$};
    \node [below right] at (-0.3,1.399) {\scriptsize $c_6$};
    \node [above right] at (-0.3,0.333) {\scriptsize $c_4$};
    \node [above] at (0.5,0) {\scriptsize $c_7$};


    \draw [very thick, color=red] (0.5,0) -- (0.5,1.732);
    \draw [very thick, color=red, dashed] (0,0) -- (-0.5,0.866) -- (0,1.732);
\end{scope}
    
\begin{scope}[scale=0.8, xshift=2cm, yshift=1.732cm]
    \draw (0,0) -- (1,0) -- (1.5,0.866) -- (1,1.732) -- (0,1.732) -- (-0.5,0.866) -- cycle;
    \fill (0,0) circle [radius=0.05];
    \fill (0,1.732) circle [radius=0.05];
    \fill (1.5,0.866) circle [radius=0.05];

    \node [below right] at (1.15,0.583) {\scriptsize $9$};
    \node [above right] at (1.15,1.149) {\scriptsize $19$};
    \node [above] at (0.5,1.682) {\scriptsize $8$};
    \node [above left] at (-0.15,1.149) {\scriptsize $6$};
    \node [below left] at (-0.15,0.583) {\scriptsize $1$};
    \node [below] at (0.5,0.05) {\scriptsize $14$};

    \node [above left] at (1.3,0.333) {\scriptsize $c_7$};
    \node [below left] at (1.3,1.399) {\scriptsize $c_4$};
    \node [below] at (0.5,1.732) {\scriptsize $c_6$};
    \node [below right] at (-0.3,1.399) {\scriptsize $c_0$};
    \node [above right] at (-0.3,0.333) {\scriptsize $c_1$};
    \node [above] at (0.5,0) {\scriptsize $c_2$};

    \draw [very thick, color=blue] (0.5,0) -- (0.5,1.732);
    \draw [very thick, color=blue, dashed] (0,0) -- (-0.5,0.866) -- (0,1.732);
\end{scope}

\begin{scope}[scale=0.8, xshift=4cm]
    \draw (0,0) -- (1,0) -- (1.5,0.866) -- (1,1.732) -- (0,1.732) -- (-0.5,0.866) -- cycle;
    \fill (0,0) circle [radius=0.05];
    \fill (0,1.732) circle [radius=0.05];
    \fill (1.5,0.866) circle [radius=0.05];

    \node [below right] at (1.15,0.583) {\scriptsize $18$};
    \node [above right] at (1.15,1.149) {\scriptsize $2$};
    \node [above] at (0.5,1.682) {\scriptsize $11$};
    \node [above left] at (-0.15,1.149) {\scriptsize $21$};
    \node [below left] at (-0.15,0.583) {\scriptsize $20$};
    \node [below] at (0.5,0.05) {\scriptsize $16$};

    \node [below left] at (0.2,0.15) {\tiny $b_0$};

    \node [above left] at (1.3,0.333) {\scriptsize $c_0$};
    \node [below left] at (1.3,1.399) {\scriptsize $c_2$};
    \node [below] at (0.5,1.732) {\scriptsize $c_6$};
    \node [below right] at (-0.3,1.399) {\scriptsize $c_7$};
    \node [above right] at (-0.3,0.333) {\scriptsize $c_6$};
    \node [above] at (0.5,0) {\scriptsize $c_3$};

\end{scope}

\begin{scope}[scale=0.8, xshift=6cm, yshift=1.732cm]
    \draw (0,0) -- (1,0) -- (1.5,0.866) -- (1,1.732) -- (0,1.732) -- (-0.5,0.866) -- cycle;
    \fill (0,0) circle [radius=0.05];
    \fill (0,1.732) circle [radius=0.05];
    \fill (1.5,0.866) circle [radius=0.05];

    \node [below right] at (1.15,0.583) {\scriptsize $21$};
    \node [above right] at (1.15,1.149) {\scriptsize $13$};
    \node [above] at (0.5,1.682) {\scriptsize $22$};
    \node [above left] at (-0.15,1.149) {\scriptsize $7$};
    \node [below left] at (-0.15,0.583) {\scriptsize $6$};
    \node [below] at (0.5,0.05) {\scriptsize $5$};

    \node [above left] at (1.3,0.333) {\scriptsize $c_7$};
    \node [below left] at (1.3,1.399) {\scriptsize $c_5$};
    \node [below] at (0.5,1.732) {\scriptsize $c_1$};
    \node [below right] at (-0.3,1.399) {\scriptsize $c_5$};
    \node [above right] at (-0.3,0.333) {\scriptsize $c_0$};
    \node [above] at (0.5,0) {\scriptsize $c_4$};

\end{scope}

\begin{scope}[scale=0.8, xshift=8cm]
    \draw (0,0) -- (1,0) -- (1.5,0.866) -- (1,1.732) -- (0,1.732) -- (-0.5,0.866) -- cycle;
    \fill (0,0) circle [radius=0.05];
    \fill (0,1.732) circle [radius=0.05];
    \fill (1.5,0.866) circle [radius=0.05];

    \node [below right] at (1.15,0.583) {\scriptsize $12$};
    \node [above right] at (1.15,1.149) {\scriptsize $1$};
    \node [above] at (0.5,1.682) {\scriptsize $7$};
    \node [above left] at (-0.15,1.149) {\scriptsize $9$};
    \node [below left] at (-0.15,0.583) {\scriptsize $23$};
    \node [below] at (0.5,0.05) {\scriptsize $18$};

    \node [above left] at (1.3,0.333) {\scriptsize $c_5$};
    \node [below left] at (1.3,1.399) {\scriptsize $c_1$};
    \node [below] at (0.5,1.732) {\scriptsize $c_5$};
    \node [below right] at (-0.3,1.399) {\scriptsize $c_7$};
    \node [above right] at (-0.3,0.333) {\scriptsize $c_4$};
    \node [above] at (0.5,0) {\scriptsize $c_0$};

\end{scope}

\begin{scope}[scale=0.8, xshift=10cm, yshift=1.732cm]
    \draw (0,0) -- (1,0) -- (1.5,0.866) -- (1,1.732) -- (0,1.732) -- (-0.5,0.866) -- cycle;
    \fill (0,0) circle [radius=0.05];
    \fill (0,1.732) circle [radius=0.05];
    \fill [color=blue] (1.5,0.866) circle [radius=0.05];

    \node [below right] at (1.15,0.583) {\scriptsize $14$};
    \node [above right] at (1.15,1.149) {\scriptsize $15$};
    \node [above] at (0.5,1.682) {\scriptsize $13$};
    \node [above left] at (-0.15,1.149) {\scriptsize $4$};
    \node [below left] at (-0.15,0.583) {\scriptsize $3$};
    \node [below] at (0.5,0.05) {\scriptsize $23$};

    \node [above left] at (1.3,0.333) {\scriptsize $c_2$};
    \node [below left] at (1.3,1.399) {\scriptsize $c_3$};
    \node [below] at (0.5,1.732) {\scriptsize $c_5$};
    \node [below right] at (-0.3,1.399) {\scriptsize $c_3$};
    \node [above right] at (-0.3,0.333) {\scriptsize $c_1$};
    \node [above] at (0.5,0) {\scriptsize $c_4$};

    \draw [very thick, color=blue] (1.25,0.5) arc (240:120:0.433);
\end{scope}

\begin{scope}[scale=0.8, xshift=12cm]
    \draw (0,0) -- (1,0) -- (1.5,0.866) -- (1,1.732) -- (0,1.732) -- (-0.5,0.866) -- cycle;
    \fill (0,0) circle [radius=0.05];
    \fill (0,1.732) circle [radius=0.05];
    \fill (1.5,0.866) circle [radius=0.05];

    \node [below right] at (1.15,0.583) {\scriptsize $4$};
    \node [above right] at (1.15,1.149) {\scriptsize $11$};
    \node [above] at (0.5,1.682) {\scriptsize $10$};
    \node [above left] at (-0.15,1.149) {\scriptsize $8$};
    \node [below left] at (-0.15,0.583) {\scriptsize $17$};
    \node [below] at (0.5,0.05) {\scriptsize $0$};

    \node [below left] at (0.2,0.15) {\tiny $b_0$};

    \node [above left] at (1.3,0.333) {\scriptsize $c_3$};
    \node [below left] at (1.3,1.399) {\scriptsize $c_6$};
    \node [below] at (0.5,1.732) {\scriptsize $c_7$};
    \node [below right] at (-0.3,1.399) {\scriptsize $c_6$};
    \node [above right] at (-0.3,0.333) {\scriptsize $c_2$};
    \node [above] at (0.5,0) {\scriptsize $c_0$};

    \draw [very thick, color=red] (0.5,0) -- (0.5,1.732);
    \draw [very thick, color=blue] (0,0) -- (-0.5,0.866) -- (0,1.732);
    \draw [very thick, color=red, dashed] (0,0) -- (-0.5,0.866) -- (0,1.732);

    \draw [very thick, color=blue] (-0.25,0.433) arc (-60:60:0.5);
\end{scope}

\begin{scope}[scale=0.8, xshift=14.2cm, yshift=1.732cm]
    \draw (0,0) -- (1,0) -- (1.5,0.866) -- (1,1.732) -- (0,1.732) -- (-0.5,0.866) -- cycle;
    \fill (0,0) circle [radius=0.05];
    \fill (0,1.732) circle [radius=0.05];
    \fill (1.5,0.866) circle [radius=0.05];

    \node [below right] at (1.15,0.583) {\scriptsize $19$};
    \node [above right] at (1.15,1.149) {\scriptsize $22$};
    \node [above] at (0.5,1.682) {\scriptsize $15$};
    \node [above left] at (-0.15,1.149) {\scriptsize $12$};
    \node [below left] at (-0.15,0.583) {\scriptsize $16$};
    \node [below] at (0.5,0.05) {\scriptsize $17$};

    \node [below left] at (0.2,0.15) {\tiny $b_0$};

    \node [above left] at (1.3,0.333) {\scriptsize $c_4$};
    \node [below left] at (1.3,1.399) {\scriptsize $c_1$};
    \node [below] at (0.5,1.732) {\scriptsize $c_3$};
    \node [below right] at (-0.3,1.399) {\scriptsize $c_5$};
    \node [above right] at (-0.3,0.333) {\scriptsize $c_3$};
    \node [above] at (0.5,0) {\scriptsize $c_2$};

    \draw [very thick, color=blue] (0.5,0) -- (0.5,1.732);
    \draw [very thick, color=blue, dashed] (0,0) -- (-0.5,0.866) -- (0,1.732);
\end{scope}

\end{tikzpicture}
\caption{The punctured disk boundaries on $S_{2,1}$.}
\label{fig: S_21 with curves}
\end{figure}

By tracking hexagon edge-pairings, it is straightforward to check that the union of the red and blue curves does not separate $S_{2,1}$. Reading the dashed curves off as a concatenation of edge paths we obtain $e_{17}\bar{e}_8e_5\bar{e}_{20} = D^{-1}$ for the red curve, for $D$ as in \Cref{basis change}, and $e_{16}\bar{e}_{12}e_1\bar{e}_6e_8\bar{e}_{17} = E_8$ for the blue curve.\end{proof}

\begin{theorem}\label{main by hand} For the cover $M_{2,1}\to \doh^{333}_2$ prescribed by the right-permutation representation $\sigma_{2,1}\co\dprgr^{333}_2\to S_{24}$ given in \Cref{table_pr}, and the meridian $\mu_0$ for the cusp of $M_{2,1}$ identified in \Cref{filling slope}, the Dehn filled manifold $M_{2,1}(\mu_0)$ satisfies $\pi_1 M_{2,1}(\mu_0)\cong\mathbb{Z}/13\mathbb{Z}$. The preimage $\widetilde{M}_{2,1}$ of $M_{2,1}$ in the universal cover of $M_{2,1}(\mu_0)$ has one cusp.
\end{theorem}

\begin{remark} Geometrization implies that the universal cover of $M_{2,1}(\mu_0)$ is $\mathbb{S}^3$, hence that $\widetilde{M}_{2,1}$ is a knot complement in $\mathbb{S}^3$.\end{remark}

\begin{proof} Applying \Cref{M_21 minus}, we regard $M_{2,1}(\mu_0)$ as having been obtained from $M_{2,1}^+$ by attaching two-handles along a pair of disjoint simple closed curves in $S_{2,1}$---the red and blue curves of Figures \ref{fig: comp disks from cusp} and \ref{fig: S_21 with curves}---then capping off the resulting two-sphere boundary component with a ball. By \Cref{M_21 minus}, the two-handles' attaching curves are represented in $\pi_1 S_{2,1}$ by $D$ and $E_8$, which respectively include to $C_6C_4^{-1}C_6$ and $C_3C_5^{-1}C_1C_6$ in $\pi_1 M_{2,1}^+$ by \Cref{basis change} and \Cref{M_21plus}. We therefore have that
\[ \pi_1 M_{2,1}(\mu_0) \cong \pi_1 M_{2,1}^+/\langle\langle C_6C_4^{-1}C_6,C_3C_5^{-1}C_1C_6\rangle\rangle, \]
In the quotient, $C_4 = C_6^2$ and $C_6 = C_1^{-1}C_5C_3^{-1}$. We may thus produce a presentation for $\pi_1 M_{2,1}(\mu_0)$ by eliminating the generators $C_4$ and $C_6$ from the presentation for $\pi_1 M_{2,1}^+$ given in \Cref{M_21plus}, and substituting for them in the original four relations, yielding:\begin{align*}
    & C_1^{-1}C_7^{-1}(C_1^{-1}C_5C_3^{-1}) && C_1C_3^{-1}C_5C_3^{-1}(C_3C_5^{-1}C_1)(C_3C_5^{-1}C_1) = C_1C_3^{-1}C_1C_3C_5^{-1}C_1 \\ 
    & C_3(C_3C_5^{-1}C_1)C_7(C_3C_5^{-1}C_1) && C_5^{-1}C_1C_5^{-1}C_7(C_3C_5^{-1}C_1)(C_3C_5^{-1}C_1) \end{align*}
The top left relation gives $C_7 = C_1^{-1}C_5C_3^{-1}C_1^{-1}$. We may thus eliminate that relation and the generator $C_7$ by substituting for it in the other relations. The bottom two above become:\begin{align*}
    & C_3(C_3C_5^{-1}C_1)(C_1^{-1}C_5C_3^{-1}C_1^{-1})(C_3C_5^{-1}C_1) = C_3C_1^{-1}C_3C_5^{-1}C_1,\ \mbox{and} \\
    & C_5^{-1}C_1C_5^{-1}(C_1^{-1}C_5C_3^{-1}C_1^{-1})(C_3C_5^{-1}C_1)(C_3C_5^{-1}C_1)
\end{align*}
Solving the first of these two relations gives $C_5 = C_1C_3C_1^{-1}C_3$, which we use to eliminate that relation and the generator $C_5$. Substituting into the second relation above, and noting that $C_3C_5^{-1}C_1 = C_1C_3^{-1}$, yields:\begin{align}
    & (C_3^{-1}C_1C_3^{-1}C_1^{-1})C_1(C_3^{-1}C_1C_3^{-1}C_1^{-1})C_1^{-1}(C_1C_3C_1^{-1}C_3)C_3^{-1}C_1^{-1}(C_1C_3^{-1})(C_1C_3^{-1}) = \nonumber\\
    & (C_3^{-1}C_1C_3^{-1})(C_3^{-1}C_1C_3^{-1})C_1^{-1}C_3C_1^{-1}(C_3^{-1}C_1C_3^{-1}) \label{cable}
\end{align}
Upon substituting for $C_5$ in the final remaining relation, which was top right of the original four recorded above, we obtain:
\begin{align}\label{fable} C_1C_3^{-1}C_1C_3(C_3^{-1}C_1C_3^{-1}C_1^{-1})C_1 = C_1C_3^{-1}C_1^2C_3^{-1} \end{align}
This can be rewritten as $C_3^{-1}C_1C_3^{-1} = C_1^{-2}$. Using this to simplify (\ref{cable}) yields:
\[ (C_1^{-2})(C_1^{-2})C_1^{-1}C_3C_1^{-1}(C_1^{-2}). \]
This gives $C_3 = C_1^{8}$, so we can eliminate the generator $C_3$ and the relation (\ref{cable}). The final relation (\ref{fable}) becomes $C_1^{13}$ after substituting for $C_3$. Thus $\pi_1 M_{2,1}(\mu_0)\cong\mathbb{Z}/13\mathbb{Z}$.

The inclusion-induced map $\pi_1 M_{2,1}\to \pi_1 M_{2,1}(\mu_0)$ therefore factors through the abelianization $H_1 M_{2,1}$. Since the cusp of $M_{2,1}$ is $H_1$-surjective by \Cref{filling slope}, it thus has connected preimage in the universal cover of $M_{2,1}(\mu_0)$; hence $\widetilde{M}_{2,1}$ is one-cusped as claimed.
\end{proof}

\begin{remark} Tracing through the substitutions in the proof of \Cref{main by hand}, we obtain the following values for the generators of $\pi_1 M_{2,1}^+$ in $\pi_1 M_{2,1}(\mu_0) = \langle C_1\,|\,C_1^{13}\,\rangle$:
\[ C_3 = C_1^8,\quad C_4 = C_1,\quad C_5 = C_1^3,\quad C_6 = C_1^7,\quad C_7 = C_1^6. \]
\end{remark}

\section{Maps, surfaces, and mutation}\label{sec: mutants}
This section focuses on understanding isometries of the four $M_{i,j}$ manifolds,  their covers $\widetilde{M}_{i,j}$, and of the submanifolds obtained by cutting along the separating totally geodesic surfaces that \Cref{geodesic surface} implies they possess. We first give a simple combinatorial criterion describing isometries between covers of the $\doh^{333}_i$, $i=2$ or $3$.


\begin{lemma}\label{comb isom}
    For $i=2$ or $3$, suppose $M$ and $M'$ are manifolds covering $\doh^{333}_i$ with identical degree $n$, respectively prescribed by right-permutation representations $\sigma$ and $\sigma'$ as in \Cref{permrep}, each equipped with the orientation lifted from $\doh^{333}_i$. The set of isometries $M\to M'$ corresponds bijectively to the set of permutations $\phi\in S_n$ with the following properties:\begin{itemize}
        \item For $\phi$ corresponding to an orientation-preserving isometry, $\phi~\sigma(g) = \sigma'(g)\phi\in S_n$ for each generator $g\in\{x,y,z,w\}$ from (\ref{able}).
        \item For $\phi$ corresponding to an orientation-reversing isometry, $\phi~\sigma(g) = \sigma'(g)^{-1}\phi\in S_n$ for each generator $g\in\{x,y,z,w\}$ from (\ref{able}).
    \end{itemize}
Each $\phi\in S_n$ satisfying a criterion above determines an isometry $f\co M\to M'$ such that $f(\dpr_k) = \dpr_{\phi(k)}$ for each $k\in\{0,\hdots,n-1\}$, where the $\dpr_k$ belong to the respective decompositions of $M$ and $M'$ described in \Cref{cell decomp}.
\end{lemma}

\begin{proof} We note first that it follows from \Cref{no cover}, and can also be established directly, that each of the doubled prisms $\dpr^{333}_2$ and $\dpr^{333}_3$ has exactly one non-trivial self-isometry: the orientation-reversing reflection $\iota$ across the prism's doubling quadrilateral face. It follows that for covers $M$ and $M'$ as in the statement above, a polyhedral decomposition-preserving isometry $f\co M\to M'$ takes the faces of $\dpr_k$ corresponding to those of $\dpr^{333}_i$ exchanged by $x^{\pm 1}$ to the same pair of faces of $f(\dpr_k)$; and likewise for the pairs of faces exchanged by the other generators $y,z,w$ for $\dprgr^{333}_i$. The combinatorial conditions recorded above are the necessary ones for such a map $f$ (orientation-preserving or -reversing, respectively) to be well-defined on faces, given the face-pairings described in \Cref{cell decomp} that produce $M$ and $M'$.

These conditions are also \emph{sufficient} to determine a well-defined map $f$ from a given permutation $\phi$ satisfying one of them: for each $k$, take $f\vert_{\dpr_k}$ to be the inverse of the marking $\dpr^{333}_i\to\dpr_k$, followed by the marking to $\dpr_{\phi(k)}$, with $\iota$ in the middle if $\phi$ satisfies the second criterion. Then $f$ takes faces of distinct $\dpr_k$, $\dpr_{k'}$ that are paired in $M$ to image faces that are paired in $M'$; hence it is well-defined since the equivalence relations defining $M$ and $M'$ from the $\dpr_k$ are generated by the face-pairings. Since $\phi$ is a bijection, $f^{-1}$ is the map associated to $\phi^{-1}$; hence is a homeomorphism. That it is an isometry can thus be seen from the fact that it is a local isometry. This in turn is clear on the interior of each $\dpr_k$ by construction, and for points in the two-skeleton by the fact that $\phi$ is a bijection.

We finally note that the only isometries $M\to M'$ are polyhedral decomposition-preserving. This follows from the fact that the prism orbifold $O^{333}_i$ is minimal in its commensurability class, by \Cref{commensurator}. For if $f\co M\to M'$ is an isometry, taking $M = \mathbb{H}^3/\Gamma$ and $M' = \mathbb{H}^3/\Gamma'$ for subgroups $\Gamma$ and $\Gamma'$ of $\dpr^{333}_i$ we have that a lift $f_*$ of $f$ to the universal cover $\mathbb{H}^3$ conjugates $\Gamma$ to $\Gamma'$. Thus $f_*$ belongs to the commensurator of $\dpr^{333}_i$, and hence, by \Cref{commensurator} to the full prism orbifold group $\Pi^{333}_i$. This leaves invariant the tiling of $\mathbb{H}^3$ by lifts of $\dpr^{333}_i$.
\end{proof}

Applying \Cref{comb isom} to the $M_{i,j}$, we obtain the following result.

\begin{corollary}\label{isoms}
    For each $i\in\{2,3\}$ and $j\in\{1,2\}$, $M_{i,j}$ from \Cref{main_tech_thm} has no non-trivial isometries, and $\widetilde{M}_{i,j}$ is a chiral knot complement in $\mathbb{S}^3$.
\end{corollary}

\begin{proof} The combinatorial criteria of \Cref{comb isom} translate to algorithms that take as input two permutation representations $\sigma$ and $\sigma'$ to $S_n$, and test all possibilities for an element $\phi\in S_n$ conjugating $\sigma$ to $\sigma'$ or $(\sigma')^{-1}$. These are implemented as the methods \texttt{OPisoms} and \texttt{ORisoms}  included in the ancillary files (see \path{anc/Analysis_of_covers/code/IsomTests.py}). For each $i$ and $j$, \texttt{OPisoms($\sigma_{i,j},\sigma_{i,j}$)} returns only the identity permutation, and \texttt{ORisoms($\sigma_{i,j},\sigma_{i,j}$)} returns an empty list.

Recall from the proof of \Cref{main_tech_thm} that $\widetilde{M}_{i,j}$ is the preimage of $M_{i,j}$ in the universal cover $\mathbb{S}^3$ of a lens space $L(p,q)$ obtained by filling $M_{i,j}$ (where $(p,q) = (13,3)$ or $(22,5)$ for $j=1$ or $2$, respectively). In particular the action of the cyclic covering transformation group of $\mathbb{S}^3\to L(p,q)$ restricts to one on $\widetilde{M}_{i,j}$, and as a result, the covering map's restriction to $p\co\widetilde{M}_{i,j}\to M_{i,j}$ is also cyclic, with covering transformation group $\mathbb{Z}/p\mathbb{Z}$. If $\widetilde{M}_{i,j}$ was an achiral knot complement for some $i,j$, then there would be an orientation-reversing self-isometry $\rho$ of $\widetilde{M}_{i,j}$, so $p\circ\rho$ would be another covering map to $\widetilde{M}_{i,j}\to M_{i,j}$ with cyclic covering transformation group generated by $\rho\tau\rho$, where $\tau$ generates the covering transformation group of $p$. But since $H_1(M_{i,j})\cong\mathbb{Z}$, $M_{i,j}$ has a unique cyclic cover of any fixed degree, so this is not possible.
\end{proof}

\begin{table}
\centering
\centering
\begin{tabular}{ccrl}
\toprule
\textbf{i} & name & & permutation representation \\
\midrule
$\mathbf{2}$ & $\sigma_{2,1}'$ & $x\mapsto$ & $[1, 2, 0, 12, 10, 14, 20, 3, 21, 8, 16, 6, 7, 5, 13, 23, 4, 19, 17, 18, 11, 9, 15, 22]$ \\ 
    & & $y\mapsto$ & $[3, 6, 10, 4, 0, 19, 7, 1, 22, 5, 11, 2, 20, 8, 15, 17, 12, 14, 23, 9, 16, 18, 13, 21]$ \\ 
    & & $z\mapsto$ & $[2, 8, 5, 13, 17, 0, 11, 12, 9, 1, 22, 21, 19, 14, 3, 20, 10, 18, 4, 7, 23, 6, 16, 15]$ \\ 
    & & $w\mapsto$ & $[1, 0, 9, 15, 12, 8, 16, 17, 5, 2, 21, 22, 4, 23, 20, 3, 6, 7, 19, 18, 14, 10, 11, 13]$ \\
\midrule
$\mathbf{2}$ & $\sigma_{2,2}'$ & $x\mapsto$ & $[1, 2, 0, 10, 11, 17, 3, 16, 22, 8, 6, 14, 7, 20, 4, 19, 12, 18, 5, 23, 21, 13, 9, 15]$ \\ 
    & & $y\mapsto$ & $[3, 7, 11, 4, 0, 19, 1, 6, 15, 18, 16, 12, 2, 23, 10, 17, 14, 8, 21, 20, 5, 9, 13, 22]$\\ 
    & & $z\mapsto$ & $[2, 8, 5, 6, 14, 0, 13, 21, 9, 1, 23, 18, 16, 3, 15, 4, 22, 11, 17, 10, 7, 20, 12, 19]$ \\ 
    & & $w\mapsto$ & $[6, 10, 3, 2, 16, 13, 0, 11, 19, 23, 1, 7, 14, 5, 12, 22, 4, 21, 20, 8, 18, 17, 15, 9]$ \\
\midrule
$\mathbf{3}$ & $\sigma_{3,1}'$ & $x\mapsto$ & $[1, 2, 0, 13, 10, 17, 20, 3, 14, 8, 15, 6, 16, 7, 9, 4, 23, 18, 5, 21, 11, 22, 19, 12]$ \\ 
    & & $y\mapsto$ & $[3, 6, 10, 4, 0, 9, 7, 1, 12, 19, 11, 2, 22, 20, 17, 13, 18, 23, 21, 5, 15, 16, 8, 14]$ \\ 
    & & $z\mapsto$ & $[2, 8, 5, 7, 15, 0, 11, 14, 9, 1, 18, 21, 23, 19, 3, 16, 4, 10, 17, 22, 12, 6, 13, 20]$ \\ 
    & & $w\mapsto$ & $[6, 3, 12, 1, 7, 13, 0, 4, 10, 15, 8, 22, 2, 5, 18, 9, 17, 16, 14, 20, 19, 23, 11, 21]$ \\
\midrule
$\mathbf{3}$ & $\sigma_{3,2}'$ & $x\mapsto$ & $[1, 2, 0, 13, 10, 17, 3, 19, 20, 8, 15, 7, 14, 6, 22, 4, 23, 18, 5, 11, 9, 16, 12, 21]$ \\ 
    & & $y\mapsto$ & $[3, 7, 10, 4, 0, 9, 1, 6, 21, 14, 11, 2, 8, 19, 5, 13, 18, 23, 22, 15, 17, 12, 16, 20]$ \\ 
    & & $z\mapsto$ & $[2, 8, 5, 14, 15, 0, 18, 11, 9, 1, 12, 20, 22, 3, 13, 16, 4, 6, 17, 21, 7, 23, 10, 19]$ \\ 
    & & $w\mapsto$ & $[6, 4, 12, 7, 1, 19, 0, 3, 11, 13, 21, 8, 2, 9, 15, 14, 20, 22, 23, 5, 16, 10, 17, 18]$
\end{tabular}
\caption{Permutation representations for $\dprgr^{333}_i$, $i = 2$ or $3$.}
\label{table_pr prime}
\end{table}

Recall that each $\doh^{333}_i$ is amphichiral in that it has an orientation-reversing isometry coming from a lift of a reflection. One takeaway from the corollary above, since each $M_{i,j}$ is chiral, is that it should have a ``chiral twin''. The following remark identifies these.

\begin{remark}\label{OR pairs} In light of \Cref{isoms}, for each $i$ and $j$ one should expect the orientation-reversing self-isometry of $\doh^{333}_i$ ($i=2$ or $3$) to lift to an orientation-reversing isometry $M_{i,j}\to M_{i,j}'$ for another degree-$24$ cover $M_{i,j}'$ of $\doh^{333}_i$ that does not have an orientation-preserving isometry to $M_{i,j}$; that is, $M_{ij}'$ has a different permutation representation than $M_{i,j}$. 
This is indeed the case. The right-permutations $\sigma_{i,j}'$ corresponding to the $M_{i,j}'$ are in \Cref{table_pr prime}, and below we record the permutations $\phi_{i,j}$, as in \Cref{comb isom}, corresponding to the isometries $M_{i,j}\to M_{i,j}'$.\begin{align*}
    \phi_{2,1} & = [21, 23, 22, 20, 5, 6, 14, 9, 12, 7, 15, 13, 8, 18, 11, 10, 19, 2, 0, 1, 17, 3, 16, 4] \\
    \phi_{2,2} & = [17, 5, 18, 21, 23, 15, 19, 14, 2, 0, 20, 16, 9, 7, 22, 10, 8, 11, 4, 12, 13, 3, 1, 6] \\
    \phi_{3,1} & = [22, 23, 12, 17, 9, 13, 18, 16, 10, 15, 14, 19, 2, 20, 4, 8, 1, 3, 7, 21, 5, 6, 11, 0] \\
    \phi_{3,2} & = [3, 7, 13, 0, 4, 14, 1, 6, 21, 12, 15, 19, 17, 2, 5, 10, 16, 8, 9, 11, 20, 23, 18, 22]
\end{align*}\end{remark}

We turn attention to the surfaces $S_{i,j}$ and the complementary submanifolds $M_{i,j}^{\pm}$ to them.

\begin{prop}\label{volsides} For $i\in\{2,3\}$ and $j=\{1,2\}$, with $M_{i,j}$ as in \Cref{main_tech_thm}, the closed, embedded, totally geodesic surface $S_{i,j}\subset M_{i,j}$ supplied by \Cref{geodesic surface} is orientable, connected, and of genus two.
\end{prop}

\begin{proof}
    For each $i,j$, by \Cref{geodesic surface}, $S_{i,j}$ is separating since $a_9 = 3$ (for $i=2$) or $a_8 =3$ ($i=3$). As a two-sided surface in the orientable manifold $M_{i,j}$, $S_{i,j}$ is itself orientable. For such a surface $S_{i,j}$ and a lift $\dpr_k$ of the doubled prism $\dpr^{333}_i$ to $M_{i,j}$, $S_{i,j}\cap \dpr_k$ is isometric to $T\cup\bar{T}$, where $T\subset P^{333}_i$ is the embedded triangle of \Cref{corandreev} and $\bar{T}$ is its mirror image in the reflected copy of $P^{333}_i$ in $\dpr^{333}_i$. As $T$ has angles $\pi/3$, $\pi/3$, and $\pi/4$, it has area $\pi-\pi/3-\pi/3-\pi/4 = \pi/12$, and $T\cup\bar{T}$ has area $\pi/6$. Since $S_{i,j}$ is tiled by $24$ copies of $T\cup\bar{T}$ it has area $4\pi$, so as an orientable surface it is the connected genus-two surface.
\end{proof}

\begin{theorem}\label{mutants}
    For $i=2$ or $3$, $j=1$ or $2$, let $M_{i,j}$, $S_{i,j}$, $M_{i,j}^+$ and $M_{i,j}^-$ be as in \Cref{volumes}. For each $j\in\{1,2\}$, $M_{2,j}^{\pm}$ is isometric to $M_{3,j}^{\pm}$; hence $M_{2,j}$ is isometric to a manifold obtained from $M_{3,j}$ by mutation along $S_{i,j}$, for $j=1,2$.
\end{theorem}

In \Cref{mutants}, what we mean by \emph{mutation along $S_{i,j}$} is to cut $M_{i,j}$ open along the surface and re-identify the resulting pieces along their totally geodesic boundaries by a non-trivial self-isometry of $S_{i,j}$.

\begin{proof}
    As recalled in the proof of \Cref{volumes}, $S_{i,j}$ is the union of copies of $T\cup\bar{T}$ in the copies of the doubled prism $\dpr^{333}_i$ decomposing $M_{i,j}$, where $T\subset P^{333}_i$ is the embedded triangle of \Cref{corandreev} and $\bar{T}$ is its mirror image in the reflected copy of $P^{333}_i$ in $\dpr^{333}_i$. 
    Again as in the proof of \Cref{volumes},  the compact and non-compact submanifolds $M_{i,j}^+$ and $M_{i,j}^-$ of $M_{i,j}$, respectively, that are bounded by $S_{i,j}$ divide into copies of the respective doubles $\dpr^+$ and $\dpr^-$ of $P^+$ and $P^-$, where $P^+$ and $P^-$ are the sub-prisms of $P^{333}_i$ bounded by $T$. 
    
    For a fixed $i\in\{2,3\}$ and $j\in\{1,2\}$, and any $k\in\{0,1,\hdots,23\}$, we denote by $\dpr_k^{\pm}$ the intersection of $M_{i,j}^{\pm}$ with the part of $\dpr_k$ isometric to $\dpr^{\pm}$, where $\dpr_k$ is the $k$th tile of the cell decomposition of $M_{i,j}$ from \Cref{cell decomp}. We also give each face of $\dpr_k^{\pm}$ that is contained in a face of $\dpr_k$ the same number, corresponding to those prescribed in \Cref{numberings}.

    Having numbered the three-cells of the polyhedral decomposition of each $M_{i.j}^{\pm}$ from $0$ to $23$, we specify each map $\rho_j^{\pm}\co M_{2,j}^{\pm}\to M_{3,j}^{\pm}$, for each $j\in\{1,2\}$, as a permutation in $S_{24}$. The isometry so-determined should be understood to take each cell $\dpr_k^{\pm}$ in the decomposition of $M_{2,j}^{\pm}$ to the cell $\dpr_{\rho_j(k)}^{\pm}$ in the decomposition of $M_{3,j}^{\pm}$. We have:
    \begin{para}\label{rho plus} $\rho_1^+ = [10, 14, 11, 23, 15, 2, 3, 7, 12, 19, 21, 20, 1, 16, 18, 17, 9, 13, 8, 22, 5, 0, 4, 6]$ \end{para}
    On any $\dpr_k^+$ in the decomposition of $M_{2,1}^+$ this map takes the copy of $P^+$ coming from $P^{333}_2$ to the corresponding $P^+$ from $P^{333}_3$ via the orientation-reversing isometry described in \Cref{scis cong}, and likewise for the mirror images. Thus it takes the faces $1_+$ and $1_-$ of $\dpr_k^+$ to $2_+$ and $2_-$, respectively, in $\dpr_{\rho_1^+(k)}$. It takes the faces $2_+$ and $2_-$ to $1_+$ and $1_-$, respectively, and preserves the labeling of the faces $3_+$ and $3_-$.
    \begin{para}\label{rho minus} $\rho_1^- = [21, 6, 11, 13, 15, 20, 1, 7, 18, 19, 10, 2, 3, 22, 12, 17, 4, 23, 0, 16, 5, 8, 9, 14]$\end{para}
    As in the plus case, this reverses the orientation of each $\dpr_k^-$, but this time by taking the copy of $P_-$ in $\dpr_k$ to the copy of the mirror image of $P_-$ in $\dpr_{\rho_1^-(k)}$. Thus it swaps the faces labeled $1_+$ and $1_-$, those labeled $2_+$ and $2_-$, and those labeled $3_+$ and $3_-$.
    \[ \rho_2^+ = [0, 20, 4, 5, 2, 3, 6, 19, 8, 22, 16, 10, 7, 12, 14, 21, 11, 17, 23, 13, 1, 15, 9, 18] \]
    This map preserves orientation of each $\dpr_k^+$, by composing the orientation-reversing map described for $\rho_1^+$ with the one described for $\rho_1^-$. Thus it takes face $1_+$ of $\dpr_k^+$ to face $2_-$ of $\dpr_{\rho_2^+(k)}$; face $1_-$ to $2_+$, faces $2_+$ and $2_-$ to $1_-$ and $1_+$, respectively; and swaps faces $3_+$ and $3_-$.
    \[ \rho_2^- = [0, 1, 2, 3, 4, 5, 6, 7, 8, 9, 11, 10, 19, 13, 14, 15, 16, 17, 18, 12, 20, 21, 22, 23] \]
    This map preserves orientation of each $\dpr_k^-$ in the most vanilla way possible, taking each face to the face of $\dpr_{\rho_2^-(k)}$ with the same label.
    
    As each $\rho_j^{\pm}$ is visibly a bijection, the necessary and sufficient condition for it to define an isometry is that it preserve the face-pairing: that is, if cells $\dpr_k$ and $\dpr_{k'}$ of the domain intersect along a face, then their images in the target intersect along the image face(s). This translates to conjugacy conditions in $S_{24}$ that depend on the descriptions above of the isometries' restrictions to individual cells. Those listed below are enough to ensure well-definedness:\begin{align*}
        & \rho_1^+ \circ \sigma_{2,1}(y) = \sigma_{3,1}(z) \circ \rho_1^+ && \rho_1^+ \circ \sigma_{2,1}(z) = \sigma_{3,1}(y) \circ \rho_1^+ && \rho_1^+ \circ \sigma_{2,1}(w) = \sigma_{3,1}(w) \circ \rho_1^+ \\
        & \rho_1^- \circ \sigma_{2,1}(y) = \sigma_{3,1}(y)^{-1} \circ \rho_1^- && \rho_1^- \circ \sigma_{2,1}(z) = \sigma_{3,1}(z)^{-1} \circ \rho_1^- && \rho_1^- \circ \sigma_{2,1}(x) = \sigma_{3,1}(x)^{-1} \circ \rho_1^- \\
        & \rho_2^+ \circ \sigma_{2,2}(y) = \sigma_{3,2}(z)^{-1} \circ \rho_2^+ && \rho_2^+ \circ \sigma_{2,2}(z) = \sigma_{3,2}(y)^{-1} \circ \rho_2^+ && \rho_2^+ \circ \sigma_{2,2}(w) = \sigma_{3,2}(w)^{-1} \circ \rho_2^+ \\
        & \rho_2^- \circ \sigma_{2,2}(y) = \sigma_{3,2}(y) \circ \rho_2^- && \rho_2^- \circ \sigma_{2,2}(z) = \sigma_{3,2}(z) \circ \rho_2^- && \rho_2^- \circ \sigma_{2,2}(x) = \sigma_{3,2}(x) \circ \rho_2^-
        \end{align*}
    These can be verified directly.

    Identifying the boundaries of each of $M_{3,1}^{\pm}$ with $S_{3,1}$ by their inclusions into $M_{3,1}$, we define 
    \[ \mu_1 = (\rho_1^-)^{-1}\circ\rho_1^+ \co \partial M^+_{2,1}\to \partial M^-_{2,1},\mbox{ and }
    M_{2,1}^{\mu_1} = M_{2,1}^- \sqcup M_{2,1}^+/x\sim\mu_1(x)\ \forall x\in\partial M_{2,1}^+. \]
    That is, $M_{2,1}^{\mu_1}$ is obtained from $M_{2,1}$ by the mutation $\mu_1$ along $S_{2,1}$. We now simply observe that there is a well-defined isometry from $M_{2,1}^{\mu_1}$ to $M_{3,1}$ that restricts on $M_{2,1}^-$ to $\rho_1^-$ and on $M_{2,1}^+$ to $\rho_1^+$. It similarly follows that $M_{3,2}$ is isometric to a mutant of $M_{2,2}$ along $S_{2,2}$.
\end{proof}

\begin{example}\label{mutation}
    Using the decomposition into quadrilaterals that $S_{2,1}$ inherits from $M_{2,1}$, we obtain the following representation of $\mu_1 = (\rho_1^-)^{-1}\circ\rho_1^+$ from \ref{rho plus} and \ref{rho minus}:\begin{align*}
    \mu_1 & = [10,23,2,17,4,11,12,7,14,9,0,5,6,19,8,15,22,3,21,13,20,18,16,1] \\
        & = (0,10)(1,23)(2)(3,17)(4)(5,11)(6,12)(7)(8,14)(9)(13,19)(15)(16,22)(18,21)(20) \end{align*}
The top line above represents $\mu_1$ as a ``destination sequence'' following \ref{rho plus} and \ref{rho minus}; the second line gives the cycle decomposition of this permutation.

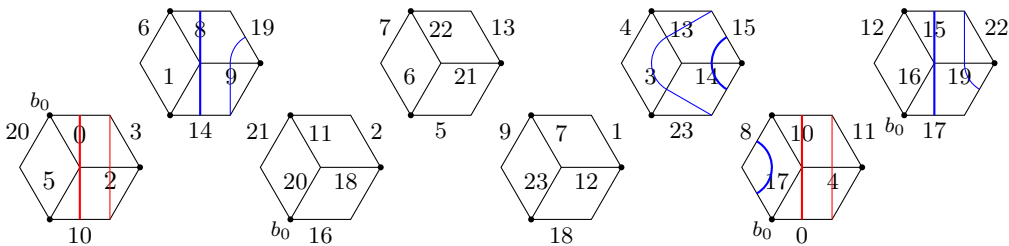
\begin{figure}[ht]
\begin{tikzpicture}

\begin{scope}[scale=0.8]
    \draw (0,0) -- (1,0) -- (1.5,0.866) -- (1,1.732) -- (0,1.732) -- (-0.5,0.866) -- cycle;
    \fill (0,0) circle [radius=0.05];
    \fill (0,1.732) circle [radius=0.05];
    \fill (1.5,0.866) circle [radius=0.05];

    \node [above right] at (1.15,1.149) {\scriptsize $3$};
    \node [above left] at (-0.15,1.149) {\scriptsize $20$};
    \node [below] at (0.5,0.05) {\scriptsize $10$};

    \node [above left] at (0.2,1.632) {\tiny $b_0$};

    \draw (0,0) -- (0.5,0.866) -- (1.5,0.866);
    \draw (0.5,0.866) -- (0,1.732);
    \node [above left] at (1.3,0.333) {\footnotesize $2$};
    \node [below] at (0.5,1.732) {\footnotesize $0$};
    \node [above right] at (-0.3,0.333) {\footnotesize $5$};

    \draw [thick, color=red] (0.5,0) -- (0.5,1.732);
    \draw [color=red] (1,0) -- (1,1.732);
\end{scope}
    
\begin{scope}[scale=0.8, xshift=2cm, yshift=1.732cm]
    \draw (0,0) -- (1,0) -- (1.5,0.866) -- (1,1.732) -- (0,1.732) -- (-0.5,0.866) -- cycle;
    \fill (0,0) circle [radius=0.05];
    \fill (0,1.732) circle [radius=0.05];
    \fill (1.5,0.866) circle [radius=0.05];

    \node [above right] at (1.15,1.149) {\scriptsize $19$};
    \node [above left] at (-0.15,1.149) {\scriptsize $6$};
    \node [below] at (0.5,0.05) {\scriptsize $14$};

    \draw (0,0) -- (0.5,0.866) -- (1.5,0.866);
    \draw (0.5,0.866) -- (0,1.732);
    \node [above left] at (1.3,0.333) {\scriptsize $9$};
    \node [below] at (0.5,1.732) {\scriptsize $8$};
    \node [above right] at (-0.3,0.333) {\scriptsize $1$};

    \draw [thick, color=blue] (0.5,0) -- (0.5,1.732);
    \draw [color=blue] (1,0) -- (1,0.866);
    \draw [color=blue] (1,0.866) arc (180:120:0.5);
\end{scope}

\begin{scope}[scale=0.8, xshift=4cm]
    \draw (0,0) -- (1,0) -- (1.5,0.866) -- (1,1.732) -- (0,1.732) -- (-0.5,0.866) -- cycle;
    \fill (0,0) circle [radius=0.05];
    \fill (0,1.732) circle [radius=0.05];
    \fill (1.5,0.866) circle [radius=0.05];

    \node [above right] at (1.15,1.149) {\scriptsize $2$};
    \node [above left] at (-0.15,1.149) {\scriptsize $21$};
    \node [below] at (0.5,0.05) {\scriptsize $16$};

    \node [below left] at (0.2,0.15) {\tiny $b_0$};

    \draw (0,0) -- (0.5,0.866) -- (1.5,0.866);
    \draw (0.5,0.866) -- (0,1.732);
    \node [above left] at (1.3,0.333) {\scriptsize $18$};
    \node [below] at (0.5,1.732) {\scriptsize $11$};
    \node [above right] at (-0.3,0.333) {\scriptsize $20$};

\end{scope}

\begin{scope}[scale=0.8, xshift=6cm, yshift=1.732cm]
    \draw (0,0) -- (1,0) -- (1.5,0.866) -- (1,1.732) -- (0,1.732) -- (-0.5,0.866) -- cycle;
    \fill (0,0) circle [radius=0.05];
    \fill (0,1.732) circle [radius=0.05];
    \fill (1.5,0.866) circle [radius=0.05];

    \node [above right] at (1.15,1.149) {\scriptsize $13$};
    \node [above left] at (-0.15,1.149) {\scriptsize $7$};
    \node [below] at (0.5,0.05) {\scriptsize $5$};

    \draw (0,0) -- (0.5,0.866) -- (1.5,0.866);
    \draw (0.5,0.866) -- (0,1.732);
    \node [above left] at (1.3,0.333) {\scriptsize $21$};
    \node [below] at (0.5,1.732) {\scriptsize $22$};
    \node [above right] at (-0.3,0.333) {\scriptsize $6$};

\end{scope}

\begin{scope}[scale=0.8, xshift=8cm]
    \draw (0,0) -- (1,0) -- (1.5,0.866) -- (1,1.732) -- (0,1.732) -- (-0.5,0.866) -- cycle;
    \fill (0,0) circle [radius=0.05];
    \fill (0,1.732) circle [radius=0.05];
    \fill (1.5,0.866) circle [radius=0.05];

    \node [above right] at (1.15,1.149) {\scriptsize $1$};
    \node [above left] at (-0.15,1.149) {\scriptsize $9$};
    \node [below] at (0.5,0.05) {\scriptsize $18$};

    \draw (0,0) -- (0.5,0.866) -- (1.5,0.866);
    \draw (0.5,0.866) -- (0,1.732);
    \node [above left] at (1.3,0.333) {\scriptsize $12$};
    \node [below] at (0.5,1.732) {\scriptsize $7$};
    \node [above right] at (-0.3,0.333) {\scriptsize $23$};

\end{scope}

\begin{scope}[scale=0.8, xshift=10cm, yshift=1.732cm]
    \draw (0,0) -- (1,0) -- (1.5,0.866) -- (1,1.732) -- (0,1.732) -- (-0.5,0.866) -- cycle;
    \fill (0,0) circle [radius=0.05];
    \fill (0,1.732) circle [radius=0.05];
    \fill (1.5,0.866) circle [radius=0.05];

    \node [above right] at (1.15,1.149) {\scriptsize $15$};
    \node [above left] at (-0.15,1.149) {\scriptsize $4$};
    \node [below] at (0.5,0.05) {\scriptsize $23$};

    \draw (0,0) -- (0.5,0.866) -- (1.5,0.866);
    \draw (0.5,0.866) -- (0,1.732);
    \node [above left] at (1.3,0.333) {\scriptsize $14$};
    \node [below] at (0.5,1.732) {\scriptsize $13$};
    \node [above right] at (-0.3,0.333) {\scriptsize $3$};

    \draw [thick, color=blue] (1.25,0.433) arc (240:120:0.5);
    \draw [color=blue] (1,0) -- (0.25,0.433);
    \draw [color=blue] (1,1.732) -- (0.25,1.299);
    \draw [color=blue] (0.25,0.433) arc (240:120:0.5);
\end{scope}

\begin{scope}[scale=0.8, xshift=12cm]
    \draw (0,0) -- (1,0) -- (1.5,0.866) -- (1,1.732) -- (0,1.732) -- (-0.5,0.866) -- cycle;
    \fill (0,0) circle [radius=0.05];
    \fill (0,1.732) circle [radius=0.05];
    \fill (1.5,0.866) circle [radius=0.05];

    \node [above right] at (1.15,1.149) {\scriptsize $11$};
    \node [above left] at (-0.15,1.149) {\scriptsize $8$};
    \node [below] at (0.5,0.05) {\scriptsize $0$};

    \node [below left] at (0.2,0.15) {\tiny $b_0$};

    \draw (0,0) -- (0.5,0.866) -- (1.5,0.866);
    \draw (0.5,0.866) -- (0,1.732);
    \node [above left] at (1.3,0.333) {\scriptsize $4$};
    \node [below] at (0.5,1.732) {\scriptsize $10$};
    \node [above right] at (-0.3,0.333) {\scriptsize $17$};

    \draw [thick, color=red] (0.5,0) -- (0.5,1.732);
    \draw [thick, color=blue] (-0.25,0.433) arc (-60:60:0.5);

    \draw [color=red] (1,0) -- (1,1.732);
\end{scope}

\begin{scope}[scale=0.8, xshift=14.2cm, yshift=1.732cm]
    \draw (0,0) -- (1,0) -- (1.5,0.866) -- (1,1.732) -- (0,1.732) -- (-0.5,0.866) -- cycle;
    \fill (0,0) circle [radius=0.05];
    \fill (0,1.732) circle [radius=0.05];
    \fill (1.5,0.866) circle [radius=0.05];

    \node [above right] at (1.15,1.149) {\scriptsize $22$};
    \node [above left] at (-0.15,1.149) {\scriptsize $12$};
    \node [below] at (0.5,0.05) {\scriptsize $17$};

    \node [below left] at (0.2,0.15) {\tiny $b_0$};

    \draw (0,0) -- (0.5,0.866) -- (1.5,0.866);
    \draw (0.5,0.866) -- (0,1.732);
    \node [above left] at (1.3,0.333) {\scriptsize $19$};
    \node [below] at (0.5,1.732) {\scriptsize $15$};
    \node [above right] at (-0.3,0.333) {\scriptsize $16$};

    \draw [thick, color=blue] (0.5,0) -- (0.5,1.732);
    \draw [color=blue] (1,1.732) -- (1,0.866);
    \draw [color=blue] (1,0.866) arc (180:240:0.5);
\end{scope}

\end{tikzpicture}
\caption{The decomposition that $S_{2,1}$ inherits from $M_{2,1}$, with compressing disk boundaries for $M_{2,1}^-$ (thick) and their $\mu_1$-images (thin).}
\label{fig: S_21 with diamonds}
\end{figure}

The decomposition into quadrilaterals that $S_{2,1}$ inherits from the cell decomposition of $M_{2,1}$ from \Cref{cell decomp} is pictured in \Cref{fig: S_21 with diamonds} as a subdivision of the hexagon decomposition from \Cref{fig: S_21 with curves}. In the Figure, the label on the inside of each quadrilateral refers to the index $k$ of the polyhedron $\dpr_k$ that it lies in. (This is the same as the label on its adjacent exterior edge in \Cref{fig: S_21 with curves}, due to our labeling convention for those edges and for the faces of the visible pre-spine $\Sigma_0$; cf.~\Cref{ex: M_21 spine}.) The labels outside edges refer to the index of the quadrilateral attached to that edge.

The action of the mutation $\mu_1$ can be seen in this figure as taking the quadrilateral labeled ``$k$'' to the one labeled by the image of $k$ under the permutation above, for each $k\in\{0,\hdots,23\}$. We claim that it ``rotates by $180$ degrees'', in the sense that it takes the sub-triangle of each that contains its label to the one in the image that does not. The ``sub-triangles'' to which we refer here are the diamond's intersection with the copies of $P^{333}_2$ and $\bar{P}^{333}_2$ in $\widetilde{P}_k$, which is isometric to their union $\dpr^{333}_2$. The sub-triangle in $P^{333}_2$ is parallel to its face labeled ``$3_-$'' in \Cref{numberings}, and the one in $\bar{P}^{333}_2$ to its face $3_+$ (compare Figure \ref{compact triangle}). Each quadrilateral's label $k$ lies in the sub-triangle in parallel to $3_+$.

The sub-triangles of each quadrilateral meet along its short edge, with its vertices in $a$-labeled edges. 
The claim follows from the natures of $\rho_1^{\pm}$ laid out in \Cref{mutants}: $\rho_1^+$ acts on quadrilaterals preserving the signs of their decompositions into sub-triangles, thus reflecting in the diagonal joining the $b$-labeled vertices; whereas $\rho_1^-$ reflects in the short diagonal, taking each face parallel to $3_+$ to one parallel to $3_-$ and vice-versa. In particular, $\mu_1$ takes the $a$-labeled vertex of any quadrilateral that lies at the center of the hexagon containing it to the $a$-labeled vertex of its image that lies on the boundary of the hexagon containing it.

The compressing disk boundaries for $S_{2,1}$ in the handlebody $M_{2,1}^-(\mu_0)$ are also pictured in \Cref{fig: S_21 with diamonds} (cf.~\Cref{M_21 minus}), along with their $\mu_1$-images. It can be easily checked that each curve bounds an annulus with its $\mu_1$-image, hence that they are isotopic although with orientations reversed. One can also check that closed curves representing the other generators $E_3$ and $E_7$ for $\pi_1 S_{2,1}$ from \Cref{basis change} are also taken to curves freely homotopic to their inverses. Thus at the level of $H_1$, $\mu_1$ acts as multiplication by $-1$.
\end{example}

\begin{theorem}\label{cover mutants} For $i=2$ or $3$, let $\widetilde{S}_{i,1}$ be the preimage in $\widetilde{M}_{i,1}$ of $S_{i,1}$ from \Cref{volsides}, where $\widetilde{M}_{i,1}\to M_{i,1}$ is the cover from \Cref{main_tech_thm}. $\widetilde{M}_{3,1}$ is isometric to a manifold obtained from $\widetilde{M}_{2,1}$ by mutation along $\widetilde{S}_{2,1}$.\end{theorem}

\begin{remark}
    We do not know that $\widetilde{M}_{2,2}$ and $\widetilde{M}_{3,2}$ are related by mutation. In order to use the proof strategy below to show this, we would need to know that $M_{2,2}^-$ is a knot complement in a handlebody (cf.~\Cref{good stuff}) and to understand the action of $(\rho_2^-)^{-1}\circ\rho_2^+$ on a full set of compressing disk boundaries for this handlebody. However, \Cref{scis cong} does at least imply that $\widetilde{M}_{2,2}$ and $\widetilde{M}_{3,2}$ are  OR-scissors congruent in the sense defined there.
\end{remark}

\begin{proof}
    We will show that the isometries $\rho_1^{\pm}$ from \Cref{mutants} lift to maps $\widetilde{M}_{2,1}^{\pm}\to\widetilde{M}_{3,1}^{\pm}$, where $\widetilde{M}^{\pm}_{i,1}$ is the preimage in $\widetilde{M}_{i,1}$ of $M^{\pm}_{i,1}$ for $i=2,3$. The result will then follow as in the final paragraph of the proof of \Cref{mutants}.

    We begin by noting that by \Cref{filling slope}, the Euclidean geodesic $\mu_0$ defined there is the only simple closed geodesic of length less than $6$ on its cusp cross-section. Therefore the image $\mu_0'$ of $\mu_0$ under $\rho_1^{-}$ has the same property, and it follows from the six theorem that the lens space filling of $M_{3,1}$ identified in \Cref{lem:lifts} must be along $\mu_0'$. Therefore $\rho_1^-$ extends to a homeomorphism $M_{2,1}^-(\mu_0)\to M_{3,1}^-(\mu_0')$ to a handlebody in the lens space $M_{3,1}(\mu_0')$ bounded by $S_{3,1}$.

    It follows that we may regard the lens space $M_{3,1}(\mu_0')$ as obtained from $M_{3,1}^+$ by attaching two-handles along the boundaries of compressing disks for $S_{3,1}$ in  $M_{3,1}^-(\mu_0')$, then capping off two-sphere boundary components with balls. The $\rho_1^-$-image of the two compressing disks for $S_{2,1}$ in $M_{2,1}^-(\mu_0)$ identified in  \Cref{M_21 minus} forms a full set of compressing disks for $S_{3,1}$ in $M_{3,1}^-(\mu_0')$. Then $D' = (\rho_1^-)_*(D)$ and $E_8' = (\rho_1^-)_*(E_8)$ represent their images in $\pi_1(S_{3,1})$, by \Cref{M_21 minus}. Since, as observed in \Cref{mutation}, the $\mu_1$-image of each compressing disk boundary is freely homotopic to itself with reversed orientation, up to conjugacy in $\pi_1 S_{2,1}$ we have $(\mu_1)_*(D) = D^{-1}$ and $(\mu_1)_*(E_8) = E_8^{-1}$. Therefore since $\mu_1 = (\rho_1^+)^{-1}\circ\rho_1^-$, we have that $(\rho_1^+)_*(D) = (D')^{-1}$ and $(\rho_1^+)_*(E_8) = (E_8')^{-1}$, again up to conjugacy in $\pi_1 S_{3,1}$.

    For $i=2$ or $3$, the cover $\widetilde{M}_{i,1}\to M_{i,1}$ corresponds to the kernel of the inclusion-induced map $\pi_1M_{i,1}\to\pi_1 M_{i,1}(\mu_0) \cong\mathbb{Z}/13\mathbb{Z}$ (replacing ``$\mu_0$'' by ``$\mu_0'$'' if appropriate). The same inclusion-induced map determines the covers $\widetilde{M}^{\pm}_{i,1}\to M^{\pm}_{i,1}$ for each $i$. Because $D$ and $E_8$ normally generate the kernel of the inclusion-induced map $\pi_1 M_{2,1}^+\to \pi_1 M_{2,1}(\mu_0)$, to see that $\rho_1^+$ lifts it is enough to know that $(\rho_1^+)_*$ maps these elements trivially. The previous paragraph's final sentence implies this. That $\rho_1^-$ lifts follows from the fact that it extends to a map $M_{2,1}^-(\mu_0)\to M_{3,1}^-(\mu_0')$.
\end{proof}

\section{Restrictions on prism orbifolds covered by knot complements}\label{sec:which prisms}

As mentioned in \Cref{subsec: compute}, the manifolds of \Cref{main_tech_thm} were found by an extensive, computer-aided search. We now give more details on this search, specifically on the \emph{pre-filtering}, \emph{enumeration}, and \emph{filtering} steps respectively outlined in \ref{pre-filter}, \ref{enumerate}, and \ref{filter}.

\subsection{Pre-filtering 1: cusps}\label{subsec: prefilter cusp} 
The pre-filtering process tested the prism orbifolds of Tables \ref{table236} and \ref{table333} for some necessary conditions to be covered by a knot complement, so that we could eliminate those not covered by one from further consideration. Tables \ref{table236Obstructions} and \ref{table333obsturctions} below record our end result. This subsection focuses on a pair of conditions involving the orbifolds' cusps. 

The first condition that we tested uses the ``cusp-killing homomorphism''.

\begin{definition} \label{cusp kill}
Let \( Q = \mathbb{H}^3 / \Gamma_Q \) be an orientable, one-cusped hyperbolic 3–orbifold of finite volume with a non-torus cusp. Denote the order of an element $\gamma$ in $\Gamma_Q$ by $|\gamma|$, denote the peripheral subgroup of $\Gamma_Q$ by $P_Q$, and let \( R = \left\{ \gamma \in P_Q \;|\; |\gamma| < \infty \right\} \). We define the \textit{cusp killing homomorphism} to be
\[
f: \Gamma_Q \rightarrow \Gamma_Q/\langle\langle R \rangle\rangle_{\Gamma_Q}.
\] 
\end{definition}

The Lemma below is a version of \cite[Corollary 4.11]{BBCW} and \cite[Proposition 2.3]{Neil_SmallKnots}.

\begin{lemma}[Boileau-Boyer-Cebanu-Walsh, Hoffman]\label{BBCW-H}
Let \( Q = \mathbb{H}^3 / \Gamma_Q \) be an orientable hyperbolic 3-orbifold with a non-torus cusp covered by knot complement. Let \( f \) denote the cusp killing homomorphism. Then, \( f(\Gamma_Q) \) is trivial. Additionally, \( |Q| \cong B^3 \), every component of the isotropy graph of \( Q \) is connected to the cusp, and the fundamental group is generated by peripheral torsion elements (\( |Q| \) denotes the underlying topological space of \( Q \)). 
 \end{lemma}

Each orientable prism orbifold has underlying space $B^3$, and its isotropy graph's sole component---the prism's one-skeleton---is connected to the cusp. The following algorithm checks the effect of cusp-killing on the orbifold fundamental group; cf.~\cite[Remark 2.4]{Neil_SmallKnots}.

First, erase the edges corresponding to peripheral torsion. Next check any degree 2 vertices in the isotropy graph and relabel the torsion on these edges using the gcd of the two orders of torsion connected to that vertex. If the gcd is 1, prune these two edges. Then prune any edges incident to a degree 1 vertex.  Remove any vertices not incident to any edge. Finally, resolve each degree two vertex to be part of an edge (or a loop).  Repeat the process until either the graph is empty or cannot be resolved further. In the case that the resulting graph is non-empty, image of the fundamental group under cusp-killing is non-trivial; hence the orbifold is not covered by a knot complement. 

We now illustrate the algorithm above on the first two orbifolds from \Cref{table236}.

\begin{example}
 The effects of the cusp killing homomorphism on the isotropy graphs of $O^{236}_{1}$ and $O^{236}_{2}$ are demonstrated in Figure \ref{cuspkill}. The figure follows two parallel implementations of cusp-killing with $O^{236}_{1}$ on the left and $O^{236}_{2}$ on the right. It is easy to see that after one more iteration on the graphs in \ref{cuspkill}(d), the isotropy graph of $O^{236}_{1}$ remains a closed loop labeled 2 while that of $O^{236}_{2}$ disappears entirely. Therefore $O^{236}_{1}$ is not covered by a knot complement.

\begin{figure}[ht]
\centering
\begin{tikzpicture}
\begin{scope}[shift={(-4,3)}]
    \begin{scope}[shift={(-0.5,0.5)}]
        \draw [thick] (0,-2) -- (-2,-2);
        \fill [color=white] (-1.2,-2.6) -- (-1.35,-1.9) -- (-1.05,-1.9);
        \draw [thick] (0,-2) -- (0,0) -- (-1.2,-0.6) -- (-1.2,-2.6) -- (-2,-2);
        \draw [thick] (0,0) -- (0,-2) -- (-1.2,-2.6) -- (-1.2,-0.6) -- (0,0) -- (-2,0) -- (-2,-2) -- (-1.2,-2.6);
        \draw [thick] (-1.2,-0.6) -- (-2,0);
        \fill [color=white] (-1.2,-2.6) circle [radius=0.1];
        \draw (-1.2,-2.6) circle [radius=0.1];
        \node [right] at (0,-1) {\scriptsize $2$};
        \node [below right] at (-0.7,-2.2) {\scriptsize $3$};
        \node [below left] at (-1.5,-2.2) {\scriptsize $2$};
        \node [above] at (-0.6,-2.05) {\scriptsize $3$};
        \node [right] at (-1.25,-1.35) {\scriptsize $6$};
        \node [left] at (-2,-1) {\scriptsize $4$};
        \node [below right] at (-0.7,-0.2) {\scriptsize $2$};
        \node [below left] at (-1.5,-0.2) {\scriptsize $2$};
        \node [above] at (-1,-0.05) {\scriptsize $2$};
    \end{scope}
    
    \begin{scope}[shift={(-0.5,-3.5)}]
        \draw [thick] (0, 0) -- (-2, 0);
        \draw [thick] (0,0) .. controls (0,-2.5) and (-2,-2.5) .. (-2,0);
        \draw [thick] (0,0) .. controls (-1.2,-1.2) .. (-2,0);

        \node [above] at (-1,-2.35) {\scriptsize $1$};
        \node [right] at (-1.25,-1.1) {\scriptsize $2$};
        \node [above] at (-1,-0.05) {\scriptsize $2$};
    \end{scope}
    \node [below] at (-0.5,-2.5) {\small (a) The isotropy graphs before the process};

    \begin{scope}[shift={(2.5,0.5)}]
        \draw [thick] (0,-2) -- (-2,-2);
        \fill [color=white] (-1.2,-2.6) -- (-1.35,-1.9) -- (-1.05,-1.9);
        \draw [thick] (0,-2) -- (0,0) -- (-1.2,-0.6) -- (-1.2,-2.6) -- (-2,-2);
        \draw [thick] (0,0) -- (0,-2) -- (-1.2,-2.6) -- (-1.2,-0.6) -- (0,0) -- (-2,0) -- (-2,-2) -- (-1.2,-2.6);
        \draw [thick] (-1.2,-0.6) -- (-2,0);
        \fill [color=white] (-1.2,-2.6) circle [radius=0.1];
        \draw (-1.2,-2.6) circle [radius=0.1];
        \node [right] at (0,-1) {\scriptsize $2$};
        \node [below right] at (-0.7,-2.2) {\scriptsize $3$};
        \node [below left] at (-1.5,-2.2) {\scriptsize $2$};
        \node [above] at (-0.6,-2.05) {\scriptsize $3$};
        \node [right] at (-1.25,-1.35) {\scriptsize $6$};
        \node [left] at (-2,-1) {\scriptsize $4$};
        \node [below right] at (-0.7,-0.2) {\scriptsize $2$};
        \node [below left] at (-1.5,-0.2) {\scriptsize $2$};
        \node [above] at (-1,-0.05) {\scriptsize $3$};
    \end{scope}
    
    \begin{scope}[shift={(2.5,-3.5)}]
        \draw [thick] (0, 0) -- (-2, 0);
        \draw [thick] (0,0) .. controls (0,-2.5) and (-2,-2.5) .. (-2,0);
        \draw [thick] (0,0) .. controls (-1.2,-1.2) .. (-2,0);
        
        \node [above] at (-1,-2.35) {\scriptsize $1$};
        \node [right] at (-1.25,-1.1) {\scriptsize $2$};
        \node [above] at (-1,-0.05) {\scriptsize $3$};
    \end{scope}
    
    \node [below] at (-0.5,-6) {\small (c) The vertices with valency 2 resolved };
\end{scope}

\begin{scope}[shift={(2.5,3)}]
    \begin{scope}[shift={(-0.5,0.5)}]

        \draw [thick] (0, 0) -- (0,-2) -- (-2,-2) -- (-2, 0) -- (0, 0);
        \draw [thick] (0, 0) -- (-1.2,-0.6) -- (-2, 0);
        
        \draw [thick] (-1.2,-0.6) -- (-2,0);

        \node [right] at (0,-1) {\scriptsize $2$};
        \node [above] at (-1,-2.05) {\scriptsize $3$};
        \node [left] at (-2,-1) {\scriptsize $4$};
        \node [below right] at (-0.7,-0.2) {\scriptsize $2$};
        \node [below left] at (-1.5,-0.2) {\scriptsize $2$};
        \node [above] at (-1,-0.05) {\scriptsize $2$};
    \end{scope}
    
    \begin{scope}[shift={(-0.5,-3.5)}]
        \draw [thick] (0, 0) -- (-2, 0);

        \draw [thick] (0,0) .. controls (-1.2,-1.2) .. (-2,0);

        \node [right] at (-1.25,-1.1) {\scriptsize $2$};
        \node [above] at (-1,-0.05) {\scriptsize $2$};
    \end{scope}
    
    \begin{scope}[shift={(2.5,0.5)}]
        \draw [thick] (0, 0) -- (0,-2) -- (-2,-2) -- (-2, 0) -- (0, 0);
        \draw [thick] (0, 0) -- (-1.2,-0.6) -- (-2, 0);
        \draw [thick] (-1.2,-0.6) -- (-2,0);

        \node [right] at (0,-1) {\scriptsize $2$};
        \node [above] at (-1,-2.05) {\scriptsize $3$};
        \node [left] at (-2,-1) {\scriptsize $4$};
        \node [below right] at (-0.7,-0.2) {\scriptsize $2$};
        \node [below left] at (-1.5,-0.2) {\scriptsize $2$};
        \node [above] at (-1,-0.05) {\scriptsize $3$};
    \end{scope}
    \node [below] at (0.5,-2.5) {\small (b) Torsion edges at the cusp removed};

    \begin{scope}[shift={(2.5,-3.5)}]
        \draw [thick] (0, 0) -- (-2, 0);
        \draw [thick] (0,0) .. controls (-1.2,-1.2) .. (-2,0);
        
        \node [right] at (-1.25,-1.1) {\scriptsize $2$};
        \node [above] at (-1,-0.05) {\scriptsize $3$};
    \end{scope}
    
    \node [below] at (0,-6) {\small (d) The edges labeled 1 removed};
\end{scope}


\end{tikzpicture}
\caption{A step by step graphical interpretation of the effects of the cusp killing homomorphism for $O^{236}_{1}$ (left) and $O^{236}_{2}$ (right).}
\label{cuspkill}
\end{figure}
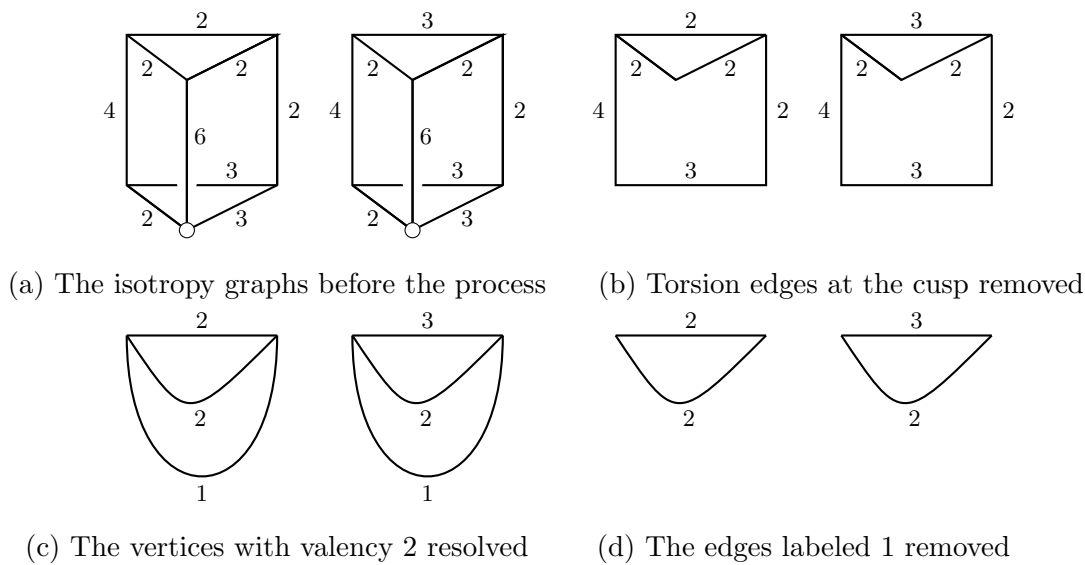

\end{example}

Applying the algorithm above to each orbifold from Tables \ref{table236} and \ref{table333} yields:

\begin{prop}\label{CK}
Each orbifold listed below is not covered by a knot complement because its orbifold fundamental group maps non-trivially under the cusp-killing homomorphism.
\begin{itemize}
        \item $\{\doh^{236}_{i} \mid\  i \in \{1, 3, 9, 10, 11, 14, 16, 18, 20, 23, 24, 25, 27, 28, 29, 32, 34, 36, 38\} \} $ 
       \item $\{\doh^{236}_{5,n}, \doh^{236}_{7,n} \mid n \geq 7\} \cup \{\doh^{236}_{13,n}, \doh^{236}_{22,n}, \doh^{236}_{31,n}, \doh^{236}_{40,n} \mid n \geq 6\}$
       \item $ \{\doh^{236}_{6,2k}, \doh^{236}_{8,2k} \mid k \geq 4 \} $ 
       \item $\{\doh^{333}_{i} \mid i \in \{1, 5, 7, 11, 13, 14, 15, 16, 20\}\}$
\end{itemize}
\end{prop}

The next condition that we tested follows from the result below, due to the third author.

\begin{lemma}[\cite{NeilOrbiCusps}]\label{factors through}
If a hyperbolic knot complement \(M = \mathbb{S}^3 - K\) covers an orbifold \(\mathcal{O}\) with a \((2,3,6)\)-cusp, then the cover \(\pi : M \to \mathcal{O}\) factors through an orbifold \(\mathcal{O}_0\) with a \((3,3,3)\)-cusp and a two-fold cover to \(\mathcal{O}\).
\end{lemma}

This implies that a \((2,3,6)\)-cusped orbifold covered by a knot complement is also \emph{double}-covered by a $(3,3,3)$-cusped orbifold. A simple combinatorial condition tests for this.

\begin{lemma}\label{dub cuv}
Suppose \( O = \mathbb{H}^3 / \Pi \) is an orientable hyperbolic 3-orbifold with a single $(2,3,6)$-cusp, and $O_0 = \mathbb{H}^3/\Pi_0$ is a  $(3,3,3)$-cusped double cover of $O$, corresponding to an index-two subgroup $\Pi_0<\Pi$. Then the isotropy graph of $O$ has a cycle $C$ with the following properties: \begin{enumerate}
    \item $C$ contains the edges incident to the cusp having labels $2$ and $6$;
    \item Each vertex of $C$ belongs to exactly two edges of $C$; and
    \item Each edge of $C$ has an even label.\end{enumerate}
\end{lemma}

\begin{proof} Because $\Pi_0$ has index two in $\Pi$, it is normal. Let $\phi\co \Pi \to \Pi/\Pi_0 \cong \mathbb{Z}/2\mathbb{Z} = \{\pm 1\}$ be the quotient map. A peripheral subgroup $\Lambda$ of $\Pi$ is a Euclidean $(2,3,6)$-rotation group with presentation $\langle a,b\,|\,a^2 = b^3 = (ab)^6\rangle$, where $a$, $b$, and $ab$ correspond to the edges of the isotropy graph that are incident to the cusp and labeled with their respective orders. Since the corresponding peripheral subgroup $\Lambda_0$ of $\Pi_0$ is a $(3,3,3)$-rotation group we must have $\phi(a) = \phi(ab) = -1$, hence $\phi(b) = 1$.

Let $C_0$ be the set of edges of the isotropy graph with the property that the corresponding fundamental group elements map to $-1$. The union of the edges of $C_0$ divides into components; let $C\subseteq C_0$ consist of the edges in the component containing those incident to the cusp. Note that every edge in $C_0$ must have an even label, since an edge's label is the order of elements of the corresponding conjugacy class in $\Pi$ and $\mathbb{Z}/2\mathbb{Z}$ has order $2$. We thus have properties (1) and (3) by construction.

To see property (2), recall that the subgroup of $\Pi$ corresponding to a vertex of the isotropy graph is a spherical rotation group with presentation $\langle a,b\,|\,a^p = b^q = (ab)^r\rangle$, where $p$, $q$ and $r$ are the labels of the edges incident to the vertex. We now note that if at least one of $a$, $b$, and $ab$ maps to $-1$ then exactly two of the three do.
\end{proof}

Testing \Cref{dub cuv}'s criterion on the orbifolds of \Cref{table236} yields the following:

\begin{prop}\label{NO2}
    For each $i\in\{3,4,14,15,16,17,20,21,34,35,38,39\}$, $\doh^{236}_i$ does not have a $(3,3,3)$-cusped double cover; hence by \Cref{factors through} it is not covered by a knot complement.
\end{prop}

Even among the $(2,3,6)$-cusped orbifolds that do satisfy cusp-killing and are covered by a $(3,3,3)$-cusped orbifold, combining the two criteria eliminates a few more.

\begin{prop}\label{CK2}
    For each $k\in\{12,26,30\}$, $\doh_{k}^{236}$ has a unique double cover by a $(3,3,3)$-cusped orbifold $\widetilde{P}^{236}_k$, and $\widetilde{P}_k^{236}$ maps non-trivially under cusp killing. Hence by Lemmas \ref{factors through} and \ref{BBCW-H} these are not covered by knot complements.
\end{prop}

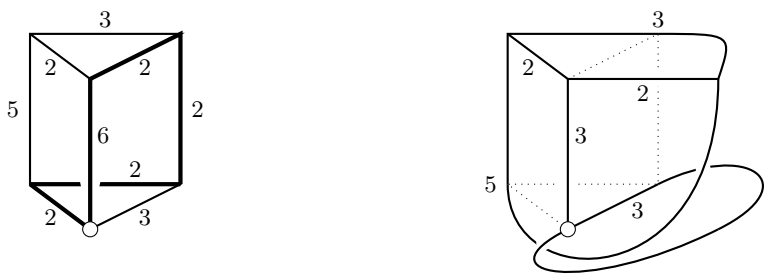
\begin{figure}[ht]
\begin{tikzpicture}

\begin{scope}[xshift=-2.5in]
    \draw [ultra thick] (0,-2) -- (-2,-2);
    \fill [color=white] (-1.2,-2.6) -- (-1.35,-1.9) -- (-1.05,-1.9);
    \draw [ultra thick] (0,-2) -- (0,0) -- (-1.2,-0.6) -- (-1.2,-2.6) -- (-2,-2);
    \draw [thick] (0,0) -- (0,-2) -- (-1.2,-2.6) -- (-1.2,-0.6) -- (0,0) -- (-2,0) -- (-2,-2) -- (-1.2,-2.6);
    \draw [thick] (-1.2,-0.6) -- (-2,0);
    \fill [color=white] (-1.2,-2.6) circle [radius=0.1];
    \draw (-1.2,-2.6) circle [radius=0.1];

    \node [right] at (0,-1) {\scriptsize $2$}; 
    \node [below right] at (-0.7,-2.2) {\scriptsize $3$}; 
    \node [below left] at (-1.5,-2.2) {\scriptsize $2$}; 
    \node [above] at (-0.6,-2.05) {\scriptsize $2$}; 
    \node [right] at (-1.25,-1.35) {\scriptsize $6$}; 
    \node [left] at (-2,-1) {\scriptsize $5$}; 
    \node [below right] at (-0.7,-0.2) {\scriptsize $2$}; 
    \node [below left] at (-1.5,-0.2) {\scriptsize $2$}; 
    \node [above] at (-1,-0.05) {\scriptsize $3$}; 
\end{scope}

\begin{scope}
    \draw [dotted] (-1.2,-0.6) -- (0,0) -- (0,-2) -- (-2,-2) -- (-1.2,-2.6);
    \fill [color=white] (-1.2,-2.6) -- (-1.35,-1.9) -- (-1.05,-1.9);
    \draw [thick] (-1.2,-0.6) -- (-1.2,-2.6) -- (0,-2);
    \draw [thick] (0,0) -- (-2,0) -- (-2,-2);
    \draw [thick] (-1.2,-0.6) -- (-2,0);

    \node [below right] at (-0.5,-2.1) {\scriptsize $3$}; 
    \node [right] at (-1.25,-1.35) {\scriptsize $3$}; 
    \node [left] at (-2,-2) {\scriptsize $5$}; 
    \node [below left] at (-1.5,-0.2) {\scriptsize $2$}; 
    \node [above] at (0,-0.05) {\scriptsize $3$}; 
    \node [below] at (-0.2,-0.55) {\scriptsize $2$};

    \fill [color=white] (0,-0.6) circle [radius=0.1];
    \draw [thick] (-1.2,-0.6) -- (0.8,-0.6);
    \draw [thick] (0,0) .. controls (1,0) .. (0.8,-0.6);
    \draw [thick] (0,-2) .. controls (1.2,-1.4) and (2,-2) .. (0.8,-2.6);
    \fill [color=white] (0.6,-1.8) circle [radius=0.1];
    \draw [thick] (-2,-2) .. controls (-2,-3.5) and (0.8,-3.5) .. (0.8,-0.6);
    \fill [color=white] (-1.52,-2.78) circle [radius=0.1];
    \draw [thick] (-1.2,-2.6) .. controls (-2.4,-3.2) and (-1,-3.5) .. (0.8,-2.6);

    \fill [color=white] (-1.2,-2.6) circle [radius=0.1];
    \draw (-1.2,-2.6) circle [radius=0.1];
\end{scope}

\end{tikzpicture}
\caption{$\doh^{236}_{12}$ (left) and its $(3,3,3)$-cusped double cover $\widetilde{P}^{236}_k$ (right).}
\label{fig: dc ck}
\end{figure}

\begin{proof} We give the proof for $\doh^{236}_{12}$; the rest are similar. The orbifold's isotropy graph is pictured on the left in \Cref{fig: dc ck}. It has a unique cycle $C$ that satisfies the criteria of \Cref{dub cuv}, pictured bold in the Figure. This determines a homomorphism $\phi\co\dprgr^{236}_{12}\twoheadrightarrow\mathbb{Z}/2\mathbb{Z}$, which maps the conjugacy class corresponding to an edge non-trivially if and only if that edge belongs to $C$. (As is visible in the presentation (\ref{able}), $\dprgr^{236}_{12}$ is generated by elements corresponding to edges of the isotropy graph.)

As argued in the proof of \Cref{NO2}, if a homomorphism $\dprgr^{236}_{12}\twoheadrightarrow\mathbb{Z}/2\mathbb{Z}$ maps the conjugacy classes corresponding to edges of $C$ nontrivially, then it must trivially map the conjugacy class corresponding to each edge not belonging to $C$ but having a vertex on $C$. Since every edge of the isotropy graph has a vertex on $C$, it follows that there is a unique homomorphism $\dprgr^{236}_{12}\twoheadrightarrow\mathbb{Z}/2\mathbb{Z}$ that maps all edges of $C$ nontrivially. Therefore by \Cref{NO2}, $\doh^{236}_{12}$ has a unique double cover by a $(3,3,3)$-cusped orbifold  $\widetilde{P}^{236}_{12}$, the one corresponding to $\ker\phi$.

The orbifold $\widetilde{P}^{236}_{12}$ is pictured on the right in \Cref{fig: dc ck}. Its underlying space is the double branched cover of the underlying space of $\doh^{236}_{12}$---which is $\mathbb{S}^3$ with one puncture---over the unknotted cycle $C$; hence it is again $\mathbb{S}^3$ with one puncture. The edge of $C$ labeled ``$6$'' yields an edge of the isotropy graph of $\widetilde{P}^{236}_k$ labeled ``$3$''; the edges of $C$ labeled ``$2$'' disappear since the corresponding conjugacy classes of order-two elements do not intersect $\ker\phi$.

Each edge $e$ of the isotropy graph of $\doh^{236}_{12}$ outside $C$ yields two edges of the isotropy graph of $\widetilde{P}^{236}_k$, with the same label as $e$ and sharing both their endpoints, which are exchanged by the covering transformation. At a vertex $v$ that $e$ shares with two edges of $C$ labeled ``$2$'', these edges meet at a $180$-degree angle (reflecting that the corresponding vertex group of $\dprgr^{236}_{12}$ has cyclic intersection with $\ker\phi$), so their union is a geodesic segment in the hyperbolic metric on $\widetilde{P}^{236}_k$, and the preimage of $v$ is not a vertex of the isotropy graph of $\widetilde{P}^{236}_{12}$. The edge labeled ``$5$'', and each one labeled ``$3$'', has a unique vertex with this property.

Inspecting the right side of \Cref{fig: dc ck}, we find that the cusp-killing algorithm described below \Cref{BBCW-H} terminates leaving a $\theta$-graph with edges labeled ``$2$'', ``$3$'', and ``$5$''. Therefore $\widetilde{P}^{236}_k$ maps non-trivially under cusp-killing.
\end{proof}

\begin{remark} Unlike Propositions \ref{CK} and \ref{NO2}, \Cref{CK2} does not provide a comprehensive list of orbifolds failing its condition---it only addresses cases that survived the previous two tests.\end{remark}

\subsection{Pre-filtering 2: arithmetic}\label{subsec: prefilter arith} We can further prune the list of candidate prism orbifolds to be covered by a knot complement using the following obstruction from arithmetic, recorded in work of the third author. Below, a \emph{ring of $S$-integers} is a subring of a number field, containing the field's ring of integers, in which a finite (possibly empty) collection of primes is inverted. For a reference on $S$-integers and units in the ring of $S$-integers see \cite[3.2.2]{narkiewicz1974elementary}.

\begin{lemma}[Hoffman, \cite{Neil_SmallKnots}]\label{unit}
Suppose that a complete hyperbolic $3$-manifold $M = \Hthree/\Gamma$ is homeomorphic to a knot complement $\knotcomp$, where $\Gamma \subset \mathrm{PSL}(2,O_s)$ for a ring of $S$-integers $O_s$. If $\mu \in \Gamma$ fixes $\infty$ and represents a meridian of $K$, then $\mu = \begin{bmatrix} 1 & x \\ 0 & 1 \end{bmatrix}$ for a unit $x$ in $O_s$. 
\end{lemma}

It is a standard consequence of Mostow-Prasad rigidity that every complete, finite-volume hyperbolic $3$-manifold is of the form $\Hthree/\Gamma$, where $\Gamma \subset \mathrm{PSL}(2,O_s)$ for some ring of $S$-integers $O_s$. We apply this in the context of matrix groups explicitly described in \cite{LakelandRoth}, isomorphic to the $\dprgr(\be)$ presented in (\ref{able}), generated by rotations pairing faces of the doubled prisms $\dpr(\be)$ embedded as discussed here in \Cref{subsec: embedding333}. For these explicit matrix groups, which we still refer to as ``$\dprgr(\be)$'', we have the following consequence of the famous ``Six Theorem'' \cite{AgolSix}, \cite{Lackenby_sixes}.

\begin{lemma}\label{meridians}
Suppose that a complete hyperbolic $3$-manifold $M = \Hthree/\Gamma$ is homeomorphic to a knot complement $\knotcomp$ covering a prism orbifold $O(\be)$, so that up to conjugation in $\mathrm{PSL}_2(\mathbb{C})$ we may assume that $\Gamma<\dprgr(\be)$. Let $O_s$ be the smallest ring of s-integers containing the entries of $\dprgr(\be)$. Then there exists $\mu \in \Gamma$ of the form $\begin{bmatrix} 1 & x \\ 0 & 1 \end{bmatrix}$, for a unit $x$ in $O_s$ with complex modulus strictly less than $6\cdot\max\{1,r\}$, for $r$ as in \Cref{maximal general}. 
\end{lemma}

\begin{proof} By \Cref{maximal general}, the horoball centered at $\infty$ of height $\max\{1,r\}$ projects to a maximal embedded cusp cross-section in $\doh(\be)$; therefore also in $M$. An element $\mu$ of $\Gamma$ stabilizing this horoball and representing a meridian of $K$ has the form $\begin{bmatrix} 1 & x \\ 0 & 1 \end{bmatrix}$ for some unit $x$ in $O_s$, by \Cref{unit}. Its translation length along the horoball is $x/\max\{1,r\}$, so there is a Euclidean geodesic of this length on the maximal embedded cusp cross-section of $M$ that is homotopic to a meridian of $K$. By the Six Theorem, this length is less than $6$.\end{proof}

One can use \Cref{meridians} to eliminate prism orbifolds from contention as follows: for such an orbifold $\doh(\be)$, assemble a list of parabolic elements of the peripheral subgroup of $\dprgr(\be)$ fixing $\infty$ having the form $\begin{bmatrix} 1 & x \\ 0 & 1 \end{bmatrix}$ for $|x|<6$, with one representative from every conjugacy class of such elements. If none of the upper-right entries $x$ is a unit in $O_s$ then $\doh(\be)$ is not covered by a knot complement. Implementing this in the prismRep class of \path{anc/Analysis_of_covers/code/MatrixReps.sage}, we obtain:

\begin{prop}
    For $i \in \{ 2,12,26,33,37\}$ $\doh^{236}_i$  is not covered by a knot complement and for $j \in \{ 21,22\} $, $\doh^{333}_j$ is not covered by a knot complement.
\end{prop}

\begin{remark}
$\doh^{236}_{26}$ has translations with off-diagonal entries contained in a prime ideal that lies over $5$ but the primes inverted in the representation lie over $2$ so this case is eliminated. 
\end{remark}

\begin{table}
\centering

\resizebox{\textwidth}{!}{%
\begin{tabular}{lrrrrrlrrrrrr}
\toprule
\char"0023  & CK & DC & UT & IR & MCD \\
\midrule
$\doh^{236}_{ 1 }$ &  0 & 1 & 0 & I & 24\\
$\doh^{236}_{ 2 }$ &  1 & 1 & 0 & I & 24\\
$\doh^{236}_{ 3 }$ &  0 & 0 & 0 & I & 60\\
$\doh^{236}_{ 4 }$ &  1 & 0 & 0 & I & 60\\
$\doh^{236}_{ 5 ,n}$ & 0 & 1 & - & - &  lcm(12,n) \\
$\doh^{236}_{ 6 ,2k}$ & 0 & 1 & - & - & lcm(12,n)\\
 \rowcolor{lightgray}
$\doh^{236}_{ 6 ,2k+1}$ & 1 & 1 & - & - &  lcm(12,n) \\
$\doh^{236}_{ 7 ,n}$ & 0 & 1 & - & - &  lcm(24,n) \\
$\doh^{236}_{ 8 ,2k}$ & 0 & 1 & - & - &  lcm(60,n)\\
 \rowcolor{lightgray}
$\doh^{236}_{ 8 ,2k+1}$ & 1 & 1 & - & - &  lcm(60,n) \\
$\doh^{236}_{ 9 }$ &  0 & 1 & 0 & I & 24\\
$\doh^{236}_{ 10 }$ &  0 & 1 & 0 & I & 24\\
$\doh^{236}_{ 11 }$ &  0 & 1 & 0 & I & 60\\
$\doh^{236}_{ 12 }$ &  1 & 1 & 0 & I & 60\\
$\doh^{236}_{ 13 ,n}$ & 0 & 1 & - & - &  lcm(12,n) \\
$\doh^{236}_{ 14 }$ &  0 & 0 & 1 & \{2\} & 12\\
$\doh^{236}_{ 15 }$ &  1 & 0 & 1 & \{2\} & 12\\
$\doh^{236}_{ 16 }$ &  0 & 0 & 1 & \{2\} & 24\\
$\doh^{236}_{ 17 }$ &  1 & 0 & 1 & \{2\} & 60\\
$\doh^{236}_{ 18 }$ &  0 & 1 & 1 & \{2\} & 24\\
 \rowcolor{lightgray}
$\doh^{236}_{ 19 }$ &  1 & 1 & 1 & \{2\} & 24\\
$\doh^{236}_{ 20 }$ &  0 & 0 & 0 & I & 60\\
\bottomrule
\end{tabular}
\begin{tabular}{lrrrrlrrrrr}
\toprule
\char"0023  & CK & DC & UT & IR & MCD \\
\midrule
$\doh^{236}_{ 21 }$ &  1 & 0 & 0 & I & 60\\
$\doh^{236}_{ 22 ,2k}$ & 0 & 1 & - & - & lcm(12,n) \\
$\doh^{236}_{ 22 ,2k+1}$ & 0 & 0 & - & - &  lcm(12,n) \\
$\doh^{236}_{ 23 }$ &  0 & 1 & 0 & I & 24\\
$\doh^{236}_{ 24 }$ &  0 & 1 & 0 & I & 24\\
$\doh^{236}_{ 25 }$ &  0 & 1 & 0 & \{2\} & 24\\
$\doh^{236}_{ 26 }$ &  1 & 1 & 0 & \{2\} & 24\\
$\doh^{236}_{ 27 }$ &  0 & 1 & 0 & \{3\} & 24\\
$\doh^{236}_{ 28 }$ &  0 & 1 & 0 & \{3\} & 24\\
$\doh^{236}_{ 29 }$ &  0 & 1 & 1 & I & 120\\
$\doh^{236}_{ 30 }$ &  1 & 1 & 1 & I & 120\\
$\doh^{236}_{ 31 ,n}$ & 0 & 1 & - & - &  lcm(24,n) \\
$\doh^{236}_{ 32 }$ &  0 & 1 & 0 & I & 60\\
$\doh^{236}_{ 33 }$ &  1 & 1 & 0 & I & 60\\
$\doh^{236}_{ 34 }$ &  0 & 0 & 0 & I & 60\\
$\doh^{236}_{ 35 }$ &  1 & 0 & 0 & I & 60\\
$\doh^{236}_{ 36 }$ &  0 & 1 & 0 & I & 120\\
$\doh^{236}_{ 37 }$ &  1 & 1 & 0 & I & 120\\
$\doh^{236}_{ 38 }$ &  0 & 0 & 1 & I & 60\\
$\doh^{236}_{ 39 }$ &  1 & 0 & 1 & I & 60\\
$\doh^{236}_{ 40 ,2k}$ & 0 & 1 & - & - &  lcm(60,n)\\
$\doh^{236}_{ 40 ,2k+1}$ & 0 & 0 & - & - &  lcm(60,n)\\
\bottomrule
\end{tabular}
}
\caption{A table recording the possible obstructions to prism orbifolds being covered by knot complements. In each of the columns ``CK'', ``DC'', and ``UT'', a ``0'' indicates that the orbifold in question fails the relevant test, while a ``1'' indicates that it passes. ``CK'' refers to the cusp killing check of \Cref{BBCW-H}; ``DC'' refers to the double cover check of \Cref{factors through}; and ``UT'' refers to the unit translation test of \Cref{meridians}. 
 ``IR'' stands for integral representation: here, ``I'' indicates the representation is integral, and otherwise $\{p\}$ indicates that the rational prime $p$ is inverted in the representation. (This is not an obstruction to being covered by a knot complement, however it is used in the computation of unit translations.) Prism orbifolds that could be covered by knot complements are highlighted in gray. (For $\doh^{236}_{30}$, note \Cref{CK2}.)
  The column ``MCD'' records the minimum degree possible for a cover by a manifold. Here, in the cases depending on a parameter $k$, ``$n$'' equals $2k$ or $2k+1$ as appropriate, and $k$ must be at least $3$.
}
\label{table236Obstructions}
\end{table}

\begin{table}[ht]
\centering
\resizebox{\textwidth}{!}{%
\begin{tabular}{lrrrrrlrrrrrr}
\toprule
\char"0023  & CK & Unit trans. & integral rep & MCD \\
\midrule
$\doh^{333}_{ 1 }$ &  0 & 1 & \{2\} & 24\\
 \rowcolor{lightgray}
$\doh^{333}_{ 2 }$ &  1 & 1 & \{2\} & 24\\
 \rowcolor{lightgray}

$\doh^{333}_{ 3 }$ &  1 & 1 & \{2\} & 24\\
 \rowcolor{lightgray}

$\doh^{333}_{ 4 }$ &  1 & 1 & I & 24\\
$\doh^{333}_{ 5 }$ &  0 & 1 & \{2\} & 24\\
 \rowcolor{lightgray}
$\doh^{333}_{ 6 }$ &  1 & 1 & \{2\} & 120\\
$\doh^{333}_{ 7 }$ &  0 & 1 & \{5\} & 60\\
 \rowcolor{lightgray}
$\doh^{333}_{ 8 }$ &  1 & 1 & \{5\} & 60\\
 \rowcolor{lightgray}
$\doh^{333}_{ 9 }$ &  1 & 1 & \{5\} & 60\\
 \rowcolor{lightgray}
$\doh^{333}_{ 10 }$ &  1 & 1 & \{5\} & 60\\
$\doh^{333}_{ 11 }$ &  0 & 1 & \{5\} & 120\\
\bottomrule
\end{tabular}
\begin{tabular}{lrrrrlrrrrr}
\toprule
\char"0023  & CK & Unit trans. & integral rep & MCD\\
\midrule
 \rowcolor{lightgray}
$\doh^{333}_{ 12 }$ &  1 & 1 & \{5\} & 60\\
$\doh^{333}_{ 13 }$ &  0 & 0 & I & 24\\
$\doh^{333}_{ 14 }$ &  0 & 0 & I & 24\\
$\doh^{333}_{ 15 }$ &  0 & 0 & \{3\} & 24\\
$\doh^{333}_{ 16 }$ &  0 & 1 & I & 120\\
 \rowcolor{lightgray}
$\doh^{333}_{ 17 }$ &  1 & 1 & I & 120\\
 \rowcolor{lightgray}
$\doh^{333}_{ 18 }$ &  1 & 1 & I & 120\\
 \rowcolor{lightgray}
$\doh^{333}_{ 19 }$ &  1 & 1 & I & 120\\
$\doh^{333}_{ 20 }$ &  0 & 0 & I & 60\\
$\doh^{333}_{ 21 }$ &  1 & 0 & I & 60\\
$\doh^{333}_{ 22 }$ &  1 & 0 & \{3\} & 60\\
\bottomrule
\end{tabular}
}
\caption{Prism Orbifolds with (3,3,3) cusp. Here the double cover is not relevant but the other columns are labeled as in the previous table and record the same information. Prism orbifolds that could be covered by knot complements are highlighted in gray. }
\label{table333obsturctions}

\end{table}

\subsection{Enumeration} Having significantly cut down the number of prism orbifolds that might be covered by knot complements, we turn to the problem of identifying which of these orbifolds' covers might be knot complements. We begin with a general observation about the degrees of manifold covers of orbifolds.

\begin{lemma} \label{vertex groups}
For a finite-degree cover $p\co M\to \doh$ from a manifold $M$ to an orientable hyperbolic $3$-orbifold $\doh=\mathbb{H}^3/\dprgr$, and any $\tilde{x}\in\mathbb{H}^3$, the order of the stabilizer $G_{\tilde{x}}$ of $\tilde{x}$ in $\dprgr$ divides the degree of $p$.\end{lemma}

\begin{proof} Let $M = \mathbb{H}^3 / \Gamma$ for a subgroup $\Gamma$ of $\dprgr$. Taking $\tilde{x}\in\mathbb{H}^3$ and $G_{\tilde{x}}<\dprgr$ as above, for any non-identity element $g\in G_{\tilde{x}}$ and any left coset $ \gamma_i \Gamma$ of $\Gamma$ in $\dprgr$, if $g\gamma_i \Gamma = \gamma_i \Gamma$ then $\gamma_i^{-1} g \gamma_i$ is in $\Gamma$. But this can’t happen because $\gamma_i^{-1} g \gamma_i$ is also of finite order, and $\Gamma$ is torsion-free since $M$ is a manifold. It follows that each orbit of the $G_{\tilde{x}}$-action by left multiplication on the left cosets of $\Gamma$ has size $|G_{\tilde{x}}|$. Since the set of all left cosets is a disjoint union of $G_{\tilde{x}}$-orbits, $|G_{\tilde{x}}|$ divides $[\dprgr:\Gamma]$.
\end{proof}

We apply \Cref{vertex groups} to the \emph{vertex groups} of the prism orbifolds, meaning the stabilizers of vertices of geometric realizations of prisms in $\mathbb{H}^3$ in their associated rotation groups.

\begin{corollary}
    For a finite-degree manifold cover $p\co M\to\doh$, where $\doh$ is an orientable prism orbifold, and a non-ideal vertex $v$ of the associated prism $P$ at the intersection of edges labeled $a_i$, $a_j$ and $a_k$, the order $|\widetilde{G}_v| = 2/n_v$ of the vertex group $\widetilde{G}_v$ divides the degree of $p$, where $n_v = \frac{1}{a_i} + \frac{1}{a_j} + \frac{1}{a_k} - 1$.
\end{corollary}

\begin{proof}
    Referring again by $P$ to a geometric realization of $P$ in $\mathbb{H}^3$ given by Andreev's theorem, the vertex $v$ is a triple intersection of faces that pairwise meet at dihedral angles of $\pi/a_i$, $\pi/a_j$, and $\pi/a_k$. Therefore the set of unit tangent vectors in $T_v\mathbb{H}^3$ that point into $P$ is a spherical triangle with the same set of dihedral angles, hence with area $\pi n_v$, for $n_v$ as above, by the Gauss-Bonnet formula. This triangle is a fundamental domain for the action of $G_v$ on $T_v\mathbb{H}^3$, where $G_v$ is the stabilizer of $v$ in the group generated by reflections in the faces of $P$. So since the unit tangent sphere has area $4\pi$, $|G_v| = 4/n_v$. The stabilizer $\widetilde{G}_v$ of $v$ in the orientable prism group has index $2$ in $G_v$, hence order $2/n_v$. The result now follows from \Cref{vertex groups}.
\end{proof}

\begin{prop}
The prism orbifolds satisfying all tests of Sections \ref{subsec: prefilter cusp} and \ref{subsec: prefilter arith} are highlighted in gray in Tables \ref{table236Obstructions} and \ref{table333obsturctions}, together with a lower bound for the minimum degree of their manifold covers given by the least common multiple (lcm) of the vertex group orders. Based on how torsion groups lift in double covers, the lcms for the two infinite families listed in Table \ref{table236Obstructions} should be doubled to give the best known lower bound for the covering degree by a knot complement.
\end{prop}

\noindent\textbf{Enumeration Details.} For the $(3,3,3)$-cusped orbifolds having the minimal lcm ($24$) in \Cref{table333obsturctions}---$O^{333}_i$ for $i\in\{2,3,4\}$---we used the \texttt{low-index} Python module \cite{lowindex} to enumerate all subgroups of $\dprgr^{333}_i$ up to index $24$. We used version 1.1.0 of the module, called as follows:\smallskip\\
    \indent \texttt{>> from low\_index import *}\\
    \indent \texttt{>> t = SimsTree(7, idx, ['aaa', 'ccc', 'eee', 'gg', 'ab', 'cd', 'ef',}\\
    \mbox{}~\hfill \texttt{'dada',' fafafa', 'fcfcfcfc', 'gcgcgc', 'gege'])}\\
    \indent \texttt{>> sgrps = t.list()}\smallskip

Above, the three inputs to \emph{SimsTree} are, in order, the number of generators for the input group; the maximum index to enumerate to; and a list containing the group's relations. By default, the script names generators by letters of the alphabet, in order. The sample relations listed above correspond to the presentation for $\dprgr^{333}_2$ produced by following (\ref{able}). Note that the fifth, sixth, and seventh relations assign $b$, $d$, and $f$ as the respective inverses of $a$, $c$, and $e$. (The final generator $g$ is its own inverse.)

The output ``\texttt{sgrps}'' is a list of lists, each one corresponding to a right-permutation representation of the given group defined by a subgroup of index at most the given maximum. With the group represented as above, each such representation consists of seven lists, each recording the representation of a generator analogous to \Cref{table_pr}. We strip out the second, fourth, and sixth, as they represent inverses of the first, third, and fifth, respectively.

Per its documentation, \texttt{low\_index} produces one right-permutation representation per conjugacy class of subgroups up to the given index bound; the subgroups being recovered as stabilizers of letters under the representation (compare \Cref{itsapermrep}). There are respectively $32\,245$, $29\,432$, and $306\,552$ such representations for $\dprgr^{333}_2$, $\dprgr^{333}_3$, and $\dprgr^{333}_4$. \footnote{The complete lists of permutation representations can be downloaded following instructions in \path{/anc/Finding_low_index_subgroups/infoOnDataFiles.txt}}

\begin{remark} It has thus far proven computationally infeasible to enumerate subgroups to indices greater than $24$, and hence to address other candidate orbifolds. Indeed, we were not even able to enumerate all degree-$24$ covers of $\doh^{236}_{19}$. However, by enumerating all degree-$12$ covers of its unique $(3,3,3)$-cusped double cover, and appealing to \Cref{factors through}, we were able to find all degree-$24$ covers of $\doh^{236}_{19}$ that could be knot complements (but see \Cref{no 236}). 
\end{remark}

\subsection{Filtering.}\label{subsec: filtering} Each right-permutation representation produced in the enumeration step corresponds to a cover of some $\doh^{333}_i$ that is a candidate knot complement in a lens space. We identified the $M_{i,j}$ from \Cref{main_tech_thm} by first testing a sequence of necessary conditions on these for this to hold, using custom-coded scripts included with the ancillary files.\begin{enumerate}
    \item We first test whether the conjugacy class of subgroups corresponds to a \emph{manifold} cover using \texttt{is\_it\_tor\_free}, which implements the condition of \Cref{torsion order}. This results in $142$, $142$, and $148$ manifold covers of $\doh^{333}_2$, $\doh^{333}_3$, and $\doh^{333}_4$, respectively. The resulting permutation representations are stored in the ancillary files at \path{anc/Finding_low_index_subgroups/data/24manifold_covers_of_O333-*.txt}.
    \item We next count cusps of the covers remaining after step (1), using the method \texttt{cusp\_seqs}, which implements the condition of \Cref{cusp count}. We find that there are $46$, $46$, and $51$ \emph{single-cusped} manifold covers of $\doh^{333}_2$, $\doh^{333}_3$, and $\doh^{333}_4$, respectively. The resulting permutation representations are stored in the ancillary files at \path{anc/Finding_low_index_subgroups/data/single_cusped_O333-*.txt}.
    \item Finally, we compute the first homology of the single-cusped manifold covers remaining after step (2). We find that there are $20$, $22$, and $12$ single-cusped manifold covers of $\doh^{333}_2$, $\doh^{333}_3$, and $\doh^{333}_4$, respectively, having \emph{first homology isomorphic to $\mathbb{Z}$}. The homology data for all single-cusped index-$24$ manifold covers of these orbifolds is recorded in the ancillary files at \path{anc/Finding_low_index_subgroups/data/24_manifold_covers_of_O333-*_single_cusped_with_hom_info.txt}.
\end{enumerate}

The filtering procedure described in steps (1)--(3) is implemented using a sequence of Python modules included in the ancillary files in \path{anc/Finding_low_index_subgroups/code} (see the \texttt{README} file therein). In particular, these scripts process the permutation representations of subgroups, identify the relevant manifold covers, and compute the structure of the spine for $M=\mathbb{H}^3/H$ described in \Cref{spine}, using the \texttt{NetworkX} Python library \cite{networkx}. The fundamental group of $M$ is then computed as described in the proof of \Cref{homologyZ}: identify a maximal tree in the spine's one skeleton, after which each remaining edge determines a generator and each two-cell gives a relation. The presentation thus computed is fed to GAP \cite{GAP4} to abelianize, hence producing $H_1(M)$.

Each one-cusped manifold cover of an $\doh^{333}_i$ ($i\in\{2,3,4\}$) identified in Step (3) as having first homology $\mathbb{Z}$ is finally triangulated as described in \Cref{triangulation}, and tested for lens space fillings as described in \Cref{lem:lifts}. This yielded the $M_{i,j}$, for $i\in\{2,3\}$ and $j\in\{1,2\}$, of \Cref{first manifold covers}, as well as their ``OR twins'' $M_{i,j}'$ prescribed by the permutation representations of \Cref{table_pr prime}.

\begin{remark}\label{no 236}
After building the $(3,3,3)$-cusped 2-fold cover of 
$\doh^{236}_{19}$ and then building the 12-fold covers of that orbifold, we found that $\doh^{236}_{19}$ does not have a 24-fold cover which is an integral homology knot complement. 
Hence, if $\doh^{236}_{19}$ is covered by a knot complement, then the smallest manifold quotient of that knot complement is a degree $d$ cover of $\doh^{236}_{19}$ with $d\geq 48$. 
\end{remark}

\section{Further questions}\label{sec: ques}

While we have exhibited new hyperbolic knot complements with hidden symmetries, it remains unclear how many there are, or how they are distributed among the collection of all hyperbolic knot complements. For instance, the following is open.

\begin{question}
    Are there infinitely many hyperbolic knot complements in $\mathbb{S}^3$ with hidden symmetries?
\end{question}

Even if there are infinitely many, the following still-open conjecture postulates that they are sparsely distributed.

\begin{conjecture}[\cite{CDM}, Conjecture 0.1]
    For any $R > 0$, at most finitely hyperbolic knot complements in $\mathbb{S}^3$
have hidden symmetries and volume less than $R$.
\end{conjecture}

The search for knot complements with hidden symmetries is also related to the following conjecture of Reid--Walsh.

\begin{conjecture}[\cite{ReidWalsh}, Conjecture 5.2(i)]
    A hyperbolic knot complement in $\mathbb{S}^3$ has at most three knot complements in its commensurability class.
\end{conjecture}

Here we recall from above \Cref{commensurator} that the \emph{commensurability class} of a manifold $M$ refers to the collection of manifolds commensurable to $M$, where two spaces are \emph{commensurable} if they have a common cover of finite degree. For instance, it follows from their construction as covers of $O^{333}_2$ that $\widetilde{M}_{2,1}$ and $\widetilde{M}_{2,2}$ from \Cref{main_tech_thm} are commensurable; as are $\widetilde{M}_{3,1}$ and $\widetilde{M}_{3,2}$. But the $\widetilde{M}_{2,j}$ are not commensurable to the $\widetilde{M}_{3,j'}$, by \Cref{commensurator}.

Boileau--Boyer--Cebanu--Walsh proved in \cite{BBCW} that the Reid--Walsh conjecture holds for hyperbolic knot complements \emph{without} hidden symmetries. The case of knot complements with hidden symmetries remains open. In particular, we do not even know answers to:

\begin{question}
    Does there exist a commensurability class containing infinitely many hyperbolic knot complements in $\mathbb{S}^3$?
\end{question}

\begin{question}\label{ques: comm Mij} How many knot complements in $\mathbb{S}^3$ belong to the commensurability class of $\widetilde{M}_{2,1}$? Of $\widetilde{M}_{3,1}$? \end{question}

The corresponding question to \ref{ques: comm Mij} has been answered for the previously-known hyperbolic knot complements with hidden symmetries: recall that the figure-eight knot complement is the only knot complement in its commensurability class, by \cite{reidarith}; and the known dodecahedral knot complements are the only two in their commensurability class, by \cite{NeilExperiment}.

Turning to the subject of totally geodesic surfaces in knot complements, we recall that Adams-Schoenfeld constructed the first examples of properly embedded totally geodesic surfaces in hyperbolic knot complements, in \cite{adamsurface}. Their examples are \emph{Seifert surfaces}---non-compact and single-cusped---and they remark that ``one could ask whether there are embedded totally geodesic surfaces in knot complements other than Seifert surfaces" (\cite[\S 1]{adamsurface}, final sentence). Our examples show that there are, but the following remains open:

\begin{question}
    Does there exist a hyperbolic knot complement in $\mathbb{S}^3$ containing a properly embedded, \emph{non-compact} totally geodesic surface that is not a Seifert surface?
\end{question}

There is also a rich literature on \emph{immersed} totally geodesic surfaces in hyperbolic knot complements, including algebraic obstructions to their existence---\cite[Th.~5.3.1, Cor.~5.3.2]{book} for closed surfaces, and \cite{calegarirealplaces} for all surfaces---and counting results \cite{twistsurfaces}, \cite{LePalmer}. These are lent additional import by the recent, spectacular characterization of arithmeticity in terms of totally geodesic surfaces \cite{arithmeticsurface}. As mentioned above \Cref{cor: immersed geod}, the list of known knot complements containing immersed \emph{closed} totally geodesic surfaces is short. Comparing that result with \Cref{cor:HidSymKnots} sheds light on the following perhaps surprising questions that Genevieve Walsh has asked:

\begin{question}
    Does there exist a hyperbolic knot complement in $\mathbb{S}^3$ without hidden symmetries that contains a \emph{closed}, immersed totally geodesic surface?\end{question}
    
\begin{question} Conversely, does every hyperbolic knot complement with hidden symmetries contain a closed, immersed totally geodesic surface?
\end{question}

We close with a challenge: to overcome one aspect of the present paper which may not be entirely satisfying---the indirect way in which we know the knot complements $\widetilde{M}_{i,j}$. Specifically:\begin{enumerate}
   \item Draw a diagram of a knot $K_{i,j}$ such that $\widetilde{M}_{i,j}\cong \mathbb{S}^3-K_{i,j}$, for $i\in\{2,3\}$ and $j\in\{1,2\}$, in a way that makes visible the totally geodesic surface $S_{i,j}$ and/or the covering transformation of its projection to a knot in a lens space.
   \item Compute knot invariants of the $K_{i,j}$ that would typically depend on a diagram, for instance their Jones and/or HOMFLY-PT polynomials.
\end{enumerate}

Toward item (1), we recall that Adams proved that for a knot $K\subset\mathbb{S}^3$ the volume of $\mathbb{S}^3-K$ is at most $(c-5)v_8+4v_3$ \cite{Adams_spansfc}, improving an estimate of D.~Thurston, where $c$ is the crossing number of $K$ and $v_3\approx 1.015$ and $v_8\approx 3.664$ are respectively the volumes of an ideal tetrahedron and octahedron. This implies that $\widetilde{M}_{2,1}$ and $\widetilde{M}_{3,1}$ are the complements of at-least-$159$-crossing knots, and $\widetilde{M}_{2,2}$ and $\widetilde{M}_{3,2}$ of at-least-$270$-crossing knots, cf.~\Cref{knot vol}. (We know of no reason to believe that these bounds are near sharp.)

Should (1) prove too challenging, we note that the non-compact side $M_{2,1}^-$ of $S_{2,1}$ in $M_{2,1}$ is itself a knot in a handlebody, by \Cref{M_21 minus}, with volume ``only'' $\approx 32.293$, cf.~\Cref{volumes}. It would already be interesting to give a diagram for this knot in a standardly-embedded handlebody.

Although we were not able to produce knot diagrams for the $\widetilde{M}_{i.j}$, we were able to produce surgery diagrams of each $M_{i,j}$ using SnapPy's \texttt{exterior\_to\_link} function: two-component links in $\mathbb{S}^3$, one component of which is an unknot that can be filled to produce a lens space in which the complement of the other component is $M_{i,j}$. The links produced by SnapPy for $M_{2,1}$, $M_{2,2}$, $M_{3,1}$ and $M_{3,2}$ have 234, 1193, 213, and 449 crossings, respectively. Diagrams and PD code for these links are available in the ancillary files \path{anc/LinkFiles} of this post.


\end{document}